\documentclass{amsart}
\usepackage{graphicx}
\usepackage{bm,caption,amsbsy,amsmath,amsthm,
amssymb,mathtools,amsfonts,multirow,verbatim,tikz,tikz-cd,thmtools,thm-restate}

\usepackage[alphabetic]{amsrefs}
\usetikzlibrary{matrix,arrows,decorations.pathmorphing}
\usepackage[top=1.25in, bottom=1.25in, left=1.25in, right=1.25in,marginpar=1in]{geometry}
\usepackage{url}
\DefineSimpleKey{bib}{myurl}
\usepackage[hidelinks]{hyperref}

\usepackage{chngcntr}
\counterwithin{figure}{section}

\usepackage{enumitem}
\setlist[enumerate]{leftmargin=*}

\newtheorem{innercustomthm}{Theorem}
\newenvironment{customthm}[1]
  {\renewcommand\theinnercustomthm{#1}\innercustomthm}
  {\endinnercustomthm}

\newtheorem{innercustomcor}{Corollary}
\newenvironment{customcor}[1]
  {\renewcommand\theinnercustomcor{#1}\innercustomcor}
  {\endinnercustomcor}
  
  \newtheorem{innercustomprop}{Proposition}
\newenvironment{customprop}[1]
  {\renewcommand\theinnercustomprop{#1}\innercustomprop}
  {\endinnercustomprop}

\newtheorem{thm}{Theorem}[section]
\newtheorem{prop}[thm]{Proposition}
\newtheorem{cor}[thm]{Corollary}
\newtheorem{lem}[thm]{Lemma}

\theoremstyle{definition}
\newtheorem{define}[thm]{Definition}

\theoremstyle{remark}
\newtheorem{rem}[thm]{Remark}

\newcommand{\ve}[1]{\boldsymbol{\mathbf{#1}}}
\newcommand{\R}{\mathbb{R}}

\newcommand{\Z}{\mathbb{Z}}
\newcommand{\N}{\mathbb{N}}

\renewcommand{\d}{\partial}
\renewcommand{\subset}{\subseteq}

\renewcommand{\tilde}{\widetilde}
\renewcommand{\bar}{\overline}
\renewcommand{\hat}{\widehat}
\newcommand{\iso}{\cong}

\DeclareMathOperator{\codim}{{codim}}
\DeclareMathOperator{\cotr}{{cotr}}
\DeclareMathOperator{\Cone}{{Cone}}

\DeclareMathOperator{\Crit}{{Crit}}

\DeclareMathOperator{\gr}{{gr}}

\DeclareMathOperator{\Hom}{{Hom}}

\DeclareMathOperator{\id}{{id}}
\DeclareMathOperator{\ind}{{ind}}
\DeclareMathOperator{\Int}{{int}}

\DeclareMathOperator{\Lef}{{Lef}}

\DeclareMathOperator{\rank}{{rank}}

\DeclareMathOperator{\red}{{red}}

\DeclareMathOperator{\Spin}{{Spin}}

\DeclareMathOperator{\sing}{{sing}}

\DeclareMathOperator{\Top}{{Top}}

\DeclareMathOperator{\Sym}{{Sym}}

\DeclareMathOperator{\Tors}{{Tors}}

\DeclareMathOperator{\tr}{{tr}}

\DeclareMathOperator{\coker}{{coker}}

\newcommand{\PD}{\mathit{PD}}

\newcommand{\bD}{\mathbb{D}}

\newcommand{\bF}{\mathbb{F}}

\newcommand{\bK}{\mathbb{K}}
\newcommand{\bL}{\mathbb{L}}

\newcommand{\bT}{\mathbb{T}}

\newcommand{\cC}{\mathcal{C}}
\newcommand{\cD}{\mathcal{D}}
\newcommand{\cE}{\mathcal{E}}

\newcommand{\cH}{\mathcal{H}}

\newcommand{\cL}{\mathcal{L}}
\newcommand{\cM}{\mathcal{M}}

\newcommand{\cR}{\mathcal{R}}

\newcommand{\cU}{\mathcal{U}}

\newcommand{\frS}{\mathfrak{S}}
\newcommand{\frT}{\mathfrak{T}}
\newcommand{\frU}{\mathfrak{U}}

\newcommand{\frj}{\mathfrak{j}}

\newcommand{\frs}{\mathfrak{s}}
\newcommand{\frt}{\mathfrak{t}}
\newcommand{\fru}{\mathfrak{u}}

\newcommand{\as}{\ve{\alpha}}
\newcommand{\bs}{\ve{\beta}}
\newcommand{\gs}{\ve{\gamma}}
\newcommand{\ds}{\ve{\delta}}
\newcommand{\zetas}{\ve{\zeta}}
\newcommand{\xis}{\ve{\xi}}
\newcommand{\taus}{\ve{\tau}}
\newcommand{\Ds}{\ve{\Delta}}
\newcommand{\sigmas}{\ve{\sigma}}
\newcommand{\ps}{\ve{p}}
\newcommand{\ws}{\ve{w}}
\newcommand{\zs}{\ve{z}}
\newcommand{\xs}{\ve{x}}
\newcommand{\ys}{\ve{y}}

\newcommand{\GY}{\Gamma^{\mathsf{Y}}}
\newcommand{\GYbar}{\bar{\Gamma}{}^{\mathsf{Y}}}
\newcommand{\GYup}{\Gamma^{\mathbin{\text{\rotatebox[origin=c]{180}{$\mathsf{Y}$}}}}}
\newcommand{\GYupbar}{\bar{\Gamma}{}^{\mathbin{\text{\rotatebox[origin=c]{180}{$\mathsf{Y}$}}}}}

\newcommand{\boldHF}{\ve{\HF}}
\newcommand{\boldCF}{\ve{\CF}}
\newcommand{\bCF}{\ve{\CF}}

\newcommand{\mix}{\mathrm{mix}}

\newcommand{\CF}{\mathit{CF}}
\newcommand{\HF}{\mathit{HF}}

\renewcommand{\a}{\alpha}
\renewcommand{\b}{\beta}
\newcommand{\g}{\gamma}
\newcommand{\Dt}{\Delta}

\newcommand{\dotsimeq}{\mathrel{\dot{\simeq}}}

\DeclareMathOperator{\grHom}{\mathrm{grHom}}

\title{Duality and mapping tori in Heegaard Floer homology}
\author{Ian Zemke}
\address{Department of Mathematics\\Princeton University\\  Princeton, NJ 08544, USA}
\email{izemke@math.princeton.edu}
\thanks{This research was supported by NSF grant DMS-1703685}

\begin{document}

\begin{abstract} We show that the graph TQFT for Heegaard Floer homology satisfies a strong version of Atiyah's duality axiom for a TQFT. As an application, we compute some Heegaard Floer mixed invariants of 4-dimensional mapping tori in terms of Lefschetz numbers on $\HF^+$. 
\end{abstract}

\maketitle

\tableofcontents

\section{Introduction}

Heegaard Floer homology is a package of invariants for 3- and 4-manifolds introduced by Ozsv\'{a}th and Szab\'{o} \cite{OSDisks} \cite{OSTriangles}. To a closed and oriented 3-manifold $Y$ with a $\Spin^c$ structure $\frs$, they associated $\Z[U]$-modules 
\[
\HF^-(Y,\frs), \quad \HF^\infty(Y,\frs),\quad  \text{and} \quad  \HF^+(Y,\frs), 
\]
and a $\Z$-module  $\hat{\HF}(Y,\frs).$ Additionally, they constructed a reduced group $\HF^{\pm}_{\red}(Y,\frs)$, which we can view  either as a submodule of $\HF^-(Y,\frs)$, or a quotient of $\HF^+(Y,\frs)$.

  If $Y_1$ and $Y_2$ are connected, closed, and oriented 3-manifolds and $W$ is a cobordism from $Y_1$ to $Y_2$ with a $\Spin^c$ structure $\frt$, they constructed a cobordism map
\[
F_{W,\frt}:\Lambda^*(H_1(W)/\Tors)\otimes_{\Z} \HF^\circ(Y_1,\frt|_{Y_1})\to \HF^\circ(Y_2,\frt|_{Y_2}),
\]
 for $\circ\in \{+,-,\infty,\wedge\}$.

If $X$ is a smooth, closed, and oriented 4-manifold with a $\Spin^c$ structure $\frt$,  Ozsv\'{a}th and Szab\'{o} defined a 4-manifold invariant
\begin{equation}
\Phi_{X,\frt}:  \Lambda^*(H_1(X)/\Tors)  \otimes_{\Z}\Z[U]\to \Z/{\pm 1}. \label{eq:mixed-invariant}
\end{equation}
The invariant $\Phi_{X,\frt}$ is often called the \emph{mixed invariant} since its construction uses both $\HF^-$ and $\HF^+$. The mixed invariants have the power to distinguish smooth structures on homeomorphic 4-manifolds, and conjecturally agree with the Seiberg-Witten invariants. Computations are often challenging.

Ozsv\'{a}th and Szab\'{o}'s cobordism maps satisfy many standard TQFT axioms, though there are several properties not satisfied by their description. For example, there is no invariant assigned to a disconnected 3-manifold. This is unsatisfying for certain purposes. For example, suppose $\bK$ is a field and $Z$ is a (3+1)-dimensional TQFT, taking values in $\bK$-vector spaces. This means that to a 3-manifold $Y$, there is a $\bK$-vector space $Z(Y)$, and to a cobordism $W$ from $Y_1$ to $Y_2$, there is a $\bK$-linear map $Z(W)$ from $Z(Y_1)$ to $Z(Y_2)$. Since $Z(\emptyset)=\bK$, a closed 4-manifold determines an element of $\bK$.  A general philosophy suggests that the numerical invariant $Z(Y\times S^1)\in \bK$ should be the rank or Euler characteristic of $Z(Y)$. A natural proof of this fact involves decomposing $Y\times S^1$ into two copies of $Y\times [0,1]$, which are glued together along $Y\sqcup -Y$. Next, one usually expects $Z(Y\sqcup -Y)\iso Z(Y)\otimes_{\bK} Z(Y)^\vee$, where $\vee$ denotes the dual of $Z(Y)$, and also that the map for $Y\times [0,1]$ is either the identity, the algebraic trace, or the cotrace map, depending on which ends are incoming or outgoing. Such a manipulation is not possible using the machinery from \cite{OSTriangles}, since there is no Heegaard Floer group assigned to $Y\sqcup -Y$, and no map assigned to $Y\times [0,1]$, viewed as a cobordism from $\emptyset$ to $Y\sqcup -Y$. We note that monopole Floer homology has a similar limitation to Heegaard Floer homology in this regard, though a description of the Seiberg--Witten invariant of $Y\times S^1$ (conjecturally equivalent to the Ozsv\'{a}th--Szab\'{o} mixed invariant) is known \cite{BaldridgeSWCircleActions}. 

The goal of this paper is compute the Heegaard Floer mixed invariant of $Y\times S^1$ using an argument of the same style as the one in the previous paragraph. To accomplish this, we use the author's graph TQFT \cite{ZemGraphTQFT}, which is a more flexible framework for cobordism maps than Ozsv\'{a}th and Szab\'{o}'s original framework. 

Implicit in Ozsv\'{a}th and Szab\'{o}'s construction of $F_{W,\frt}$ is a choice of path which connects the two components of $\d W$. It turns out that after taking homology, this dependence vanishes on $\HF^-$, $\HF^+$ and $\HF^\infty$ (but interestingly, not for $\widehat{\HF}$, and also not on the chain level) \cite{ZemGraphTQFT}*{Corollary~F}. The graph TQFT uses cobordisms which are decorated with graphs, instead of paths. This allows for cobordisms with disconnected ends. When the graph is a path connecting the two ends, we recover Ozsv\'{a}th and Szab\'{o}'s map. With this in mind, we often refer to Ozsv\'{a}th and Szab\'{o}'s cobordism maps as \emph{path cobordism maps}.

The reader may interpret $\Phi_{X,\frs}(1)$ as corresponding to the mixed invariant of $X$, decorated with a graph which is an arc. If $\xi_1,\dots, \xi_n\in H_1(X)$, then we can interpret $\Phi_{X,\frt}(\xi_1\wedge \cdots \wedge \xi_n)$ as the mixed invariant for $X$, decorated with a graph consisting of the wedge sum of an arc and the $n$ loops $\xi_1,\dots, \xi_n$. This is made precise in Theorem~\ref{thm:mixed-invariants-from-graphs}, below.

Throughout this paper, we work over $\bF_2:=\Z/2\Z$.

\subsection{Heegaard Floer mixed invariants of mapping tori}

If $V=V_0\oplus V_1$ is a finite dimensional, $\Z/2\Z$-graded vector space over a field $\bK$, and  $F:V\to V$ is a map which preserves the $\Z/2\Z$-grading, the \emph{Lefschetz number} of $F$ is the quantity
\[
\Lef(F:V\to V):=\tr(F|_{V_0})-\tr(F|_{V_1})\in \bK.
\]

 Although $\HF^\circ(Y,\frs)$ is often not a finitely generated $\bF_2$-vector space, for $\circ\in \{+,-,\infty\}$, the group $\HF^+_{\red}(Y,\frs)$ is always finitely generated over $\bF_2$. In this paper, we prove the following result about 4-manifolds which admit a non-separating cut:

\begin{thm}\label{thm:mixedinvariantmappingtorus}
Suppose $X^4$ is a closed, oriented 4-manifold with $b_2^+(X)> 1$ and $Y^3\subset X$ is a closed, oriented, connected and non-separating 3-dimensional submanifold. Write $W$ for the cobordism obtained by cutting $X$ along $Y$. Suppose $\frs\in \Spin^c(W)$ is a $\Spin^c$ structure whose restrictions to both copies of $Y$ in $\d W$ agree. Suppose further that at least one of the following holds:
\begin{enumerate}
\item $\frs|_Y$ is non-torsion, or
\item $b_2^+(W)>0$.
\end{enumerate}
 If $\xi\in \Lambda^*(H_1(W)/\Tors)\otimes \bF_2[U]$, then the $\bF_2$ mixed invariants of $X$ satisfy
\[
\Lef\big(F_{W,\frs}(\xi\otimes -):\HF^+_{\red}(Y,\frs|_{Y})\to \HF^+_{\red}(Y,\frs|_{Y})\big)=\sum_{\substack{\frt\in \Spin^c(X)\\ \frt|_W=\frs}}\Phi_{X,\frt}(\xi).
\]
\end{thm}

\begin{rem}
Some care is required to define the map $F_{W,\frs}$ because the cobordism map $F_{W,\frs}$ requires a choice of path connecting the two components of $\d W$. According to \cite{ZemGraphTQFT}*{Theorem~F}, the dependence vanishes on $\HF^+$, and hence the Lefschetz number of $F_{W,\frs}$ on $\HF^+_{\red}$ is also independent of the path.
\end{rem}

A motivating example of a 4-manifold which admits a non-separating cut is a mapping torus. If $Y$ is a closed, oriented 3-manifold and $\phi:Y\to Y$ is an orientation preserving diffeomorphism, the mapping torus $X_\phi$ of the pair $(Y,\phi)$ is the 4-manifold
\[X_\phi:=\frac{Y\times [0,1]}{(y,1)\sim (\phi(y),0)} .\]  By specializing Theorem \ref{thm:mixedinvariantmappingtorus}, we obtain the following:

\begin{cor}\label{cor:mixedinvariantofactualmappingtori}Suppose $Y^3$ is closed, oriented 3-manifold and $\phi:Y\to Y$ is an orientation preserving diffeomorphism such that $b_2^+(X_\phi)>1$. If $\frs\in \Spin^c(Y)$ is non-torsion and  $\phi_*(\frs)=\frs$, then the mixed invariants of $X_\phi$ satisfy
\[
\Lef\big(\phi_*:\HF^+_{\red}(Y,\frs)\to \HF^+_{\red}(Y,\frs)\big)=\sum_{\substack{\frt\in \Spin^c(X_\phi)\\ \frt|_Y=\frs}}\Phi_{X_\phi,\frt}(1).\]
\end{cor}

The simplest example of a mapping torus is $Y\times S^1$. In this case, the projection map $\pi:Y\times S^1\to Y$ induces a map
\[\pi^*:\Spin^c(Y)\to \Spin^c(Y\times S^1).\] We say that a $\Spin^c$ structure on $Y\times S^1$ is \textit{$S^1$-invariant} if it is in the image of $\pi^*$. Together with the adjunction inequality, Corollary~\ref{cor:mixedinvariantofactualmappingtori} gives the following:

\begin{cor}
\label{cor:invariantsofYxS1}
If $Y^3$ has $b_1(Y)>1$ and $\frs\in \Spin^c(Y)$ is non-torsion, then 
 \[\Phi_{Y\times S^1,\pi^*(\frs)}(1)=\chi(\HF^+(Y,\frs)).\] Furthermore, if $\frt\in \Spin^c(Y\times S^1)$ is not $S^1$-invariant, then
 \[\Phi_{Y\times S^1,\frt}(1)=0.\]
 \end{cor}

\subsection{Perturbed coefficients}

We prove a more robust version of Theorem~\ref{thm:mixedinvariantmappingtorus} and its corollaries by working with twisted coefficients. If $\omega$ is a closed 2-form on $Y^3$, we can consider the \emph{perturbed Heegaard Floer homology} of $Y$, denoted $\HF^\circ(Y,\frs;\Lambda_{\omega})$, for $\circ\in \{+,-,\infty\}$. The construction is due to Ozsv\'{a}th and Szab\'o \cite{OSGenusBounds}. The groups $\HF^\circ(Y,\frs;\Lambda_{\omega})$ are modules over $\Lambda[U]$, where $\Lambda$ is the Novikov ring. Recall that the Novikov ring is generated by formal sums of the form $\sum_{\a\in \R} c_{\a} e^\a$, where $\{\a\in (-\infty, C]: c_{\a}\neq 0\}$  is finite for each $C$. Here, $e$ denotes the group-ring variable.

 If $\omega$ is a closed 2-form on a cobordism $W$ from $Y_1$ to $Y_2$, then there is also a perturbed version of the cobordism maps:
 \[
 F_{W,\frs;\omega}\colon \HF^\circ(Y_1,\frs_1;\Lambda_{\omega_1})\to \HF^\circ(Y_2,\frs_2;\Lambda_{\omega_2}),
 \] 
 where $\frs_i=\frs|_{Y_i}$ and $\omega_i=\omega|_{Y_i}$.

Theorem~\ref{thm:mixedinvariantmappingtorus} has the following refinement:
\begin{thm}
\label{thm:non-separating-cut-mixed-invariant}
Suppose $X^4$ is a closed, oriented 4-manifold with $b_2^+(X)> 1$ and $Y^3\subset X$ is a closed, oriented, connected, and non-separating 3-dimensional submanifold. Write $W$ for the cobordism obtained by cutting $X$ along $Y$. Suppose $\frs\in \Spin^c(W)$ is a $\Spin^c$ structure whose restrictions to both copies of $Y$ in $\d W$ agree, $\omega$ is a closed 2-form on $W$, and $\xi\in \Lambda^*(H_1(W)/\Tors)\otimes \bF_2[U]$. Furthermore, suppose that at least one of the following holds:
\begin{enumerate}
\item $c_1(\frs)$ is non-torsion on $Y$,
\item $[\omega|_Y]\neq 0\in H^2(Y;\R)$, or
\item $b_2^+(W)>0$.
\end{enumerate}
 Then the $\bF_2$ mixed invariants of $X$ satisfy
\[
\begin{split}
&\Lef\big(F_{W,\frs;\omega|_{W}}(\xi\otimes -):\HF^+_{\red}(Y,\frs|_{Y};\Lambda_{\omega|_Y})\to \HF^+_{\red}(Y,\frs|_{Y};\Lambda_{\omega|_Y})\big)\\
\doteq&\sum_{\substack{\frt\in \Spin^c(X)\\ \frt|_W=\frs}}e^{\langle (\frt-\frt_0)\cup \omega, [X]\rangle }\cdot \Phi_{X,\frt}(\xi),
\end{split}
\]
where $\frt_0$ denotes any choice of base $\Spin^c$ structure. Here $\doteq$ denotes equality up to an overall factor of $e^z$.
\end{thm}

We can apply Theorem~\ref{thm:non-separating-cut-mixed-invariant} to mapping tori. We highlight the case of $Y\times S^1$:

\begin{cor}\label{cor:perturbed-S1xY} Suppose that $Y^3$ has $b_1(Y)>1$, $\frs\in \Spin^c(Y)$, and $\omega$ is a closed 2-form which induces a non-zero element of $H^2(Y;\R)$. Then
\[
\Phi_{Y\times S^1, \pi^*(\frs)}(1)=\chi(\HF^+_{\red}(Y,\frs;\Lambda_\omega)),
\]
and $\Phi_{Y\times S^1,\frt}(1)=0$ if $\frt$ is not $S^1$-invariant.
\end{cor}

The simplest illustration of Corollary~\ref{cor:perturbed-S1xY} is $X=\bT^4$. It is well known that if $\omega$ is a non-zero 2-form on $\bT^3$, when $\HF^+(\bT^3;\Lambda_{\omega})\iso \Lambda$, and furthermore $\HF^+$ is supported only in the torsion $\Spin^c$ structure. See the work of Ai--Peters \cite{AiPetersTwisted}*{Theorem~1.3}, Jabuka--Mark~\cite{JabukaMarkProduct}*{Theorem~10.1}, Lekili~\cite{LekiliBrokenFibrations}*{Theorem~14} and Wu~\cite{WuPerturbedFibered}. In particular, our theorem computes that $\Phi_{\bT^4,\frt_0}= 1$, where $\frt_0\in \Spin^c(\bT^4)$ is the torsion $\Spin^c$ structure, and $\Phi_{\bT^4,\frt}=0$ for all other $\Spin^c$ structures. Of course, $\Phi_{\bT^4,\frt}$ may also be computed using the fact that $\bT^4$ is symplectic, using \cite{OSTrianglesandSymplectic}*{Theorem~1.1}.
 
\subsection{Trace and cotrace cobordisms and the graph TQFT}

In \cite{OSProperties}*{Proposition~2.5}, Ozsv\'{a}th and Szab\'{o} describe the effect of orientation reversal on the Heegaard Floer complexes. They show that if $c_1(\frs)$ is torsion, then
\begin{equation}
\CF^-(-Y,\frs)\iso \grHom_{\bF_2}(\CF^+(Y,\frs), \bF_2), \label{eq:duality-old}
\end{equation}
where $\grHom$ denotes the vector space spanned by homogeneously graded $\bF_2$-linear maps. There is an alternate way of describing the effect of orientation reversal on Heegaard Floer homology, which is more natural from the perspective of a TQFT, as follows. Similar to equation~\eqref{eq:duality-old}, there is a canonical chain isomorphism
\begin{equation}
\CF^-(-Y,\frs)\iso \CF^-(Y,\frs)^\vee:= \Hom_{\bF_2[U]}(\CF^-(Y,\frs),\bF_2[U]).\label{eq:multipointeddualityiso}
\end{equation}
In particular, evaluation gives a natural $\bF_2[U]$-equivariant pairing 
\[
\tr\colon \CF^-(Y,\frs)\otimes_{\bF_2[U]}\CF^-(-Y,\frs)\to \bF_2[U],
\]
 which we call the \emph{trace pairing}. See Lemma~\ref{lem:F-pairing-equivalence} for the relation between the duality statements in equations~\eqref{eq:duality-old} and~\eqref{eq:multipointeddualityiso}.

 Following influential papers of Witten \cite{WittenTQFT} and Segal \cite{SegalCFT}, Atiyah describes  an axiomatic framework for TQFTs \cite{AtiyahTQFT} which features a duality axiom concerning orientation reversal. Accordingly, one should expect the canonical trace pairing to coincide with the cobordism map for $Y\times [0,1]$, viewed as a cobordism from $Y\sqcup -Y$ to $\emptyset$. However, as a cobordism from $Y\sqcup -Y$ to $\emptyset$, $Y\times [0,1]$  is not assigned a map in Ozsv\'{a}th and Szab\'{o}'s TQFT framework.

The graph TQFT of \cite{ZemGraphTQFT} does assign a map to $Y\times [0,1]$, viewed as a cobordism from $Y\sqcup -Y$ to $\emptyset$, as we now describe. If $(Y,\ws)$ is a multi-pointed 3-manifold, the 4-manifold with embedded graph $(Y\times [0,1],\ws\times [0,1])$ can be viewed as a graph cobordism in three ways, depending on which ends we identify as incoming and outgoing. We suggestively call these the \emph{identity cobordism}, the \emph{trace cobordism}, and the \emph{cotrace cobordism}. We illustrate the three configurations in Figure~\ref{fig::55}.  The graph TQFT from \cite{ZemGraphTQFT} assigns maps to the trace and cotrace cobordisms, though it is not immediate from their definitions that they are related to the canonical trace and cotrace maps induced by the pairing in equation~\eqref{eq:multipointeddualityiso}. A key step towards proving Theorem~\ref{thm:mixedinvariantmappingtorus} is the following:

  \begin{figure}[ht!]
  	\centering
  	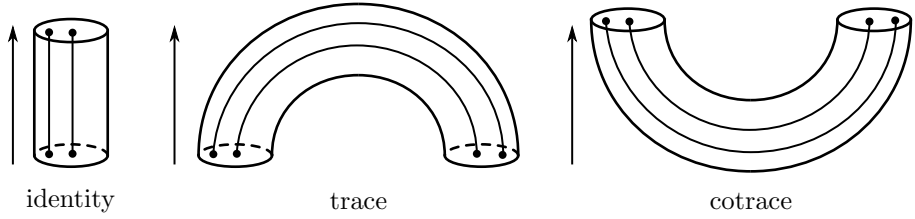
  	\caption{\textbf{The identity, trace and cotrace graph cobordisms.} All are equal to $(Y\times [0,1], \ws\times [0,1])$, but have different ends identified as incoming or outgoing.\label{fig::55}}
  \end{figure}

\begin{thm}\label{thm:dualityv1}If $(Y,\ws)$ is a multi-pointed 3-manifold, the trace graph cobordism $(Y\times [0,1], \ws\times [0,1]): (Y\sqcup -Y,\ws\sqcup \ws)\to \emptyset$ induces the canonical trace map
\[
\tr:\CF^-(Y,\ws,\frs)\otimes_{\bF_2[U]} \CF^-(-Y,\ws,\frs)\to \bF_2[U].
\]
 Similarly, the cotrace graph cobordism $(Y\times [0,1],\ws\times [0,1]):\emptyset\to (Y\sqcup -Y,\ws\sqcup \ws)$ induces the canonical cotrace map
\[
\cotr: \bF_2[U]\to \CF^-(Y,\ws,\frs)\otimes_{\bF_2[U]} \CF^-(-Y,\ws,\frs).
\]
\end{thm}

\begin{rem} If $C$ is a finitely generated, free module over a ring $\cR$, then there are canonical isomorphisms
\[
\Hom_{\cR}(C,C)\iso \Hom_{\cR}(C\otimes_{\cR} C^\vee, \cR)\iso \Hom_{\cR}(\cR,C^\vee\otimes_{\cR} C).
\] 
Under the above isomorphisms, the identity map, the trace map and the cotrace map are all identified. Hence Theorem~\ref{thm:dualityv1} implies that the cobordism map for $(Y\times[0,1],\ws\times [0,1])$ is independent of which ends of $Y\times [0,1]$ are identified as incoming, and which ends are identified as outgoing. Using the composition law, one sees that the map for an arbitrary cobordism is independent of which ends are incoming or outgoing, in the above sense.
\end{rem}

 If $(\Sigma,\as,\bs,\gs,\ws)$ is a Heegaard triple, Ozsv\'{a}th and Szab\'{o} constructed a 3-ended 4-manifold $X_{\a,\b,\g}$. Although Ozsv\'{a}th and Szab\'{o} extensively considered these 4-manifolds when the underlying Heegaard triple was adapted to a framed link in a 3-manifold,  their construction does not in general assign an invariant cobordism map to the 4-manifold $X_{\a,\b,\g}$. Nonetheless, there is a natural trivalent graph $\Gamma_{\a,\b,\g}$ inside of $X_{\a,\b,\g}$, and the construction from \cite{ZemGraphTQFT} assigns a cobordism map to the pair $(X_{\a,\b,\g},\Gamma_{\a,\b,\g})$. A key component of our proof of Theorem~\ref{thm:dualityv1} is proving that for an arbitrary Heegaard triple, the graph cobordism map for $(X_{\a,\b,\g}, \Gamma_{\a,\b,\g})$ coincides with the map obtained by counting holomorphic triangles on $(\Sigma,\as,\bs,\gs,\ws)$, i.e.
 \begin{equation}
 F_{X_{\a,\b,\g},\Gamma_{\a,\b,\g},\frs}\simeq F_{\a,\b,\g,\frs}.\label{eq:triangle-map-intro}
 \end{equation}
 See Theorem~\ref{thm:triplesandgraphcobordismmaps} for a precise statement. Much of the technical work of this paper is devoted to proving equation~\eqref{eq:triangle-map-intro}.

Theorem~\ref{thm:dualityv1} also generalizes another duality result of Ozsv\'{a}th and Szab\'{o}. They proved that if $W:Y_1\to Y_2$ is a cobordism between two connected 3-manifolds, and $W^\vee:-Y_2\to -Y_1$ is the cobordism obtained by turning around $W$, then 
 \begin{equation}
 F_{W^\vee,\frs}=(F_{W,\frs})^{\vee}.\label{eq:OSturningaroundcob}
 \end{equation}
See \cite{OSTriangles}*{Theorem~3.5}. It is a straightforward algebraic exercise to show that equation~\eqref{eq:OSturningaroundcob} can be derived from Theorem~\ref{thm:dualityv1} and the composition law.

\subsection{A Lefschetz number formula for complexes over $\bK[U]$}

Suppose $\bK$ is a field of characteristic 2. We now describe an algebraic result about the Lefschetz number of a map on a chain complex over $\bK[U]$.  We use this result as a key step in using the graph cobordism maps to prove Theorem~\ref{thm:mixedinvariantmappingtorus}.

Recall that if $C$ is a relatively $\Z/2\Z $-graded, finitely generated chain complex over  $\bK$, and $F:C\to C$ is a chain map which preserves the relative grading, then
\begin{equation}
(\tr\circ(F\otimes \id)\circ \cotr)(1)=\Lef(F_*:H_*(C)\to H_*(C))\in \bK\label{eq:Eulercharacteristicoverfield}.
\end{equation} 

If $C$ is a finitely generated, free chain complex over $\bK[U]$, and $F\colon C\to C$ is a chain map, then the left hand side of equation~\eqref{eq:Eulercharacteristicoverfield} is usually not interesting. We now describe a natural analog of equation~\eqref{eq:Eulercharacteristicoverfield} for chain complexes over $\bK[U]$.

Mirroring the algebra of Heegaard Floer homology, we define chain complexes
\begin{equation}
C^-:=C, \qquad C^\infty:=C\otimes_{\bK[U]} \bK[U,U^{-1}],\qquad \text{and} \qquad C^+:=(C\otimes_{\bK[U]} \bK[U,U^{-1}])/ C, \label{eq:Cminus-infty-plus-intro}
\end{equation}
and write $H^-(C),$ $H^\infty(C)$ and $H^+(C)$ for the respective homology groups. The short exact sequence
\[
0\to C^-\to C^\infty\to C^+\to 0,
\]
 induces a long exact sequence on homology. We let 
 \[
 \delta:H^+(C)\to H^-(C)
 \]
  denote the connecting homomorphism.  We define
  \[
H^-_{\red}(C)=\ker( H^-(C)\to H^\infty(C))\quad \text{and} \quad H^+_{\red}(C)=\coker(H^\infty(C)\to H^+(C)),
  \]
  and note that $\delta$ induces an isomorphism from $H_{\red}^+(C)$ to $H_{\red}^-(C)$.

If $F\colon C\to C$ is a $\bK[U]$-equivariant chain map, then there is an induced map
\[
F_*\colon H_{\red}^{+}(C)\to H_{\red}^{+}(C).
\]
Furthermore, $H_{\red}^{+}(C)$ is a finitely generated $\bK$-module, so the Lefschetz number of $F$ on $H_{\red}^{+}(C)$ is defined.

Our Lefschetz number formula also involves a $+1$ graded endomorphism $\Phi\colon C\to C$. The map $\Phi$ is obtained by writing the differential $\d$ of $C$ as a matrix in terms of a $\bK[U]$ basis of $C$, and then differentiating each entry of the matrix with respect to $U$. The map $\Phi$ is independent of the choice of basis, up to chain homotopy. 

To state the formula, we must tensor $C$ with $\bK[[U]]$. Chain complexes $C^\infty$ and $C^+$ may also be defined over $\bK[[U]]$, similarly to equation~\eqref{eq:Cminus-infty-plus-intro}. This does not change the Lefschetz number of $F$ over $H_{\red}^+$, since
  \[
H_{\red}^{+}(C\otimes_{\bK[U]} \bK[[U]])\iso H_{\red}^{+}(C).
  \]

We can now state our algebraic Lefschetz number formula:

\begin{prop}\label{prop:algebraicmappingtorus} Suppose that $C$ is a free, finitely generated, $\Z/2\Z $-graded chain complex over $\bK[[U]]$. Suppose that $F:C\to C$ is a $\bK[[U]]$-equivariant chain map which preserves the $\Z/2\Z$-grading, and vanishes on $H^\infty(C)$. Then 
$\Lef\big(F_*:H^+_{\red}(C)\to H^+_{\red}(C)\big)$ is equal to the coefficient of $U^{-1}$ in the expression 
\[
(\tr \circ (F\otimes \id)\circ  \delta^{-1}\circ (\id\otimes \Phi^\vee) \circ \cotr)(1).
\]
\end{prop}

In  Proposition~\ref{prop:algebraicmappingtorus},  $\Phi^\vee$ denotes the dual of $\Phi:C\to C$, and $\delta$ denotes the connecting homomorphism $\delta: H^+(C\otimes C^\vee)\to H^-(C\otimes C^\vee)$. Implicit in the statement is the claim that the maps $(\id\otimes \Phi^\vee)$ and $(F\otimes \id)$ factor through $H_{\red}^{\pm}(C\otimes C^\vee)$, so the above expression makes sense. The reason for working over $\bK[[U]]$ instead of $\bK[U]$ is that it allows us to factor $(F\otimes \id)$ and $(\id\otimes \Phi^{\vee})$ through $H_{\red}^{\pm}$.

In the context of Heegaard Floer complexes, the formal derivative map $\Phi$ appearing in Proposition~\ref{prop:algebraicmappingtorus} is the map induced by the ``broken path'' graph cobordism  in Figure~\ref{fig::42} (see Lemma~\ref{lem:phi=brokenpathcobordism}, below). When dealing with the Heegaard Floer complexes, we will usually write $\Phi_w$ for the broken path graph cobordism map, where $w\in Y$ is the basepoint. The composition in Proposition~\ref{prop:algebraicmappingtorus} naturally appears when computing the mixed invariants of mapping tori using graph cobordisms; see Figure~\ref{fig:59}.

  \begin{figure}[ht!]
  	\centering
  	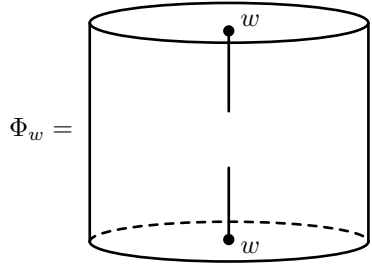
  	\caption{\textbf{The ``broken path'' graph cobordism inducing the map $\Phi_w:\CF^-(Y,w,\frs)\to \CF^-(Y,w,\frs)$.} This graph cobordism induces the map which features algebraically in Proposition~\ref{prop:algebraicmappingtorus}. The underlying 4-manifold is $Y\times [0,1]$.\label{fig::42}}
  \end{figure}

 The map $\Phi_w$ and some analogs have been studied in other contexts. An early appearance was in \cite{SarkarMovingBasepoints}, where Sarkar proved a formula for a mapping class group action that involved an analogous map on link Floer homology.  The map $\Phi_w$ appears in the formula for  the $\pi_1(Y,w)$-action on $\CF^-(Y,w,\frs)$, proven by the author in \cite{ZemGraphTQFT}. Analogous maps appear in the link Floer TQFT \cite{JCob} \cite{ZemCFLTQFT}, where they have a similar graphical interpretation in terms of certain dividing sets on cylindrical link cobordisms. The link Floer homology analogs of $\Phi_w$ also feature in a connected sum formula for involutive knot Floer homology \cite{ZemKnotConnectedSums}.

\subsection{Comparison to Seiberg-Witten theory} 

Theorem~\ref{thm:mixedinvariantmappingtorus} is inspired by several analogous results in monopole Floer homology.
One analog to Theorem~\ref{thm:mixedinvariantmappingtorus} was proven by Fr\o yshov using a version of monopole Floer homology \cite{FroyshovMonopoleFloerHomology}*{Theorem~7}. Another analog is Baldridge's computation of the Seiberg-Witten invariants of 4-manifolds with free circle actions \cite{BaldridgeSWCircleActions}.

There are isomorphisms between the 3-dimensional Heegaard Floer invariants from \cite{OSDisks} and the 3-dimensional Seiberg-Witten Floer invariants appearing in \cite{KMMonopole}. This has been established  by Kutluhan, Lee and Taubes (\cite{KLTHF=HM1,KLTHF=HM2,KLTHF=HM3,KLTHF=HM4,KLTHF=HM5}) and independently by Colin, Ghiggini and Honda (\cite{CGHHF=ECH0,CGHHF=ECH1,CGHHF=ECH2,CGHHF=ECH3}), the latter using work of Taubes \cite{TaubesECH=SW1}. An equivalence between the 4-dimensional theories is expected, though to the author's knowledge has not appeared in the literature yet.

Hence, one alternate route toward proving Theorem~\ref{thm:mixedinvariantmappingtorus} would be to prove that the 4-dimensional Heegaard Floer and monopole Floer theories are equivalent, and then use one of the aforementioned results about monopole Floer homology.

Lefschetz number formulas are fairly common in the study of TQFTs. Another recent construction involving Lefschetz numbers is due to Mrowka, Ruberman and Saveliev \cite{MRSEndPeriodic}, and takes the form of an invariant $\lambda_{SW}(X)$ for a 4-manifold $X$ with the homology of $S^1\times S^3$. The invariant has been further considered by Lin, Ruberman and Saveliev in \cite{LRShomologyS1S3}, and under the assumption that a generator of $H_3(Y)$ can be realized as a rational homology sphere $Y\subset X$, it is shown in \cite{LRShomologyS1S3} that the invariant $\lambda_{SW}(X)$ is related to the Lefschetz number of the map on monopole Floer homology induced by the cobordism obtained by cutting $X$ along $Y$. It would be interesting to see if the techniques of this paper could be applied to 4-manifolds with the homology of $S^1\times S^3$.

\subsection{Organization} In Section~\ref{sec:background} we provide background on Heegaard Floer homology. Section~\ref{sec:algebraic-Lefschetz}  covers the proof of Proposition~\ref{prop:algebraicmappingtorus}. Section~\ref{sec:graphTQFT} covers  background and preliminary results about the graph TQFT for Heegaard Floer homology. In Section~\ref{sec:handledecomposition} we describe a handle decomposition of the trace cobordism. In Sections~\ref{sec:generalized1--handleand3--handlemaps} and \ref{sec:doubleddiagrams}, we describe two technical tools which will be useful later in the paper: the generalized 1- and 3-handle maps, and doubled Heegaard diagrams. In Section~\ref{sec:connectedsumsandgraphTQFT} we describe the behavior of the graph TQFT with respect to connected sums, and prove that the maps Ozsv\'{a}th and Szab\'{o} used to prove the K\"{u}nneth theorem are in fact graph cobordism maps. In Section~\ref{sec:Heegaardtriplesandgraphcobordisms} we show that the graph cobordism map for the 4-manifold obtained from a Heegaard triple is chain homotopic to the holomorphic triangle map on that Heegaard triple. In Section~\ref{sec:traceandcotrace} we compute the cobordism maps for the trace and cotrace graph cobordisms, proving Theorem~\ref{thm:dualityv1}. In Section~\ref{sec:mixedinvariants}, we prove the untwisted versions of our theorems on the mixed invariants. In Section~\ref{sec:perturbed}, we explain how to adapt the results of the paper to the setting of twisted coefficients.

\subsection{Acknowledgments} I would like to thank Ciprian Manolescu, my Ph.D. advisor, for suggesting the mapping torus problem, as well as providing valuable suggestions and direction along the way. I would also like to thank Andr\'{a}s Juh\'{a}sz,  Jianfeng Lin, Robert Lipshitz, Thomas Mark, Anubhav Mukherjee, Peter Ozsv\'{a}th, and Luya Wang for valuable discussions and suggestions. The author would also like to thank an anonymous referee for their helpful comments. In addition, the author is indebted to Robert Lipshitz, Peter Ozsv\'{a}th and Dylan Thurston for providing some suggestions for the proof of Theorem~\ref{thm:triplesandgraphcobordismmaps}.

\section{Background on Heegaard Floer homology}
\label{sec:background}
\subsection{Multi-pointed Heegaard diagrams}

Heegaard Floer homology for multi-pointed 3-manifolds is constructed using the following notion of a multi-pointed Heegaard diagram:

\begin{define}\label{def:multipointedheegaarddiagram}If $(Y,\ws)$ is a connected, closed 3-manifold with a nonempty collection of basepoints $\ws$, we say that $\cH=(\Sigma,\as,\bs,\ws)$ is a \emph{multi-pointed Heegaard diagram for} $(Y,\ws)$ if the following are satisfied:
\begin{enumerate}[label=(HD-\arabic*), ref=HD-\arabic*,leftmargin=*, widest=IIII]
\item\label{def:mphd1} $\Sigma\subset Y$ is an embedded surface containing the points $\ws$, which splits $Y$ into two handlebodies, $U_{\a}$ and $U_{\b}$, oriented so that $\Sigma=\d U_{\a}=-\d U_{\b}$.
\item\label{def:mphd2} $\as=\{\alpha_1,\dots, \alpha_{g(\Sigma)+|\ws|-1}\}$ and $\bs=\{\beta_1,\dots, \beta_{g(\Sigma)+|\ws|-1}\}$ each consist of $g(\Sigma)+|\ws|-1$ pairwise disjoint, simple closed curves on $\Sigma$.
\item\label{def:mphd3} The collection $\as$ bounds pairwise disjoint compressing disks in $U_{\a}$, and $\bs$ bounds pairwise disjoint compressing disks in $U_{\b}$.
\item\label{def:mphd4} Each of the sets $\as$ and $\bs$ are homologically independent in $H_1(\Sigma\setminus \ws; \Z)$.
\end{enumerate}
\end{define}

It follows from the above definition that if $(\Sigma,\as,\bs,\ws)$ is a multi-pointed Heegaard surface, then each connected component of $\Sigma\setminus \as$ and $\Sigma\setminus \bs$ is planar and contains exactly one $\ws$ basepoint.

\subsection{The Heegaard Floer complexes}

If $(Y,\ws)$ is a connected, closed, oriented 3-manifold with basepoints $\ws$ and a $\Spin^c$ structure $\frs$, Ozsv\'{a}th and Szab\'{o} \cite{OSDisks} \cite{OSLinks} define $\bF_2[U]$-modules
\[
\HF^-(Y,\ws,\frs), \qquad \HF^\infty(Y,\ws,\frs), \qquad \HF^+(Y,\ws,\frs), \qquad \text{and} \qquad \hat{\HF}(Y,\ws,\frs).
\] We view  $U$ as acting by zero on $\hat{\HF}(Y,\ws,\frs)$.

 To define these homology groups, one first picks a multi-pointed Heegaard diagram $\cH=(\Sigma,\as,\bs,\ws)$ for $(Y,\ws)$.
  There are two tori,
 \[
 \bT_{\a}:=\alpha_1\times \dots \times \alpha_n \qquad \text{and} \qquad \bT_{\b}:=\beta_1\times \dots \times \beta_n,
 \] inside of the symmetric product $\Sym^n(\Sigma)$, where $n:=|\as|=|\bs|=g(\Sigma)+|\ws|-1$. There is a map 
 \[
 \frs_{\ws}:\bT_{\a}\cap \bT_{\b}\to \Spin^c(Y),
 \] 
 defined in \cite{OSDisks}*{Section~2.6}. 
 
 Suppose $J$ is an almost complex structure on $\Sym^n(\Sigma)$. The chain complex $\CF^-_J(\cH,\frs)$ is freely generated over $\bF_2[U]$ by  intersection points $\ve{x}\in \bT_{\a}\cap \bT_{\b}$ with $\frs_{\ws}(\xs)=\frs$.  If $\phi$ is a homology class of disks in $\Sym^n(\Sigma)$, with boundary on $\bT_{\a}\cup \bT_{\b}$, and $\mu(\phi)=1$, then the moduli space $\cM_J(\phi)$ is generically 1-dimensional, and has a free action of $\R$. We write
 \[
 \hat{\cM}(\phi):=\cM(\phi)/\R.
 \]
 The differential on $\CF^-_{J}(\cH,\frs)$ counts Maslov index 1 holomorphic strips in $\Sym^n(\Sigma)$ via the formula
 \[
 \d\ve{x}=\sum_{\ys\in \bT_{\a}\cap \bT_{\b}} \sum_{\substack{\phi\in \pi_2(\xs,\ys)\\
 \mu(\phi)=1}} \# \hat{\cM}_J(\phi) U^{n_{\ws}(\phi)} \cdot \ve{y}.
 \]
 In the above expression, $n_{\ws}(\phi)$ denotes the total multiplicity of the homology class $\phi$ over the $\ws$ basepoints. We usually omit the subscript $J$ from the notation. Under the strong $\frs$-admissibility assumption on the diagram $\cH$  (see \cite{OSDisks}*{Section~4.2.2}), the total number of holomorphic disks contributing to $\d\xs$ is finite.  
 
 The complexes $\CF^\infty(\cH,\frs)$, $\CF^+(\cH,\frs)$ and $\hat{\CF}(\cH,\frs)$ are obtained algebraically from $\CF^-(\cH,\frs)$ by the formulas
  \[\CF^\infty:=\CF^-\otimes_{\bF_2[U]} \bF_2[U,U^{-1}],\qquad \CF^+:=\CF^\infty/\CF^-\qquad \text{and}\]\[  \hat{\CF}:=\CF^-\otimes_{\bF_2[U]} \bF_2[U]/(U=0).\] The short exact sequence $0\to \CF^-(\cH,\frs)\to \CF^\infty(\cH,\frs)\to \CF^+(\cH,\frs)\to 0$ yields a long exact sequence on homology. We let
  \[\delta:\HF^+(\cH,\frs)\to \HF^-(\cH,\frs)\] denote the connecting homomorphism.
  
  The groups $\HF_{\red}^{+}(\cH,\frs)$ and $\HF_{\red}^-(\cH,\frs)$ are defined as
   \[\HF_{\red}^-(\cH,\frs):=\ker\big(\HF^-(\cH,\frs)\to \HF^\infty(\cH,\frs)\big)\] and
   \[\HF^+_{\red}(\cH,\frs):=\coker\big(\HF^\infty(\cH,\frs)\to \HF^+(\cH,\frs)\big).\] The connecting homomorphism $\delta$ induces an isomorphism from $\HF_{\red}^+(\cH,\frs)$ to $\HF_{\red}^-(\cH,\frs)$. The modules $\HF_{\red}^-$ and $\HF^+_{\red}$ are always finitely generated over $\bF_2$.

We need the following naturality result:

\begin{prop} If $\cH_1$ and $\cH_2$ are two strongly $\frs$-admissible diagrams for $(Y,\ws)$, equipped with almost complex structures $J_1$ and $J_2$, then there is a transition map
\[
\Psi_{(\cH_1,J_1)\to (\cH_2,J_2)}: \CF^-_{J_1}(\cH_1,\frs)\to \CF^-_{J_2}(\cH_2,\frs),
\] 
which is well-defined up to chain homotopy. Furthermore
\[
\Psi_{(\cH,J)\to (\cH,J)}\simeq \id_{\CF^-_J(\cH,\frs)}.
\]
If $\cH_1,$ $\cH_2$ and $\cH_3$ are three $\frs$-admissible Heegaard diagrams, with almost complex structures $J_1,$ $J_2$ and $J_3$, then
\[
\Psi_{(\cH_1,J_1)\to (\cH_3,J_3)}\simeq \Psi_{(\cH_2,J_2)\to (\cH_3,J_3)}\circ \Psi_{(\cH_1,J_1)\to (\cH_2,J_2)}.
\]
\end{prop}
 
   We refer the reader to \cite{JTNaturality} for more about the problem of naturality, however we make a few remarks. If $\cH_1$ and $\cH_2$ are two admissible Heegaard diagrams for $(Y,\ws)$, one can always connect $\cH_1$ and $\cH_2$ by a sequence of elementary Heegaard moves.  Using this fact, Ozsv\'{a}th and Szab\'{o} construct a transition map $\Psi_{(\cH_1,J_1)\to (\cH_2,J_2)}$ from $\CF^-_{J_1}(\cH_1,\frs)$ to $\CF^-_{J_2}(\cH_2,\frs)$ in \cite{OSDisks}.  They show that $\Psi_{(\cH_1,J_1)\to (\cH_2,J_2)}$ is a quasi-isomorphism, though it is not obviously independent of the sequence of intermediate Heegaard diagrams between $\cH_1$ and $\cH_2$. The main result of \cite{JTNaturality} is that the map $\Psi_{(\cH_1,J_1)\to (\cH_2,J_2)}$ is independent on homology from the choice of Heegaard moves from $\cH_1$ to $\cH_2$. Using this, it is possible to define a single $\bF_2[U]$-module $\HF^-(Y,\ws,\frs)$  as the transitive limit of the groups $\HF^-(\cH,\frs)$ (see \cite{JTNaturality}*{Definition~1.1}). In fact, using some additional results proven by Lipshitz (in particular \cite{LipshitzCylindrical}*{Proposition~11.4}), one can show that $\Psi_{(\cH_1,J_1)\to (\cH_2,J_2)}$ and $\Psi_{(\cH_2,J_2)\to (\cH_1,J_1)}$ are homotopy inverses, and show that $\Psi_{(\cH_1,J_1)\to (\cH_2,J_2)}$ is well-defined up to $\bF_2[U]$-equivariant chain homotopy. We refer the reader to \cite{HMInvolutive}*{Proposition~2.3} for an overview of this last fact.

 For some of the neck-stretching arguments  in this paper, an important tool will be Lipshitz's cylindrical reformulation of Heegaard Floer homology \cite{LipshitzCylindrical}. Lipshitz constructs a chain complex, generated by intersection points $\xs\in\bT_{\a}\cap \bT_{\b}$, as above, but with a differential that counts holomorphic curves in $\Sigma\times [0,1]\times \R$,  with boundary on $\bs\times \{0\}\times \R$ and $\as\times \{1\}\times \R$. We describe some additional technical details about this approach in Section~\ref{section:analyticalaspects}. The cylindrical reformulation is motivated by the ``tautological correspondence'' between $(\frj_{\bD}, \Sym^{n}(\frj_\Sigma))$-holomorphic maps $u: \bD\to \Sym^{n}(\Sigma)$ and
$(\frj_S, \frj_{\Sigma}\times \frj_{\bD})$-holomorphic maps $u':S\to \Sigma\times \bD,$ where  $S$ is a Riemann surface and $\pi_{\bD}\circ u'$ is an $n$-fold branched cover of $\bD$. Here $\Sym^n(\frj_{\Sigma})$ and $\frj_{\Sigma}\times \frj_{\bD}$ denote product almost complex structures. See \cite{OSDisks}*{Lemma~3.6} and \cite{LipshitzCylindrical}*{Section~13} for more details about the tautological correspondence and the equivalence between the two constructions.

\subsection{Duality and the Heegaard Floer complexes}
\label{sec:dualityofcomplexes}
If $\cH=(\Sigma,\as,\bs,\ws)$ is a diagram for $(Y,\ws)$, then $\cH^\vee:=(\Sigma,\bs,\as,\ws)$ is a diagram for $(-Y,\ws)$.  In \cite{OSTriangles}*{Section~5}, Ozsv\'{a}th and Szab\'{o} define a pairing map
\[
\langle ,\rangle :\CF^\infty(\cH,\frs)\otimes_{\bF_2} \CF^\infty(\cH^\vee,\frs)\to \bF_2,
\] 
by the formula
\begin{equation}
\langle U^i\cdot \xs, U^j\cdot \ys \rangle=\begin{cases}1& \text{ if }i+j=-1 \text{ and } \xs=\ys\\
0& \text{ otherwise.}
\end{cases}\label{eq:F2pairingdefinition}
\end{equation}
 We emphasize that $\langle, \rangle$ is not $\bF_2[U]$-equivariant. As such, the pairing map $\langle,\rangle$ cannot have an interpretation in terms of the cobordism maps.  Using the pairing $\langle, \rangle$, Ozsv\'{a}th and Szab\'{o} show  in \cite{OSProperties}*{Proposition~2.5} that there is a chain isomorphism 
  \begin{equation}
  \CF^-(-Y,\ws,\frs)\iso \grHom_{\bF_2}(\CF^+(Y,\ws,\frs),\bF_2),\label{eq:OSdualcomplex}
  \end{equation}
  where $\grHom_{\bF_2}$ denotes the $\bF_2$ span of homogeneously graded $\bF_2$-linear homomorphisms.

As described in the introduction, there is also a  trace pairing, which takes the form of a map
\[
\tr: \CF^-(\cH,\frs)\otimes_{\bF_2[U]} \CF^-(\cH^\vee,\frs)\to \bF_2[U],
\] 
given by
\[
\tr( U^i \cdot \xs, U^j\cdot \ys )=\begin{cases}U^{i+j}& \text{ if }\xs=\ys\\
0& \text{ otherwise}.
\end{cases}
\]
 Since there is a bijection between flowlines from $\xs$ to $\ys$ on $\cH$ and flowlines from $\ys$ to $\xs$ on $\cH^\vee$, it follows that
\[
\tr(\d_{\cH} (\xs),\ys) =\tr( \xs, \d_{\cH^\vee} (\ys)).
\] Hence there is a chain isomorphism
\begin{equation}
\CF^-(-Y,\ws,\frs)\iso \Hom_{\bF_2[U]}(\CF^-(Y,\ws,\frs), \bF_2[U]).
\label{eq:F[U]duality}
\end{equation}

Despite its slightly different appearance, the duality isomorphism from equation~\eqref{eq:F[U]duality} equivalent to the isomorphism from equation~\eqref{eq:OSdualcomplex}, as we explain in the following lemma:
\begin{lem}
\label{lem:F-pairing-equivalence}
If $C$ is a free, finitely generated chain complex over $\bF_2[U]$ which is relatively $\Z$-graded, then
\[
\grHom_{\bF_2}(C^+, \bF_2)\iso \Hom_{\bF_2[U]}(C,\bF_2[U]),
\]
 where $C^+:=(C\otimes \bF_2[U,U^{-1}])/C$.
\end{lem}
\begin{proof} We can write $C^+$ as $C\otimes_{\bF_2[U]} (\bF_2[U,U^{-1}]/\bF_2[U])$. Using tensor-hom adjunction, we have
\begin{align*}\grHom_{\bF_2}(C^+,\bF_2)
&= \grHom_{\bF_2}(C\otimes_{\bF_2[U]} (\bF_2[U,U^{-1}]/\bF_2[U]),\bF_2)\\
&\iso \Hom_{\bF_2[U]}(C, \grHom_{\bF_2}(\bF_2[U,U^{-1}]/\bF_2[U],\bF_2)).
\end{align*}
 It is easy to construct an isomorphism of $\bF_2[U]$-modules
\[
\grHom_{\bF_2}(\bF_2[U,U^{-1}]/\bF_2[U],\bF_2)\iso \bF_2[U].
\]
  The main claim now follows.
\end{proof}

\begin{rem} Omitting the assumption that $C$ is relatively $\Z$-graded, one has
\[
\Hom_{\bF_2}(C^+,\bF_2)\iso \Hom_{\bF_2[U]}(C, \bF_2[[U]]),
\]
since $\Hom_{\bF_2}(\bF_2[U,U^{-1}], \bF_2)\iso \bF_2[[U]]$.
\end{rem}

There is also a cotrace map
\[
\cotr:\bF_2[U]\to \CF^-(Y,\ws,\frs)\otimes_{\bF_2[U]} \CF^-(-Y,\ws, \frs),
\] 
which we can define as the dual of the trace map with domain $\CF^-(-Y,\ws,\frs)\otimes_{\bF_2[U]} CF^-(Y,\ws,\frs)$. On the level of generators, the cotrace map takes the form
\[
\cotr(1)=\sum_{i=1}^n \xs_i \otimes \xs_i^\vee,
\] 
for a basis $\xs_1,\dots, \xs_n$ of $\CF^-(Y,\ws,\frs)$.

\subsection{Heegaard Floer mixed invariants} 

To a closed, oriented 4-manifold $X$ with $b_2^+(X)>1$, Ozsv\'{a}th and Szab\'{o} define a mixed invariant $\Phi_{X,\frt}$ \cite{OSTriangles}, which is a map
\[
\Phi_{X,\frt}:\Lambda^*( H_1(X)/\Tors)\otimes_{\bF_2} \bF_2[U]\to \bF_2.
\]
 In this section, we describe Ozsv\'{a}th and Szab\'{o}'s construction, and state some basic properties. For notational reasons, we will focus on $\Phi_{X,\frs}(1)$.

An important component of the construction of the mixed invariant is the following definition:
\begin{define}
\label{def:admissible-cut}
An \emph{admissible cut} of a 4-manifold $X$ is a closed, connected 3-manifold $N^3\subset X$, which separates $X$ into two connected submanifolds, $X_1$ and $X_2$, meeting along $N$, such  that $b_2^+(X_i)>0$ and such that the restriction map
\[H^2(X)\to H^2(X_1)\oplus H^2(X_2)\] is an injection.
\end{define}

Given an admissible cut $N\subset X$, we construct a cobordism $W_1:S^3\to N$ by removing a 4-ball from $X_1$. We construct a cobordism $W_2:N\to S^3$ similarly. The condition that $b_2^+(X_i)>0$ ensures that both
\[
F_{W_1, \frt|_{W_1}}:\Lambda^*( H_1(W_1)/\Tors) \otimes \HF^\infty(S^3)\to \HF^\infty(N, \frt|_{N}) 
\] and
\[F_{W_2, \frt|_{W_2}}: \Lambda^*( H_1(W_2)/\Tors)\otimes  \HF^\infty(N, \frt|_{N})\to \HF^\infty(S^3)\] vanish \cite{OSTriangles}*{Lemma~8.2}. It follows that if $N$ is an admissible cut, the maps  $F_{W_1,\frt|_{W_1}}$ and $F_{W_2,\frt|_{W_2}}$ factor through $\HF_{\red}$, as in the following diagram:
\[\begin{tikzcd}\, &&& \HF^-(S^3)\arrow{d}{F_{W_1,\frt|_{W_1}}}\arrow[dashed]{dl}\\
\HF^+(N,\frt|_N)\arrow{r}\arrow[swap]{d}{F_{W_2,\frt|_{W_2}}}& \HF^+_{\red}(N,\frt|_N) \arrow[dashed]{dl}\arrow{r}{\delta}[swap]{\iso}&\HF^-_{\red}(N,\frt|_{N})\arrow{r}& \HF^-(N,\frt|_{N})\\
\HF^+(S^3)&&&
\end{tikzcd}.\]
The mixed invariant $\Phi_{X,\frt}$ is then defined as the coefficient of $U^{-1}$ in the expression 
\[
(F_{W_2,\frt|_{W_2}}\circ \delta^{-1}\circ F_{W_1,\frt|_{W_1}})(1)\in \HF^+(S^3).
\]

 More generally, if $\xi_1\in \bF_2[U]\otimes \Lambda^*(H_1(X_1)/\Tors)$ and $\xi_2\in \Lambda^* (H_1(X_2)/\Tors)$, then the invariant $\Phi_{X,\frt}(\xi_1\wedge \xi_2)\in \bF_2$ is defined as the coefficient of $U^{-1}$ in the expression 
\begin{equation}F_{W_2, \frt|_{W_2}}(\xi_2\otimes \delta^{-1}( F_{W_1, \frt|_{W_1}}^{-}(\xi_1))).\label{eq:mixedinvariantwithhomologyaction}
\end{equation}

It is also often convenient to compute the mixed invariant with coefficients in $\bF_2[[U]]$. It is easy to see that the mixed invariant over $\bF_2[[U]]$ contains the same information as the mixed invariant with coefficients in $\bF_2[U]$. The advantage is that one can sometimes use more general cuts than those in Definition~\ref{def:admissible-cut}. See \cite{OSTrianglesandSymplectic}*{Section~2} and \cite{JabukaMarkProduct}*{Definition~8.12} for examples, as well as Section~\ref{sec:mixedinvariants} of our present paper.

 \subsection{Almost complex structures, moduli spaces and transversality}
 \label{section:analyticalaspects}
  We now describe the moduli spaces which appear in this paper, and state some transversality results which will be helpful for some gluing arguments that appear in Section~\ref{sec:generalized1--handleand3--handlemaps}.

 If $(\Sigma,\as,\bs,\ws)$ is a multi-pointed Heegaard diagram, we will primarily be interested in almost complex structures on the cylindrical 4-manifold $\Sigma\times [0,1]\times \R$ which satisfy the following axioms (taken from \cite{LipshitzCylindrical}):

\begin{enumerate}[label=($J$\arabic*),leftmargin=*, widest=III]
\item\label{def:J1} $J$ is tamed by the product symplectic form.
\item\label{def:J2} $J$ is split (i.e. equal to $\frj_\Sigma\times \frj_{\bD}$) in a cylindrical neighborhood of $\ws\times [0,1]\times \R$.
\item\label{def:J3} $J$ is translation invariant in the $\R$ factor.
\item\label{def:J4} $J(\d/\d s)=\d/\d t$.
\item\label{def:J5} $J$ preserves the 2-planes $T(\Sigma\times \{(s,t)\})$ for all $(s,t)\in [0,1]\times \R$.
\end{enumerate}

For the purposes of a gluing argument in Section~\ref{sec:generalized1--handleand3--handlemaps}, these will not be generic enough, so we state an alternate fifth axiom (also from \cite{LipshitzCylindrical}):

\begin{enumerate}[label=($J$\arabic*$'$),leftmargin=*, widest=III]
\setcounter{enumi}{4}
\item \label{def:J5'} There is a 2-plane distribution $\xi$ on $\Sigma\times [0,1]$ such that the restriction of $\omega$ to $\xi$ is non-degenerate, $J$ preserves $\xi$ and the restriction of $J$ to $\xi$ is compatible with $\omega$. We further assume that $\xi$ is tangent to $\Sigma\times \{pt\}$ near $(\as\cup \bs)\times [0,1]$ and near $\Sigma\times \{0,1\}$.
\end{enumerate}

Similar to \cite{LipshitzCylindrical}, if $J$ is an almost complex structure on $\Sigma\times [0,1]\times \R$ satisfying \ref{def:J1}--\ref{def:J5}, we define the moduli space $\cM_J(\phi)$ to consist of equivalence classes of triples $(S,j,u)$, where $S$ is a Riemann surface with boundary, $n:=g(\Sigma)+|\ws|-1$ positive punctures $p_1,\dots, p_n$ and $n$ negative punctures $q_1,\dots, q_n$, and $u\colon S\to \Sigma\times [0,1]\times \R$ is a $(j,J)$-holomorphic map representing the homology class $\phi$, satisfying the following: 
 \begin{enumerate}[label=($M$\arabic*),leftmargin=*, widest=III]
 \item\label{def:M1} $S$ is smooth (not nodal).
 \item\label{def:M2} $u(\d S)\subset (\as\times \{1\}\times \R)\cup (\bs\times \{0\}\times \R)$.
 \item\label{def:M4} $\lim_{z\to p_i} (\pi_{\R}\circ u)(z)=-\infty$ and $\lim_{z\to q_i} (\pi_{\R}\circ u)(z)=\infty$.
 \item\label{def:M5} $u$ has finite energy.
  \item\label{def:M3} $\pi_{\bD}\circ u$ is locally non-constant.
 \item\label{def:M6} $u$ is an embedding.
 \end{enumerate}
 Here, we say two triples $(S',j',u')$ and $(S,j,u)$ are equivalent if they are related by a reparametrization of $S$.

We also will need to consider a weaker version of the \ref{def:M3} axiom:

\begin{enumerate}[label=($M$\arabic*$'$),leftmargin=*, widest=III]
\setcounter{enumi}{4}
\item \label{def:M3'} There is no non-empty open subset $U\subset S$ such that $\pi_{\bD}\circ u|_U$ is constant, and takes value near $\{0,1\}\times \R$ (in the sense of \ref{def:J5'}).
\end{enumerate}

It is important for our purposes to compute the expected dimension of moduli spaces.   To deal with the presence of curves which are potentially non-embedded, it is helpful to consider a refinement of the moduli space $\cM(\phi)$ which takes into account the topological source curve $S$. If $S$ is fixed Riemann surface, and $\phi$ is a homology class, we can consider the moduli space
\[\cM_J(S,\phi)\] consisting of the elements of $\cM_J(\phi)$ which have source $S$. Near a holomorphic curve $u$ where $J$ achieves transversality, $\cM_J(S,\phi)$ will be a smooth manifold of dimension equal to the Fredholm index of $D\bar{\d}$ at $u$.

 It follows from \cite{LipshitzCylindrical}*{Corollary 4.3} that if $u:S\to \Sigma\times [0,1]\times \R$ is a holomorphic curve which is an embedding, then the Fredholm index agrees with the Maslov index, so the expected dimension of $\cM(S,\phi)$ is $\mu(\phi)$. More generally, the Fredholm index satisfies
\[
\ind(u)=\mu(\phi)-2\sing(u),
\] 
where $\sing(u)$ is the local self-intersection number of $u$, in the sense of \cite{McDuffLocal}*{Section~4}. By definition, if $u$ is holomorphic, then $\sing(u)\ge 0$, and $\sing(u)=0$ if and only if $u$ is an embedding. Immersed, interior  double points contribute $\pm 1$, and immersed boundary double points contribute $\pm \tfrac{1}{2}$. See \cite{LipshitzErrata}*{Proposition~4.2'} or \cite{LOTBordered}*{Proposition~5.69} for a proof.
Suppose $\phi\in \pi_2(\xs,\ys)$ is a homology class, $p\in \Sigma\setminus (\as\cup \bs)$ is a point, and $X\subset \Sym^n(\bD)$ is a submanifold (where $n=n_p(\phi)$). We will need to consider the matched moduli space
\[
\cM(S,\phi,X):=\{u\in \cM(S,\phi): \rho^p(u)\in X\},
\] 
where $n=n_p(\phi)$ and $\rho^p:\cM(S,\phi)\to \Sym^n(\bD)$ is the map
\begin{equation}
\rho^p(u):=(\pi_{\bD}\circ u)\big((\pi_{\Sigma}\circ u)^{-1}(p)\big).\label{eq:rhopdefinition}
\end{equation} 
To simplify a few of the arguments, we focus on subsets $X\subset \Sym^n(\bD)$ which avoid the fat diagonal in $\Sym^n(\bD)$, i.e., the codimension 2 subset consisting of tuples with at least one repeated entry.

 We need the following transversality result:

\begin{prop}\label{prop:transversalitydisks}Suppose $J$ is a generic almost complex structure on $\Sigma\times [0,1]\times \R$ satisfying \textup{\ref{def:J1}--\ref{def:J5}}. Then, near a holomorphic curve $u:S\to \Sigma\times [0,1]\times \R$ satisfying \textup{\ref{def:M1}--\ref{def:M3}}, the moduli space $\cM(S,\phi)$ is a smooth manifold of dimension
\[
\ind(u)=\mu(\phi)-2\sing(u).
\] 
Similarly, if $X\subset \Sym^n(\bD)$ is a submanifold which avoids the fat diagonal, then near any curve $u\in \cM(S,\phi,X)$ satisfying \textup{\ref{def:M1}--\ref{def:M3}}, the space $\cM(S,\phi,X)$ is a smooth manifold of dimension
\[
\ind(u)=\mu(\phi)-2\sing(u)-\codim(X).
\]

If $J$ is a generic almost complex structure on $\Sigma\times [0,1]\times \R$ which satisfies \textup{\ref{def:J1}--\ref{def:J4}} and \textup{\ref{def:J5'}}, then the same statements hold at a holomorphic curve $u:S\to \Sigma\times [0,1]\times \R$ which satisfies \textup{\ref{def:M1}}, \textup{\ref{def:M2}}, \textup{\ref{def:M4}}, \textup{\ref{def:M5}} and \textup{\ref{def:M3'}},  with no multiply covered closed components, and with no components $S_0$ such that $\pi_{\bD}\circ u|_{S_0}$ is constant and takes on a value near $\{0,1\}\times \R$ (in the sense of \textup{\ref{def:J5'}}). 
\end{prop}

The proof of the statements involving the unmatched moduli spaces $\cM(S,\phi)$ can be found in \cite{LipshitzCylindrical}*{Sections~3, 4} for embedded curves. Some corrections, and proofs for curves which are not embedded can be found in \cite{LipshitzErrata}.  We refer the reader to \cite{JTNaturality}*{Section~9.3} for a proof of the statement about the matched moduli spaces $\cM(\phi,S,X)$. In analogy to the situation in \cite{MS04:HolomorphicCurvesSymplecticTopology}*{Theorem~3.4.1}, the proof that $\cM(S,\phi,X)$ is transversely cut out is substantially simplified by assuming that $X$ avoids the fat diagonal in $\Sym^n(\bD)$. We note that the condition that $X$ avoids the fat diagonal also implies that there are no multiply covered closed components.

We now describe the moduli spaces of holomorphic triangles appearing in this paper. Let $\Delta$ denote a triangular region in the complex plane, which has three boundary components, and three cylindrical ends, each identified with $[0,1]\times [0,\infty)$. As in \cite{LipshitzCylindrical}, we will primarily consider almost complex structures on $\Sigma\times \Delta$ which satisfy the following axioms:
\begin{enumerate}[label=($J'$\arabic*),leftmargin=*, widest=III]
\item\label{def:J'1} $J$ is tamed by the split symplectic form on $\Sigma\times \Delta$.
\item \label{def:J'2}There is a finite collection of points $P\subset \Sigma\setminus(\as\cup\bs\cup\gs)$ with at least one point in each component of $\Sigma\setminus (\as\cup \bs\cup \gs)$ such that $J$ is split on a product neighborhood of $P\times \Delta$. 
\item\label{def:J'3} In the cylindrical ends of $\Delta$, $J$ is equal to a cylindrical almost complex structure satisfying \ref{def:J1}--\ref{def:J5}.
\item\label{def:J'4} The projection map $\pi_\Delta:\Sigma\times \Delta\to \Delta$ is holomorphic and the tangent space of each fiber of $\pi_\Sigma$ is a complex line.
\end{enumerate}

There will be some instances when we need to consider a more generic set of almost complex structures on $\Sigma\times \Delta$. We need the following alternate axioms:

\begin{enumerate}[label=($J'$\arabic*$'$),leftmargin=*, widest=III]
\setcounter{enumi}{2}
\item\label{def:J'3'} In the cylindrical ends of $\Sigma\times \Delta$, $J$ agrees with cylindrical almost complex structures satisfying $(J1)$--$(J4)$ and $(J5')$, above.
\item\label{def:J'4'} The 2-planes of $T(\{p\}\times \Delta)$ are complex lines of $J$ for all $p\in \Sigma$.
\item\label{def:J'5'} The 2-planes of $T(\Sigma\times \{d\})$, for $d\in \Delta$, are complex lines for $J$ near $(\as\cup \bs\cup \gs)\times \Delta$ and on $\Sigma\times U$ for an open subset $U\subset \Delta$ containing the three  components of $\d \Delta$.
\end{enumerate}

Given a source surface $S$ and a homology class of triangles $\psi$, we can consider the moduli space $\cM(S,\psi)$ of curves  satisfying the natural analogs of \ref{def:M1}--\ref{def:M5}, for triangles. If $p\in \Sigma\setminus (\as\cup \bs\cup \gs)$ is a point, we can also consider the map $\rho^p:\cM(S,\psi)\to \Sym^{n_p(\psi)}(\Delta)$, defined analogously to equation~\eqref{eq:rhopdefinition}. If $X\subset \Sym^n(\Delta)$ is a subset (which we will always assume avoids the fat diagonal), then we can consider the matched moduli space $\cM(S,\psi,X)$, as before. In analogy to Proposition~\ref{prop:transversalitydisks}, we state the following transversality result:

\begin{prop}\label{prop:transversalitytriangles} Suppose that $J$ is a generic almost complex structure on $\Sigma\times \Delta$ which satisfies \textup{\ref{def:J'1}--\ref{def:J'4}}. Then near any holomorphic curve $u:S\to \Sigma\times \Delta$, satisfying the analogs of \textup{\ref{def:M1}--\ref{def:M5}} for triangles, the moduli space $\cM(S,\psi)$ is a smooth manifold of dimension
\[
\ind(u)=\mu(\psi)-2\sing(u).
\]
 If $X\subset \Sym^n(\Delta)$ is a submanifold which avoids the fat diagonal, then near any curve $u\in \cM(S,\psi,X)$ satisfying \textup{\ref{def:M1}--\ref{def:M5}}, the space $\cM(S,\psi,X)$ is a  smooth manifold of dimension
\[
\ind(u)=\mu(\psi)-2\sing(u)-\codim(X).
\]

If $J$ is a generic almost complex structure on $\Sigma\times \Delta$ which satisfies \textup{\ref{def:J'1}}, \textup{\ref{def:J'2}},  \textup{\ref{def:J'3'}}, \textup{\ref{def:J'4'}} and \textup{\ref{def:J'5'}}, then the same statements hold at any holomorphic curve $u:S\to \Sigma\times \Delta$ which satisfies the analogs of \textup{\ref{def:M1}}, \textup{\ref{def:M2}}, \textup{\ref{def:M4}}, \textup{\ref{def:M5}} and \textup{\ref{def:M3'}} for triangles, with no multiply covered closed components,  and with no components $S_0$ such that $\pi_{\Delta}\circ u|_{S_0}$ is constant and takes on a value near $\d \Delta$ (in the sense of \textup{\ref{def:J'5'}})
\end{prop}

A proof is sketched in \cite{JTNaturality}*{Section~9.3}.

\section{On Lefschetz numbers over \texorpdfstring{$\bK[[U]]$}{K[[U]]}}

\label{sec:algebraic-Lefschetz}

In this section, we prove Proposition~\ref{prop:algebraicmappingtorus}, our Lefschetz number formula.

\subsection{Background on chain complexes over \texorpdfstring{$\bK[[U]]$}{K[[U]]}}

Suppose that $\bK$ is a field of characteristic 2, and that $C$ is a finitely generated, free chain complex over $\bK[[U]]$. We are mostly interested in the case that $\bK=\bF_2$, though in Section~\ref{sec:perturbed} we consider the case where $\bK=\Lambda$, the Novikov field. We assume that $C$ has a relative $\Z/2\Z$ grading, which is lowered by $\d$, and which is preserved by the action of $U$. We define chain complexes $C^-,$ $ C^\infty$ and $C^+$ by the formulas
   \[
   C^-:=C, \qquad C^\infty:=C^-\otimes_{\bK[[U]]} \bK[[U,U^{-1}]\qquad\text{and} \qquad C^+:=C^\infty/C^-.
   \] 
   We write $H^\circ(C)$ for the homology group $H_*(C^\circ)$, for $\circ\in \{+,-,\infty\}$. 
   Write $\delta\colon H^+(C)\to H^-(C)$ for the connecting homomorphism, and define
   \[
   H_{\red}^-(C):=\ker(H^-(C)\to H^\infty(C) )\qquad \text{and} \qquad H_{\red}^+(C):=\coker(H^\infty(C)\to H^+(C)).
   \]
     The connecting homomorphism $\delta$ induces an isomorphism from $H_{\red}^+(C)$ to $H_{\red}^-(C)$. 

\begin{rem}
In the context of Heegaard Floer homology, one often works with chain complexes over $\bK[U]$, instead of $\bK[[U]]$. For the purposes of computing Lefschetz numbers on $H_{\red}^{+}(C)$, we lose no generality by working over $\bK[[U]]$. Indeed if $C$ is a finitely generated, free chain complex over $\bK[U]$, then $H_{\red}^{+}(C)\iso H_{\red}^{+}(C\otimes \bK[[U]])$, since $\bK[[U]]$ is a flat $\bK[U]$-module, so $H_{\red}^{+}(C\otimes \bK[[U]])\iso H_{\red}^{+}(C)\otimes \bK[[U]]$. On the other hand, for large $n$, $U^n$ annihilates $H_{\red}^{+}(C)$, so $H_{\red}^{+}(C)\otimes \bK[[U]]\iso H_{\red}^{\pm}(C)$.
\end{rem}

We recall some notation. If $\cR$ is a ring,  we say a chain complex $C$ is a \emph{1-step complex} if $C\iso \cR$ with vanishing differential. We say $C$ is a \emph{2-step complex} if $C\iso \cR\oplus \cR$, with generators $\ve{a}$ and $\ve{b}$, and the differential is given by $\d \ve{a}=p\cdot \ve{b}$ and $\d \ve{b}=0$, for some non-zero $p\in \cR$. We have the following:

\begin{lem}\label{lem:classificationfgPID}If $C$ is a free, finitely generated chain complex over $\bK[[U]]$, then $C$ is chain isomorphic to a direct sum of 1-step complexes and 2-step complexes of the form 
\[
\ve{a}\xrightarrow{U^n} \ve{b}.
\]
\end{lem}
\begin{proof}The classification theorem for finitely generated chain complexes over a PID (see, e.g., \cite{HMZConnectedSum}*{Lemma~6.1}) says that any free, finitely generated chain complex over a PID $\cR$ decomposes into a direct sum of 1-step complexes, and 2-step complexes $\ve{a}\xrightarrow{p} \ve{b}$, for various $p\in \cR$. In our case, $\cR=\bK[[U]]$, and we need to reason that for any 2-step complex which appears, the element $p$ can be taken to be a nonnegative power of $U$. Write $p=U^n(\a+Uq(U))$ for some $q(U)\in \bK[[U]]$, and $\a\in \bK^\times$. Since $\a+Uq(U)$ is a unit in $\bK[[U]]$, the complex 
\[
\ve{a}\xrightarrow{U^n(\a+Uq(U))} \ve{b}
\] 
is chain isomorphic to the complex 
\[
\ve{a}'\xrightarrow{U^n} \ve{b}'
\]
 via the map $\ve{a}\mapsto \ve{a}'$ and $\ve{b}\mapsto (\a+Uq(U))^{-1}\cdot \ve{b}'$.
\end{proof}

\subsection{Trace and cotrace maps}

Let $C$ be a free and finitely generated chain complex over $\bK[[U]]$. Recall the natural trace map
\begin{equation}\tr: C\otimes_{\bK[[U]]} C^\vee\to \bK[[U]]\label{eq:trace}\end{equation} defined by the formula $\tr(\xs\otimes \ys)=\ys(\xs)$. Since $C$ is finitely generated, there is a cotrace map
\[\cotr: \bK[[U]]\to C\otimes_{\bK[[U]]} C^\vee,\]
which can be defined as the dual of a trace map. If $\xs_1,\dots, \xs_n$ is a basis for $C$, then the cotrace map takes the form
\[\cotr(1)=\sum_{i=1}^n \xs_i\otimes  \xs_i^\vee.\] The trace and cotrace maps are easily seen to be chain maps.

\subsection{The  endomorphism $\Phi$}
\label{sec:Phi-def}
We now describe our special endomorphism $\Phi:C\to C$. Suppose that $C$ is a free, finitely generated chain complex over $\bK[[U]]$, with a chosen basis $B=\{\xs_1,\dots, \xs_n\}$. (The construction also works over $\bK[U]$, as above)

 We can write
\begin{equation}
\d\ve{x}_i=\sum_{j=1}^n P_{i,j}\cdot \ve{x}_j,\label{eq:Pijdef}
\end{equation}
 for $P_{i,j}\in \bK[[U]]$. Let $P_{i,j}'$ denote the derivative of $P_{i,j}$ with respect to $U$. We  define the map
\[
\Phi_B:C\to C
\] by the formula
\[
\Phi_B(\xs_i)=\sum_{j=1}^n P_{i,j}' \cdot \xs_j.
\]
 Viewing $\d$ as a matrix over the basis $B$, we can differentiate the expression $\d^2=0$ using the Leibniz rule to see that $\Phi_B$ is a chain map. Similarly, if $x\in C$, applying the Leibniz rule to the expression $\Phi_B(x)$ (viewed as a product of a matrix and a column vector) implies the relation
\begin{equation}
\Phi_B=\d\circ \frac{d}{d U}\bigg|_B+\frac{d}{d U}\bigg|_B \circ \d.
\label{eq:Phi-null-homotopic}
\end{equation}
 The map $d/d U|_B$ is defined by writing an element $x\in C$ in terms of the basis $B$, and then differentiating the coefficients of the basis elements. The map $d/d U|_B$ does not commute with the action of $U$. Note that equation~\eqref{eq:Phi-null-homotopic} implies that $\Phi_B$ is chain homotopic to 0, but not $U$-equivariantly.

The map $\Phi_B$ is independent of the chosen basis, up to chain homotopy, in the following sense:
   
   \begin{lem}\label{lem:Phiwelldefineduptochainhomotopy}If $B_1$ and $B_2$ are two bases of $C$ over $\bK[[U]]$, then the maps $\Phi_{B_1}$ and $\Phi_{B_2}$ are chain homotopic over $\bK[[U]]$. 
   \end{lem}
   \begin{proof}Consider the more general situation, where $(C_1,\d_1)$ and $(C_2,\d_2)$ are chain complexes over $\bK[[U]]$ with bases $B_1$ and $B_2$ and $F:C_1\to C_2$ is an $\bK[[U]]$-equivariant chain map. Differentiating the matrix equation (written in terms of the bases $B_1$ and $B_2$)
   \[
   F\circ \d_1+\d_2\circ F=0,
   \]  
   we get that
   \[
   F\circ\Phi_{B_1}+\Phi_{B_2}\circ F\simeq 0.
   \]
    By specializing to the case that $(C_1,\d_1)=(C_2,\d_2)$ and $F:C_1\to C_2$ is the identity map, but the bases $B_1$ and $B_2$ are different, we obtain the lemma statement.
   \end{proof}

We henceforth write just $\Phi$, for the map $\Phi_B$ for some chosen basis $B$.

\subsection{The Lefschetz number formula}

We now prove our main Lefschetz number formula, Proposition~\ref{prop:algebraicmappingtorus} of the introduction.

We being with a helpful algebraic lemma:
\begin{lem}\label{lem:0-map-tensor}
Suppose $C$ is a free, finitely generated complex over $\bK[[U]]$, and $F\colon C\to C$ is a chain map such that $F_*\colon H^\infty(C)\to H^\infty(C)$ vanishes. Then $F\otimes \id$ and $\id \otimes F^\vee$ vanish on $H^\infty(C\otimes C^\vee)$.
\end{lem}
\begin{proof} Clearly it is sufficient to consider just $F\otimes \id$. We use Lemma~\ref{lem:classificationfgPID} to write $C$ as a direct sum of two complexes, $C_1$ and $C_2$, where $C_1$ is a direct sum of 1-step complex, and $C_2$ is a direct sum of 2-step complexes of the form $\ve{a}\xrightarrow{U^n} \ve{b}$. Over the ring $\bK[[U,U^{-1}]$, the complex $C_2$ is homotopy equivalent to the 0 complex, since $U^n$ is a unit in $\bK[[U,U^{-1}]$. In particular, the canonical inclusion $I$ of $C_1$ into $C$, and the canonical projection $\Pi$ of $C$ onto $C_1$ are both homotopy equivalences. Hence, $F$ vanishes on $H^\infty(C)$ if and only if $\Pi\circ F\circ I=0$.

Suppose $F$ vanishes on $H^\infty(C)$ so $\Pi\circ F\circ I=0$, by above. Then
\[
(\Pi\otimes I^\vee)\circ (F\circ \id) \circ (I\otimes \Pi^\vee)=(\Pi\circ F\circ I)\otimes (I^\vee\circ \Pi^\vee)=0,
\]
which implies that $F\otimes \id\simeq 0$, since $\Pi\otimes I^\vee$ and $I\otimes \Pi^\vee$ are homotopy equivalences between $C\otimes C^\vee$ and $C_1\otimes C_1^\vee$.
\end{proof}

\begin{rem} Lemma~\ref{lem:0-map-tensor} fails over $\bK[U]$. For example, if $C$ is the 2-step complex $\ve{a}\xrightarrow{1+U} \ve{b}$ and $F=\Phi$, it is easy to check that $\Phi\otimes \id$ is non-zero on $H^\infty(C\otimes C^\vee)$.
\end{rem}

Suppose that $C$ is a free, finitely generated chain complex over $\bK[[U]]$ and $F:C\to C$ is a $\Z/2\Z$-grading preserving chain map. By Lemma~\ref{lem:0-map-tensor}, both of the maps $(F\otimes \id)$ and $\id \otimes \Phi^\vee$ vanish on $H^\infty(C\otimes C^\vee)$. Hence, we may define a quantity $\Delta(C,F)\in \bK$ to be the coefficient of $U^{-1}$ in the expression
\begin{equation}
\left(\tr\circ (F\otimes \id)\circ \delta^{-1}\circ(\id\otimes \Phi^\vee) \circ \cotr\right)(1).\label{eq:DeltaCFdefinition}
\end{equation}

The main algebraic result of this section is the following:
\begin{customprop}{\ref{prop:algebraicmappingtorus}}
 Suppose that $\bK$ is a field of characteristic 2, and $C$ is a finitely generated, relatively $\Z/2\Z$ graded, free chain complex over $\bK[[U]]$. Suppose that $F\colon C\to C$ is a chain map which preserves the relative $\Z/2\Z$ grading, and such that the induced map $F^\infty_*\colon H_*(C^\infty)\to H_*(C^\infty)$ vanishes. Then
 \[
\Lef\left(F_* \colon H_{\red}^+(C)\to H_{\red}^+(C) \right) \in \bK
 \]
 coincides with $\Delta(C,F)$.
\end{customprop}
\begin{proof}
Applying Lemma~\ref{lem:classificationfgPID} shows that $C$ can be decomposed as a direct sum of 1-step and 2-step complexes of the form $\ve{a}\xrightarrow{U^n} \ve{b}$. 

We first consider the case that $C$ is a 1-step complex. In this case, $\Phi=0$, since the differential vanishes. Hence $\Delta(C,\id)=0$. On the other hand, $H^+_{\red}(C)=0$, so the claim holds.

 We now consider the claim when $C=(\ve{a}\xrightarrow{U^n} \ve{b})$ and $F=\id$. In this case
\[
\Lef\big(F_*:H^+_{\red}(C)\to H^+_{\red}(C)\big)=\chi(H^+_{\red}(C))=n.
\] 
The dual complex $C^\vee$ is the 2-step complex $\ve{b}^\vee\xrightarrow{U^n} \ve{a}^\vee$. The complex $(C\otimes C^\vee)^-$ is shown below:
  \[(C\otimes C^\vee)^-=\begin{tikzcd}& \ve{a}\ve{b}^\vee\arrow[swap]{dl}{U^n} \arrow{dr}{U^n}&\\
  \ve{bb}^\vee\arrow[swap]{dr}{U^n} && \ve{aa}^\vee\arrow{dl}{U^n}\\
  & \ve{ba}^\vee&
  \end{tikzcd}.\]
  With this notation, the map $1\otimes \Phi^\vee$ takes the form
  \[(1\otimes \Phi^\vee)=\begin{tikzcd}& \ve{ab}^\vee \arrow{dr}{nU^{n-1}}&\\
  \ve{bb}^\vee\arrow{dr}{nU^{n-1}} && \ve{aa}^\vee\\
  & \ve{ba}^\vee&
  \end{tikzcd}.\]

   It is easy to compute that $H^-_{\red}(C\otimes C^\vee)$ is generated over $\bF_2$ by the classes $[U^i\cdot(\ve{a}\ve{a}^\vee+\ve{b}\ve{b}^\vee)]$ and $[U^i\cdot \ve{b}\ve{a}^\vee]$ for $0\le i\le n-1$, with no relations.
   
   Similarly $H^+_{\red}(C\otimes C^\vee)$ is generated over $\bK$ by the classes $[U^i\cdot \ve{a}\ve{b}^\vee]$, $[U^i\cdot \ve{a}\ve{a^\vee}]$ and $[U^i\cdot\ve{b}\ve{b}^\vee]$, for $-n\le i\le -1$, with the only relation being $[U^i\cdot\ve{a}\ve{a}^\vee]=[U^i\cdot\ve{b}\ve{b}^\vee]$.

  The connecting homomorphism $\delta$ satisfies
  \[
  \delta([U^i\cdot \ve{a}\ve{b}^\vee])=[U^{i+n}\cdot(\ve{aa}^\vee+\ve{bb}^\vee)]\qquad \text{and} \qquad \delta([U^i\cdot\ve{aa}^\vee])=[U^{i+n}\cdot \ve{ba}^\vee].
  \]
   It is now an easy matter to compute
  \[
  (\tr\circ \id\circ \delta^{-1}\circ (\id\otimes\Phi^\vee)\circ \cotr) (1)=nU^{-1},
  \] 
  which verifies the claim in this case, since $n=\chi(H^+_{\red}(C))=\Lef(\id:H^+_{\red}(C)\to H^+_{\red}(C))$.
  
  We now consider the case that $C$ is still the 2-step complex $\ve{a}\xrightarrow{U^n} \ve{b}$, but $F:C\to C$ is an arbitrary chain map which preserves the relative $\Z/2\Z$-grading. Since $F$ preserves the relative grading, it follows that $F(\ve{a})=p(U) \cdot\ve{a}$ and  $F(\ve{b})=q(U)\cdot \ve{b}$ for some $p(U),q(U)\in \bK[[U]]$. Since $F$ is a chain map, it follows that $p(U)=q(U)$. Hence $F$ is equal to multiplication by $p(U)$, for some $p(U)\in \bK[[U]]$. Write
  \[
  p(U)=\alpha+U p_0(U),
  \]
   where $\alpha\in  \bK$ and $p_0(U)\in \bK[[U]]$. Clearly
  \[
  \Lef\big(F_*:H^+_{\red}(C)\to H^+_{\red}(C)\big)=\alpha n.
  \]
   On the other hand, it is easy to compute that
  \[
  (\tr\circ (F\otimes\id )\circ \delta^{-1}\circ(\id \otimes \Phi^\vee) \circ \cotr)(1)=\alpha nU^{-1}.
  \] 
  In particular,
   \[
   \Delta(C,F)=\Lef\big(F_*:H^+_{\red}(C)\to H^+_{\red}(C)\big).
   \]

  Finally, we consider the case that $F$ is an arbitrary sum of 1 and 2-step complexes. Write
  \[
C=C_1\oplus \cdots \oplus C_n
  \]
  where each $C_{1}$ is either a 1-step complex, or a 2-step complex, as above. 
    We have
  \[
  H^+_{\red}(C)=H^+_{\red}(C_1)\oplus \cdots \oplus H^+_{\red}(C_n),
  \]
   and we can write
    \[
    C\otimes C^{\vee}=\sum_{i,j} C_i\otimes C_j^\vee,
    \]
   We can decompose $F$ and $\Phi^\vee$ as  
  \[
  F=\sum_{1\le i,j\le n} F_{i,j},\qquad\text{and} \qquad \Phi^\vee=\sum_{k=1}^n \Phi_k^\vee \]
   where $F_{i,j}=I_j\circ \Pi_j\circ F\circ I_i\circ \Pi_i$, where $\Pi_i:C\to C_i$ is projection and $I_i\colon C_i\to C$ is inclusion. The maps $\Phi_k^\vee\colon C_k^\vee\to C_k^\vee$ are defined similarly. (Note that only one index is necessary for the decomposition of $\Phi^\vee$, since $\Phi$ maps $C_k$ to $C_k$).
   
   In a similar manner, the trace and cotrace maps on $C\otimes C^\vee$ decompose as
  \[
  \cotr=\sum_{i=1}^n \cotr_i, \qquad \text{and} \qquad \tr=\sum_{i=1}^n \tr_i,
  \]
   where $\tr_i$ and $\cotr_i$ are the trace and cotrace maps on $C_i\otimes C_i^\vee$. We can write 
  \[
  \delta=\sum_{1\le i,j\le n}\delta_{i}^j
  \]
   where $\delta_{i}^j$ is the connecting homomorphism in the minus-infinity-plus long exact sequence for the subcomplex $C_i\otimes C_j^\vee$. Finally,  we compute 
  \begin{align*}&(\tr\circ (F\otimes \id)\circ \delta^{-1}\circ (\id \otimes \Phi^\vee)\circ \cotr)(1)\\
  =&\Bigg(\bigg(\sum_{i=1}^n \tr_i\bigg)\circ \bigg(\sum_{1\le i,j\le n} F_{i,j}\otimes \id\bigg)\circ \bigg(\sum_{1\le i,j\le n} (\delta_i^j)^{-1}\bigg) \circ \bigg(\sum_{k=1}^n \id \otimes \Phi_k^\vee\bigg)\circ \bigg(\sum_{i=1}^n \cotr_i\bigg) \Bigg)(1)\\
  =&\sum_{i=1}^n (\tr_i \circ (F_{i,i}\otimes\id)\circ (\delta_{i}^{i})^{-1}\circ (\id \otimes \Phi^\vee_i)\circ  \cotr_i)(1).
  \end{align*}   
   By our result for 1-step and 2-step complexes, the $U^{-1}$ coefficient of the above expression is exactly
  \[
  \sum_{i=1}^n \Lef\big((F_{i,i})_*:H^+_{\red}(C_i)\to H^+_{\red}(C_i)\big),
  \]
   which is  $\Lef\big(F_*:H^+_{\red}(C)\to H^+_{\red}(C) \big)$, completing the proof.
\end{proof}

\section{The graph TQFT for Heegaard Floer homology}
\label{sec:graphTQFT}
In this section, we provide an overview of the graph TQFT for Heegaard Floer homology \cite{ZemGraphTQFT}, and prove some properties which are relevant to this paper. The graph TQFT uses the following notion of cobordism between multi-pointed 3-manifolds:

\begin{define}
\begin{enumerate}
\item A \emph{ribbon graph} is a graph with no valence zero vertices, together with a choice of cyclic ordering of the edges adjacent to each vertex.
\item A \emph{ribbon graph cobordism} $(W,\Gamma):(Y_1,\ve{w}_1)\to (Y_2,\ve{w}_2)$ between two multi-pointed 3-manifolds is a pair consisting of a 4-manifold $W$ with $\d W=-Y_1\sqcup Y_2$ as well as a ribbon graph $\Gamma\subset W$ such that $\Gamma\cap Y_i=\ve{w}_i$ and each basepoint of $\ve{w}_i$ has valence 1 in $\Gamma$.
\end{enumerate}
\end{define}

To a ribbon graph cobordism $(W,\Gamma):(Y_1,\ve{w}_1)\to (Y_2,\ve{w}_2)$, equipped with a $\Spin^c$ structure $\frs$, the author \cite{ZemGraphTQFT} associates two cobordism maps
\[
F_{W,\Gamma,\frs}^A, \quad F_{W,\Gamma,\frs}^B: \CF^-(Y_1,\ve{w}_1,\frs|_{Y_1})\to \CF^-(Y_2,\ve{w}_2,\frs|_{Y_2}).
\] 

 The graph cobordism maps coincide with Ozsv\'{a}th and Szab\'{o}'s cobordism maps from \cite{OSTriangles} when $Y_1$ and $Y_2$ are connected and $\Gamma$ is a path connecting $Y_1$ to $Y_2$ \cite{ZemGraphTQFT}*{Theorem~B}. In particular, the type-$A$ maps and type-$B$ maps coincide for such cobordisms. More generally, there is a symmetry
\[
F_{W,\Gamma,\frs}^A\simeq F_{W,\bar{\Gamma},\frs}^B,
\]
where $\bar{\Gamma}$ is the graph obtained by reversing the cyclic orders of $\Gamma$ \cite{HMZConnectedSum}*{Lemma~5.9}. Consequently the set of type-$A$ maps and the set of type-$B$ maps contain the same information, though an asymmetry in the construction makes it more convenient to work with both versions.

In Sections~\ref{sec:mapsandrelations} and \ref{sec:outlineofconstruction}, we provide an outline of the construction. In Sections~\ref{sec:furtherrelationsofGraphTQFT}, \ref{sec:graphsandhomologyactions} and \ref{sec:relativehomologyandTriangles} we prove several useful relations for our paper.

\subsection{Ingredients of the graph TQFT}
\label{sec:mapsandrelations}

The following  maps are used in the construction of the graph cobordism maps in \cite{ZemGraphTQFT}:
\begin{enumerate}
\item \label{graphmap1} Handle attachment maps for 1-, 2-, and 3-handles, attached away from the basepoints.
\item \label{graphmap2}Handle attachment maps for 0- and 4-handles, which add or remove a copy of $S^3$, with a single basepoint.
\item \label{graphmap3}\textit{Free-stabilization maps} for adding or removing basepoints in a 3-manifold.
\item \label{graphmap4}\textit{Relative homology maps} associated to paths between two  basepoints in a multi-pointed 3-manifold.
\end{enumerate}

The original cobordism maps from \cite{OSTriangles} are built as a composition of maps 1-, 2- and 3-handles. We now describe the maps of type~\eqref{graphmap2}, \eqref{graphmap3} and \eqref{graphmap4}, which are new to the construction in \cite{ZemGraphTQFT}.

The 0-handle and 4-handle maps are defined using the canonical isomorphism
\[\CF^-(Y\sqcup S^3, \ws\cup \{w_0\},\frs\sqcup \frs_0)\iso \CF^-(Y,\ws,\frs)\otimes_{\bF_2[U]} \CF^-(S^3,w_0,\frs_0).\] Under this isomorphism, the 0-handle map $F_0$ and the 4-handle map $F_4$ take the form
\[
F_0(\ve{x})=\ve{x}\otimes \ve{c}_0\qquad \text{and} \qquad F_4(\ve{x}\otimes \ve{c}_0)=\ve{x},
\] 
where $\ve{c}_0\in \CF^-(S^3,w_0,\frs)$ is a cycle which generates the homology group $\HF^-(S^3,w_0,\frs_0)\iso \bF_2[U]$.
The free-stabilization maps
\[S_{w}^+:\CF^-(Y,\ve{w},\frs)\to \CF^-(Y,\ve{w}\cup \{w\},\frs)\] and
\[S_w^-:\CF^-(Y,\ve{w}\cup \{w\},\frs)\to \CF^-(Y,\ve{w},\frs)\]  are constructed somewhat analogously to the 1-handle and 3-handle maps defined by Ozsv\'{a}th and Szab\'{o}. One picks a diagram $(\Sigma,\as,\bs,\ws)$ for $(Y,\ws)$ such that $w\in \Sigma\setminus (\as\cup \bs)$. A diagram $(\Sigma,\as\cup \{\alpha_0\},\bs\cup \{\beta_0\},\ws\cup \{w\})$ for $(Y,\ws\cup \{w\})$ is constructed by  adding the basepoint  $w$, as well two new curves, $\alpha_0$ and $\beta_0$, both contained in a small disk on $\Sigma\setminus (\as\cup \bs\cup \ws)$, such that $|\alpha_0\cap \beta_0|=2$. Writing $\theta^+$ and $\theta^-$ for the higher and lower graded intersection points of $\alpha_0\cap \beta_0$, the free-stabilization maps are defined by the formulas
\[S_w^+(\ve{x})=\ve{x}\times \theta^+,\]
\[S_w^-(\ve{x}\times \theta^+)=0\qquad \text{and} \qquad S_w^-(\ve{x}\times \theta^-)=\ve{x},\]
extended $\bF_2[U]$-equivariantly. For appropriately stretched almost complex structure, these are chain maps, which commute with the transition maps for changing the Heegaard diagram $(\Sigma,\as,\bs,\ws)$.  See \cite{ZemGraphTQFT}*{Section~6} for more details on the construction. They turn out to be the graph cobordism maps induced by the graphs in $Y\times [0,1]$ which are shown on the left side of Figure~\ref{fig::29}. 

The formulas for the free-stabilization maps resemble the formulas for the 1-handle and 3-handle maps from \cite{OSTriangles}. This is, of course, no accident. We can decompose the cobordism for $S_w^+$ as a 0-handle (which adds a copy of $S^3$ and the basepoint $w$), followed by a 1-handle  which cancels the 0-handle topologically, but leaves the basepoint $w$ in $Y$. Analogously, the graph cobordism for $S_{w}^-$ can be written as a composition of a 3-handle (with attaching 2-sphere bounding a small ball containing the basepoint $w$), followed by a 4-handle.

Another map which appears in the graph TQFT is the map $\Phi_w$,  which is an endomorphism of $\CF^-(Y,\ve{w},\frs)$ when $w\in \ve{w}$. It is defined on $\CF^-(Y,\ws,\frs)$ by the formula
\[
\Phi_w(\ve{x})=U^{-1}\sum_{\ve{y}\in \bT_{\a}\cap \bT_{\b}} \sum_{\substack{\phi\in \pi_2(\ve{x},\ve{y})\\
\mu(\phi)=1}}n_w(\phi) \# \hat{\cM}(\phi) U^{n_{\ve{w}}(\phi)}\cdot \ve{y},
\]
extended $\bF_2[U]$-equivariantly. In Lemma~\ref{lem:phi=brokenpathcobordism}, we prove that the broken graph cobordism $(Y\times [0,1], \Gamma_w)$ on the right side of  Figure~\ref{fig::29} induces the map $\Phi_w$. As the cobordism for $\Phi_w$ in Figure~\ref{fig::29} suggests, the map $\Phi_w$ satisfies $\Phi_w\simeq S_w^+S_w^-$ when $w$ is not the only basepoint (so that $S_w^+$ and $S_w^-$ can be defined). 

\begin{rem} If $w$ is the only basepoint on $Y$, then the map $\Phi_w$ coincides with the formal derivative of the differential on $\CF^-(Y,w,\frs)$, as defined in Section~\ref{sec:Phi-def}. If $\ws=\{w_1,\dots, w_n\}$ is a collection of basepoints on $Y$, then the formal derivative map of $\CF^-(Y,\ws,\frs)$ is instead chain homotopic to $\Phi_{w_1}+\cdots +\Phi_{w_n}$.
\end{rem}

\begin{figure}[ht!]
	\centering
	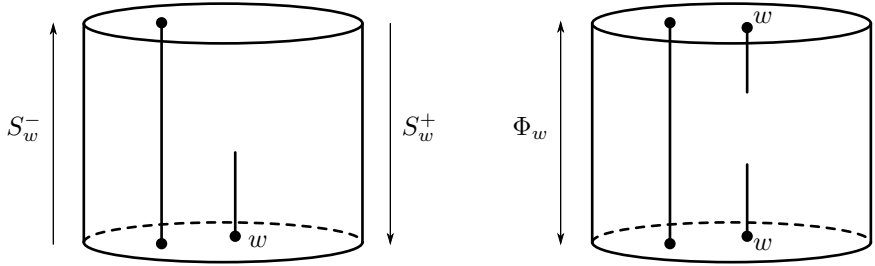
	\caption{\textbf{Left: the graph cobordism inducing the free-stabilization maps $S_w^+$ and $S_w^-$. Right: the broken path graph cobordism $(Y\times [0,1], \Gamma_w)$ inducing the map $\Phi_w$.} The maps $S_w^+$ and $S_w^-$ are only defined in the presence of an additional basepoint (so that both ends have at least one basepoint), whereas the map $\Phi_w$ is defined regardless of whether there are additional basepoints or not. \label{fig::29}}
\end{figure}

The next maps which feature in the construction of the graph TQFT are the relative homology maps. Suppose $(\Sigma,\as,\bs,\ve{w})$ is a multi-pointed Heegaard diagram, and $\lambda$ is an immersed path on $\Sigma$ between two basepoints $w_1$ and $w_2$. If $\phi\in \pi_2(\xs,\ys)$, we can define a quantity $a(\lambda,\phi)\in \bF_2$ by summing the changes in multiplicity of $\phi$ across each of the $\as$ curves as one travels along the path $\lambda$. Using the quantities $a(\lambda,\phi)$, we can define a $-1$ graded endomorphism
\[A_{\lambda}:\CF^-(\Sigma,\as,\bs,\ws,\frs)\to \CF^-(\Sigma,\as,\bs,\ve{w},\frs)\]  using the formula
\begin{equation}
A_{\lambda}(\ve{x})=\sum_{\ve{y}\in \bT_{\a}\cap \bT_{\b}}\sum_{\substack{\phi\in \pi_2(\xs,\ys)\\ \mu(\phi)=1}} a(\phi,\lambda) \# \hat{\cM}(\phi)U^{n_{\ws}(\phi)}\cdot \ys,
\label{def:rel-homology-action}
\end{equation}
extended $\bF_2[U]$-equivariantly.

  By counting the ends of 2-dimensional moduli spaces, one obtains the equality
\[\d \circ A_\lambda+A_\lambda\circ \d=0.\] Note that one can consider more general Heegaard Floer complexes than the ones considered in this paper, where one  associates a variable to each basepoint. In this more general setting, if $\lambda$ is a path from $w_1$ to $w_2$, then $\d \circ A_{\lambda}+A_{\lambda}\circ \d=U_{w_1}+U_{w_2}$. See \cite{ZemGraphTQFT}*{Lemma~5.1}.

One could replace the factor of $a(\lambda,\phi)$ in the definition of $A_\lambda$, with the symmetric quantity $b(\lambda,\phi)$, obtained by counting changes across only the $\bs$ curves. Doing so yields a map $B_\lambda$, which is also a chain map. For a path $\lambda$ with ends on two basepoints, $w_1$ and $w_2$, the maps $A_\lambda$ and $B_{\lambda}$ are in general not equal, or even chain homotopic. Instead, since the quantities $a(\lambda,\phi)$ and $b(\lambda,\phi)$ satisfy
\[a(\lambda,\phi)+b(\lambda,\phi)=n_{w_1}(\phi)-n_{w_2}(\phi),\] it follows that $A_{\lambda}$ and $B_{\lambda}$ satisfy the  relation
\begin{equation}A_{\lambda}+B_{\lambda}=U\Phi_{w_1}+U\Phi_{w_2}.\label{eq:Alambda+Blambda}\end{equation}

For an immersed, closed curve $\gamma$ in $\Sigma$, one can define maps $A_{\gamma}$ and $B_{\gamma}$, using the same formula as the maps $A_{\lambda}$ and $B_{\lambda}$. When $\gamma$ is a closed curve in $\Sigma$, the maps $A_{\gamma}$ and $B_{\gamma}$ are chain maps. In contrast to equation~\eqref{eq:Alambda+Blambda}, since $a(\gamma,\phi)=b(\gamma,\phi)$ when $\gamma$ is a closed loop, we have 
\[
A_\gamma=B_{\gamma}.
\] 
The map $A_\gamma$ is the map induced by the action of $H_1(Y)/\Tors$ described in \cite{OSDisks}, using the homology class induced by $\gamma$ under the inclusion $\Sigma\hookrightarrow Y$.

We now list some algebraic relations from \cite{ZemGraphTQFT} which will be useful for this paper:

\begin{enumerate}[label=($R$\arabic*), leftmargin=*, widest=IIII]
\item \label{rel:R1}$S_{w}^-S_{w}^+\simeq 0$ \cite{ZemGraphTQFT}*{Lemma~6.15}.
\item \label{rel:R2}$S_{w}^+S_{w}^-\simeq \Phi_w$ \cite{ZemGraphTQFT}*{Lemma~14.15}.
\item \label{rel:R3'} $S_{w'}^{\circ} \Phi_w\simeq \Phi_w S_{w'}^{\circ}$ if $w\neq w'$ and $\circ \in \{+,-\}$ \cite{ZemGraphTQFT}*{Lemma~14.17}.
\item\label{rel:R3} $A_{\lambda_1}A_{\lambda_2}+A_{\lambda_2}A_{\lambda_1}\simeq \#((\d \lambda_1)\cap (\d \lambda_2))\cdot  U$ \cite{ZemGraphTQFT}*{Lemma~5.4}.
\item\label{rel:R4} If $\lambda$ is a path from $w_1$ to $w_2$ and $\lambda'$ is a path from $w_2$ to $w_3$ then $A_{\lambda}+A_{\lambda'}=A_{\lambda'*\lambda}$, where $*$ denotes concatenation \cite{ZemGraphTQFT}*{Lemma~5.3}.
\item\label{rel:R6} If $\lambda$ is a path from $w_1$ to $w_2$ and $w_1\neq w_2$, then $A_{\lambda}^2\simeq U$ \cite{ZemGraphTQFT}*{Lemma~5.5}.
\item\label{rel:R6'} If $\lambda$ is a path from $w$ to another basepoint, then $\Phi_w A_{\lambda}+A_{\lambda}\Phi_w\simeq \id$ \cite{ZemGraphTQFT}*{Lemma~14.16}.
\item \label{rel:R7} $S_{w}^{\circ}S_{w'}^{\circ'}\simeq S_{w'}^{\circ'}S_{w}^{\circ}$, for $\circ,\circ'\in \{+,-\}$ \cite{ZemGraphTQFT}*{Proposition~6.14}.
\item \label{rel:R7'} If $e$ is an edge which we can write as the concatenation of two edges, $e_1$ and $e_2$, whose intersection consists of a single vertex $v$, then $A_e\simeq S_v^- A_{e_2}A_{e_1} S_v^+$ \cite{ZemGraphTQFT}*{Lemma~7.11}. 
\item \label{rel:R8} If $\lambda$ is a path from $w$ to another basepoint, then $S_{w}^-A_{\lambda}S_{w}^+\simeq \id$ \cite{ZemGraphTQFT}*{Lemma~7.10}.
\item \label{rel:R9}If $\lambda$ is a path from $w_1$ to $w_2$ (and $w_1\neq w_2$) then $S_{w_1}^- A_{\lambda} S_{w_2}^+\simeq \phi_*$, where $\phi$ is a diffeomorphism of $Y$ which moves $w_1$ to $w_2$ along $\lambda$, and is supported in a neighborhood of $\lambda$ \cite{ZemGraphTQFT}*{Theorem~14.11}.
\end{enumerate}

\subsection{Outline of the construction of the graph cobordism maps}
\label{sec:outlineofconstruction}
We now briefly summarize the construction of the graph cobordism maps, in terms of the maps described in the last section.

The first step is to define maps for graph cobordisms of the form $(Y\times [0,1],\Gamma):(Y,\ws_1)\to (Y,\ws_2)$. For such cobordisms, the maps are defined as a composition of the free-stabilization maps, and the relative homology maps, as we now describe. It is more convenient to project the graph $\Gamma\subset Y\times [0,1]$ into $Y$, and define a map for a ribbon graph embedded in $Y$, which has designated incoming and outgoing vertices. We call such a graph, embedded in $Y$, a \emph{ribbon flow graph}, and write $\Gamma:\ws_1\to \ws_2$. For a ribbon flow graph $\Gamma:\ws_1\to \ws_2$ in $Y$, a map
\[A_{\Gamma}: \CF^-(Y,\ws_1,\frs)\to \CF^-(Y,\ws_2,\frs)\] is constructed in \cite{ZemGraphTQFT}*{Section~7}, which we call the \emph{type-$A$ graph action map}. The type-$A$ graph cobordism map for a graph cobordism $(Y\times [0,1],\Gamma)$ is equal to the type-$A$ graph action map for the graph obtained by projecting $\Gamma$ into $Y$.

Given a flow graph $\Gamma:\ws_1\to \ws_2$ in $Y$, to construct the graph action map, one decomposes the graph $\Gamma$ into a sequence of \textit{elementary flow graphs}, which are flow-graphs taking one of the following three forms: 

\begin{enumerate}[label= ($E\Gamma$-\arabic*):, ref=Type ($E\Gamma$-\arabic*),leftmargin=*, widest=IIII]
\item\label{elementarygraphtype1} (\emph{translation}) $|\ws_1|=|\ws_2|$ and each edge of $\Gamma$ connects a vertex in $\ws_1$ to a vertex in $\ws_2$.
\item\label{elementarygraphtype2} (\emph{interior vertex}) There is a single vertex $v_0$ of $\Gamma$ which is not in $\ws_1$ or $\ws_2$, and all edges of $\Gamma$ connect either $\ws_1$ to $\ws_2$, or connect a point of $\ws_1$ or $\ws_2$ to $v_0$.
\item\label{elementarygraphtype3} (\emph{local extrema}) $|\ws_1|=|\ws_2|\pm 2$, and all edges of $\Gamma$, except for a single edge $e$, connect $\ws_1$ to $\ws_2$. Furthermore $e$ connects two vertices of $\ws_1$ together, or connects two vertices of $\ws_2$ together.
\end{enumerate}

Examples of elementary flow-graphs are shown in Figure~\ref{fig::56}. We call such a decomposition of a flow-graph into elementary flow-graphs a \textit{Cerf decomposition} of the graph.

\begin{figure}[ht!]
	\centering
	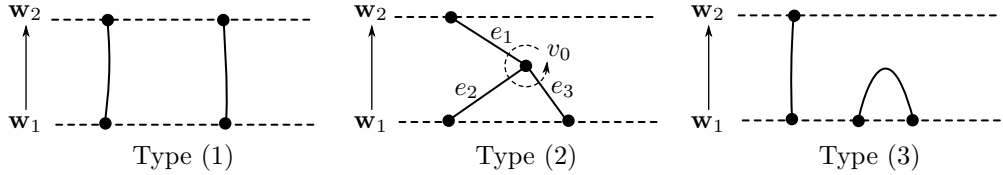
	\caption{\textbf{Examples of the three types of elementary flow graphs.} These graphs are embedded in a fixed 3-manifold $Y$. \label{fig::56}}
\end{figure}

The graph action map $A_\Gamma$ for an elementary flow graph $\Gamma:\ws_1\to \ws_2$ of  \ref{elementarygraphtype1} is equal to the composition of $| E(\Gamma)|$ terms of the form 
\begin{equation}S_{w_1}^- A_{e} S_{w_2}^+,\label{eq:elemengraphtype1}\end{equation} where $e$ is an edge of $\Gamma$ and $w_1$ and $w_2$ are the incoming and outgoing ends of $e$. Note that by \ref{rel:R9}, the induced map $A_{\Gamma}$ is the diffeomorphism map obtained by moving $\ws_1$ to $\ws_2$ along the edges of $\Gamma$. Also, we note that it is easy to use Relations~\ref{rel:R1}--\ref{rel:R9} to see that the map is independent of the ordering of the terms of the form $S_{w_1}^- A_e S_{w_2}^+$.

The graph action map $A_{\Gamma}$ for an elementary flow graph $\Gamma:\ws_1\to \ws_2$ of \ref{elementarygraphtype2} is defined as follows. Let $v_0$ be the interior vertex, and let $e_1,\dots, e_n$ denote the edges adjacent to $v_0$, ordered in any way which is compatible with the cyclic ordering.  The graph action is defined as a composition of expressions of the form shown in equation~\eqref{eq:elemengraphtype1}, for the edges $e$ which are not incident to $v_0$, as well as the map
\begin{equation}
S_{\ws_1'\cup \{v_0\}}^- A_{e_n}\cdots A_{e_1} S_{\ws_2'\cup \{v_0\}}^+,
\label{eq:elemengraphtype2}\end{equation} where $\ws_i'$ denotes the
subset of vertices in $\ws_i$ which are connected by an edge to $v_0$. Also, $S_{\ws_2'\cup \{v_0\}}^{+}$ denotes the the composition of the maps $S_v^+$ for $v\in \ws_2'\cup \{v_0\}$, and similarly for $S_{\ws_1'\cup \{v_0\}}^-$. Note that Relation~\ref{rel:R7} implies that the order of the vertices in $\ws_i'\cup \{v_0\}$ does not affect the composition.

\begin{rem}
The map appearing in equation~\eqref{eq:elemengraphtype2} turns out to be invariant under cyclic permutation of the edges $e_1,\dots, e_n$ \cite{ZemGraphTQFT}*{Lemma~7.13}. In general, the map is not invariant under arbitrary permutations, which is the reason that we decorate graphs with a ribbon structure in the TQFT. Note that the ribbon structure need not be respected by the embedding of the graph in any geometric sense (e.g lie in a 2-plane field in a natural way). Instead, the ribbon structure only determines the order of terms appearing in equation~\eqref{eq:elemengraphtype2}.
\end{rem}

Finally, the graph action map $A_{\Gamma}$ for an elementary flow graph of \ref{elementarygraphtype3} is defined as follows. Let $e$ denote the edge of $\Gamma$ which does not have a vertex in both $\ws_1$ and $\ws_2$, and write $v_1$ and $v_2$ for the two vertices of $e$. If $v_1,v_2\in \ws_1$, the map $A_{\Gamma}$ is defined as a composition of terms like equation~\eqref{eq:elemengraphtype1}, ranging over the edges of $\Gamma$ which have an end in both $\ws_1$ and $\ws_2$, as well as a single term of the form
\[
S^-_{v_1}S^-_{v_2} A_e.
\] 
If instead $v_1,v_2\in \ws_2$, the map is defined by replacing the above expression with $A_e S_{v_1}^+S_{v_2}^+$.

The map $A_{\Gamma}$ is independent up to chain homotopy of the choice of Cerf decomposition of the flow graph $\Gamma$, and is also invariant under subdivision of the edges of $\Gamma$ \cite{ZemGraphTQFT}*{Theorem~B}.

We note that the map $A_{\Gamma}$ is defined somewhat asymmetrically, since we chose to use the $A_{\lambda}$ maps, which count changes across the alpha curves. We could instead define a graph action map $B_{\Gamma}$, by using the construction described above and replacing each instance of $A_{\lambda}$ with $B_{\lambda}$.

Having constructed the graph action maps, the next step is to define maps for 4-dimensional handles. If each component of $W$ has non-empty incoming and outgoing ends, the maps are defined using the graph action map, as well as 1-, 2- and 3-handle maps which are similar to the ones described by Ozsv\'{a}th and Szab\'{o} in \cite{OSTriangles}.

If a cobordism $(W,\Gamma)$ has a connected component which is missing either an incoming or outgoing end, one must remove a collection of 4-balls from $W$. A single arc is added to $\Gamma$ for each 4-ball we remove. Each arc has one endpoint on a new boundary 3-sphere, as well as a point on $\Gamma$. Writing $(W',\Gamma')$ for a graph cobordism obtained by puncturing $W$ and adding strands to $\Gamma$ in the above manner, the map $F_{W,\Gamma,\frs}^A$ is then defined as the composition of $F_{W',\Gamma',\frs|_{W'}}^A$ together with 0-handle and 4-handle maps for the excised 4-balls. Using \cite{ZemGraphTQFT}*{Proposition~11.1}, it follows that the induced cobordism map is independent from which 4-balls we removed, and which arcs we pick to connect them to $\Gamma$.

The type-$B$ graph cobordism map is defined similarly, but using the type-$B$ graph action map. The type-$A$ and type-$B$ versions of the graph cobordism maps are related as follows:

\begin{lem}[\cite{HMZConnectedSum}*{Lemma~5.9}] \label{lem:reversecyclicordering} The type-$A$ and type-$B$ graph cobordism maps satisfy the relation
\[F_{W,\Gamma,\frs}^A\simeq F_{W,\bar{\Gamma},\frs}^B,\] where $\bar{\Gamma}$ is the ribbon graph obtained by reversing the cyclic orderings on $\Gamma$.
\end{lem}

\subsection{Further relations in the graph TQFT}
\label{sec:furtherrelationsofGraphTQFT}

In this section, we prove two useful results about the graph TQFT. The first is the \emph{vertex breaking relation} shown in Figure~\ref{fig::39}, which describes the effect of changing the cyclic ordering at a vertex. The second result is a proof that the broken path cobordism induces the map $\Phi_w$.

\begin{lem}\label{lem:vertexbreakingrelation}Suppose that $(W,\Gamma)$ is a graph cobordism, $v_0$ is a vertex in the interior of $\Gamma$, and $e_1$ and $e_2$ are two edges incident to $v_0$, which are adjacent in the cyclic ordering. Let $\Gamma'$ denote the ribbon graph obtained by switching the relative ordering of $e_1$ and $e_2$. Let $\Gamma''$ denote the graph obtained by removing a connected subarc from the interiors of  $e_1$ and $e_2$ (as in Figure~\ref{fig::39}). Then
\[F_{W,\Gamma,\frs}^A+F_{W,\Gamma',\frs}^A\simeq U\cdot F_{W,\Gamma'',\frs}^A.\] The same relation holds for the type-$B$ maps.
\end{lem}
\begin{figure}[ht!]
	\centering
	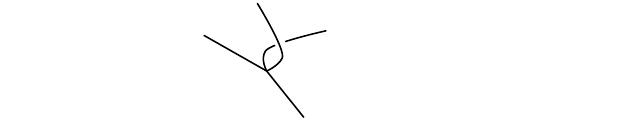
	\caption{\textbf{The vertex breaking relation for changing the relative ordering of two vertices.}  \label{fig::39}}
\end{figure}

\begin{proof} It is sufficient to show the analogous relation for the graph action map. Suppose that $\Gamma:\ws_1\to \ws_2$ is a ribbon flow graph, embedded in $Y$. Since the graph action map is defined as a composition of the graph action maps for elementary flow graphs, only one of which will contain the vertex $v_0$, it is sufficient to show the claim for an elementary flow graph of \ref{elementarygraphtype2}, which contains the vertex $v_0$ as its interior vertex. Let $e_1,\dots, e_n$ be the edges adjacent to $v_0$, indexed compatibly with their cyclic ordering. Let $\ws_1'$ denote the set of vertices in $\ws_1$ which are adjacent to one of $e_1,\dots, e_n$, and let $\ws_2'$ denote the vertices in $\ws_2$ which are adjacent to one of $e_1,\dots, e_n$. For notational simplicity, we assume that $\ws_2=\ws_2'$ and $\ws_1=\ws_1'$, and that there are no edges which are not incident to $v_0$. The more general case is handled by the argument we give presently, by simply adding in extra terms of the form shown in equation~\eqref{eq:elemengraphtype1} for edges not adjacent to $v_0$, though these maps have no interaction with any of the maps related to the component containing $v_0$. The map $A_{\Gamma}$ is equal to
\[
S_{\ws_1\cup \{v_0\}}^- A_{e_n}\cdots A_{e_2}A_{e_1}S_{\ws_2\cup \{v_0\}}^+.
\]
 Using Relation~\ref{rel:R3}, we obtain
\[
S_{\ws_1\cup \{v_0\}}^- A_{e_n}\cdots A_{e_2}A_{e_1}S_{\ws_2\cup \{v_0\}}^++S_{\ws_1\cup \{v_0\}}^- A_{e_n}\cdots A_{e_1}A_{e_2}S_{\ws_2\cup \{v_0\}}^+\simeq U \cdot S_{\ws_1\cup \{v_0\}}^- A_{e_n}\cdots A_{e_3} S_{\ws_2\cup \{v_0\}}^+.
\]
 By definition, the expression $S_{\ws_1\cup \{v_0\}}^- A_{e_n}\cdots A_{e_1}A_{e_2}S_{\ws_2\cup \{v_0\}}^+$ is the graph action map $A_{\Gamma'}$. We now claim that the third expression 
 \begin{equation}
 A_{\Gamma''}\simeq S_{\ws_1'\cup \{v_0\}}^- A_{e_n}\cdots A_{e_3} S_{\ws_2'\cup \{v_0\}}^+
\label{eq:A-Gamma''}
 \end{equation}
The right hand side of equation~\eqref{eq:A-Gamma''} is almost the definition of $A_{\Gamma''}$, but there are some small differences. Nonetheless, by adding trivial strands using Relation~\ref{rel:R8}, and subdividing edges using Relation~\ref{rel:R7'}, it is straightforward to manipulate the expression so that it is, on the nose, the map induced by a Cerf decomposition for the graph $\Gamma''$.
\end{proof}

We now consider the broken path cobordism $(Y\times [0,1], \Gamma_w)$,  shown in Figure~\ref{fig::29}.

\begin{lem}\label{lem:phi=brokenpathcobordism}If $(Y,\ws)$ is a multi-pointed 3-manifold (possibly with only one basepoint) and $(Y\times [0,1],\Gamma_w)$ is the broken path cobordism shown in Figure~\ref{fig::29}, then
\[F_{Y\times [0,1],\Gamma_w,\frs}^A\simeq \Phi_w.\] The same holds for the $B$ versions of the cobordism maps.
\end{lem}

\begin{proof}If $|\ws|>1$, the result follows from Relation~\ref{rel:R2} since $(Y\times [0,1],\Gamma_w)$ is a composition of a free-destabilization cobordism and a free-stabilization cobordism. In the case when $\ws$ consists of a single point $w$, we use invariance of the graph cobordism maps under isotopies of the graph. We pick a new basepoint $w'\not\in \ws$, as well as a path $\lambda$ from $w$ to $w'$. There is a \textit{basepoint swapping} graph cobordism $(Y\times [0,1],\Gamma_{\lambda}^X)$ from $(Y,\{w,w'\})$ to $(Y, \{w,w'\})$ which swaps the basepoints $w$ and $w'$ along the path $\lambda$. The graph $\Gamma_{\lambda}^X\subset Y\times [0,1]$ is well-defined up to isotopy, since $Y\times [0,1]$ is a 4-manifold (unlike the analogous situation in 3-manifolds, where there are two non-isotopic choices of crossings). We can decompose the broken path cobordism $(Y\times [0,1],\Gamma_w)$ into a composition of a free-stabilization graph cobordism at $w'$, followed by a basepoint swapping cobordism $(Y\times [0,1],\Gamma_{\lambda}^X)$, followed by a free-destabilization cobordism at $w'$. The configuration is shown in Figure~\ref{fig::44}.

\begin{figure}[ht!]
	\centering
	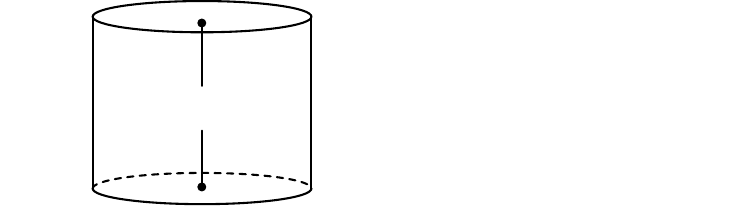
	\caption{\textbf{Manipulating the broken path cobordism $(Y\times [0,1],\Gamma_w)$ so that a basepoint swapping cobordism appears in the middle.} Note that since $Y\times [0,1]$ is 4-dimensional, the crossing shown in our picture carries no meaning.  \label{fig::44}}
\end{figure}

By \cite{ZemGraphTQFT}*{Proposition~14.24}, the cobordism $(Y\times [0,1],\Gamma_{\lambda}^X)$ induces the map
\[F_{Y\times [0,1],\Gamma_{\lambda}^X,\frs}^A\simeq \Phi_{w} A_{\lambda}+A_{\lambda} \Phi_{w'}.\] Using this, we compute that
\begin{align*}
F_{Y\times [0,1],\Gamma_w,\frs}^A&\simeq S_{w'}^-(\Phi_{w} A_{\lambda}+A_{\lambda} \Phi_{w'}) S_{w'}^+&&\\
&\simeq S_{w'}^- \Phi_{w} A_{\lambda} S_{w'}^+&& \text{\ref{rel:R1}, \ref{rel:R2}}\\
& \simeq \Phi_w S_{w'}^- A_{\lambda} S_{w'}^+&&\text{\ref{rel:R3'}}\\
& \simeq \Phi_w&& \text{\ref{rel:R8}}.
\end{align*}
Replacing the $A_\lambda$ maps with $B_{\lambda}$ yields the result for the type-$B$ maps, as well.
\end{proof}

\subsection{Graphs for the actions of \texorpdfstring{$U$}{U} and \texorpdfstring{$H_1(Y)/\Tors$}{H1(Y)/Tors}}
\label{sec:graphsandhomologyactions}
In this section, we describe how the $\bF_2[U]$-module action and the action of $H_1(Y)/\Tors$ are encoded into the graph cobordism maps. We consider the graph cobordisms $(Y\times[0,1],\Gamma_\gamma)$ and $(Y\times [0,1], \Gamma_U)$ shown in Figure~\ref{fig::26}.

\begin{figure}[ht!]
	\centering
	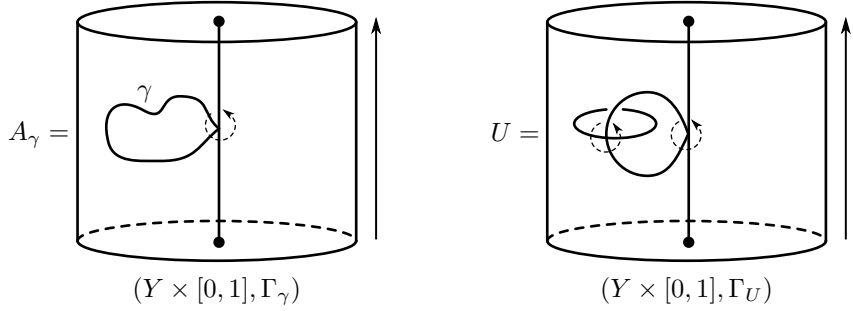
	\caption{\textbf{Graph cobordisms for the action of  $[\gamma]\in H_1(Y)/\Tors$ and the action of $U$.} The two loops on the cobordism for the $U$ map are both null-homologous. \label{fig::26}}
\end{figure}

\begin{prop}\label{prop:spliceinloopsforUandH_1} Suppose $\gamma$ is a closed, embedded loop in $Y$, which intersects $w$. Let $\Gamma_\gamma\subset Y\times [0,1]$ be the graph $\Gamma_{\gamma}:=(\{w\}\times [0,1])\cup (\gamma\times \{\tfrac{1}{2}\})$, shown in Figure~\ref{fig::26}. Then
\[
F_{Y\times [0,1],\Gamma_\gamma,\frs}^A\simeq F_{Y\times [0,1],\Gamma_\gamma,\frs}^B\simeq A_{\gamma},
\]
 where $A_\gamma$ denotes the action of $H_1(Y)/\Tors$. Let $\Gamma_U\subset Y\times [0,1]$ denote the graph on the right side of Figure~\ref{fig::26}. Then
\[
F_{Y\times[0,1],\Gamma_U,\frs}^A\simeq F_{Y\times[0,1],\Gamma_U,\frs}^B\simeq U.
\]
\end{prop}

It is convenient to break the computation into two pieces. If $\lambda$ is a path from $w$ to $w'$ in $Y$, then let $\GY_{\lambda}$ and $\GYup_{\lambda}$ denote the Y-shaped graphs in $Y\times [0,1]$ shown in Figure~\ref{fig::48}.
\begin{figure}[ht!]
	\centering
	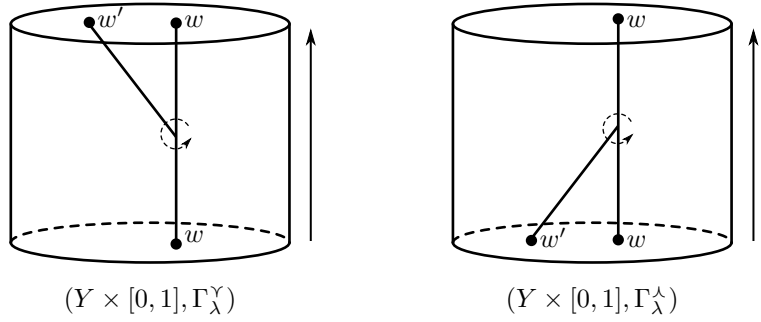
	\caption{\textbf{The graph cobordisms $(Y\times [0,1], \GY_\lambda)$ and $(Y\times [0,1],\protect\GYup_{\lambda})$ considered in Lemma~\ref{lem:computationofYshapedgraphcobordisms}.} These depend on a choice of path, $\lambda$, from $w$ to $w'$ in $Y$. \label{fig::48}}
\end{figure}

\begin{lem}\label{lem:computationofYshapedgraphcobordisms}The graph cobordism maps for $(Y\times [0,1], \GY_{\lambda})$ and $(Y\times [0,1], \GYup_{\lambda})$ satisfy 
\[
F_{Y\times [0,1], \GY_{\lambda},\frs}^A \simeq B_{\lambda} S_{w'}^+ \qquad \text{and} \qquad F_{Y\times [0,1], \GYup_{\lambda},\frs}^A\simeq S_{w'}^-B_{\lambda}.
\] 
If $\GYbar_{\lambda}$ and $\GYupbar_{\lambda}$ denote the graphs with the opposite cyclic order, then
\[
F_{Y\times [0,1], \GYbar_{\lambda},\frs}^A \simeq A_{\lambda} S_{w'}^+ \qquad \text{and} \qquad F_{Y\times [0,1], \GYupbar_{\lambda},\frs}^A\simeq S_{w'}^-A_{\lambda}.
\]
\end{lem}
\begin{proof}In \cite{HMZConnectedSum}*{Lemmas~5.5 and 5.6} it is computed (directly from the definition of the graph action map) that 
\[
F_{Y\times [0,1], \GY_{\lambda},\frs}^A \simeq (A_{\lambda}+U\Phi_{w}) S_{w'}^+ \qquad \text{and} \qquad F_{Y\times [0,1], \GYup_{\lambda},\frs}^A\simeq S_{w'}^-(A_\lambda+U\Phi_{w}).
\]
 By using  equation~\eqref{eq:Alambda+Blambda} and Relations~\ref{rel:R1} and \ref{rel:R2}, we see that
\[(A_{\lambda}+U\Phi_{w}) S_{w'}^+\simeq (B_{\lambda}+U\Phi_{w'})S_{w'}^+\simeq B_{\lambda}S_{w'}^+,\] and similarly
\[S_{w'}^-(A_{\lambda}+U\Phi_{w}) \simeq S_{w'}^-(B_{\lambda}+U\Phi_{w'})\simeq S_{w'}^-B_{\lambda}.\]

To prove the statements about the graph cobordisms with the opposite cyclic orders, we use Lemma~\ref{lem:reversecyclicordering}, which shows that the effect of switching the cyclic ordering is to replace the $A_{\lambda}$ maps with the $B_{\lambda}$ maps, and vice-versa.
\end{proof}

We can now compute the graph cobordism maps for $(Y\times [0,1], \Gamma_\gamma)$ and $(Y\times [0,1],\Gamma_U)$:

\begin{proof}[Proof of Proposition~\ref{prop:spliceinloopsforUandH_1}] We first consider the cobordism $(Y\times [0,1], \Gamma_\gamma)$, shown on the left side of Figure~\ref{fig::26}. Write $\gamma$ as a concatenation of two arcs, $\lambda_1$ and $\lambda_2$, which both have one endpoint at $w$, and one endpoint at a new basepoint, $w'$. We first claim that the graph cobordism map for $(Y\times [0,1], \Gamma_\gamma)$ is invariant under splitting the single vertex in the interior of the graph into two vertices connected by an edge, as shown in Figure~\ref{fig::49}. One way of establishing this equality would be to compute directly from the definition. This is not hard, though somewhat tedious. For convenience, we will instead appeal to the link cobordism interpretation of the graph cobordism maps \cite{ZemCFLTQFT}*{Theorem~C}, from which it follows that ribbon-equivalent graphs induce the same map (see \cite{ZemCFLTQFT}*{Corollary~D}). After splitting this vertex into two vertices, the resulting graph cobordism can be written as the composition of the two graph cobordisms, $(Y\times [0,1], \GYup_{\lambda_2})$ and $(Y\times [0,1], \GY_{\lambda_1})$.

\begin{figure}[ht!]
	\centering
	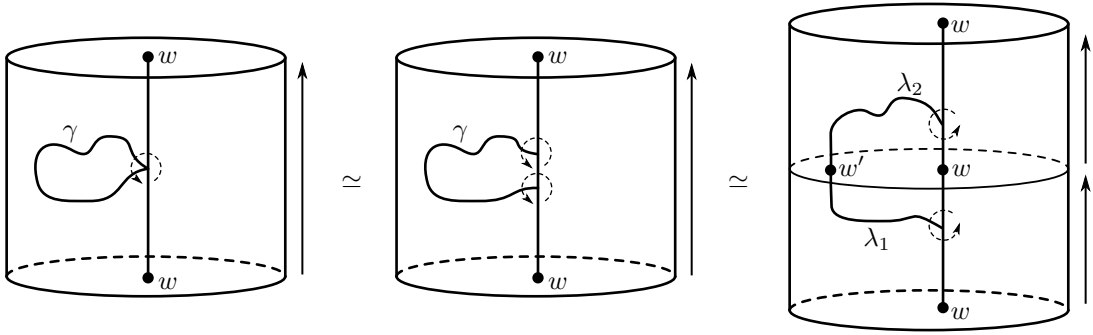
	\caption{\textbf{Computing the cobordism map for $(Y\times[0,1], \Gamma_\gamma)$.} On the left is the graph cobordism $(Y\times [0,1], \Gamma_\gamma)$. In the middle is a ribbon equivalent graph cobordism. On the right is the composition of $(Y\times [0,1], \protect\GYup_{\lambda_2})$ and $(Y\times [0,1], \GY_{\lambda_1})$, The concatenation $\lambda_2* \lambda_1$ is equal to the closed curve $\gamma$. In Proposition~\ref{prop:spliceinloopsforUandH_1}, we show that the induced map is the $H_1(Y)/\Tors$ action map, $A_\gamma$.\label{fig::49}}
\end{figure}

Lemma~\ref{lem:computationofYshapedgraphcobordisms} describes the maps induced by $(Y\times [0,1], \GYup_{\lambda_2})$ and $(Y\times [0,1], \GY_{\lambda_1})$. Accordingly,
\[
F_{Y\times [0,1],\Gamma_\gamma,\frs}^A\simeq S_{w'}^- B_{\lambda_2} B_{\lambda_1} S_{w'}^+.
\] By Relation~\ref{rel:R7'}, this is chain homotopic to $B_{\gamma}$. We recall also that $B_{\gamma}=A_{\gamma}$, since the quantities $a(\gamma,\phi)$ and $b(\gamma,\phi)$ agree for any homology class $\phi$, when $\gamma$ is a closed curve, completing the proof  for $(Y\times [0,1],\Gamma_\gamma)$.

We now compute the map for the graph cobordism $(Y\times[0,1],\Gamma_U)$ from Figure~\ref{fig::26}. We use the previous result for $(Y\times [0,1],\Gamma_\gamma)$, as well as the vertex breaking relation from Lemma~\ref{lem:vertexbreakingrelation}.  It is convenient to consider the more general case that the loops in $\Gamma_U$ are not necessary null-homologous, and instead have homology classes $\gamma_1$ and $\gamma_2$.  As shown in Figure~\ref{fig::28}, by using the vertex breaking relation, and the computation of map induced by $(Y\times [0,1], \Gamma_{\gamma_i})$, we have that the cobordism on the right of Figure~\ref{fig::26} has induced map
\begin{equation}
U+A_{\gamma_1}A_{\gamma_2}.\label{eq:U+Ag1Ag2}
\end{equation}
 The graph $\Gamma_U$ from the main statement is obtained by picking $\gamma_1$ and $\gamma_2$ to be null-homotopic in $Y$, so $A_{\gamma_1}$ and $A_{\gamma_2}$ vanish. The remaining term in equation~\eqref{eq:U+Ag1Ag2} is $U$, completing the proof.
\end{proof}
\begin{figure}[ht!]
	\centering
	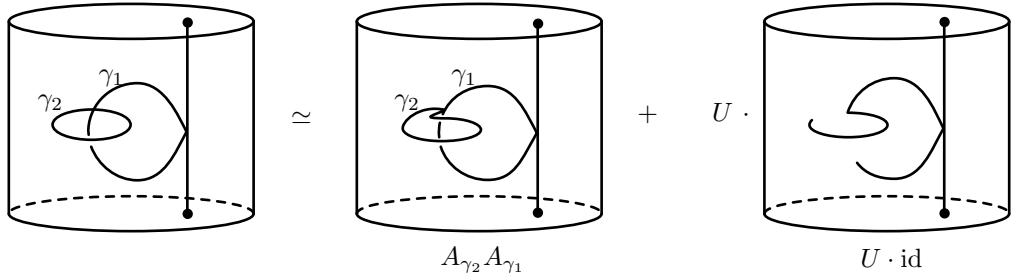
	\caption{\textbf{Computing the graph cobordism on the left by using the vertex breaking relation.} The vertex breaking relation is proven in Lemma~\ref{lem:vertexbreakingrelation}. The cyclic orders are counterclockwise, with respect to the page.\label{fig::28}}
\end{figure}

\subsection{Relative homology maps and holomorphic triangles}
\label{sec:relativehomologyandTriangles}
In this section, we describe the interaction of the relative homology maps with the holomorphic triangle maps. 

Suppose that $(\Sigma,\ve{\alpha},\ve{\beta},\ve{\gamma},\ve{w})$ is a multi-pointed Heegaard triple. In \cite{OSDisks}*{Section~8}, Ozsv\'{a}th and Szab\'{o} construct a 4-manifold $X_{\a,\b,\g}$, as well as a map
\[
\frs_{\ws}:\pi_2(\xs,\ys,\zs)\to \Spin^c(X_{\a,\b,\g}).
\] 
If $\frs\in \Spin^c(X_{\a,\b,\g})$, the holomorphic triangle map 
\[
F_{\a,\b,\g,\frs}:\CF^-(\Sigma,\ve{\alpha},\ve{\beta},\frs_{\a,\b})\otimes_{\bF_2[U]} \CF^-(\Sigma,\ve{\beta},\ve{\gamma},\frs_{\b,\g})\to \CF^-(\Sigma,\ve{\alpha},\ve{\gamma},\frs_{\a,\g})
\]
 is defined by counting holomorphic triangles which represent Maslov index 0 homology classes $\psi$ with $\frs_{\ws}(\psi)=\frs$.

 If $\lambda$ is a path between two basepoints on the Heegaard surface $\Sigma$, and $\phi$ is a homology class of disks, let $a(\lambda,\phi)$, $b(\lambda,\phi)$ and $c(\lambda,\phi)$ denote the sum of the changes of multiplicities of $\phi$, across only the $\ve{\alpha},$ $\ve{\beta}$ or $\ve{\gamma}$ curves, respectively. Let $A_{\lambda},$ $B_\lambda$ and $C_\lambda$ denote the maps which count holomorphic disks with an extra factor of $a(\lambda,\phi),$ $b(\lambda,\phi)$ or $c(\lambda,\phi)$, respectively (all three maps may be defined on any of the three complexes involved in the triple). We have the following:

\begin{lem}\label{lem:graphactionandtriangles} If $\frs\in \Spin^c(X_{\a,\b,\g})$, then the holomorphic triangle map
$F_{\a,\b,\g,\frs}$ satisfies 
	\begin{align*}F_{\a,\b,\g,\frs}(A_\lambda\otimes \id)+A_\lambda \circ F_{\a,\b,\g,\frs}(\id\otimes \id)&\simeq 0\\
	F_{\a,\b,\g,\frs}(B_\lambda\otimes \id)+F_{\a,\b,\g,\frs}(\id\otimes B_\lambda)&\simeq 0\\
	F_{\a,\b,\g,\frs}(\id\otimes C_\lambda)+C_\lambda\circ  F_{\a,\b,\g,\frs}(\id\otimes \id)&\simeq 0.
	\end{align*}
\end{lem}

\begin{proof} We will focus on the first relation, as the other two follow from similar arguments. If $\psi\in \pi_2(\xs,\ys,\zs)$ is a Maslov index 1 homology class of triangles, then the moduli space $\cM(\psi)$ is 1-dimensional, and can be compactified into a compact 1-manifold $\bar{\cM(\psi)}$ with boundary, by adjoining pairs consisting of an index 0 holomorphic triangle and an index 1 holomorphic disk.  Hence, summing over all such classes, for fixed $\xs,$ $\ys$ and $\zs$, we get
\[
0=\sum_{\substack{\psi\in \pi_2(\ve{x},\ve{y},\ve{z})\\
	\mu(\psi)=1\\
\frs_{\ve{w}}(\psi)=\frs}} a(\lambda,\psi)\# \d \left(\bar{\cM(\psi)}\right)U^{n_{\ve{w}}(\psi)}.
\]
 Furthermore, if $\psi$ is a homology class of triangles and $\phi$ is a homology class of disks, we have $a(\lambda,\psi+\phi)=a(\lambda,\psi)+a(\lambda,\phi)$. Also, if $\phi$ is a homology class of disks on $(\Sigma,\ve{\beta},\ve{\gamma})$ then $a(\lambda,\phi)=0$. Hence 
\[
F_{\a,\b,\g,\frs}(A_\lambda\otimes \id)+A_\lambda F_{\a,\b,\g,\frs}(\id\otimes \id)=\d_{\a,\g} H_{\a,\b,\g,\frs}^{A}+H_{\a,\b,\g,\frs}^{A}(\d_{\a,\b}\otimes \id+\id\otimes \d_{\b,\g}),
\]
 where $H_{\a,\b,\g,\frs}^{A}$ counts holomorphic triangles with $\frs_{\ve{w}}(\psi)=\frs$ with an additional factor of $a(\lambda,\psi)$.
\end{proof}

\section{A handle decomposition of the trace cobordism}
\label{sec:handledecomposition}
In this section, we describe a handle decomposition of the trace cobordism $Y\times [0,1]:-Y\sqcup Y\to \emptyset$, and also recall the singular homology of mapping tori.

\subsection{A handle decomposition of the trace cobordism}
\label{sec:handledecomptracecobordism}
In this section, we describe how a Heegaard diagram for $Y$ induces a handle decomposition of the trace cobordism. In this section, it is more convenient to view the trace cobordism as $Y\times [-1,1]$, instead of $Y\times [0,1]$.

Suppose $f:Y\to [1,3]$ is a Morse function which induces the diagram $(\Sigma,\as,\bs,\ws)$, by which we mean that $f$ admits a gradient like vector field such that the following hold:

\begin{enumerate}
\item $f$ has $|\ws|$ index 0 critical points, all with critical value 1.
\item $f$ has $g(\Sigma)+|\ws|-1$ index 1 critical points, all with critical values in $(1,2)$, whose ascending manifolds intersect $\Sigma$ along $\as$.
\item $f^{-1}(2)=\Sigma$.
\item $f$ has $g(\Sigma)+|\ws|-1$ index 2 critical points, all with critical values in $(2,3)$, whose descending manifolds intersect $\Sigma$ along $\bs$.
\item $f$ has $|\ws|$ index 3 critical points, all with critical value $3$.
\end{enumerate}

We construct a Morse function $F:Y\times [-1,1]\to [0,3]$ by the formula
\[F(y,s)=(1-s^2)\cdot f(y).\] It is easy to see that the critical set of $F$ is equal to $\Crit(f)\times \{0\}\subset Y\times [-1,1]$. Furthermore, if $p$ is a critical point of $f$, then
\[\ind_{(p,0)}(F)=\ind_p(f)+1.\]

For our purposes, it is important to precisely describe the attaching spheres of the handles. To this end, we define the following submanifolds:
\[W_t:=F^{-1}([0,t]),\qquad 
M_t:=F^{-1}(t),\qquad
Y_t:= f^{-1}([t,3]),\qquad \text{and}\qquad
\Sigma_t :=f^{-1}(t).
\]
 The following lemma describes the level sets of the trace cobordism:

\begin{lem}\label{lem:levelsetdescription}Suppose that $t\in [1,3]$ is a regular value of $f$. The projection map $\pi_Y:Y\times [-1,1]\to Y$ restricts to a homeomorphism between $M_t\cap (Y\times [-1,0])$ and $Y_t$. The map $\pi_Y$ also restricts to a homeomorphism between $M_t\cap (Y\times [0,1])$ and $Y_t$. On each of the above sets, the map $\pi_Y$ is a diffeomorphism away from $Y\times \{0\}$. Putting these maps together yields a homeomorphism between $M_t$ and  $Y_t\cup_{\Sigma_t} -Y_t$, which is a diffeomorphism away from $M_t\cap (Y\times \{0\})$.
\end{lem}
\begin{proof}To see that $\pi_Y$ induces a homeomorphism on the stated spaces, it is sufficient to show that it is bijective, and maps between the stated subsets, which is an easy exercise from the definitions of the maps $f$ and $F$. To see that $\pi_Y$ induces a local diffeomorphism on each of $M_t\cap (Y\times [-1,0))$ and $M_t\cap (Y\times (0,1])$,  one simply needs to show that $\d/\d s\not \in T_{(y,s)}M_t$. However it is easily checked that $\d/\d s\in T_{(y,s)} M_t$ if and only if $s=0$.
\end{proof}

Using Lemma~\ref{lem:levelsetdescription}, we obtain the following description of the critical points of $F$ on $Y\times [-1,1]$ and their attaching spheres:

\begin{enumerate}[label=(Index \arabic*),leftmargin=*, widest=IIIII]
\item $F$ has $|\ws|$ index 1 critical points, corresponding to the index 0 critical points of $f$. All have critical value equal to 1.
The attaching 0-sphere in $Y\sqcup -Y$ of each of these critical points is equal to the union of the corresponding index 0 critical point of $f$ in $Y$, together with its image in $-Y$.
\item $F$ has $g(\Sigma)+|\ws|-1$ index 2 critical points, which have attaching spheres equal to the union of the descending flow lines of the index 1 critical points of $f$ in $Y_{1+\epsilon}$  (which have their boundary on $\d Y_{1+\epsilon}$) and their images in $-Y_{1+\epsilon}$. They have critical values in $(1,2)$. The framings are discussed in more detail, below.
\item $F$ has $g(\Sigma)+|\ws|-1$ index 3 critical points. We can view the attaching spheres as being in $M_2=F^{-1}(2)$, which is homeomorphic to $U_{\b}\cup_{\Sigma} -U_{\b}$  by Lemma~\ref{lem:levelsetdescription}. The descending manifolds of the index $2$ critical points of $f$ intersect $U_{\b}$ in a set of 2-dimensional compressing disks which meet $\Sigma$ along the $\bs$ curves. The union of these disks with their images in $-U_{\b}\subset -Y$ are spheres in $M_2=U_{\b}\cup_{\Sigma} -U_{\b}$, and these spheres are the attaching spheres of the index 3 critical points of $F$. They have critical values in $(2,3)$.
\item $F$ has $|\ws|$ index 4 critical points, corresponding to the $|\ws|$ index 3 critical points of $f$. They all have critical value equal to $3$.
\end{enumerate}

 We now discuss the framings of the index 2 critical points in somewhat more detail. Note that the exact framing of the attaching spheres of the 2-handles in $M_{1+\epsilon}\iso Y\, \#_{\ws} -Y$ (where $Y\, \#_{\ws} -Y$ denotes the manifold obtained by adding a connected sum tube near each point in $\ws$) depends on some additional data (such as a choice of gradient like vector field), however the framing of the portion of the link in $Y_{1+\epsilon}$ can be taken to be the mirror of the framing of the portion in $-Y_{1+\epsilon}$. Up to isotopy, a framing is uniquely determined by this property.

\begin{rem}A handlebody decomposition of the cotrace cobordism can be obtained by turning around the above decomposition for the trace cobordism.
\end{rem}

\subsection{Singular homology of mapping tori}

We recall in this section that the singular homology of a mapping torus is given by a mapping cone.

First, let us recall the algebraic mapping cone construction. If $(C_1,\d_1)$ and $(C_2,\d_2)$ are two chain complexes, and $F:C_1\to C_2$ is a degree zero chain map, the \emph{mapping cone of $F$}, written as 
\[\Cone(C_1\xrightarrow{F} C_2),\] is defined to be the complex
\[\Cone(C_1\xrightarrow{F} C_2):=C_1[1]\oplus C_2,\] with differential
\[\d=\begin{pmatrix}-\d_1& 0\\
F& \d_2
\end{pmatrix}.\]

\begin{lem}\label{lem:CWhomologyofmappingtori}Suppose that $X^4$ is a smooth, oriented 4-manifold and $Y^3\subset X^4$ is a smooth, oriented, non-separating cut. Let $W^4$ be the result of cutting $X$ along $Y$, and let $\iota_0$ and $\iota_1$ denote the two inclusions of $Y$ into $W$ (corresponding to the two copies of $Y$ in $\d W$). Then $C_*^{CW}(X;\Z)$ is quasi-isomorphic to
\[\Cone(C_*^{CW}(Y;\Z)\xrightarrow{(\iota_0)_*-(\iota_1)_*} C_*^{CW}(W;\Z)).\]
\end{lem}
 
\begin{rem}Note that Lemma~\ref{lem:CWhomologyofmappingtori} specializes in the case of a mapping torus to show that $C_*^{CW}(X_\phi;\Z)$ is quasi-isomorphic to $\Cone(C_*^{CW}(Y;\Z)\xrightarrow{\id_*-\phi_*} C_*^{CW}(Y;\Z))$.
\end{rem}

\begin{proof}[Proof of Lemma~\ref{lem:CWhomologyofmappingtori}] The proof is by explicit construction of a $CW$ decomposition of $X$ whose homology is that of the mapping cone. Pick a $CW$ decomposition of $Y$, and pick a $CW$ decomposition of $W$ which extends this fixed decomposition (on both boundary components). The $CW$ decomposition of $Y$ naturally yields a $CW$ decomposition of $Y\times [0,1]$, via the product construction. If $e_i$ is a cell of dimension $i$ in our decomposition for $Y$, then there are three cells in our decomposition for $Y\times [0,1]$, namely $e_i\times \{0\}, e_i\times \{1\}$ and $e_i\times [0,1]$. Furthermore
\[\d(e_i\times [0,1])=e_i\times \{1\}-e_i\times \{0\}.\] We can construct a $CW$ decomposition of $X$, by taking our $CW$ decomposition for $W$, and adding in the cells of the form $e_i\times [0,1]$, where $e_i$ is a cell in $Y$. Manifestly, we have an isomorphism of groups
\[C_*^{CW}(X;\Z)\iso C_*^{CW}(Y;\Z)[1]\oplus C_*^{CW}(W;\Z),\] and the differential on $C_*^{CW}(X;\Z)$ is given by
\[\d=\begin{pmatrix}-\d_Y&0\\
(\iota_1)_*-(\iota_0)_*& \d_W
\end{pmatrix},\] which is the mapping cone.
\end{proof}

\section{Generalized 1-handle and 3-handle maps}
\label{sec:generalized1--handleand3--handlemaps}

In \cite{OSTriangles}, Ozsv\'{a}th and Szab\'{o} define cobordism maps associated to attaching a 4-dimensional 1-handle or 3-handle. In this section, we describe a map for attaching many 1-handles or 3-handles simultaneously, which will be useful for our analysis of the trace and cotrace cobordisms.

\subsection{Definition of the generalized 1- and 3-handle maps}
\label{section:defgeneralized1-handlemaps}

Suppose that $(\Sigma_0,\xis,\zetas,\ws_0)$ is a multi-pointed Heegaard diagram for $(S^1\times S^2)^{\# g(\Sigma_0)}$. We claim that
\begin{equation}
\HF^-(\Sigma_0,\xis,\zetas,\ws_0)\iso V^{\otimes (g(\Sigma_0)+|\ws_0|-1)}\otimes_{\bF_2} \bF_2[U],
\label{eq:Floer-homology-s1s2}
\end{equation}
where $V\iso H^1(S^1;\bF_2)$. Ozsv\'{a}th and Szab\'{o} prove equation~\eqref{eq:Floer-homology-s1s2} in \cite{OSDisks}*{Lemma~9.1} in the case that $|\ws_0|=1$. Furthermore, Ozsv\'{a}th and Szab\'{o} describe the effect on the Heegaard Floer complexes of adding a basepoint in \cite{OSLinks}*{Proposition~6.5}. It follows from equation~20 therein that if we work with the complexes with just one $U$ variable, then adding a basepoint has the effect of tensoring over $\bF_2$ with $V$. In particular, equation~\eqref{eq:Floer-homology-s1s2} holds for any number of basepoints.

 In particular, it follows from equation~\eqref{eq:Floer-homology-s1s2} that there is a well-defined top degree element on homology. We will further restrict to diagrams where 
 \begin{equation}
 |\xi_i\cap \zeta_j|=2 \delta_{ij}, \label{eq:kronecker-delta}
 \end{equation}
  where $\delta_{ij}$ denotes the Kronecker-delta. This last assumption implies that the top degree element of homology is realized as a unique top degree intersection point $\Theta_{\xi,\zeta}^+\in \bT_{\xi}\cap \bT_{\zeta}$. There is also a well-defined bottom degree intersection point $\Theta^-_{\xi,\zeta}$. Of course, equation~\eqref{eq:kronecker-delta} is satisfied if $\xis$ are appropriately chosen and sufficiently small isotopies of $\zetas$, however this case will not be sufficient for our purposes.

Now suppose that  $\cH=(\Sigma,\as,\bs,\ws)$ is a multi-pointed Heegaard diagram for an arbitrary 3-manifold, and $f:\ws_0\to \Sigma\setminus (\as\cup \bs\cup \ws)$ is an embedding. Write $\ps\subset \Sigma$ for the image $\ws_0$ under $f$.  We join the diagrams $(\Sigma,\as,\bs,\ws)$ and $(\Sigma_0,\xis,\zetas,\ws_0)$ by adding a connected sum tube at each pair of points identified by $f$. We remove the basepoints $\ws_0$, and obtain a diagram
\[
(\Sigma\, \#_f \Sigma_0, \as\cup \xis, \bs\cup \zetas, \ws).
\]
 An example is shown in Figure~\ref{fig::31}. We define the generalized 1-handle map 
\[
F_1^{\xi,\zeta}:\CF^-(\Sigma,\as,\bs,\ve{w})\to \CF^-(\Sigma\, \#_f \Sigma_0, \as\cup \xis, \bs\cup \zetas, \ve{w})
\]
 by the formula
\begin{equation}
F_1^{\xi,\zeta}(\ve{x})=\ve{x}\otimes \Theta^+_{\xi,\zeta}.\label{eq:formula1-handlemap}
\end{equation}
 Dually, we define the generalized 3-handle map via the formula
\[
F_3^{\xi,\zeta}(U^i\cdot \ve{x}\times \theta)=
\begin{cases}
U^i \cdot \ve{x}& \text{if } \theta=\Theta_{\xi,\zeta}^-\\
0& \text{otherwise}.
\end{cases}
\]

\begin{figure}[ht!]
	\centering
	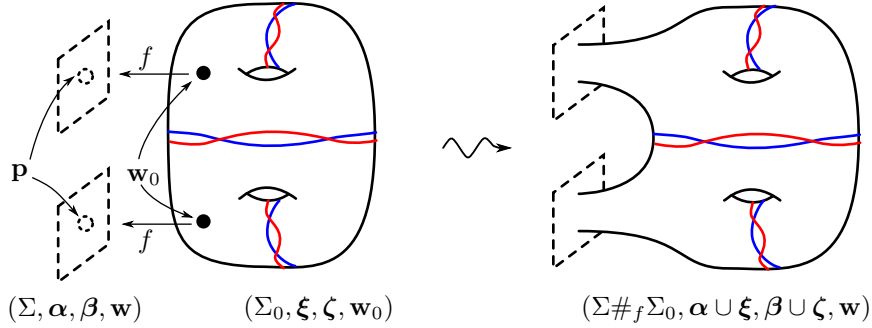
	\caption{\textbf{An example of the generalized 1-handle operation.} The connected sum is taken at the points $\ps\subset \Sigma$ and $\ws_0\subset \Sigma_0$, using the identification given by $f$. Only a small portion of the Heegaard surface $\Sigma$ is shown.\label{fig::31}}
\end{figure}

\subsection{Holomorphic disks and the generalized 1-handle and 3-handle maps}

We now show that the generalized 1-handle and 3-handle maps defined in Section~\ref{section:defgeneralized1-handlemaps} are chain maps, for  appropriate choices of almost complex structures.

 Write $\d_0$ for the differential on $\CF^-(\Sigma,\as,\bs)$, defined using an almost complex structure $J$ which is split on a cylindrical neighborhood of $\ve{p}\times [0,1]\times \R$. Write $\d_{J(T)}$ for the differential on $\CF^-(\Sigma\,\#_f \Sigma_0,\as\cup \xis,\bs\cup \zetas, \ws)$, for an almost complex structure $J(T)$, defined using the connected sum of $J$ with an almost complex structure on $\Sigma_0\times [0,1]\times \R$, and inserting a neck of length $T$  into each connected sum tube. We prove the following (compare \cite{OSProperties}*{Proposition~6.4}):

\begin{prop}\label{prop:differentialcomp}
Suppose that $(\Sigma_0,\xis,\zetas,\ws_0)$ is a diagram for $(S^1\times S^2)^{\# g(\Sigma_0)}$, with $|\xi_i\cap \zeta_j|=2 \delta_{ij}$ and $(\Sigma,\as,\bs,\ws)$ is a Heegaard diagram for an arbitrary 3-manifold, equipped with an embedding $f\colon \ws_0\to \Sigma\setminus (\as\cup \bs\cup \ws)$.  For sufficiently large $T$, the generalized 1-handle and 3-handle maps satisfy
\[
\d_{J(T)}\circ F_1^{\xi,\zeta}=F_1^{\xi,\zeta}\circ \d_0\qquad \text{and}\qquad \d_0\circ F_3^{\xi,\zeta}=F_3^{\xi,\zeta}\circ \d_{J(T)}.
\]
\end{prop}
\begin{proof}
 We introduce some notation. If $\cH$ is a Heegaard diagram, let $\langle , \rangle':\CF^-(\cH)\otimes_{\bF_2[U]} \CF^-(\cH)\to \bF_2[U]$ denote the map
\begin{equation}
\langle U^i\cdot \ve{x}, U^j\cdot \ve{y} \rangle'=\begin{cases} U^{i+j}& \text{if } $\xs=\ys$,\\
0& \text{otherwise.}
\end{cases}\label{eq:pairingdefinition}
\end{equation} 

The first relation of the main claim can be restated as
\begin{equation}
\langle \d_{J(T)}(\xs\times \Theta^+_{\xi,\zeta}), \ys\times \theta\rangle'=\langle \d_0(\xs),\ys\rangle'\cdot \langle' \Theta^+_{\xi,\zeta}, \theta\rangle',
\label{eq:gen-1handle-chain-restate}
\end{equation}
for any $\theta\in \bT_{\xi}\cap \bT_{\zeta}$. The second relation can be restated as
\[
\langle \d_{J(T)}(\xs\times \theta), \ys\times \Theta^-_{\xi,\zeta} \rangle'=\langle \d_0(\xs), \ys \rangle' \cdot \langle \theta,\Theta^-_{\xi,\zeta} \rangle'.
\]

We first consider the claim about $F_1^{\xi,\zeta}$.  If $ \phi_0\in \pi_2( \theta, \theta')$ is a class on $(\Sigma_0,\xis,\zetas,\ws_0)$, then by the definition of the Maslov grading on $\CF^-(\Sigma_0,\xis,\zetas,\ws_0)$, one has
\begin{equation}
\mu(\phi_0)=2n_{\ve{w}_0}(\phi_0)+\gr(\theta,\theta'),
\label{eq:Maslovindexgeneralized1-handle}
\end{equation}
where $\gr(\theta,\theta')$ is the drop in grading from $\theta$ to $\theta'$.

Next, we recall the Maslov index formula of Lipshitz \cite{LipshitzCylindrical}*{Corollary~4.10} (see also \cite{LipshitzErrata}*{Proposition~4.8'}), which says that
\begin{equation}
\mu(\phi)=e(\cD(\phi))+n_{\ve{x}}(\phi)+n_{\ve{y}}(\phi).
\label{eq:Euler-measure-maslov}
\end{equation}
Here, $\cD(\phi)$ denotes the domain of $\phi$ on $\Sigma$, and $e(\cD(\phi))$ is its \emph{Euler measure} (recall that the Euler measure is $1/2\pi$ times the integral of the Gaussian curvature of $\Sigma$ with respect to a metric where the $\as$ and $\bs$ are geodesics which intersect at right angles). Furthermore, $n_{\ve{x}}(\phi)$ and $n_{\ve{y}}(\phi)$ denote the sum of the average multiplicities at each point in $\ve{x}$ and $\ve{y}$, respectively.

Using equation~\eqref{eq:Euler-measure-maslov}, we can compute $\mu(\phi\# \phi_0)$. Namely, the only quantity in equation~\eqref{eq:Euler-measure-maslov} which is not additive under connected sum is the Euler measure. Since the Euler measure of a disk is 1, 
\[
e(\cD(\phi\# \phi_0))=e(\cD(\phi))+e(\cD(\phi_0))-2n_{\ws_0}(\phi_0).
\]
Hence, by using equation~\eqref{eq:Maslovindexgeneralized1-handle} as well, we see that if $\phi\# \phi_0\in \pi_2(\xs\times \theta,\ys\times \theta')$ is a class on $(\Sigma\#_f \Sigma_0, \as\cup \xis,\bs\cup \zetas,\ws),$ then
\begin{equation}
\label{eq:Maslovindexgeneralized1--handledisk}
\begin{split}\mu(\phi\# \phi_0)&=\mu(\phi)+\mu(\phi_0)-2n_{\ve{w}_0}(\phi_0) 
\\&=\mu(\phi)+\gr(\theta,\theta').
\end{split}
\end{equation}
 Furthermore, if $\phi\#\phi_0$ has a representative for arbitrarily large $T$, then we can extract a limit to a broken representative of both $\phi$ and $\phi_0$. By transversality, $\mu(\phi)\ge 0$, with equality to zero if and only if $\phi$ is the constant class. 

Since the claim concerns $\d_{J(T)}(\ve{x}\times \Theta_{\xi,\zeta}^+)$, we assume $\theta=\Theta^+_{\xi,\zeta}$. Since  $\Theta^+_{\xi,\zeta}$ is the top graded intersection point, we have $\gr(\Theta_{\xi,\zeta}^+,\theta')\ge 0$. 

Since $\gr(\Theta_{\xi,\zeta}^+,\theta')\ge 0$ and $\mu(\phi)\ge 0$, equation~\eqref{eq:Maslovindexgeneralized1--handledisk} implies there are two possible cases: 
\begin{enumerate}[label= (c-\arabic*), ref=c-\arabic*,leftmargin=*, widest=III]
\item\label{case:c1} $\mu(\phi)=0\qquad\text{and} \qquad \gr(\Theta_{\xi,\zeta}^+,\theta')=1$;
\item\label{case:c2} $\mu(\phi)=1\qquad\text{and} \qquad \gr(\Theta_{\xi,\zeta}^+,\theta')=0$.
\end{enumerate}

In Case~\eqref{case:c1}, $\phi$ must be a constant class, since it has Maslov index zero and admits a broken holomorphic representative for a cylindrical almost complex structure. Hence $\phi_0$ is an index 1 disk which has zero multiplicity over $\ws_0$.  As $\Theta_{\xi,\zeta}^+$ is a cycle in $\hat{\CF}(\Sigma_0,\xis,\zetas,\ws_0)$, all holomorphic disks of this form cancel, modulo 2, and hence (together) make no contribution to  $\d_{J(T)}(\ve{x}\times \Theta_{\xi,\zeta}^+)$.

 In Case~\eqref{case:c2}, we have $\theta=\theta'=\Theta_{\xi,\zeta}^+$. If $\phi\#\phi_0$ admits holomorphic representatives for arbitrarily long neck lengths, then we can extract (potentially broken) limiting curves $U$ and $U_0$ representing $\phi$ and $\phi_0$, respectively. Since $\mu(\phi)=1$, by transversality the broken curve $U$ consists of a single, non-broken holomorphic strip, $u$.
 
  Write $\ws_0=\{w_1,\dots, w_k\}$ and write $\ps=\{p_1,\dots, p_k\}$ for the image of $f$, where $f(w_i)=p_i$. There must be a holomorphic strip $u_0$ in the broken curve $U_0$ which matches $u$, i.e., which satisfies
 \[
 \rho^{\ve{p}}(u)=\rho^{\ws_0}(u_0),
 \] where 
 \[
 \rho^{\ps}: \cM(\phi)\to \Sym^{n_1}(\bD)\times \cdots \times \Sym^{n_k}(\bD)
 \]
  is the map 
 \[
 \rho^{\ps}(u)=\big((\pi_{\bD}\circ  u)((\pi_{\Sigma}\circ u)^{-1}(p_1)),\dots,(\pi_{\bD} \circ u)((\pi_\Sigma\circ u)^{-1}(p_k))\big),
 \]
  and
 \[
 n_i:=n_{p_i}(\phi)=n_{w_i}(\phi_0).
 \]
  The map $\rho^{\ws_0}$ is defined analogously to $\rho^{\ps}$.
 
 The argument now diverges slightly, depending on whether $|\ws_0|=1$ or $|\ws_0|>1$.  In the case that $|\ws_0|>1$, we can consider almost complex structures satisfying \ref{def:J1}--\ref{def:J5}, whereas when $|\ws_0|=1$, we will have to consider slightly generic almost complex structures, satisfying instead \ref{def:J1}--\ref{def:J4} and \ref{def:J5'}.

  We first consider the case that $|\ws_0|>1$, as this case is slightly simpler. We claim that the broken curve $U_0$, described above, consists of only the unbroken holomorphic strip $u_0$. This follows from expected dimension counts, and transversality, as we now describe. 
 
Write $\phi_0'$ for the homology class of $u_0$. For a generic choice of almost complex structure on $\Sigma\times [0,1]\times \R$, the set $\rho^{\ps}(u)$ will be disjoint from the fat diagonal whenever $u$ has Maslov index 1. In the case that $|\ws_0|>1$, it is not hard to see that the curve $u_0$ will satisfy conditions \ref{def:M1}--\ref{def:M3} (\ref{def:M1}--\ref{def:M5} are straightforward to verify, and \ref{def:M3} follows since $\rho^{\ws_0}(u_0)$ is not in the fat diagonal, so there can be no components which map constantly to $\bD$). Using Proposition~\ref{prop:transversalitydisks} we see that for a generic almost complex structure on $\Sigma_0\times [0,1]\times \R$, if $S_0$ denotes the source curve of $u_0$, and $X\subset \Sym^{n_1}(\bD)\times \cdots \times  \Sym^{n_k}(\bD)$ is a smooth submanifold avoiding the fat diagonal, then $\cM(S_0,\phi_0',X)$ is a smooth manifold of dimension $\mu(\phi_0')-\codim(X)-2\sing(u_0)$ near $u_0$. We consider
\[
X(\phi):=\{\rho^{\ps}(u): u\in \cM(\phi)\}\subset  \Sym^{n_1}(\bD)\times \cdots \times \Sym^{n_k}(\bD),
\] which has codimension $2(n_1+\cdots +n_k)-1$.

It follows from equation~\eqref{eq:Maslovindexgeneralized1-handle} that 
\[
\mu(\phi_0')\le \mu(\phi_0)=2(n_1+\cdots +n_k).
\]
 Hence, using Proposition~\ref{prop:transversalitydisks}, it follows that near $u_0$
\begin{align*}
\dim \cM(S_0,\phi_0',X(\phi))&=\mu(\phi_0')-\codim(X(\phi))-2\sing(u_0)\\
&\le \mu(\phi_0)-\codim(X(\phi))-2\sing(u_0)\\
&\le 1,
\end{align*} with equality if and only if $u_0$ is embedded and $\mu(\phi_0')=\mu(\phi_0)=2(n_1+\cdots+n_k)$. 
We conclude that $u_0$ is an embedding, and $\mu(\phi_0')=\mu(\phi_0)=2(n_1+\cdots+n_k)$. There cannot be any remaining curves of $U_0$, since they would have Maslov index at least 1 by transversality, and hence would raise the Maslov index of $\phi_0$ above $2(n_1+\dots+ n_k)$, a contradiction. Hence $\phi_0'=\phi_0$.

Summarizing, any sequence of curves representing $\phi\# \phi_0$ for a sequence of almost complex structures $J(T_i)$, with $T_i\to \infty$, limits to a pair $(u,u_0)\in \cM(\phi)\times \cM(\phi_0)$ which satisfies $\rho^{\ps}(u)=\rho^{\ws_0}(u_0)$.

 If the almost complex structures achieve transversality at $u$ and $u_0$, then it follows from \cite{LipshitzCylindrical}*{Proposition~A.2} that there is a neighborhood $\cU$ of $(u,u_0)$ in the compactification of the space of holomorphic disks on $(\Sigma\, \#_f\Sigma_0)\times [0,1]\times \R$ such that
 \[\cU\cap\bigg(\bigcup_{T>0}\hat{\cM}_{J(T)}(\phi\# \phi_0)\cup \{(u,u_0)\}\bigg)\iso (0,1].\]
 
 If  $\ve{d}\in \Sym^{n_1}(\bD)\times \cdots \times \Sym^{n_k}(\bD)$ is not in the fat diagonal, define
 \[
 \cM(\phi_0,\ve{d}):=\{u\in \cM(\phi_0): \rho^{\ve{w}_0}(u)=\ve{d}
 \}.
 \]
 If $\ve{d}$ is  point which is not in the fat diagonal, then by Proposition~\ref{prop:transversalitydisks}, for a generic choice of almost complex structures,   the space $\cM(\phi_0,\ve{d})$ is a compact 0-dimensional manifold. We will show that if $\ve{d}$ is fixed, then for a generic choice of almost complex structure,
  \begin{equation}\sum_{\substack{\phi_0\in \pi_2(\Theta_{\xi,\zeta}^+,\Theta_{\xi,\zeta}^+)\\ n_{p_i}(\phi_0)=n_i}} \#\cM(\phi_0,\ve{d})\equiv 1 \pmod{2}.\label{eq:maincountofdisks}\end{equation} 
  
  Note that the main statement is a consequence of equation~\eqref{eq:maincountofdisks}, since using the argument described thus far, we see that
  \begin{equation*}\begin{split}\d_{J(T)}(\ve{x}\times \Theta_{\xi,\zeta}^+)&=\sum_{\substack{\theta'\in \bT_{\a}\cap \bT_{\b}\\ \phi\# \phi_0\in \pi_2(\xs\times \Theta^+_{\xi,\zeta}, \ys\times \theta')\\ \mu(\phi\# \phi_0)=1}}\# \hat{\cM}(\phi\# \phi_0) U^{n_{\ws}(\phi)}\cdot \ys\times \theta'\\
  &=\sum_{\substack{\phi\in \pi_2(\xs,\ys)\\\mu(\phi)=1}} \sum_{\substack{\phi_0\in \pi_2(\Theta_{\xi,\zeta}^+, \Theta_{\xi,\zeta}^+)\\ n_{w_i}(\phi_0)=n_{p_i}(\phi)}} \#\hat{\cM}(\phi\# \phi_0) U^{n_{\ws}(\phi)} \cdot \ys\times \Theta_{\xi,\zeta}^+\\
  &=\sum_{\substack{\phi\in \pi_2(\xs,\ys)\\\mu(\phi)=1}}  \sum_{u \in \hat{\cM}(\phi)} \bigg(\sum_{\substack{\phi_0\in \pi_2(\Theta_{\xi,\zeta}^+, \Theta_{\xi,\zeta}^+)\\ n_{w_i}(\phi_0)=n_{p_i}(\phi)}}  \# \cM(\phi_0, \rho^{\ws}(u))\bigg) U^{n_{\ws}(\phi)}\cdot \ys\times \Theta^+_{\xi,\zeta}\\
  &=\sum_{\substack{\phi\in \pi_2(\xs,\ys)\\ \mu(\phi)=1}} \sum_{u\in  \hat{\cM}(\phi)} U^{n_{\ws}(\phi)}\cdot \ys\times \Theta_{\xi,\zeta}^+\\
  &=\d_0(\xs)\otimes \Theta_{\xi,\zeta}^+.
 \end{split}\end{equation*}

  Hence it remains to establish the count from equation~\eqref{eq:maincountofdisks}. We establish equation~\eqref{eq:maincountofdisks} by considering a path $\cD\colon [0,\infty)\to \Sym^{n_1}(\bD)\times \cdots \times \Sym^{n_k}(\bD)$ satisfying the following:
\begin{enumerate}
\item $\cD(0)=\ve{d}$.
\item $\cD$ has image disjoint from the fat diagonal.
\item The components of $\cD(t)$ are spaced at least distance $t$ apart, with respect to the Euclidean metric on $[0,1]\times \R$.
\item The $[0,1]$-component of each point of $\cD(t)$ approaches $1$.
\end{enumerate}
We consider the moduli space
\[
\cM(\cD):=\bigcup_{t\in [0,\infty)}\bigcup_{\substack{\phi_0\in \pi_2(\Theta_{\xi,\zeta}^+,\Theta_{\xi,\zeta}^+)\\ n_{w_i}(\phi_0)=n_i}} \cM(\phi_0,\cD(t)).
\]

For generically chosen almost complex structure, the space $\cM(\cD)$ is a 1-manifold, with three types of ends: ends at $t=0$, ends at finite $t>0$, and ends appearing as $t\to \infty$. At finite $t$, the dimension count from Proposition~\ref{prop:transversalitydisks} implies that the only ends which can appear correspond to an index 1 strip breaking off, which has 0 multiplicity over the basepoints. The total number of such ends is zero (modulo 2), since $\hat{\d}(\Theta_{\xi,\zeta}^+)=0$. The limiting curve as $t\to \infty$ consists of $n_1+\cdots +n_k$ index 2 cylindrical boundary degenerations (i.e. curves $u\colon S\to \Sigma\times [0,\infty)\times \R)$, and a holomorphic strip representing the constant class $e_{\Theta_{\xi,\zeta}^+}$. According to \cite{OSLinks}*{Theorem~5.5}, the count of index 2 boundary degenerations modulo the action of conformal automorphisms of $[0,\infty)\times \R$ is $1$. It follows that the total count of ends appearing as $t\to \infty$ is 1, so equation~\eqref{eq:maincountofdisks} holds.

 We now consider the case that $|\ws_0|=1$. In this case, the conditions \ref{def:J1}--\ref{def:J5} do not prevent curves $u_0$ from appearing in $U_0$ which do not satisfy \ref{def:M3}. For example, a closed copy of $\Sigma_0\times \{(s,t)\}$ could appear. As in the proof of stabilization invariance from \cite{LipshitzCylindrical}*{Section~12} one solution to this problem is to consider almost complex structures satisfying \ref{def:J1}--\ref{def:J4} and \ref{def:J5'}, which achieve transversality at curves satisfying \ref{def:M1}--\ref{def:M5} and \ref{def:M3'}, and have no multiply covered components. In this case, the assumption that $\ve{d}$ avoids the fat diagonal implies that no curve in $\cM(\phi_0,\ve{d})$ has a multiply covered component. Hence, for a generic choice of almost complex structure, the space $\cM(\phi_0,\ve{d})$ will be transversely cut out by Proposition~\ref{prop:transversalitydisks}.

  We now establish the count appearing in equation~\eqref{eq:maincountofdisks}. Similar to before, we consider a path $\cD$ with $\cD(0)=\ve{d}$, such that $\cD(t)$ is not in the fat diagonal, and the components of $\cD(t)$ are spaced at least distance $t$ apart. Unlike before, we now assume that the points of $\cD(t)$ are bounded away from $\{0,1\}\times \R$. Adapting the previous argument, it remains to count the elements of the moduli space $\cM(\phi_0,d)$, where $\phi_0$ is an index 2 class with domain $[\Sigma_0]$, and $d$ is a single point in $(0,1)\times\R$. We do this somewhat indirectly, using a modification of Lipshitz's argument to prove stabilization invariance (compare~\cite{LipshitzCylindrical}*{Figure~19}).
  
  We stabilize the diagram $(\Sigma_0,\xis,\zetas,w_0)$ near $w_0$, adding a new basepoint $w$ and a pair of isotopic curves, $\alpha_0$ and $\beta_0$. See Figure~\ref{fig::32} for the precise configuration. Let $\alpha_0\cap \beta_0=\{x^+,x^-\}$. There is an index 1 class in $\pi_2(x^+\times \Theta_{\xi,\zeta}^+,x^-\times \Theta_{\xi,\zeta}^+)$, formed by taking the connected sum of the bigon $B$ on $(S^2,\alpha_0,\beta_0,w,w_0)$ which has  multiplicity 1 in the connected sum region, together with the class $\phi_0$. See Figure~\ref{fig::32}. 
  
By a gluing argument, we see that for sufficiently stretched neck,  
\begin{equation}
\#\hat{\cM}(B\# \phi_0)\equiv \# \hat{\cM}(B)\cdot\#\cM(\phi_0,d)=\# \cM(\phi_0,d).\label{eq:gluing-phi-0-to-bigon}
\end{equation} 

There is another bigon class $B'\in \pi_2(x^+\times \Theta_{\xi,\zeta}^+, x^-\times \Theta_{\xi,\zeta}^+)$, which has a unique representative. See Figure~\ref{fig::32}.

For sufficiently large $T$, our previous argument implies that the differential, applied to $x^+\times \Theta_{\xi,\zeta}^+$, counts only classes which satisfy Case~\eqref{case:c1} or~\eqref{case:c2}. Consequently, $B'$ and $B\# \phi_0$ are the only classes in $\pi_2(x^+\times \Theta_{\xi,\zeta}^+, x^-\times \Theta_{\xi,\zeta}^+)$ which contribute to the differential.

However, by invariance of Heegaard Floer homology, 
\[
\hat{\HF}(\Sigma_0, \alpha_0\cup \xis, \beta_0\cup \zetas, w,w_0)\iso V^{\otimes (g(\Sigma_0)+1)},
\] 
(where $V\iso H^1(S^1;\bF_2)$) which has the same rank as $\hat{\CF}(\Sigma_0,\alpha_0\cup \xis,\beta_0\cup \zetas,w,w_0)$. Consequently, the differential on $\hat{\CF}(\Sigma_0,\alpha_0\cup \xis,\beta_0\cup \zetas,w,w_0)$ must vanish, so
\[
\# \hat{\cM}(B\# \phi_0)\equiv \# \hat{\cM}(B')\equiv 1.
\]
Combining this with equation~\eqref{eq:gluing-phi-0-to-bigon}, we conclude that $\cM(\phi_0,d)\equiv 1$, which concludes the proof of the first formula in the main statement, when $|\ws_0|=1$.

 A straightforward modification proves the relation $\d_0\circ F_3^{\xi,\zeta}=F_3^{\xi,\zeta}\circ \d_{J(T)}$, completing the proof.
\end{proof}

 \begin{figure}[ht!]
 	\centering
 	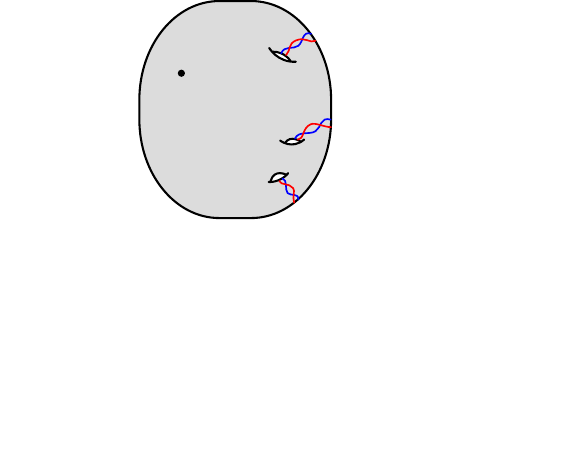
 	\caption{\textbf{Counting $\#\cM(\phi_0,d)$ indirectly when $|\ws_0|=1$.} We stretch the almost complex structure on the dashed line.\label{fig::32}}
 \end{figure}

\subsection{Holomorphic triangles and generalized 1-handle and 3-handle maps}

We now address the interaction of the holomorphic triangle maps with the generalized 1-handle and 3-handle maps. Our result should be thought of as a stronger version of the holomorphic triangle map computation used to show the well-definedness of the 1-handle map (\cite{OSTriangles}*{Theorem~4.10}).

Suppose that $(\Sigma,\as,\bs,\gs,\ws)$ and $(\Sigma_0,\xis,\zetas,\taus,\ws_0)$ are Heegaard triples. Suppose further that $(\Sigma_0,\xis,\zetas,\taus,\ws_0)$ satisfies the following:

\begin{enumerate}[label= ($T$\arabic*),ref= ($T$\arabic*),leftmargin=*, widest=III]
\item\label{cond:triple1} The Heegaard triple $(\Sigma_0,\xis,\zetas,\taus,\ws_0)$ is related by a sequence of handleslides and isotopies to a triple where all three sets of attaching circles are equal.
\item\label{cond:triple2} The collections $\xis,$ $\zetas,$ and $\taus$ can be ordered so that $|\xi_i\cap \zeta_j|=|\xi_i\cap \tau_j|=|\zeta_i\cap \tau_j|=2\delta_{ij}$, where $\delta_{ij}$ denotes the Kronecker delta function.
\end{enumerate}

Condition~\ref{cond:triple1} allows us to interpret the triangle counts on $(\Sigma_0,\xis,\zetas,\taus,\ws_0)$ as being associated to a sequence of Heegaard moves on a diagram for $(S^1\times S^2)^{\# g(\Sigma_0)}$. Together, Conditions~\ref{cond:triple1} and \ref{cond:triple2} imply that there are top degree intersection points 
\[
\Theta_{\xi,\zeta}^+\in \bT_{\xi}\cap \bT_{\zeta},\qquad \Theta^+_{\zeta,\tau}\in \bT_{\zeta}\cap \bT_{\tau} \qquad \text{and} \qquad \Theta^+_{\xi,\tau}\in \bT_{\xi}\cap \bT_{\tau}.
\] 
Similarly there are well-defined bottom degree intersection points $\Theta_{\xi,\zeta}^-,$ $\Theta^-_{\zeta,\tau},$ and $\Theta^-_{\xi,\tau}$.

If we pick an embedding $f:\ws_0\to \Sigma\setminus (\as\cup \bs\cup \gs\cup \ws)$, we can form the connected sum of the two Heegaard triples at the points identified by $f$, and obtain a Heegaard triple
\[
(\Sigma\, \#_f \Sigma_0, \as\cup \xis,\bs\cup \zetas,\gs\cup \taus,\ws).
\]
 In \cite{OSDisks}*{Section~8} Ozsv\'{a}th and Szab\'{o} describe a 4-manifold $X_{\a,\b,\g}$ associated to the Heegaard triple $(\Sigma,\as,\bs,\gs)$. We review the construction presently. Let $\Delta$ denote a triangle, with edges labeled $e_{\a},$ $e_{\b}$ and $e_{\g}$ (in that order, clockwise). We let $U_{\a},$ $ U_{\b}$ and $U_{\g}$ denote the 3-dimensional handlebody obtained by attaching compressing disks to $\Sigma\times [0,1]$, along $\as$, $\bs$, or $\gs$, and then attaching $|\ws|$ 3-handles. The 4-manifold $X_{\a,\b,\g}$ is defined by
\begin{equation}
X_{\a,\b,\g}:=\big((\Sigma\times \Delta)\sqcup (U_{\a}\times e_{\a})\sqcup (U_{\b}\times e_{\b})\sqcup (U_{\g}\times e_{\g})\big)/{\sim}
\label{eq:Xabgdef}
\end{equation} where $\sim$ is the relation determined by gluing $U_{\sigma} \times e_{\sigma}$ to $\Sigma\times \Delta$ along $\Sigma \times e_{\sigma}$ for each $\sigmas\in \{\as,\bs,\gs\}$, using the natural identification.

We begin with the following topological lemma about 4-dimensional $\Spin^c$ structures:
\begin{lem}\label{lem:connectedsumspincstructures}Suppose $(\Sigma,\as,\bs,\gs,\ws)$ and $(\Sigma_0,\xis,\zetas,\taus,\ws_0)$ are arbitrary Heegaard triples, with a chosen embedding $f:\ws_0\to \Sigma\setminus (\as\cup \bs\cup \gs\cup \ws)$, with which we take the connected sum. Writing $X_{\a\cup \xi, \b\cup \zeta,\g\cup \tau}$ for the 4-manifold constructed from the triple $(\Sigma\, \#_{f} \Sigma_0, \as\cup \xis,\bs\cup \zetas,\gs\cup \taus,\ws)$, there is a natural isomorphism
\[\Spin^c(X_{\a\cup \xi,\b\cup \zeta,\g\cup \tau})\iso \Spin^c(X_{\a,\b,\g})\times \Spin^c(X_{\xi,\zeta,\tau}).\]
\end{lem}
\begin{proof}The claim is proven by analyzing two Mayer-Vietoris exact sequences. We will define a map from $\Spin^c(X_{\a\cup \xi,\b\cup \zeta,\g\cup \tau})$ to $\Spin^c(X_{\a,\b,\g})\times \Spin^c(X_{\xi,\zeta,\tau})$ as a composition of a restriction map, and the inverse of another restriction map (which is an isomorphism). 

For notational clarity, if $(\Sigma',\as',\bs',\gs')$ is a Heegaard triple, we will write $X_{\Sigma',\a',\b',\g'}$ for the 4-manifold constructed from the Heegaard triple $(\Sigma',\as',\bs',\gs')$. The construction makes sense when $\Sigma'$ has non-empty boundary.

Let $\ve{p}\subset \Sigma$ denote the image of $f$. Note that topologically $X_{\Sigma\, \#_{f} \Sigma_0,\a\cup \xi,\b\cup \zeta,\g\cup \tau}$ is obtained by gluing $X_{\Sigma\setminus N(\ps),\a,\b,\g}$ to $X_{\Sigma_0\setminus N(\ws_0), \xi,\zeta,\tau}$ along $|\ws_0|$ thrice punctured copies of $S^3$. We leave it to the reader to analyze the Mayer-Vietoris exact sequence for cohomology and verify that the map
\[
\Spin^c(X_{\Sigma\, \#_f \Sigma_0,\a\cup \xi,\b\cup \zeta,\g\cup \tau})\to \Spin^c(X_{\Sigma\setminus N(\ps),\a,\b,\g})\times \Spin^c(X_{\Sigma_0\setminus N(\ws_0),\xi,\zeta,\tau})
\]
 is an isomorphism. Finally, it is not hard to verify that
\[
X_{\Sigma,\a,\b,\g}\setminus X_{\Sigma\setminus N(\ps),\a,\b,\g},
\]
 is a topologically a 4-ball, so a similar argument shows that the restriction maps
\[
\Spin^c(X_{\Sigma,\a,\b,\g})\to \Spin^c( X_{\Sigma\setminus N(\ps),\a,\b,\g})
\]
and
\[
\Spin^c(X_{\Sigma_0,\xi,\zeta,\tau})\to \Spin^c(X_{\Sigma_0\setminus N(\ws_0), \xi,\zeta,\tau})
\]
 are also isomorphisms.
\end{proof}

Note that Condition~\ref{cond:triple1} implies that the 4-manifold $X_{\xi,\zeta,\tau}$ is diffeomorphic to $X_{\xi,\xi,\xi}$. Upon filling in each end with 3-handles and 4-handles we obtain $(S^1\times S^3)^{\# g(\Sigma_0)}$. In particular,  there is a unique $\Spin^c$ structure $\frs_0\in \Spin^c(X_{\xi,\zeta,\tau})$ which restricts to the torsion $\Spin^c$ structure on each end of $X_{\xi,\zeta,\tau}$. Using Lemma~\ref{lem:connectedsumspincstructures}, it follows that if $\frs\in \Spin^c(X_{\a,\b,\g})$, then there is a well-defined $\Spin^c$ structure 
\[
\frs\# \frs_0\in \Spin^c(X_{\a\cup \xi,\b\cup \zeta,\g\cup \tau}).
\]

\begin{prop}\label{prop:generalized1-handlesandtriangles}
Suppose that $(\Sigma, \ve{\alpha},\ve{\beta},\ve{\gamma},\ve{w})$ and $(\Sigma_0,\xis,\zetas,\taus,\ve{w}_0)$ are Heegaard triples with a fixed embedding $f:\ws_0\to \Sigma\setminus(\as\cup \bs\cup \gs\cup \ws)$ and consider the triple $(\Sigma\#_f \Sigma_0,\as\cup \xis,\bs\cup \zetas,\gs\cup \taus, \ws)$. Furthermore, suppose that $(\Sigma_0,\xis,\zetas,\taus,\ws_0)$ satisfies Conditions~\ref{cond:triple1} and~\ref{cond:triple2}, above. If $\frs\in \Spin^c(X_{\a,\b,\g})$,  then for an almost complex structure sufficiently stretched on the connected sum necks, we have 
\begin{align*}
F_{\a\cup \xi,\b\cup \zeta,\g\cup \tau,\frs\# \frs_0}(F_{1}^{\xi,\zeta}(\xs), F_1^{\zeta,\tau}(\ys))&=F_1^{\xi,\tau}( F_{\a,\b,\g,\frs}(\xs,\ys)),\\
F_3^{\xi,\tau}( F_{\a\cup \xi,\b\cup \zeta,\g\cup \tau,\frs\# \frs_0}(F_{1}^{\xi,\zeta}(\xs), \ys\times \theta_2))&= F_{\a,\b,\g,\frs} (\xs, F_3^{\zeta,\tau}(\ys\times\theta_2)),\\
F_3^{\xi,\tau}( F_{\a\cup \xi,\b\cup \zeta,\g\cup \tau,\frs\# \frs_0}(\xs\times \theta_1, F_1^{\zeta,\tau}(\ys)))&= F_{\a,\b,\g,\frs} (F_3^{\xi,\zeta}(\xs\times \theta_1), \ys),
\end{align*}
for any choice of intersection points $\xs$, $\ys$, $\zs$, $\theta_1$ and $\theta_2$.

\end{prop}

\begin{proof}Writing out the definitions of the above maps, we wish to show
\begin{align*}F_{\a\cup \xi,\b\cup \zeta,\g\cup \tau,\frs\# \frs_0}(\ve{x}\times \Theta_{\xi,\zeta}^+, \ve{y}\times \Theta_{\zeta,\tau}^+)&=F_{\a,\b,\g,\frs}(\ve{x},\ve{y})\otimes \Theta_{\xi,\tau}^+,\\
\langle F_{\a\cup \xi,\b\cup \zeta,\g\cup \tau,\frs\# \frs_0}(\ve{x}\times \Theta_{\xi,\zeta}^+, \ve{y}\times \theta_2),\ve{z}\times \Theta_{\xi,\tau}^-\rangle'   &=\langle  F_{\a,\b,\g,\frs}(\ve{x},\ve{y}),\ve{z}\rangle' \cdot \langle \theta_2, \Theta_{\zeta,\tau}^-\rangle' ,\\
\langle  F_{\a\cup \xi,\b\cup \zeta,\g\cup \tau,\frs\# \frs_0}(\ve{x}\times \theta_1, \ve{y}\times \Theta^+_{\zeta,\tau}),\ve{z}\times \Theta_{\xi,\tau}^-\rangle'  &= \langle F_{\a,\b,\g,\frs}(\ve{x},\ve{y}),\ve{z}\rangle' \cdot \langle \theta_1, \Theta_{\xi,\zeta}^- \rangle'  ,
\end{align*}
where $\langle, \rangle' $ denotes the pairing from equation~\eqref{eq:pairingdefinition}.

We focus on the first formula involving only the generalized 1-handle maps. The two formulas involving the generalized 3-handle map follow from a straightforward adaptation of the following argument.

We first claim that if $\psi_0\in \pi_2(\theta_1,\theta_2,\theta_3)$ is a homology class on $(\Sigma_0,\xis,\zetas,\taus,\ws_0)$ and $\frs_{\ws_0}(\psi_0)$ is torsion on the 3 ends of $X_{\xi,\zeta,\tau}$, then
\begin{equation}
\mu(\psi_0)=-\gr(\Theta_{\xi,\zeta}^+,\theta_1)-\gr(\Theta_{\zeta,\tau}^+,\theta_2)+\gr(\Theta_{\xi,\tau}^+,\theta_3)+2n_{\ws_0}(\psi_0).\label{eq:Maslovindexisotopictriple}
\end{equation} 
equation~\eqref{eq:Maslovindexisotopictriple} holds for at least one class of triangles, since invariance of $\hat{\HF}((S^1\times S^2)^{\# g(\Sigma_0)}, \ws_0)$ implies that the holomorphic triangle map on $\hat{\CF}$ (which counts index 0 triangles which have zero multiplicity on all of the basepoints) maps the element $\Theta_{\xi,\zeta}^+\otimes \Theta_{\zeta,\tau}^+$ to $\Theta_{\xi,\tau}^+$. As there is a unique $\Spin^c$ structure on $X_{\xi,\zeta,\tau}$ which restricts to the torsion $\Spin^c$ structure on each end, it follows that if $\psi_0'$ is another triangle in $\pi_2(\theta_1,\theta_2,\theta_3)$, then the difference $\psi_0-\psi_0'$ is a sum of doubly periodic domains. Hence it is sufficient to show that the formula respects splicing disks into a homology class of triangles. To see that the formula respects this, we can use the Maslov index formula for disks from equation~\eqref{eq:Maslovindexgeneralized1-handle}, which implies equation~\eqref{eq:Maslovindexisotopictriple} in general.

In a similar fashion to our proof of equation~\eqref{eq:Maslovindexgeneralized1--handledisk}, above, we may compute the Maslov index of a connected sum of two triangle classes $\psi\# \psi_0$, as follows. If $\psi\# \psi_0\in \pi_2(\ve{x}\times \theta_1, \ve{y}\times \theta_2, \ve{z}\times \theta_3)$, then the formula for the index from \cite{SarkarMaslov} implies that
 \begin{equation}
 \mu(\psi\# \psi_0)=\mu(\psi)+\mu(\psi_0)-2n_{\ws_0}(\psi_0).
 \label{eq:Maslov-index-triangles-stabilization}
 \end{equation}
 (Similar to the case of disks, described in Proposition~\ref{prop:differentialcomp}, the only term in Sarkar's formula which is not additive under connected sum is the Euler measure, which must be corrected by $2n_{\ws_0}(\psi_0)$ when taking the connected sum).
  Combining equations~\eqref{eq:Maslovindexisotopictriple} and ~\eqref{eq:Maslov-index-triangles-stabilization}, we see
 \begin{equation}
 \mu(\psi\# \psi_0)=\mu(\psi)-\gr(\Theta_{\xi,\zeta}^+,\theta_1)-\gr(\Theta_{\zeta,\tau}^+,\theta_2)+\gr(\Theta_{\xi,\tau}^+,\theta_3).\label{eq:indexforgeneratriangles}
 \end{equation}
 Given a class $\psi\# \psi_0\in \pi_2(\ve{x}\times \Theta^+_{\xi,\zeta}, \ve{y}\times \Theta_{\zeta,\tau}^+, \ve{z}\times \theta_3)$, equation~\eqref{eq:indexforgeneratriangles} specializes to the formula
\begin{equation}\mu(\psi\# \psi_0)=\mu(\psi)+\gr(\Theta^+_{\xi,\tau},\theta_3).\label{eq:indexforrelevanttriangles}
\end{equation}

From here, the argument proceeds similarly to the proof of Proposition~\ref{prop:differentialcomp}. Given a sequence of holomorphic triangles in the homology class $\psi\# \psi_0$, for a sequence of almost complex structures $J(T_i)$, with necks of length $T_i$ inserted along the connected sum tubes, where $T_i\to \infty$, we can extract a limit consisting of a broken holomorphic triangle $U$ representing $\psi$, and a broken holomorphic triangle $U_0$ representing $\psi_0$. From  equation~\eqref{eq:indexforrelevanttriangles}, it follows that $\mu(\psi)=0$ and $\gr(\Theta_{\xi,\tau}^+, \theta_3)=0$. Hence $U$ consists of a single index 0 holomorphic triangle and $\theta_3=\Theta_{\xi,\tau}^+$. 

Write $\ps=\{p_1,\dots, p_k\}$ and $\ws_0=\{w_1,\dots, w_k\}$, where $f(w_i)=p_i$. We consider the map
  \[
  \rho^{\ws_0}: \cM(\psi_0)\to \Sym^{n_1}(\Delta)\times \cdots \times \Sym^{n_k}(\Delta)
  \] 
  defined by the formula
  \[
  \rho^{\ws_0}(u):=\big((\pi_\Delta \circ u)((\pi_\Sigma\circ u)^{-1}(w_1)),\dots,(\pi_\Delta\circ u)((\pi_\Sigma\circ u)^{-1}(w_k))\big).
  \]
   Here, $n_i:=n_{p_i}(\psi)=n_{w_i}(\psi_0)$.
   
    If $\ve{d}\in \Sym^{n_1}(\Delta)\times \cdots \times \Sym^{n_k}(\Delta)$, we consider the matched moduli space
  \[
  \cM(\psi_0,\ve{d}):=\{u_0\in \cM(\psi_0): \rho^{\ws_0}(u_0)=\ve{d}\}.
  \]
  
  If $\psi_{0}\in \pi_2(\Theta_{\xi,\zeta}^+,\Theta_{\zeta,\tau}^+,\Theta_{\xi,\tau}^+)$ is a class with $n_{w_i}(\psi_0)=n_i$, then equation~\eqref{eq:Maslovindexisotopictriple} implies that $\cM(\psi_0,\ve{d})$ has expected dimension 0. We will show that if $\ve{d}$ is not contained in the fat diagonal, then for a sufficiently generic almost complex structure
\begin{equation}\sum_{\substack{\psi_0\in \pi_2(\Theta_{\xi,\zeta}^+,\Theta_{\zeta,\tau}^+,\Theta_{\xi,\tau}^+)\\
n_{w_i}(\psi_0)=n_i}} \# \cM(\psi_0,\ve{d})\equiv 1\pmod{2}.
\label{eq:matchedholomorphictrianglecount}
\end{equation}
As in the proof of Proposition~\ref{prop:differentialcomp}, the precise meaning of ``sufficiently generic'' depends whether $|\ws_0|=1$ or $|\ws_0|>1$. If $|\ws_0|>1$, we may use an almost complex structure satisfying  \ref{def:J'1}--\ref{def:J'4}. If $|\ws_0|=1$, we must use almost complex structures instead satisfying \ref{def:J'1}, \ref{def:J'2}, \ref{def:J'3'}--\ref{def:J'5'}.
  
A gluing argument \cite{LipshitzCylindrical}*{Proposition~A.2} implies that if $\mu(\psi)=0$ and $\psi_0\in \pi_2(\Theta_{\xi,\zeta}^+,\Theta_{\zeta,\tau}^+,\Theta_{\xi,\tau}^+)$, then for sufficiently stretched almost complex structure 
 \[
 \# \cM(\psi\# \psi_0)=\sum_{u\in\cM(\psi)} \# \cM(\psi_0, \rho^{\ps}(u)).
 \]
  Using this, the main result follows from the following manipulation
 \begin{align*} &\qquad F_{\a\cup \xi,\b\cup \zeta,\g\cup \tau,\frs\# \frs_0}(\ve{x}\times \Theta_{\xi,\zeta}^+, \ve{y}\times \Theta_{\zeta,\tau}^+)\\
 &=\sum_{\theta_3\in \bT_{\xi}\cap \bT_{\tau}}\sum_{\substack{ 
 \psi\#\psi_0\in \pi_2(\xs\times \Theta_{\xi,\zeta}^+,\ys\times \Theta_{\zeta,\tau}^+,\zs\times \theta_3)\\\mu(\psi\# \psi_0)=0\\\frs_{\ws}(\psi\#\psi_0)=\frs\#\frs_0}} \# \cM(\psi\# \psi_0) U^{n_{\ws}(\psi\# \psi_0)}\cdot  \zs\times \theta_3\\
 &=\sum_{\substack{\psi\in \pi_2(\xs,\ys,\zs)\\ \mu(\psi)=0\\\frs_{\ws}(\psi)=\frs}} \sum_{\substack{\psi_0\in \pi_2(\Theta_{\xi,\zeta}^+,\Theta_{\zeta,\tau}^+,\Theta_{\xi,\tau}^+)\\ n_{w_i}(\psi_0)=n_i}} \# \cM(\psi\# \psi_0) U^{n_{\ws}(\psi)} \cdot \zs\times \Theta_{\xi,\tau}^+\\
 &=\sum_{\substack{\psi\in \pi_2(\xs,\ys,\zs)\\ \mu(\psi)=0\\\frs_{\ws}(\psi)=\frs}} \sum_{\substack{\psi_0\in \pi_2(\Theta_{\xi,\zeta}^+,\Theta_{\zeta,\tau}^+,\Theta_{\xi,\tau}^+)\\ n_{w_i}(\psi_0)=n_i}} \sum_{u\in\cM(\psi)} \# \cM(\psi_0, \rho^{\ps}(u))  U^{n_{\ws}(\psi)} \cdot \zs\times \Theta_{\xi,\tau}^+\\
 &=\sum_{\substack{\psi\in \pi_2(\xs,\ys,\zs)\\ \mu(\psi)=0\\\frs_{\ws}(\psi)=\frs}} \sum_{u\in\cM(\psi)}U^{n_{\ws}(\psi)}\cdot \zs\times \Theta_{\xi,\tau}^+\\
 &=F_1^{\xi,\tau}(F_{\a,\b,\g,\frs}(\xs,\ys)).
 \end{align*}

 It remains to establish equation~\eqref{eq:matchedholomorphictrianglecount}. The argument is similar to the one in Proposition~\ref{prop:differentialcomp}. We consider a path
 \[
\cD\colon [0,\infty)\to \Sym^{n_1}(\Delta)\times \cdots \times \Sym^{n_k}(\Delta),
 \]
 such that
 \begin{enumerate}
 \item  $\cD(0)=\ve{d}$.
 \item The image of $\cD$ is disjoint from the fat diagonal.
 \item The components of $\cD(t)$ approach $\infty$ in the $\xis$-$\zetas$ cylindrical end of $\Delta$.
 \item In the $\xis$-$\zetas$ cylindrical end of $\Delta$, as $t\to \infty$,  $\cD(t)$ approaches some fixed $\ve{d}'\in \Sym^{n_1}(\bD)\times \cdots \times \Sym^{n_k}(\bD)$ (up to overall translation by $\R$).
 \end{enumerate}
 
 We consider the ends of the 1-dimensional moduli space
 \[
\cM(\cD)=\bigcup_{t\in [0,\infty)}\bigcup_{\substack{\psi_0\in \pi_2(\Theta_{\xi,\zeta}^+,\Theta_{\zeta,\tau}^+,\Theta_{\xi,\tau}^+)\\ n_{w_i}(\psi_0)=n_i}} \cM(\psi_0,\cD(t)). 
 \]
 
 The ends of $\cM(\cD)$ at $t=0$ have total count equal to the left side of equation~\eqref{eq:matchedholomorphictrianglecount}. 
 
 As in the proof of Proposition~\ref{prop:differentialcomp}, the ends at finite $t$ correspond to index 1 holomorphic disks breaking off, which do not cover $\ws_0$. Since $\Theta_{\xi,\zeta}^+,$ $\Theta_{\zeta,\tau}^+$ and $\Theta_{\xi,\tau}^+$ are all cycles in their respective hat complexes, these ends cancel modulo 2. 

A broken curve appearing at $t=\infty$ must contain a level containing a holomorphic triangle, which has zero multiplicity over $\ws_0$, as well as a 
level containing a holomorphic strip in the $\xis$-$\zetas$ cylindrical end which matches $\ve{d}'$. Additional levels are prohibited since they would raise the Maslov index. Let $\phi_0\in \pi_2(\Theta_{\xi,\zeta}^+,\theta)$ and $\hat{\psi}_0\in \pi_2(\theta,\Theta_{\zeta,\tau}^+,\Theta_{\xi,\tau}^+)$ be the classes of a holomorphic disk and triangle which appear in the limit. The Maslov index of $\hat{\psi}_0$ is $-\gr(\Theta_{\xi,\zeta}^+,\theta)$ by equation~\eqref{eq:Maslovindexisotopictriple}. Hence, for $\cM(\hat{\psi}_0)$ to be non-empty, we must have $\theta=\Theta_{\xi,\zeta}^+$.

 Hence, the ends appearing as $t\to \infty$ can be identified with the Cartesian product
 \begin{equation}
\Bigg(\bigcup_{\substack{\hat{\psi}_0\in \pi_2(\Theta_{\xi,\zeta}^+,\Theta_{\zeta,\tau}^+,\Theta_{\xi,\tau}^+)\\ n_{w_i}(\hat{\psi}_0)=0\\\mu(\hat{\psi}_0)=0}}\cM(\hat{\psi}_0)\Bigg)\times \Bigg( \bigcup_{\substack{\phi_0\in \pi_2(\Theta^+_{\xi,\zeta},\Theta^+_{\xi,\zeta})\\ n_{w_i}(\phi_0)=n_i}} \cM(\phi_0, \ve{d}')\Bigg). 
 \label{eq:infinity-limit-triangles}
 \end{equation}
 
 We note
 \begin{equation}
 \sum_{\substack{\hat{\psi}_0\in \pi_2(\Theta_{\xi,\zeta}^+,\Theta_{\zeta,\tau}^+,\Theta_{\xi,\tau}^+)\\ n_{w_i}(\hat{\psi}_0)=0\\\mu(\hat{\psi}_0)=0}}\#\cM(\hat{\psi}_0)\equiv 1 \pmod{2}, \label{eq:hat-count-triangles}
 \end{equation}
 since Conditions~\ref{cond:triple1} and~\ref{cond:triple2} allow us to interpret the count of such triangles as the $\Theta_{\xi,\tau}^+$-component of the image of $\Theta_{\xi,\zeta}^+$ under the map from naturality which moves $\zetas$ to $\taus$ via a sequence of handleslides and isotopies.
 
 Combining equation~\eqref{eq:infinity-limit-triangles} with equations~\eqref{eq:hat-count-triangles} and~\eqref{eq:maincountofdisks}, we obtain that the total count of the ends of $\cM(\cD)$ which appear as $t\to \infty$ is 1. Hence the $t=0$ ends of $\cM(\cD)$ have the same total count as the $t=\infty$ end, so equation~\eqref{eq:matchedholomorphictrianglecount} follows, and the proof is complete.

\end{proof}

 \subsection{Variations of the generalized 1-handle and 3-handle maps}

We will  need a variation of the generalized 1-handle and 3-handle maps, where the diagram for $(S^1\times S^2)^{\# n}$ is allowed to have extra basepoints which are not merged to those in $Y$:

 \begin{rem}\label{rem:extrabasepoints}Our definition of the generalized 1-handle and 3-handle maps extends to the case when we wish to attach a diagram  $(\Sigma_0,\xis,\zetas,\ws_0\cup \ws_1)$ for $(S^1\times S^2)^{\# g(\Sigma_0)}$ to a diagram $(\Sigma,\as,\bs,\ws)$, using an embedding $f:\ws_0\to \Sigma\setminus (\as\cup \bs\cup \ws)$. In order to use the holomorphic curve counts from Proposition~\ref{prop:differentialcomp}, we describe how this construction is a special case, rather than a generalization, of our previous construction. To see this, we add a copy of $(S^2,w)$ to $(\Sigma,\as,\bs,\ws)$ for each basepoint $w\in \ws_1$, and then use the generalized 1-handle map for joining $(\Sigma_0,\xis,\zetas,\ws_0\cup \ws_1)$ to $(\Sigma,\as,\bs,\ws)\cup \coprod_{w\in \ws_1}(S^2,w)$ using the natural extension of $f$ to $\ws_0\cup \ws_1$. The generalized 1-handle map has domain $\CF^-(\Sigma\cup \coprod_{w\in \ws_1} S^2,\as,\bs,\ws\cup \ws_1)$, which is canonically isomorphic to $\CF^-(\Sigma,\as,\bs,\ws)$. Similarly the holomorphic triangle counts from Proposition~\ref{prop:generalized1-handlesandtriangles} can be applied when we are given two Heegaard triples, $(\Sigma,\as,\bs,\gs,\ws)$ and $(\Sigma_0,\xis,\zetas,\taus,\ws_0\cup \ws_1)$, and an embedding $f:\ws_0\to \Sigma\setminus (\as\cup \bs\cup \gs\cup \ws)$ used to form their connected sum.
 \end{rem}

\section{Doubling a Heegaard diagram}
\label{sec:doubleddiagrams}

 Suppose that $\cH=(\Sigma, \ve{\alpha},\ve{\beta},\ws)$ is a multi-pointed Heegaard diagram for $(Y,\ws)$. In this section we describe two natural diagrams for $(Y,\ws)$,
\[
D_{\a}(\cH) \qquad \text{and} \qquad D_{\b}(\cH),
\] 
 which can be constructed from the diagram $\cH$.
  We call these the \emph{doubled Heegaard diagrams of} $\cH$. They naturally appear when we compute the trace and cotrace cobordism maps. We also give an expression for the transition maps between $\cH$ and its doubles.

\subsection{Construction of the doubled diagrams}

\label{sec:doubleddiagram}

We first describe the construction of the diagram $\cD_{\a}(\cH)$. Pick a regular neighborhood $N(\Sigma)\iso \Sigma\times [-1,1]$ of $\Sigma$ in $Y$, such that $\Sigma$ is embedded as $\Sigma\times \{0\}$. Let $N'(\ws)$ denote a collection of $|\ws|$ pairwise disjoint closed disks in $\Sigma\setminus (\as\cup \bs)$,  each containing a single basepoint of $\ws$ in its boundary (i.e. the disks $N'(\ws)$ are obtained by translating a regular neighborhood $N(\ws)$ slightly so that $\ws\subset \d N'(\ws)$).

Remove $(\Int N'(\ws))\times [-1,1]$ from $N(\Sigma)$ to obtain a handlebody of genus $2g(\Sigma)+|\ws|-1$, which we denote by $U_\Sigma$. Write $\Sigma\, \#_{\ws} \bar{\Sigma}$ for $-\d U_\Sigma$, and note that $\Sigma \, \#_{\ws} \bar{\Sigma}$ is a Heegaard surface in $Y$ which contains $\ws$. Using the standard orientation convention,  $Y\setminus U_\Sigma$ becomes the $\alpha$-handlebody, and $U_\Sigma$ the $\beta$-handlebody.

 Pick a set of compressing curves $\Ds$ on $\Sigma\, \#_{\ws} \bar{\Sigma}$ for the handlebody $U_\Sigma$. Let $\as\subset \Sigma\, \#_{\ws} \bar{\Sigma}$ denote the images of the original $\as$ curves on $\Sigma\setminus N'(\ws)$, and let $\bar{\bs}\subset \Sigma\,\#_{\ws} \bar{\Sigma}$ denote the images of the original $\bs$ curves on $\bar{\Sigma}\setminus N'(\ws)$. The curves $\as\cup \bar{\bs}\subset\Sigma\, \#_{\ws} \bar{\Sigma}$  bound compressing disks in $Y\setminus U_\Sigma$. Define the diagram $D_{\a}(\cH)$ as 
\[
D_{\a}(\cH):=(\Sigma\, \#_{\ws} \bar{\Sigma},\ve{\alpha}\cup \bar{\ve{\beta}},\Ds,\ve{w}).
\]

 We define $D_{\b}(\cH)$ as the conjugate Heegaard diagram, i.e., the diagram obtained by reversing the orientation of the surface and switching the roles of the $\alpha$- and $\beta$-handlebodies:
\[
D_{\b}(\cH):=(\bar{\Sigma}\, \#_{\ws} \Sigma, \Ds,\bar{\ve{\alpha}}\cup \ve{\beta},\ve{w}).
\]

An example of a neighborhood of a basepoint in a doubled Heegaard diagram is shown in Figure~\ref{fig::10}.

\begin{figure}[ht!]
	\centering
	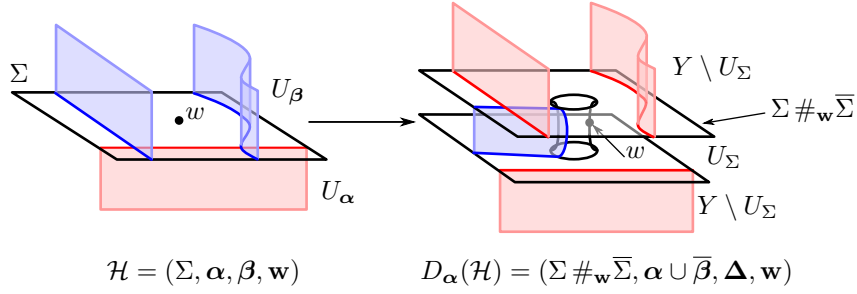
	\caption{\textbf{A neighborhood of a basepoint $w\in \ws$ in a Heegaard diagram $\cH$ and its double $D_{\a}(\cH)$.} The red and blue shaded strips denote portions of compressing disks attached to the Heegaard surface. On the right, a single $\Ds$ compressing disk is shown (blue) in $U_\Sigma$.\label{fig::10}}
\end{figure}

 There is a natural way to construct compressing curves $\Ds$ for $U_{\Sigma}$, as we now describe. Define the subsurface
 \[
 \Sigma(\ws):=\Sigma\setminus (\Int N'(\ws)),
 \] 
 where $N'(\ws)\subset \Sigma$ denotes the translated regular neighborhood of $\ws$ described previously. Pick a collection of closed arcs 
 \[
 A\subset (\d N'(\ws))\setminus \ws
 \]
  such that the set $A$ contains one arc per component of $\d \Sigma(\ws)$. We then form the surface $\Sigma(\ws)\,\natural_A \bar{\Sigma}(\ws)$, where $\natural_A$ denotes the boundary connected sum along $A$. The surface $\Sigma(\ws)\,\natural_A \bar{\Sigma}(\ws)$ has one puncture per basepoint of $\ws$, and is homeomorphic to $(\Sigma\, \#_{\ws} \bar{\Sigma})\setminus N(\ws)$.

  We now pick a collection of pairwise disjoint, properly embedded arcs $d_1,\dots, d_{2n}$ on $\Sigma(\ws)$, which have endpoints on $A$, and which form a basis of $H_1(\Sigma(\ws), A)$ (here $n=|\as|=|\bs|=g(\Sigma)+|\ws|-1$).  We take the arcs $d_i$ on $\Sigma(\ws)$, and concatenate them with their mirrors on $\bar{\Sigma}(\ws)$ to form a collection of $2n$ simple closed curves $\delta_1,\dots, \delta_{2n}$ on $\Sigma\, \#_{\ws} \bar{\Sigma}$. In the definition of the doubled Heegaard diagram, above, we can take $\Ds=\{\delta_1,\dots, \delta_{2n}\}$ as our choice of compressing curves for the handlebody $U_\Sigma$.

We  now show that any set of curves $\Ds$, constructed using the doubling procedure above, form a valid set of attaching curves in the sense of Definition~\ref{def:multipointedheegaarddiagram}. They clearly satisfy requirements \eqref{def:mphd1}--\eqref{def:mphd3}. It remains to show~\eqref{def:mphd4}:

\begin{lem}\label{lem:doublingvalid}The curves $\Ds$ are homologically independent in $ (\Sigma\,\, \#_{\ws} \bar{\Sigma})\setminus \ws$.
\end{lem}
\begin{proof}As described above, we have a diffeomorphism $(\Sigma\, \#_{\ws} \bar{\Sigma})\setminus N(\ws)\iso \Sigma(\ws)\,\natural_A\bar{\Sigma}(\ws)$. Consider the composition
\[H_1(\Sigma(\ws)\,\natural_A \bar{\Sigma}(\ws))\to H_1(\Sigma(\ws)\,\natural_A \bar{\Sigma}(\ws), \bar{\Sigma}(\ws); \Z)\to H_1(\Sigma(\ws), A).\] The first map is the natural map, and the second map is the inverse of the excision isomorphism. The composition sends $\delta_i$ to $d_i$. Since the $d_i$ are linearly independent by assumption, it follows that the $\delta_i$ are as well.
\end{proof}

\subsection{Computing the transition map between \texorpdfstring{$\cH$}{H} to \texorpdfstring{$D_{\a}(\cH)$}{Da(H)}}

 It will be important for our purposes to have a simple formula for the transition maps $\Psi_{\cH\to D_{\a}(\cH)}$ and $\Psi_{D_{\a}(\cH)\to \cH}$. In this section, we define a candidate map, and then prove that it coincides with the transition map. 

Note that if $\bs$ is any set of attaching curves on $\Sigma$, then  $(\Sigma\, \#_{\ws} \bar{\Sigma},\bs\cup \bar{\bs},\Ds,\ws)$ is a multi-pointed diagram for $(S^1\times S^2)^{\# g(\Sigma)}$, because it is a double of the diagram $(\Sigma,\bs,\bs,\ws)$. Hence $\HF^-(\Sigma\, \#_{\ws} \bar{\Sigma},\bs\cup \bar{\bs},\Ds,\ws)$ contains a top degree element $\Theta_{\b\cup \bar{\b},\Dt}^+$. 

We consider the generalized 1-handle map
\[
F_1^{\bar{\b},\bar{\b}}: \CF^-(\Sigma,\as,\bs,\ws)\to \CF^-(\Sigma\, \#_{\ws} \bar{\Sigma}, \as\cup \bar{\bs}, \bs\cup \bar{\bs}, \ws),
\]
 as described in Section~\ref{sec:generalized1--handleand3--handlemaps}. (In the above map, the second copy of $\bar{\bs}$ should be replaced with the image of $\bar{\bs}$ under a small isotopy, though we suppress this from the notation).

The main result of this section is the following:

\begin{prop}\label{prop:changeofdiagramsmapcomp}If $\cH=(\Sigma,\as,\bs,\ws)$ is a multi-pointed Heegaard diagram, and $\Ds$ is any set of attaching curves obtained by the doubling procedure described in Section~\ref{sec:doubleddiagram},  then the transition map $\Psi_{\cH\to D_{\a}(\cH)}$ satisfies the formula
\[
\Psi_{\cH\to D_{\a}(\cH)}\simeq F_{\a\cup \bar{\b}, \b\cup \bar{\b}, \Delta}( F_1^{\bar{\b},\bar{\b}}, \Theta_{\b\cup \bar{\b},\Dt}^+).
\]
\end{prop}

\begin{rem} In Proposition~\ref{lem:doublingvalid}, we have omitted a $\Spin^c$ structure in the triangle map  $F_{\a\cup \bar{\b}, \b\cup \bar{\b}, \Dt}$. We will see that the triple $(\Sigma\, \#_{\ws} \bar{\Sigma}, \as\cup \bar{\bs}, \bs\cup \bar{\bs},\Ds)$ represents surgery on a link embedded in $Y\#(S^1\times S^2)^{\# |\bs|}$ which topologically cancels the summands of $S^1\times S^2$ added by the generalized 1-handle map. Thus, by attaching 3-handles and 4-handles to $X_{\a\cup \bar{\b}, \b\cup \bar{\b}, \Dt}$ we obtain the identity cobordism $Y\times [0,1]$, so there is a unique $\Spin^c$ structure on $X_{\a\cup \bar{\b}, \b\cup \bar{\b}, \Dt}$ which extends to $\frs$ on $Y\times [0,1]$.
\end{rem}

The proof of Proposition~\ref{prop:changeofdiagramsmapcomp} is somewhat involved, though the idea is simple to state: the expression in Proposition~\ref{prop:changeofdiagramsmapcomp} represents the cobordism map for a collection of topologically canceling 1-handles and 2-handles. If the reader is satisfied with this level of reasoning, they can safely skip the remainder of the proof. For the undeterred reader, we now embark upon providing a proper proof of Proposition~\ref{prop:changeofdiagramsmapcomp}.

The first step is to specify a collection of 0-spheres and a framed link. Let $D_{\beta_i}\subset Y$ denote a choice of compressing disk for $\beta_i\in \bs$. Inside of the handlebody $U_{\b}$, we pick regular neighborhoods 
\[
N(\Sigma),\, N(D_{\beta_i})\subset U_{\b}.
\]
 Further, we pick trivializing diffeomorphisms
\[
\tau_i:N(D_{\beta_i})\to D_{\beta_i}\times [-1,1] \qquad \text{and} \qquad \tau':N(\Sigma)\to \Sigma\times [0,1],
\] 
such that $\tau_i(D_{\beta_i})=D_{\beta_i}\times \{0\}$ and $\tau'(\Sigma)=\Sigma\times \{0\}$.

Using the maps $\tau_i$ and $\tau'$, we can specify 0-spheres $S_1,\dots, S_n$ and a framed link $\bL\subset Y(S_1,\dots, S_n)$. We define the 0-sphere $S_i\subset Y$ to be 
\begin{equation}
S_i:=\{0\}\times \left\{-\tfrac{1}{2},\tfrac{1}{2}\right\} \subset D_{\beta_i}\times [-1,1],
\label{eq:defframed0spheres}
\end{equation}
 and we define the link component $\ell_i\subset Y(S_1,\dots, S_n)$ of $\bL$ to be
\begin{equation}
\ell_i:=\{0\}\times \left[-\tfrac{1}{2}, \tfrac{1}{2}\right]\subset D_{\beta_i}\times [-1,1].
\label{eq:defframedlink}
\end{equation}

 A neighborhood of the disk $D_{\beta_i}$, the nearby 0-sphere $S_i$, and link component $\ell_i$ are shown in Figure~\ref{fig::52}.

The trivialization $\tau_i$ of $N(D_{\beta_i})$ determines a framing of the link component $\ell_i$, which is given as a vector field along $\ell_i$ that projects to a single vector in $T_0 D_{\beta_i}$ under the projection map $D_{\beta_i}\times [-1,1]\to D_{\beta_i}$.

 \begin{figure}[ht!]
 	\centering
 	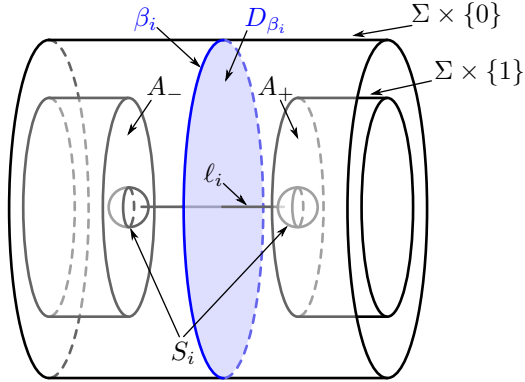
 	\caption{\textbf{The Heegaard surface $\Sigma'\subset Y(S_1,\dots, S_n)$ inside the neighborhood $N(D_{\beta_i})$ of the compressing disk $D_{\beta_i}$.} A neighborhood $N(\Sigma)\subset U_{\b}$ of the original Heegaard surface $\Sigma$ is identified with $\Sigma\times [0,1]$. The original Heegaard surface $\Sigma$ is identified with $\Sigma\times \{0\}$. The surface  $\Sigma'$ is the union of $\Sigma\times\{0\}$, a portion of $\Sigma\times \{1\}$, and the two annuli $A_-$ and $A_+$.  Surgery on the knot $\ell_i$ cancels surgery on the 0-sphere $S_i$.
 	\label{fig::52}}
 \end{figure}

We can specify a Heegaard surface 
\[\Sigma'\subset Y(S_1,\dots, S_n),\] as follows. Outside of the union of the neighborhoods $N(D_{\beta_i})$, the surface $\Sigma'$ is equal to $\Sigma\,\#_{\ws} \bar{\Sigma}= \d ((\Sigma\setminus N'(\ws))\times [0,1])$. Inside of $N(D_{\beta_i})$, we define
\begin{align*}
\Sigma'\cap N(D_{\beta_i}):=&(\Sigma\times \{0\})\cap N(D_{\beta_i})\\
&\cup (\Sigma\times \{1\})\cap \left(D_{\beta_i}\times \left(\left[-1,-\tfrac{1}{2}\right]\cup \left[\tfrac{1}{2},1\right]\right)\right)\\
&\cup A_-\cup A_+,
\end{align*}  where $A_{-}$ and $A_{+}$ are two annular subsets of $D_{\beta_i}\times \left\{-\tfrac{1}{2}\right\}$ and $D_{\beta_i} \times  \left\{\tfrac{1}{2}\right\}$, respectively. The surface $\Sigma'$  is shown in Figure~\ref{fig::52}.

It is not hard to see that the two diffeomorphisms $\tau_i$ and $\tau'$ also determine a diffeomorphism
\begin{equation}\phi:\Sigma\, \#_{\ws} \bar{\Sigma}\to \Sigma',\label{eq:embeddingphidef}\end{equation} up to isotopy (examine Figure~\ref{fig::52}).

For each 0-sphere $S_i$, there is a 1-handle map $F_{S_i}$, defined in \cite{ZemGraphTQFT}*{Section~8}. The definition is similar but not identical to the definition of the 1-handle map from \cite{OSTriangles}.

As a first step, we relate the map $F_1^{\bar{\b},\bar{\b}}$ appearing in Proposition~\ref{prop:changeofdiagramsmapcomp} with a cobordism map:

\begin{lem}\label{lem:generalized1-handlemapiscompof1-handles}Suppose $(\Sigma,\as,\bs,\ws)$ is a Heegaard diagram for $Y$, and let $S_1,\dots, S_{n}$ be the 0-spheres in $U_{\b}$, described above. Let $\phi:\Sigma\, \#_{\ws} \bar{\Sigma}\to Y(S_1,\dots, S_n)$ denote the embedding describe above. Then
\[
\phi_* \circ F_1^{\bar{\b},\bar{\b}}\simeq F_{S_n}\circ \cdots \circ F_{S_1}.
\]
\end{lem}
\begin{proof} If we ignore almost complex structures, the statement is immediate, since the two maps are both defined by the formula $\ve{x}\mapsto \ve{x}\times \Theta_{\bar{\b},\bar{\b}}^+$. However the two maps have different requirements concerning how the almost complex structure on $(\Sigma\,\#_{\ws} \bar{\Sigma})\times [0,1]\times \R$ must be stretched. 

To each 0-sphere $S_i$, there is a distinguished annular region on $\bar{\Sigma}$, which contains the corresponding pair of attaching curves in the two copies of $\bar{\bs}$. By definition, an almost complex structure can be used  to compute $F_{S_i}$ if it is stretched along the two boundary components of this annulus, and the change of almost complex structure map associated to additional stretching preserves intersection points of the form $\ve{x}\times \theta^+$ (see \cite{ZemGraphTQFT}*{Definition~8.2}). 

Analogously, an almost complex structure may be used to compute $F_1^{\bar{\b},\bar{\b}}$ if it is sufficiently stretched along the $|\ws|$ connected sum necks between $\Sigma$ and $\bar{\Sigma}$.

Schematically, the two almost complex structures are shown in Figure~\ref{fig::33}. It is not \textit{a-priori} obvious that a single almost complex structure can be chosen to compute both maps.

 \begin{figure}[ht!]
 	\centering
 	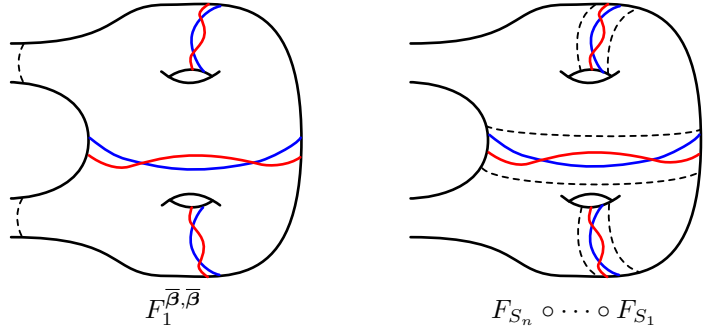
 	\caption{\textbf{A schematic of the almost complex structures used to compute the maps $F_1^{\bar{\b},\bar{\b}}$ and  $F_{S_n}\circ \cdots \circ F_{S_1}$.} The dashed lines indicate where we stretch the almost complex structures for the two maps. Shown is the subset of $\Sigma\, \#_{\ws} \bar{\Sigma}$ corresponding to $\bar{\Sigma}$. \label{fig::33}}
 \end{figure}

 We use a variation of the proof that the 1-handle maps commute with each other (\cite{ZemGraphTQFT}*{Proposition~8.3}) to see that we can pick a single almost complex, which has been stretched on $2|\bar{\bs}|+|\ws|$ necks simultaneously, and can compute both $F_1^{\bar{\b},\bar{\b}}$ and $F_{S_n}\circ \cdots \circ F_{S_1}$.

 Fix an almost complex structure $J$ on $(\Sigma\,\#_{\ws} \bar{\Sigma})\times [0,1]\times \R$. If $\ve{T}=(T_1,\dots, T_k)$ is a tuple of positive numbers with $k=2|\bs|+|\ws|$, let us write $J(\ve{T})$ for the almost complex structure on $(\Sigma\,\#_{\ws} \bar{\Sigma})\times [0,1]\times \R$ obtained from $J$ by inserting necks of length $T_1,\dots, T_k$  along  the $|\ws|$ connected sum tubes of $\Sigma\,\#_{\ws} \bar{\Sigma}$, and along the $2|\bar{\bs}|$ curves which are parallel to the $\bar{\bs}$ curves on  $\bar{\Sigma}$.

 We claim the following: If $\ve{T}$ and $\ve{T}'$ are two tuples of neck lengths, and all of the components of $\ve{T}$ and $\ve{T}'$ are sufficiently large, then
\begin{equation}
\Psi_{J(\ve{T})\to J(\ve{T}')}(\ve{x}\times \Theta_{\bar{\b},\bar{\b}}^+)=\ve{x}\times \Theta_{\bar{\b},\bar{\b}}^+.
\label{eq:cxstr1-handle=generalized1-handle}
\end{equation} Importantly, we do not assume anything about the relative sizes of the components of $\ve{T}$ and $\ve{T}'$. 

The main claim follows from equation~\eqref{eq:cxstr1-handle=generalized1-handle}, since it implies that we can find a $J(\ve{T})$ such that the change of almost complex structure map associated to additional stretching along any subcollection of the necks preserves elements of the form $\ve{x}\times \Theta_{\bar{\b},\bar{\b}}^+$.

To establish equation~\eqref{eq:cxstr1-handle=generalized1-handle},  note that the change of almost complex structure map $\Psi_{J(\ve{T})\to J(\ve{T}')}$ counts Maslov index 0 holomorphic curves for a non-cylindrical almost complex structure which agrees with $J(\ve{T})$ on $(\Sigma\, \#_{\ws} \bar{\Sigma})\times [0,1]\times (-\infty,-1]$ and agrees with $J(\ve{T}')$ on  $(\Sigma\, \#_{\ws} \bar{\Sigma})\times [0,1]\times [1,\infty)$. Suppose we are given two sequences of tuples, $\ve{T}_i$ and $\ve{T}_i'$, such that each component of each tuple individually approaches $+\infty$. We can pick a sequence of interpolating almost complex structures $\hat{J}_i$ between $J(\ve{T}_i)$ and $J(\ve{T}'_i)$ on $(\Sigma\, \#_{\ws} \bar{\Sigma}) \times [0,1]\times \R$ such that the almost complex manifold $((\Sigma\, \#_{\ws} \bar{\Sigma})\times [0,1]\times \R,\hat{J}_i)$ contains the almost complex submanifold 
\[
(\Sigma\setminus N_i(\ws)\times [0,1]\times \R,J_0|_{\Sigma\setminus N_i(\ws)\times [0,1]\times \R}),
\] 
where $J_0$ denotes a fixed, cylindrical almost complex structure on
$\Sigma\times[0,1]\times \R$, and $ N_i(\ws)$ is a sequence of regular neighborhoods of $\ws$ such that $N_{i+1}(\ws)\subset N_i(\ws)$ and $\bigcap_i N_i(\ws)=\ws$.  

Suppose $\{u_i\}_{i\in \N}$ is a sequence of $\hat{J}_i$-holomorphic curves on $(\Sigma\, \#_{\ws} \bar{\Sigma})\times [0,1]\times \R$, which represent a Maslov index 0 class $\phi\# \phi_0\in \pi_2(\xs\times \Theta_{\bar{\b},\bar{\b}}^+, \ys\times \theta)$, where $\phi$ is a class on $(\Sigma,\as,\bs)$ and $\phi_0$ is a class on $(\bar{\Sigma},\bar{\bs},\bar{\bs})$. The index formula from equation~\eqref{eq:Maslovindexgeneralized1--handledisk} shows that
\begin{equation}\mu(\phi\#\phi_0)=\mu(\phi)+\gr(\Theta_{\bar{\b},\bar{\b}}^+,\theta).\label{eq:indexformuladyamicdisks}\end{equation} 

Adapting \cite{LipshitzCylindrical}*{Sublemma A.12}, by letting $i\to \infty$, we can extract a potentially broken limiting curve on $\Sigma\times [0,1]\times \R$, for the cylindrical almost complex structure $J_0$, representing the class $\phi$. By transversality, $\mu(\phi)\ge 0$. Using equation~\eqref{eq:indexformuladyamicdisks} and the fact that $\mu(\phi\# \phi_0)=0$,  we conclude that $\mu(\phi)=0$ and $\gr(\Theta_{\bar{\b},\bar{\b}}^+,\theta)=0$. Transversality at the limiting representative of $\phi$ implies that $\phi$ is the constant homology class. Since $\gr(\Theta_{\bar{\b},\bar{\b}}^+,\theta)=0$, it also follows that $\Theta_{\bar{\b},\bar{\b}}^+=\theta$.

 It is straightforward to examine the diagram $(\bar{\Sigma},\bar{\bs},\bar{\bs},\ws)$ and observe that the only nonnegative homology class in $\pi_2(\Theta_{\bar{\b},\bar{\b}}^+, \Theta_{\bar{\b},\bar{\b}}^+)$ which has zero multiplicity on the basepoints is the constant class.

Hence, if $\ve{T}$ and $\ve{T}'$ have components which are all sufficiently large, the map $\Psi_{J(\ve{T})\to J(\ve{T}')}$ counts only the constant homology class, when applied to $\ve{x}\times \Theta_{\bar{\b},\bar{\b}}^+$. This establishes equation~\eqref{eq:cxstr1-handle=generalized1-handle}, completing the proof.
\end{proof}

 We now focus our attention on the triangle map $F_{\a\cup \bar{\b}, \b\cup \bar{\b}, \Dt}(-,\Theta_{\bar{\b},\bar{\b}}^+)$ appearing in the statement of Proposition~\ref{prop:changeofdiagramsmapcomp}. In equation~\eqref{eq:defframedlink}, we described a framed link $\bL\subset Y(S_1,\dots, S_n)$ which topologically cancels the 0-spheres $S_1,\dots, S_n\subset Y$, defined in equation~\eqref{eq:defframed0spheres}.   We give the following alternate description of the framing of $\bL$:

\begin{lem}\label{lem:welldefinedframing}
Let $\phi:\Sigma\, \#_{\ws} \bar{\Sigma}\to Y(S_1,\dots, S_n)$ denote the embedding defined above.
\begin{enumerate}
\item Suppose  $b_i\subset \Sigma\setminus N'(\ws)$ is a properly embedded arc which intersects $\beta_i$ once, intersects none of the other $\beta_j$ curves, and  intersects the boundary of $ \Sigma\setminus N'(\ws)$ at two points. Then $\bL$ is isotopic to the link obtained by doubling $b_1\cup \cdots \cup b_n$ onto $\Sigma\, \#_{\ws} \bar{\Sigma}$.
\item The framing on $\bL$ given by the trivializations of $N(\Sigma)$ and $N(D_{\beta_i})$, described above, agrees with the framing induced by a normal vector to $\Sigma\, \#_{\ws} \bar{\Sigma}$ along the curves obtained by doubling $b_1,\dots, b_n$.
\end{enumerate}
\end{lem}
\begin{proof}Both claims can be verified by examining Figure~\ref{fig::52}.
\end{proof}

   We note that Proposition~\ref{prop:changeofdiagramsmapcomp} is stated in terms of an arbitrary $\Ds$, constructed using the doubling procedure from Section~\ref{prop:changeofdiagramsmapcomp}. We describe a choice of $\Ds$ which is particularly convenient for our purposes. Pick a collection $b_1,\dots, b_n$ of pairwise disjoint arcs in $H_1(\Sigma(\ws), A; \Z)$, which have both endpoints on $A$, and which are dual to the curves $\beta_1,\dots, \beta_n$ in the sense that $|\beta_i\cap b_j|=\delta_{ij}$, where $\delta_{ij}$ denotes the Kronecker-delta (as in Lemma~\ref{lem:welldefinedframing}). Additionally, we construct arcs $b_1',\dots, b_n'$ on $\Sigma(\ws)$ by isotoping $\beta_i$ (which intersects $b_i$ in a single point) along $b_i$ (in either direction), until it intersects $\d \Sigma(\ws)$ at two points. We let $\Ds$ be the collection of curves determined by doubling $b_1,\dots, b_n,b_1',\dots, b_n'$ onto $\Sigma\, \#_{\ws} \bar{\Sigma}$. An example of the arcs $b_1,\dots, b_n,b_1',\dots, b_n'$ is shown in Figure~\ref{fig::34}.

    \begin{figure}[ht!]
    	\centering
    	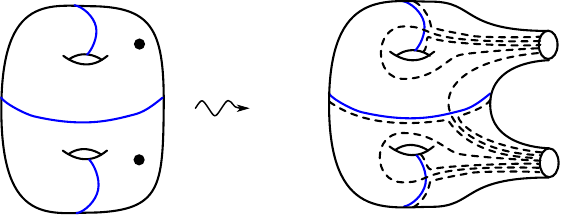
    	\caption{\textbf{The arcs $b_i$ and $b_i'$ on $\Sigma(\ws):=\Sigma\setminus \Int N'(\ws)$.} On the left are the curves $\bs\subset \Sigma$. On the right, the surface $\Sigma(\ws)$ is shown, as well as the arcs $A\subset \d N'(\ws)$ (shown in bold), and the arcs $b_i$ and $b_i'$ (both shown as dashed lines). \label{fig::34}}
    \end{figure}

   We now show that the $\Ds$ curves constructed by doubling the arcs $b_1,\dots, b_n,b_1',\dots, b_n'$ forms a valid set of attaching curves. By Lemma~\ref{lem:doublingvalid}, this amounts to proving the following:

   \begin{lem}\label{lem:particularbasisisarealbasis}If $b_1,\dots, b_n,b_1',\dots, b_n'$ denote the arcs described above, then the classes $[b_1],\dots, [b_n],$ $[b_1'],\dots, [b_n']$ form a basis of $H_1(\Sigma(\ws), A)$.
   \end{lem}
   \begin{proof}As $[b_i']$ is homologous to $[\beta_i]$, it is sufficient to show that the classes $[b_1],\dots, [b_n], [\beta_1],\dots, [\beta_n]$ form a basis of $H_1(\Sigma(\ws), A;\Z)$. The claim can then be proven by induction on the number of $\bs$ curves. In the case that $\bs$ is empty, the surface $\Sigma(\ws)$ is a collection of disks, and $H_1(\Sigma(\ws),A)$ vanishes. Assuming the claim holds for any diagram where $|\bs|=k-1$, we can prove that the claim also holds for diagrams with $|\bs|=k$ by surgering out a curve $\beta_k\in \bs$ and considering the effect on $H_1(\Sigma(\ws), A)$.  Using a Mayer-Vietoris exact sequence for the subspaces $N(\beta_k)$ and $\Sigma(\ws)\setminus \beta_k$, it is easy to see that \[\rank H_1(\Sigma(\ws),A)=\rank H_1(\Sigma(\ws)(\beta_k), A)+2.\] Furthermore, a basis of $H_1(\Sigma(\ws),A)$ is obtained from a basis of $H_1(\Sigma(\ws)(\beta_k),A)$ by adding the two generators $[\beta_k]$ and $[b_k]$, where $b_k$ is a dual arc to $\beta_k$.
    \end{proof}
 
 \begin{lem}\label{lem:randomtrianglemapis2-handlemap} Let $\Sigma\, \#_{\ws} \bar{\Sigma}$ be embedded in $Y(S_1,\dots, S_n)$ as described above. With the choice of $\Ds$ considered in Lemma~\ref{lem:particularbasisisarealbasis}, the map $F_{\a\cup \bar{\b}, \b\cup \bar{\b},\Dt}(-,\Theta_{\b\cup \bar{\b},\Dt}^+)$ is the 2-handle map, for surgery on the framed link $\bL\subset Y(S_1,\dots, S_n)$.
 \end{lem}
 \begin{proof}Let us write $\ds=\{\delta_1,\dots, \delta_n\}$ for the curves obtained by doubling $b_1,\dots, b_n$ and let us write $\ds'=\{\delta_1',\dots, \delta_n'\}$ for the curves obtained by doubling $b_1',\dots, b_n'$. By definition 
 \[
 \Ds=\ds'\cup \ds.
 \]
   The 2-handle map is defined by picking a Heegaard triple which is subordinate to a bouquet for $\bL$. In this case, the triple $(\Sigma\, \#_{\ws} \bar{\Sigma},\as\cup \bar{\bs},\bs\cup \bar{\bs},\Ds)$ is not quite a triple subordinate to a bouquet for $\bL$. Nonetheless, we will show that it is related to such a triple by a sequence of handleslides. A straightforward argument using associativity of the holomorphic triangle maps can be used to show that this implies $F_{\a\cup \bar{\b}, \b\cup \bar{\b},\Dt}(-,\Theta_{\b\cup \bar{\b},\Dt}^+)$ is chain homotopic to the 2-handle map.
 
 First handleslide each $\bs$ curve across the corresponding curve in $\bar{\bs}$. Let $\hat{\bs}$ denote the resulting curves, and let us now consider the triple $(\Sigma\, \#_{\ws} \bar{\Sigma}, \as\cup \bar{\bs}, \hat{\bs}\cup \bar{\bs}, \Ds)$. Note that to handleslide the curve $\beta_i$ over $\bar{\beta}_i$, we need to pick a path from $\beta_i$ to $\bar{\beta}_i$. The arc $b_i$ (which intersects $\beta_i$ exactly once) provides two choices of path from $\beta_i$ to $\bar{\beta}_i$, and we choose the one which is consistent with our choice for the arcs $b_i'$ in the construction of the curves $\Ds$. It follows that $\hat{\beta}_i$ is isotopic to $\delta_i'$. Clearly we can pick $\hat{\beta}_i$ so that 
 \[|\hat{\beta}_i\cap \delta_j'|=\begin{cases} 2 & \text{if } i=j,\\
 0& \text{otherwise},\end{cases}\] and $|\hat{\beta}_i\cap \delta_j|=0$ for all $i$ and $j$. According to Lemma~\ref{lem:welldefinedframing}, the link component $\ell_i\in \bL$ corresponding to the curve $\beta_i$ is isotopic to the knot obtained by pushing $\delta_i$ off of $\Sigma \, \#_{\ws} \bar{\Sigma}$. The framing from Lemma~\ref{lem:welldefinedframing} is the one which is tangent to $\Sigma\, \#_{\ws} \bar{\Sigma}$. Hence $\delta_i$ is a longitude of $\ell_i$.  Furthermore,  $\bar{\beta}_i$ is a meridian of $\ell_i$.
 
  The curves $\delta_i$ and $\bar{\beta}_i$ are dual, in the sense that 
  \[|\bar{\beta}_i\cap \delta_j|=\begin{cases} 1 & \text{if } i=j,\\
   0& \text{otherwise}.\end{cases}\] Since $\bar{\beta}_i\cup \delta_j$ is also disjoint from the $\hat{\beta}_j$ and $\delta_j'$ curves, it follows that a regular neighborhood of $\bar{\beta}_i\cup \delta_i$ is a diffeomorphic to a once punctured torus $F_i$, which does not intersect any of the other $\bar{\beta}_j$ or $\delta_j$ curves, or any of the $\hat{\beta}_j$ or $\delta'_j$ curves. We note that 
  \[(\Sigma\, \#_{\ws} \bar{\Sigma}) \setminus \bigcup_{i=1}^n(F_i\cup \hat{\beta}_i)\] is homeomorphic to a collection of punctured disks, each containing exactly one $\ws$ basepoint. There are thus disjoint, embedded arcs on $(\Sigma\, \#_{\ws} \bar{\Sigma})$ (avoiding $\hat{\bs},$ $\bar{\bs}$, $\ds$ and $\ds'$, except at their endpoints) connecting each $\delta_i$ to one of the basepoints. The union of these arcs (pushed slightly off of $\Sigma\, \#_{\ws} \bar{\Sigma}$) is a bouquet for $\bL$, and the triple $(\Sigma\, \#_{\ws} \bar{\Sigma}, \as\cup \bar{\bs}, \hat{\bs}\cup \bar{\bs}, \Ds)$ is, by definition, subordinate to this bouquet for $\bL$.
 \end{proof}

We can now prove Proposition~\ref{prop:changeofdiagramsmapcomp}, by showing that 
\[
\Psi_{\cH\to D_{\a}(\cH)}\simeq F_{\a\cup \bar{\b}, \b\cup \bar{\b}, \Dt}( F_1^{\bar{\b},\bar{\b}}, \Theta_{\b\cup \bar{\b},\Dt}^+).
\]

\begin{proof}[Proof of Proposition~\ref{prop:changeofdiagramsmapcomp}] The composition appearing in the statement is unchanged (up to chain homotopy) by isotopies or handleslides of the $\Ds$ curves amongst each other. As any two sets of attaching curves for the fixed handlebody $U_\Sigma$ are related by a sequence of handleslides and isotopies, it is sufficient to show the claim for the $\Ds$ curves considered in Lemma~\ref{lem:randomtrianglemapis2-handlemap}. The main proposition statement is now a consequence of Lemmas~\ref{lem:generalized1-handlemapiscompof1-handles} and \ref{lem:randomtrianglemapis2-handlemap}, as well as invariance of the cobordism maps under 4-dimensional handle cancellations.
\end{proof}

Analogously, the proof of Proposition~\ref{prop:changeofdiagramsmapcomp} adapts to compute the transition map in the opposite direction:

\begin{prop}\label{prop:changeofdiagramsmapcompundouble}
If $\cH=(\Sigma,\as,\bs,\ws)$ is a multi-pointed Heegaard diagram, then the transition map $\Psi_{ D_{\a}(\cH)\to \cH}$ satisfies
\[
\Psi_{ D_{\a}(\cH)\to \cH}\simeq F_3^{\bar{\b},\bar{\b}}(F_{\a\cup \bar{\b}, \Delta, \b\cup \bar{\b}}(-, \Theta^+_{\b\cup \bar{\b}, \Dt})).
\]
\end{prop}

\section{Connected sums and graph cobordisms}
\label{sec:connectedsumsandgraphTQFT}

In this Section, we show that the graph cobordism maps for connected sums have an alternate description (Proposition~\ref{prop:OSmapsaregraphcobmaps}), which we will need when proving the Heegaard triple cobordism formula in Theorem~\ref{thm:triplesandgraphcobordismmaps}.

There is a natural graph cobordism 
\[
(W,\Gamma):(Y_1\sqcup Y_2,\{w_1,w_2\})\to (Y_1\# Y_2, w),
\] 
as follows. We construct $W$ by attaching a 1-handle to $Y_1\sqcup Y_2$ at $w_1$ and $w_2$, i.e., gluing a copy of $D^3\times [-1,1]$ to $(Y_1\sqcup Y_2)\times [0,1]$ along regular neighborhoods of $(w_1,1)$ and $(w_2,1)$ in $(Y_1\sqcup Y_2)\times \{1\}$. We view $D^3$ as the unit ball in $\R^3$ and pick a point $v\in S^2=\d D^3$. We take the basepoint in $Y_1\# Y_2$ to be $w:= (v,0)\in S^2\times [-1,1]$, which sits inside of the connected sum region of $Y_1\# Y_2$. We define the graph $\Gamma$ to be
\[\Gamma:=\big(\{w_1,w_2\}\times [0,1]\big)\cup \big(\{(0,0,0)\}\times [-1,1]\big)\cup \{(t v,0): t\in [0,1]\}.\] We give $\Gamma$ the ribbon structure induced by the cyclic ordering of the boundary manifolds $Y_1,$ $Y_2,$ $Y_1\# Y_2$ (read left to right). This is shown in Figure~\ref{fig::38}.

\begin{figure}[ht!]
	\centering
	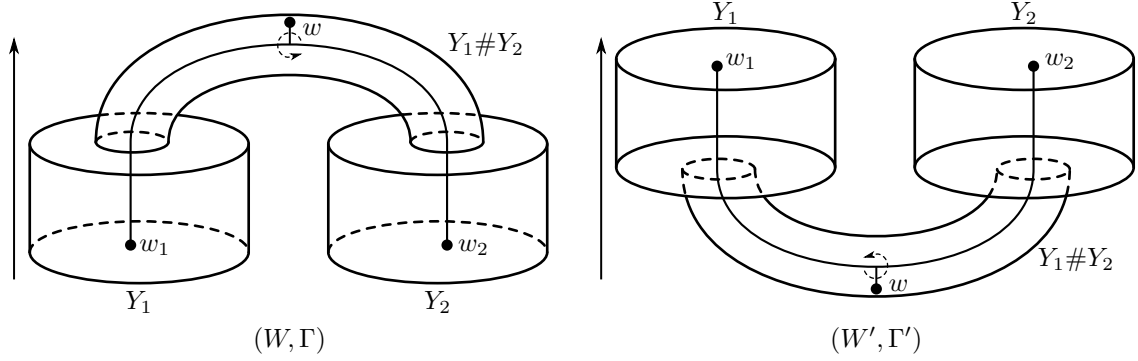
	\caption{\textbf{The graph cobordisms $(W,\Gamma)$ and $(W',\Gamma')$ for connected sums used to define the maps $E_1^A$ and $ E_1^B$ (left) as well as $G_1^A$ and $G_1^B$ (right).} In the case that $Y_1$ and $Y_2$ have more basepoints, the cobordisms have additional 1-handles or 3-handles, and additional graph components. The cyclic order on the trivalent vertices are shown.\label{fig::38}}
\end{figure}

Define graph cobordism maps $E_1^A$ and $E_2^A$ by 
\begin{equation}E_1^A:=F_{W,\Gamma,\frt}^A\qquad \text{and} \qquad E_2^A:=F_{W,\bar{\Gamma},\frt}^A,\label{eq:EiAmapsdefinition}\end{equation}  where $\bar{\Gamma}$ is the ribbon graph obtained by taking $\Gamma$ and reversing the cyclic orderings. Also $\frt$ is the unique $\Spin^c$ structure on $Y_1\# Y_2$ which extends $\frs_1\sqcup \frs_2$. We define maps $E_1^B$ and $E_2^B$ analogously, by using the type-$B$ graph cobordism maps. Note that
\[E_1^A\simeq E_2^B \qquad \text{and} \qquad E_2^A\simeq E_1^B,\] by Lemma~\ref{lem:reversecyclicordering}.

We can also define a graph cobordism $(W',\Gamma')$ from $(Y_1\# Y_2,w)$ to $(Y_1\sqcup Y_2, \{w_1,w_2\})$. The 4-manifold $W'$ is a 3-handle cobordism and $\Gamma'$ is the graph $\Gamma$, with reversed cyclic ordering. We define maps $G_1^A$ and $G_2^A$ as
\begin{equation}
G_1^A:=F_{W',\Gamma',\frt}^A \qquad \text{and} \qquad G_2^A:=F_{W',\bar{\Gamma}',\frt}^A,\label{eq:GiABdefinition}
\end{equation}
where $\frt$ is the unique 4-dimensional $\Spin^c$ structure extending $\frs_1\sqcup \frs_2$ on $Y_1\sqcup Y_2$. We define maps $G_1^B$ and $G_2^B$ analogously.

In \cite{HMZConnectedSum}*{Proposition~5.4}, the maps $E_1^A$ and $G_1^A$ are shown to be chain homotopy inverses (the same holds for the other 3 natural pairs).

In the original proof of the connected sum formula \cite{OSProperties}*{Theorem~1.5}, Ozsv\'{a}th and Szab\'{o} constructed different maps from  $\CF^-(Y_1)\otimes \CF^-(Y_2)$ to $\CF^-(Y_1\# Y_2)$, which do not obviously have a cobordism interpretation. We describe their maps presently. If $(\Sigma_1,\as_1,\bs_1,w_1)$ and $(\Sigma_2,\as_2,\bs_2,w_2)$ are two singly pointed Heegaard diagrams for $(Y_1,w_1)$ and $(Y_2,w_2)$, then the diagram $(\Sigma_1\# \Sigma_2,\as_1\cup \as_2,\bs_1\cup \bs_2,w)$ is a diagram for $(Y_1\# Y_2,w)$, where the connected sum is taken near $w_1$ and $w_2$, and $w$ is a basepoint in the connected sum region of $\Sigma_1\# \Sigma_2$. Ozsv\'{a}th and Szab\'{o} construct a map 
\[
\cE_1\colon \CF^-(\Sigma_1,\as_1,\bs_1,w_1)\otimes_{\bF_2[U]} \CF^-(\Sigma_2,\as_2,\bs_2,w_2)\to \CF^-(\Sigma_1\# \Sigma_2,\as_1\cup \as_2,\bs_1\cup \bs_2,w)
\] and prove that it is a quasi-isomorphism.

 The map $\cE_1$ is defined as a composition of two generalized 1-handle maps and a triangle map, via the formula
\begin{equation}
\cE_1(\xs,\ys):=F_{\a_1\cup \a_2, \b_1\cup \a_2, \b_1\cup \b_2}(F_1^{\a_2,\a_2}(\xs)\otimes F_1^{\b_1,\b_1}(\ys)).
\label{eq:cE_1def}
\end{equation} 

We note that the formula for $\cE_1$ is not symmetric in $Y_1$ and $Y_2$. By switching the roles of $Y_1$ and $Y_2$, we can define a potentially different map $\cE_2$ by the formula
\begin{equation}
\cE_2(\xs,\ys):=F_{\a_1\cup \a_2, \a_1\cup \b_2, \b_1\cup \b_2}(F_1^{\a_1,\a_1}(\ys)\otimes F_1^{\b_2,\b_2}(\xs)).
\label{eq:cE_2def}
\end{equation}
 We refer to $\cE_1$ and $\cE_2$ as the \textit{intertwining maps}.

 We prove the following in this section:

\begin{prop}\label{prop:OSmapsaregraphcobmaps}The connected sum maps $\cE_i$, $E_i^A$ and $E_i^B$ satisfy the relations
\[
\cE_1\simeq E_1^B\simeq E_2^A\qquad \text{and} \qquad\cE_2\simeq E_2^B\simeq E_1^A.
\]
\end{prop}

\begin{rem}\label{rem:morebasepoints}
The connected sum maps $\cE_i,$  $E_i^A$ and $E_i^B$ can all be defined when $(Y_1,\ws_1)$ and $(Y_2,\ws_2)$ are multi-pointed manifolds, as long a bijection $f:\ws_1\to \ws_2$ is specified. We will need to use this generalization when we prove Theorem~\ref{thm:triplesandgraphcobordismmaps}. 

 If $(Y_1,\ws_1)$ and $(Y_2,\ws_2)$ are two multi-pointed 3-manifolds, with a bijection $f:\ws_1\to \ws_2$,  add a connected sum tube between $Y_1$ and $Y_2$ for each pair of basepoints in $\ws_1$ and $\ws_2$ identified by $f$. In each connected sum tube, we add a basepoint. We write $\ws$ for the new basepoints, and $Y_1\, \#_{\ws} Y_2$ for the  manifold obtained by this connected sum operation. In the case that $Y_1$ and $Y_2$ are connected, we  have  $Y_1\, \#_{\ws} Y_2\iso Y_1\# Y_2\# (S^1\times S^2)^{\#(|\ws|-1)}$.

 One can define intertwining maps $\cE_i$ essentially the same as in equation~\eqref{eq:cE_1def}, using the generalized 1-handle operation from Section~\ref{sec:generalized1--handleand3--handlemaps} followed by a triangle map. A graph cobordism $(W,\Gamma)$ can be defined from $(Y_1\sqcup Y_2,\ws_1\cup \ws_2)$ to $(Y_1\, \#_{\ws} Y_2,\ws)$ by attaching $|\ws|$ 1-handles, and constructing a graph with $|\ws|$ components, each with 3 edges and a single trivalent vertex. There are $2^{|\ws|}$ potential choices of cyclic orderings on this graph. For definiteness, we pick the cyclic order on $\Gamma$ induced by the cyclic ordering of the boundaries of $W$ given by $Y_1,$ $Y_2,$ $Y_1\, \#_{\ws}Y_2$ (read left to right).
\end{rem}

\subsection{Properties of the connected sum graph cobordism maps}
\label{sec:descconnsumgraphcobs}
In this section we describe some  properties of the connected sum graph cobordism maps $E_i^A,$ $E_i^B,$ $G_i^A$ and $G_i^B$, defined in equations~\eqref{eq:EiAmapsdefinition} and~\eqref{eq:GiABdefinition}. 

 It will be useful for our purposes to have a more explicit description of the maps $E_1^A$ and $E_1^B$:

\begin{lem}\label{lem:connsumgraphcobcomp}Suppose $(Y_1,w_1)$ and $(Y_2,w_2)$ are two singly pointed 3-manifolds, and  $(Y_1\# Y_2, w)$ is their connected sum, with the connected sum taken at $w_1\in Y_1$ and $w_2\in Y_2$. The maps $E_1^A$ and $E_1^B$ from equation~\eqref{eq:EiAmapsdefinition} satisfy 
\[E_1^A\simeq \phi_* S_{\psi(w_2)}^- A_\lambda F_1\psi_*, \qquad \text{and} \qquad E_1^B\simeq \phi_*S_{\psi(w_2)}^- B_\lambda F_1\psi_*,\] where 
\begin{itemize}
\item $\psi$ is a diffeomorphism of $Y_1\sqcup Y_2$, which is supported in a neighborhood of $\{w_1,w_2\}$ and moves $w_1$ and $w_2$ along paths in $Y_1$ and $Y_2$ outside of where the 1-handle is attached;
\item $F_1$ is the 1-handle map for attaching a 1-handle with feet centered at $w_1$ and $w_2$;
\item $\lambda$ is a path from $\psi(w_1)$ to $\psi(w_2)$ in $Y_1\# Y_2$, obtained by concatenating the paths used to construct $\psi$ with a path across the connected sum region;
\item $\phi$ is a diffeomorphism of $Y_1\# Y_2$ which is supported in a neighborhood of the path  $\lambda$ and moves $\psi(w_1)$ to the point $w$.
\end{itemize}

If $(Y_1,\ws_1)$ and $(Y_2,\ws_2)$ are two multi-pointed 3-manifolds, with a chosen bijection $f:\ws_1\to \ws_2$, then the connected sum maps $E_1^A$ and $E_1^B$  are a composition of $|\ws|$ maps, each taking the above form.
\end{lem}

\begin{rem} It might appear overly fastidious to keep track of the diffeomorphisms $\psi$ and $\phi$. This is helpful for the proof of Proposition~\ref{prop:OSmapsaregraphcobmaps}, which is our main application of Lemma~\ref{lem:connsumgraphcobcomp}.
\end{rem}

\begin{proof}[Proof of Lemma~\ref{lem:connsumgraphcobcomp}]We manipulate the graph, so that it has the configuration shown in Figure~\ref{fig::37}. Viewing the cobordism as a 4-dimensional movie, the first step moves the basepoints $w_1$ and $w_2$ away from the feet of the 1-handle, using a diffeomorphism $\psi$. Next, one attaches the 1-handle. Next, one inserts a trivalent graph  which deletes $\psi(w_2)$ and does not move $\psi(w_1)$, as considered in Lemma~\ref{lem:computationofYshapedgraphcobordisms}. Finally one moves $\psi(w_1)$ into the connected sum region using the diffeomorphism $\phi$. Using the computation from Lemma~\ref{lem:computationofYshapedgraphcobordisms} for a cylindrical cobordism containing a trivalent graph, the entire graph cobordism map becomes
\[
E_1^A\simeq \phi_* S_{\psi(w_2)}^-A_\lambda F_1\psi_*.
\]

The formula for $E_1^B$ follows similarly. The claim about the formulas for $E_1^A$ and $E_1^B$ when $Y_1$ and $Y_2$ have several basepoints follows by applying the above argument at each pair of basepoints.
\end{proof}

 \begin{figure}[ht!]
	\centering
	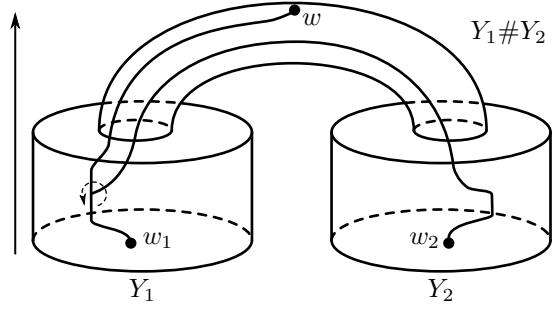
	\caption{\textbf{Computing $E_1^A$ by manipulating the graph inside the connected sum graph cobordism.} One first moves the basepoints $w_1$ and $w_2$ slightly away from the feet of the 1-handle (corresponding to the diffeomorphism $\psi$). Then one attaches a 1-handle at $w_1$ and $w_2$. This is followed by a copy of the $(Y_1\# Y_2)\times [0,1]$, containing a trivalent graph.  Finally one moves the basepoint $\psi(w_1)$ into the connected sum region, using the diffeomorphism $\phi$. \label{fig::37}}
\end{figure}

 Analogously to Lemma~\ref{lem:connsumgraphcobcomp}, we have the following:

\begin{lem}\label{lem:connsumgraphcobcomp2}Suppose that $(Y_1,w_1)$ and $(Y_2,w_2)$ are two singly pointed 3-manifolds, and $G_1^A$ and $G_1^B$ are the cobordism maps for the  3-handle graph cobordism  $(W',\Gamma')$ with trivalent graph shown in Figure~\ref{fig::38}. Then
\[G_1^A\simeq  \phi_* F_3 A_\lambda S_{w_2'}^+\psi_* \qquad \text{and} \qquad G_1^B\simeq  \phi_* F_3 B_\lambda S_{w_2'}^+\psi_*,\] where 
\begin{itemize}
\item $\psi$ is a diffeomorphism of $Y_1\# Y_2$, supported in a neighborhood of the connected sum region, which  pushes $w$ slightly into $Y_1$;
\item $w_2'$ is a new basepoint in the $Y_2$ side of $Y_1\#Y_2$, near the connected sum region;
\item  $\lambda$ is a path from $\psi(w)$ to $w_2'$ crossing over the connected sum region;
\item $\phi$ is diffeomorphism of $Y_1\sqcup Y_2$ which is supported in a neighborhood of $w_1$ and $w_2$ and moves $\psi(w)$ to $w_1$ and moves $w_2'$ to $w_2$.
\end{itemize}

If $(Y_1,\ws_1)$ and $(Y_2,\ws_2)$ are two multi-pointed 3-manifolds, with a chosen bijection $f:\ws_1\to \ws_2$, then the maps $G_1^A$ and $G_1^B$ are a composition of $|\ws_i|$ maps, each given by one of the above formulas.
\end{lem}
\begin{proof}The proof is identical to the proof of Lemma~\ref{lem:connsumgraphcobcomp}.
\end{proof}

An important property of the maps $E_i^A,$ $E_i^B,$ $ G_i^A$ and $G_i^B$ is that they are chain homotopy inverses of each other:

\begin{prop}\label{prop:connectedsummapsarehomotopyinverses}The maps $E_i^A,$ and $G_i^A$ satisfy the relations
\[E_i^A\circ G_i^A\simeq \id \qquad \text{and} \qquad G_i^A\circ E_i^A\simeq \id,\] for $i=1,2$. The same relations hold for the type-$B$ maps $E_i^B$ and $G_i^B$.
\end{prop}

A proof of Proposition~\ref{prop:connectedsummapsarehomotopyinverses} can be found in \cite{HMZConnectedSum}*{Proposition~5.2}, in the context of 3-manifolds with a single basepoint. The proof extends without change to the case when $Y_1$ and $Y_2$ each have extra basepoints which are not involved in the connected sum (i.e. when we take two multi-pointed 3-manifolds, $(Y_1,\ws_1)$ and $(Y_2,\ws_2)$, and take their connected sum at just one pair of basepoints $w_1\in \ws_1$ and $w_2\in \ws_2$). The full version of Proposition~\ref{prop:connectedsummapsarehomotopyinverses}, where the cobordisms for $E_i^A$ and $G_i^A$ involve $|\ws_i|$ 1-handles or 3-handles, with each handle containing a trivalent vertex, can then be proven by applying the composition law.

\subsection{Proof of Proposition~\ref{prop:OSmapsaregraphcobmaps}:  \texorpdfstring
{$\cE_i$}{curlyE\_i} and \texorpdfstring{$E_i^B$}{E\_i B} are equal}

\begin{proof}[Proof of Proposition~\ref{prop:OSmapsaregraphcobmaps}] We will show that $\cE_1\simeq E_1^B$. Once we establish this, the relation $\cE_2\simeq E_2^B$ will also follow, since both $\cE_2$ and $E_2^B$ are simply the maps obtained by switching the roles of $Y_1$ and $Y_2$ in the definitions of $\cE_1$ and $E_1^B$. We will also only consider the case that $Y_1$ and $Y_2$ are singly pointed. The case that they are multi-pointed follows from a straightforward modification of the argument we present.

As $E_1^B$ and $G_1^B$ are chain homotopy inverses of each other by Proposition~\ref{prop:connectedsummapsarehomotopyinverses}, it is sufficient to show that
\[G_1^B\circ \cE_1\simeq \id.\] By Lemma~\ref{lem:connsumgraphcobcomp2}, this amounts to showing that
\begin{equation}\phi_* F_3 B_\lambda S_{w_2'}^+\psi_* \cE_1\simeq \id_{\CF^-(\Sigma_1,\a_1,\b_1)\otimes \CF^-(\Sigma_2,\a_2,\b_2)},\label{eq:OS=GraphTQFT1}\end{equation}  where $\psi$ is a diffeomorphism pushing $w$ into $Y_1$ slightly, $w_2'$ is a new basepoint in the $Y_2$ side of $Y_1\# Y_2$, $\lambda$ is a path from $\psi(w)$ to $w_2'$ crossing the connected sum region, and $\phi$ is a diffeomorphism of $Y_1\sqcup Y_2$ which moves $\psi(w)$ to $w_1$ and $w_2'$ to $w_2$.

The remainder of the proof establishes equation~\eqref{eq:OS=GraphTQFT1}. We warn the reader that the proof is quite involved, though the strategy is relatively straightforward to summarize.  Using properties of the graph TQFT and the holomorphic triangle counts used to show the well-definedness of the generalized 1-handle and 3-handle maps, we will manipulate the expression in equation~\eqref{eq:OS=GraphTQFT1} until it becomes a holomorphic triangle count on $\Sigma_1\sqcup \Sigma_2$ which counts the same holomorphic triangles as appear in the transition map associated to a small isotopy of the attaching curves on each diagram.

Pick diagrams $(\Sigma_1,\as_1,\bs_1,w_1)$ and $(\Sigma_2,\as_2,\bs_2,w_2)$ for $(Y_1,w_1)$ and $(Y_2,w_2)$, respectively. We can form a diagram for $(Y_1\# Y_2,w)$ by taking the connected sum of the two diagrams at $w_1$ and $w_2$, and placing a single basepoint $w$ in the connected sum region. We also need to consider doubly pointed diagrams for $(Y_1\#Y_2,\{\psi(w),w_2'\})$, where $\psi$ and $w_2'$ are as above. To get such a diagram, we need to add an additional pair of attaching curves to $(\Sigma_1\# \Sigma_2,\as_1\cup \as_2,\bs_1\cup \bs_2)$. There are two convenient choices, which are shown in Figure~\ref{fig::36} and are labeled by $\zeta$ and $\tau$.  The curves marked with $\tau$ are chosen so that they can compute the free-stabilization maps at $w_2'$, and the curves marked with $\zeta$ can be used to compute the 3-handle map (we are abusing notation and writing $\zeta$ or $\tau$ for both a curve and a small Hamiltonian translate). We will write $F_3^{\zeta,\zeta}$ for $F_3$, for clarity.

\begin{figure}[ht!]
	\centering
	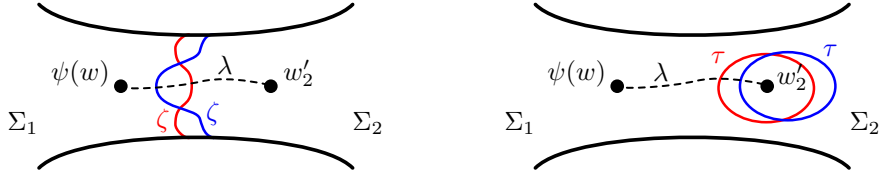
	\caption{\textbf{The pairs of attaching curves labeled $\tau$ and $\zeta$ in the connected sum region of $\Sigma_1\# \Sigma_2$.}  The $\zeta$ curves can be used to compute the 3-handle map $F_3=F_3^{\zeta,\zeta}$, while the $\tau$ curves can be used to compute the free-stabilization map $S_{w_2'}^+$. The path $\lambda$ is also shown (dashed). \label{fig::36}}
\end{figure}

 Rewriting equation~\eqref{eq:OS=GraphTQFT1} using the definition of $\cE_1$, we wish to show that
\begin{equation}
\begin{split}\phi_*F_3^{\zeta,\zeta}  B_\lambda S^+_{w_2'}\psi_* F_{\a_1\cup\a_2,\b_1\cup \a_2,\b_1\cup\b_2} (F_1^{\a_2,\a_2}(-),  F_1^{\b_1,\b_1}(-))\\
 \simeq \id_{\CF^-(\Sigma_1,\a_1,\b_1)}(-)\otimes \id_{\CF^-(\Sigma_2,\a_2,\b_2)}(-).\end{split}\label{eq:OS=GraphTQFT2'}
 \end{equation}

 Noting that $\psi$ can be chosen to fix the Heegaard surface $\Sigma_1\# \Sigma_2\subset Y_1\# Y_2$ setwise, and also fix the curves $\as_i$ and $\bs_i$, we can bring $\psi_*$ inside the triangle map so that the composition on the left side of equation~\eqref{eq:OS=GraphTQFT2'} becomes
\begin{equation}\phi_*F_3^{\zeta,\zeta}  B_\lambda S^+_{w_2'} F_{\a_1\cup\a_2,\b_1\cup \a_2,\b_1\cup\b_2} (\psi_*F_1^{\a_2,\a_2}(-),  \psi_*F_1^{\b_1,\b_1}(-))\label{eq:OS=GraphTQFT2}.
\end{equation}

We define the following sets of attaching curves on $\Sigma_1\# \Sigma_2$:
\begin{align*}\mathcal{L}_{\tau}&:=\as_1\cup\{\tau\}\cup  \as_2,& \mathcal{\cL}_{\zeta}&:= \as_1\cup \{\zeta\}\cup \as_2,\\
\mathcal{M}_{\tau}&:= \bs_1\cup \{\tau\} \cup \as_2,& \mathcal{M}_{\zeta}&:=\bs_1\cup \{\zeta\}\cup \as_2,\\
\mathcal{R}_{\tau}&:=\bs_1\cup \{\tau\} \cup \bs_2,& \mathcal{R}_{\zeta}&:=\bs_1\cup \{\zeta\}\cup \bs_2.
\end{align*}

Using the triangle counts used to show well-definedness of the free-stabilization maps \cite{ZemGraphTQFT}*{Theorem~6.7} (these can be viewed as a special case of Proposition~\ref{prop:generalized1-handlesandtriangles} of this paper; see Remark~\ref{rem:extrabasepoints}), we can pull the expression $S_{w_2'}^+$ inside the triangle map for an appropriately chosen almost complex structure, and conclude that equation~\eqref{eq:OS=GraphTQFT2} is chain homotopic to
\begin{equation}\phi_*F_3^{\zeta,\zeta} B_{\lambda} F_{\cL_{\tau},\cM_{\tau}, \cR_{\tau}}(S_{w_2'}^+ \psi_* F_1^{\a_2,\a_2}(-),  S_{w_2'}^+ \psi_* F_1^{\b_1,\b_1}(-)).\label{eq:OS=GraphTQFT3}\end{equation}

The expression in equation~\eqref{eq:OS=GraphTQFT3} is not quite sufficient to actually compute the composition since it still implicitly involves a change of diagrams map to handleslide the two $\tau$ curves into the position of the two $\zeta$ curves. Hence, we insert the transition map $\Psi_{\cL_\tau\to \cL_\zeta}^{\cR_\tau\to \cR_\zeta}$ immediately to the left of the triangle map in equation~\eqref{eq:OS=GraphTQFT3} to rewrite equation~\eqref{eq:OS=GraphTQFT3} as
\[
\phi_*F_3^{\zeta,\zeta} B_{\lambda} \Psi_{\cL_\tau\to \cL_\zeta}^{\cR_\tau\to \cR_\zeta} F_{\cL_{\tau},\cM_{\tau}, \cR_{\tau}}(S_{w_2'}^+\psi_* F_1^{\a_2,\a_2}(-),  S_{w_2'}^+ \psi_*F_1^{\b_1,\b_1}(-)).
\] By the construction of the transition maps, we have
\[
\Psi_{\cL_\tau\to \cL_\zeta}^{\cR_\tau\to \cR_\zeta}=\Psi_{\cL_\tau\to \cL_\zeta}^{\cR_{\zeta}}\circ \Psi_{\cL_\tau}^{\cR_\tau\to \cR_\zeta}.
\]
 Individually, each of $\Psi_{\cL_\tau\to \cL_{\zeta}}^{\cR_{\zeta}}$ and $\Psi^{\cR_\tau\to \cR_\zeta}_{\cL_\tau}$ can be computed by a holomorphic triangle map. For example, the map $\Psi_{\cL_\tau}^{\cR_{\tau}\to \cR_{\zeta}}$ satisfies
\[
\Psi_{\cL_\tau}^{\cR_{\tau}\to \cR_{\zeta}}(-)\simeq F_{\cL_\tau, \cR_{\tau}, \cR_{\zeta}}(-, \Theta_{\cR_\tau, \cR_{\zeta}}^+),
\] 
where $\Theta_{\cR_\tau, \cR_{\zeta}}^+\in \CF^-(\Sigma_1\# \Sigma_2, \cR_\tau, \cR_{\zeta})$ is a cycle which represents the top degree element of homology. The map $\Psi_{\cL_\tau\to \cL_{\zeta}}^{\cR_{\zeta}}$ takes a similar form. A straightforward argument using associativity of the triangle maps (twice) shows that
\begin{equation}
\begin{split}\,&\phi_*F_3^{\zeta,\zeta} B_{\lambda}\Psi_{\cL_\tau\to \cL_\zeta}^{\cR_{\zeta}} \Psi_{\cL_\tau}^{\cR_\tau\to \cR_\zeta}F_{\cL_{\tau},\cM_{\tau}, \cR_{\tau}}(S_{w_2'}^+ \psi_*F_1^{\a_2,\a_2}(-),  S_{w_2'}^+ \psi_* F_1^{\b_1,\b_1}(-)) \\
 \simeq  &\phi_*F_3^{\zeta,\zeta} B_{\lambda}F_{\cL_{\zeta},\cM_{\tau}, \cR_{\zeta}}(\Psi_{\cL_\tau\to \cL_{\zeta}}^{\cM_\tau} S_{w_2'}^+ \psi_* F_1^{\a_2,\a_2}(-),\Psi^{\cR_\tau\to \cR_\zeta}_{\cM_\tau}  S_{w_2'}^+\psi_*  F_1^{\b_1,\b_1}(-)).\end{split}
 \label{eq:OS=GraphTQFT4}
 \end{equation}

We now wish to change the $\cM_\tau$ to $\cM_{\zeta}$, in the above triangle map. Naturality of Heegaard Floer homology implies that
\begin{equation}\Psi_{\cL_\zeta}^{\cM_{\zeta}\to \cM_{\tau}}\circ \Psi_{\cL_\zeta}^{\cM_{\tau}\to \cM_{\zeta}}\simeq \id_{\CF^-(\Sigma_1\# \Sigma_2, \cL_{\zeta}, \cM_{\tau})}.\label{eq:changethenchangeback=id}\end{equation}

Using equation~\eqref{eq:changethenchangeback=id}, the fact that $\Psi_{\cL_\zeta}^{\cM_{\zeta}\to \cM_{\tau}}$ can be realized as the triangle map $F_{\cL_{\zeta},\cM_{\zeta},\cM_{\tau}}(-, \Theta_{\cM_{\zeta},\cM_\tau}^+)$, as well  as associativity of the triangle maps, we perform the following manipulation to equation~\eqref{eq:OS=GraphTQFT4}:
\begin{align*}\, &\phi_* F_3^{\zeta,\zeta} B_{\lambda} F_{\cL_{\zeta},\cM_{\tau}, \cR_{\zeta}}(\Psi_{\cL_\tau\to \cL_{\zeta}}^{\cM_\tau} S_{w_2'}^+ \psi_*F_1^{\a_2,\a_2},\Psi^{\cR_\tau\to \cR_\zeta}_{\cM_\tau}  S_{w_2'}^+\psi_* F_1^{\b_1,\b_1})\\
\simeq  &\phi_* F_3^{\zeta,\zeta} B_{\lambda}  F_{\cL_{\zeta},\cM_{\tau}, \cR_{\zeta}}(\Psi_{\cL_\zeta}^{\cM_{\zeta}\to \cM_{\tau}}\Psi_{\cL_\zeta}^{\cM_{\tau}\to \cM_{\zeta}}\Psi_{\cL_\tau\to \cL_{\zeta}}^{\cM_\tau} S_{w_2'}^+\psi_* F_1^{\a_2,\a_2},\Psi^{\cR_\tau\to \cR_\zeta}_{\cM_\tau}  S_{w_2'}^+ \psi_*F_1^{\b_1,\b_1})\\
\simeq &\phi_* F_3^{\zeta,\zeta} B_{\lambda} F_{\cL_{\zeta},\cM_{\zeta}, \cR_{\zeta}}(\Psi_{\cL_\zeta}^{\cM_{\tau}\to \cM_{\zeta}}\Psi_{\cL_\tau\to \cL_{\zeta}}^{\cM_\tau} S_{w_2'}^+\psi_* F_1^{\a_2,\a_2},\Psi^{\cR_{\zeta}}_{\cM_{\tau}\to \cM_{\zeta}}\Psi^{\cR_{\tau}\to \cR_{\zeta}}_{\cM_{\tau}}  S_{w_2'}^+\psi_* F_1^{\b_1,\b_1}).
\end{align*}
We condense the above expression slightly by combining some of the transition maps to arrive at the expression

\begin{equation}
\phi_* F_3^{\zeta,\zeta} B_{\lambda} F_{\cL_{\zeta},\cM_{\zeta}, \cR_{\zeta}}(\Psi_{\cL_\tau\to \cL_{\zeta}}^{\cM_{\tau}\to \cM_{\zeta}} S_{w_2'}^+ \psi_* F_1^{\a_2,\a_2}(-),\Psi_{\cM_{\tau}\to \cM_{\zeta}}^{\cR_{\tau}\to \cR_{\zeta}} S_{w_2'}^+\psi_* F_1^{\b_1,\b_1}(-))
.
\label{eq:OS=GraphTQFT7} 
\end{equation} 
 Lemma~\ref{lem:graphactionandtriangles} allows us to bring $B_{\lambda}$ inside the triangle map to see that equation~\eqref{eq:OS=GraphTQFT7} is chain homotopic to
\begin{equation}
\phi_* F_3^{\zeta,\zeta}  F_{\cL_{\zeta},\cM_{\zeta}, \cR_{\zeta}}(\Psi_{\cL_\tau\to \cL_{\zeta}}^{\cM_{\tau}\to \cM_{\zeta}} S_{w_2'}^+ \psi_*F_1^{\a_2,\a_2}(-),B_{\lambda}\Psi_{\cM_{\tau}\to \cM_{\zeta}}^{\cR_{\tau}\to \cR_{\zeta}} S_{w_2'}^+\psi_* F_1^{\b_1,\b_1}(-)).
\label{eq:OS=GraphTQFT5}
\end{equation}

We now define a map
\begin{equation}\Top^+_{(\Sigma_2,\a_2,\a_2)}:\CF^-(\Sigma_1,\as_1,\bs_1,w_1)\to \CF^-(\Sigma_1, \as_1, \bs_1, w_1)\otimes \CF^-(\Sigma_2,\as_2,\as_2,w_2')\label{eq:OS=GraphTQFT9}
\end{equation} 
by the formula
\[\Top^+_{(\Sigma_2,\a_2,\a_2)}(\ve{x})=\ve{x}\otimes \Theta_{\a_2,\a_2}^+,\]
extended $\bF_2[U]$-equivariantly,
 where $\Theta^+_{\a_2,\a_2}\in \CF^-(\Sigma_2,\as_2,\as_2,w_2')$ is the top degree generator. 
 We claim that 
\begin{equation}\Psi_{\cL_\tau\to \cL_{\zeta}}^{\cM_{\tau}\to \cM_{\zeta}} S_{w_2'}^+ \psi_* F_1^{\a_2,\a_2}(-)\simeq F_1^{\zeta,\zeta} \psi'_* \Top_{(\Sigma_2,\a_2,\a_2)}^+(-) ,\label{eq:OS=GraphTQFT6}
\end{equation} 
where $\psi'$ is the diffeomorphism of $\Sigma_1$ obtained by pushing $w_1$ to $\psi(w)$, and $F_1^{\zeta,\zeta}$ is the 1-handle map for attaching a 1-handle with feet at $w_1$ and $w_2$, using attaching curves in the 1-handle region equal to Hamiltonian translates of $\zeta$. Equation~\eqref{eq:OS=GraphTQFT6}  follows from the fact that the generalized 1-handle map is well-defined (i.e. the holomorphic triangle counts of Proposition~\ref{prop:generalized1-handlesandtriangles}, as well as the change of almost complex structure computation of Lemma~\ref{lem:generalized1-handlemapiscompof1-handles}). This is demonstrated schematically in Figure~\ref{fig::53}. We remark that technically we are using the well-definedness of a version of the generalized 1-handle map where we allow additional basepoints on our diagram for $(S^1\times S^2)^{\# g(\Sigma_0)}$, though the holomorphic disk and triangle counts from Section~\ref{sec:generalized1--handleand3--handlemaps} can be adapted  to this situation with only minor notational changes (see Remark~\ref{rem:extrabasepoints}).

\begin{figure}[ht!]
	\centering
	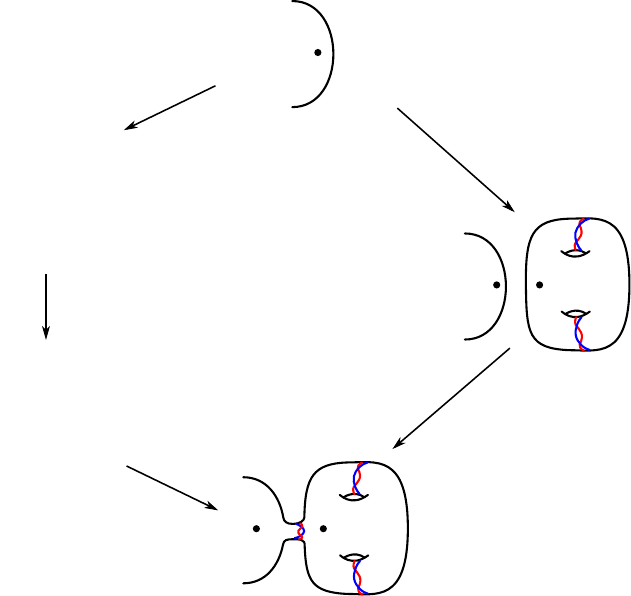
	\caption{\textbf{A schematic of the relation $\Psi_{\cL_\tau\to \cL_{\zeta}}^{\cM_{\tau}\to \cM_{\zeta}} S_{w_2'}^+ \psi_* F_1^{\a_2,\a_2}(-)\simeq F_1^{\zeta,\zeta} \psi'_* \Top_{(\Sigma_2,\a_2,\a_2)}^+(-)$.} The relation follows from the well-definedness of the generalized 1-handle map, since both compositions can be interpreted as a version of the generalized 1-handle map, with the connected sum operation being taken near $\psi(w)$.\label{fig::53}}
\end{figure}

Using equation~\eqref{eq:OS=GraphTQFT6}, equation~\eqref{eq:OS=GraphTQFT5} now becomes 
\begin{equation}\phi_*F_3^{\zeta,\zeta}  F_{\cL_{\zeta},\cM_{\zeta}, \cR_{\zeta}}(F_1^{\zeta,\zeta} \psi'_*  \Top_{(\Sigma_2,\a_2,\a_2)}^+(-),B_{\lambda}\Psi_{\cM_{\tau}\to \cM_{\zeta}}^{\cR_{\tau}\to \cR_{\zeta}} S_{w_2'}^+ \psi_* F_1^{\b_1,\b_1}(-)).\label{eq:OS=graphTQFT20}
\end{equation} 
Using the relation between the triangle maps and the generalized 3-handle maps from Proposition~\ref{prop:generalized1-handlesandtriangles} to move $F_3^{\zeta,\zeta}$ inside the triangle map and conclude that equation~\eqref{eq:OS=graphTQFT20} is equal to
\begin{equation}\phi_*F_{\a_1\cup\a_2,\b_1\cup \a_2,\b_1\cup\b_2}( \psi'_*  \Top_{(\Sigma_2,\a_2,\a_2)}^+(-), F_3^{\zeta,\zeta}B_{\lambda}\Psi_{\cM_{\tau}\to \cM_{\zeta}}^{\cR_{\tau}\to \cR_{\zeta}} S_{w_2'}^+\psi_* F_1^{\b_1,\b_1}(-)).\label{eq:OS=graphTQFT11}
\end{equation} We remark that the underlying Heegaard surface of the triangle count in equation~\eqref{eq:OS=graphTQFT11} is the disjoint union $\Sigma_1\sqcup \Sigma_2$, since we surger out $\zeta$ when moving the 3-handle map inside the triangle map.

We now wish to rearrange the terms appearing in the right component of the triangle map. Recall that $\psi$ is the diffeomorphism of $Y_1\# Y_2$ which moves $w$ into the $Y_1$ side of $Y_1\# Y_2$. The diffeomorphism $\psi$ and the curves $\zeta$ and $\tau$ can be chosen so that $\psi$ fixes $w_2',$ $\tau$ and $\zeta$, implying
\begin{equation}
\Psi_{\cM_{\tau}\to \cM_{\zeta}}^{\cR_{\tau}\to \cR_{\zeta}}S_{w_2'}^+\psi_*=\psi_*\Psi_{\cM_{\tau}\to \cM_{\zeta}}^{\cR_{\tau}\to \cR_{\zeta}}S_{w_2'}^+. \label{eq:OS=graphTQFT19}
\end{equation}
 Define
$\lambda'':=\psi^{-1}(\lambda),$ so that tautologically
\begin{equation}\psi_* B_{\lambda''}\simeq B_{\lambda}\psi_*.\label{eq:OS=graphTQFT17}\end{equation}

Using equations~\eqref{eq:OS=graphTQFT19} and~\eqref{eq:OS=graphTQFT17}, we see that
\begin{equation}
F_3^{\zeta,\zeta}B_{\lambda}\Psi_{\cM_{\tau}\to \cM_{\zeta}}^{\cR_{\tau}\to \cR_{\zeta}} S_{w_2'}^+\psi_* F_1^{\b_1,\b_1}(-)\simeq F_3^{\zeta,\zeta} \psi_* B_{\lambda''} \Psi_{\cM_{\tau}\to \cM_{\zeta}}^{\cR_{\tau}\to \cR_{\zeta}} S_{w_2'}^+ F_1^{\b_1,\b_1}(-).\label{eq:OS=GraphTQFT8}
\end{equation} We note that $F_1^{\b_1,\b_1}$ commutes with $S_{w_2'}^+$ since 1-handle and free-stabilization maps commute with each other \cite{ZemGraphTQFT}*{Lemma~8.13}.  Furthermore
\begin{equation}\Psi_{\cM_{\tau}\to \cM_{\zeta}}^{\cR_{\tau}\to \cR_{\zeta}} F_1^{\b_1,\b_1}(-)\simeq F_1^{\b_1,\b_1} \Psi_{\a_2\cup \{\tau\}\to \a_2\cup \{\zeta\}}^{\b_2\cup \{\tau\}\to \b_2\cup \{\zeta\}}(-)\label{eq:OS=graphTQFT16}\end{equation}  by the holomorphic triangle computation of Proposition~\ref{prop:generalized1-handlesandtriangles}.  

Since the relative homology map $B_{\lambda''}$ commutes with 1-handle maps by \cite{ZemGraphTQFT}*{Lemma~8.11}, we note that
\begin{equation}B_{\lambda''}F_1^{\b_1,\b_1}\simeq F_1^{\b_1,\b_1} B_{\lambda_0},\label{eq:OS=graphTQFT15}
\end{equation}
for a path $\lambda_0$ in $Y_2$ from $w_2$ to $w_2'$, contained in a neighborhood of $w_2\in Y_2$.

Using equations~\eqref{eq:OS=graphTQFT16} and~\eqref{eq:OS=graphTQFT15}, we conclude that equation~\eqref{eq:OS=GraphTQFT8} is chain homotopic to
\begin{equation} F_3^{\zeta,\zeta} \psi_*F_1^{\b_1,\b_1}B_{\lambda_0} \Psi_{\a_2\cup \{\tau\}\to \a_2\cup \{\zeta\}}^{\b_2\cup \{\tau\}\to \b_2\cup \{\zeta\}} S_{w_2'}^+(-).\label{eq:OS=GraphTQFT10}
\end{equation} 
We now claim that 
\begin{equation}F_3^{\zeta,\zeta}\psi_*F_1^{\b_1,\b_1}(-)\simeq \Top_{(\Sigma_1,\b_1,\b_1)}^+S^-_{w_2}(-),
\label{eq:OS=graphTQFT12}
\end{equation} 
where
 \[
 \Top_{(\Sigma_1,\b_1,\b_1)}^+:\CF^-(\Sigma_2,\a_2,\b_2,w_2')\to \CF^-(\Sigma_1,\bs_1,\bs_1,\psi(w_1))\otimes \CF^-(\Sigma_2,\as_2,\bs_2,w_2')
 \] 
 is defined by the formula $\Top^+_{(\Sigma_1,\b_1,\b_1)}(\xs)=\Theta_{\b_1,\b_1}^+\otimes \xs$, 
 analogously to the map in equation~\eqref{eq:OS=GraphTQFT9}. Equation~\eqref{eq:OS=graphTQFT12} essentially follows immediately from the formulas of the maps involved, however one should also use an argument similar to Lemma~\ref{lem:generalized1-handlemapiscompof1-handles} to ensure that a single almost complex structure can be chosen which allows both sides of the equivalence to be computed simultaneously.

\begin{figure}[ht!]
	\centering
	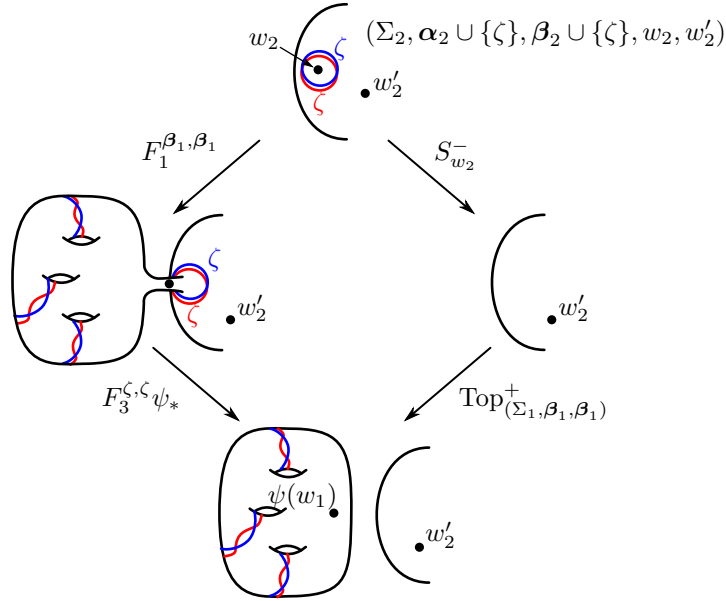
	\caption{\textbf{A schematic of the relation $F_3^{\zeta,\zeta}\psi_*F_1^{\b_1,\b_1}(-)\simeq \Top_{(\Sigma_1,\b_1,\b_1)}^+S^-_{w_2}(-)$.} The relation follows from the formulas for the maps in the composition.\label{fig::54}}
\end{figure}

 Using equation~\eqref{eq:OS=graphTQFT12}, we see that equation~\eqref{eq:OS=GraphTQFT10} is chain homotopic to
\begin{equation}
\Top_{(\Sigma_1,\b_1,\b_1)}^+ S_{w_2}^- B_{\lambda_0} \Psi_{\a_2\cup \{\tau\}\to \a_2\cup \{\zeta\}}^{\b_2\cup \{\tau\}\to \b_2\cup \{\zeta\}} S_{w_2'}^+(-).
\label{eq:OS=graphTQFT14}
\end{equation} Since the maps appearing in equation~\eqref{eq:OS=graphTQFT14} are all natural, we will omit writing the transition map. This reduces equation~\eqref{eq:OS=graphTQFT14} to
\begin{equation}
\Top_{(\Sigma_1,\b_1,\b_1)}^+ S_{w_2}^- B_{\lambda_0} S_{w_2'}^+(-).
\label{eq:OS=graphTQFT13}
\end{equation} 
By Relation~\ref{rel:R9} (the basepoint moving relation) the  expression in equation~\eqref{eq:OS=graphTQFT13} is chain homotopic to
\begin{equation}
\Top^+_{(\Sigma_1,\b_1,\b_1)}\phi^{\lambda_0}_*(-),
\label{eq:OS=graphTQFT18}
\end{equation} 
where $\phi^{\lambda_0}$ is the diffeomorphism of $Y_1\sqcup Y_2$ obtained by moving $w_2$ to $w_2'$ along the path $\lambda_0$.  Inserting equation~\eqref{eq:OS=graphTQFT18} into equation~\eqref{eq:OS=graphTQFT11}, we see that 
\[
(G_1^B\circ \cE_1)(-,-)\simeq \phi_*F_{\a_1\cup\a_2,\b_1\cup \a_2,\b_1\cup\b_2}( \psi'_*  \Top_{(\Sigma_2,\a_2,\a_2)}^+(-), \Top^+_{(\Sigma_1,\b_1,\b_1)}\phi^{\lambda_0}_*(-)).
\]
  Bringing $\phi_*$ inside the triangle map cancels the diffeomorphism maps already inside the triangle map, and we are simply left with
\[
(G_1^B\circ \cE_1)(-,-)\simeq F_{\a_1\cup\a_2,\b_1\cup \a_2,\b_1\cup\b_2}(  \Top_{(\Sigma_2,\a_2,\a_2)}^+(-), \Top^+_{(\Sigma_1,\b_1,\b_1)}(-)).
\]
 The above triangle map counts triangles on $\Sigma_1\sqcup \Sigma_2$, and is just the tensor product of two transition maps,
\[
\Psi_{\a_1}^{\b_1\to \b_1}(-)\otimes \Psi_{\a_2\to \a_2}^{\b_2}(-),
\]
 completing the proof.
 \end{proof}

\section{Heegaard triples and graph cobordisms}
\label{sec:Heegaardtriplesandgraphcobordisms}

Given a Heegaard triple $(\Sigma,\ve{\alpha},\ve{\beta},\ve{\gamma},\ve{w})$, Ozsv\'{a}th and Szab\'{o} construct a smooth 4-manifold $X_{\a,\b,\g}$ \cite{OSDisks}*{Section~8}. The manifold $X_{\a,\b,\g}$ has boundary
\[
\d X_{\a,\b,\g}=-Y_{\a,\b}\sqcup -Y_{\b,\g}\sqcup Y_{\a,\g}.
\]
 The construction of $X_{\a,\b,\g}$ is described in equation~\eqref{eq:Xabgdef} of this paper.

There is a graph $\Gamma_{\a,\b,\g}\subset X_{\a,\b,\g}$, defined as follows. Let $v_0\in \Delta$ be a chosen center point. A graph $\Gamma_0\subset \Delta$ can be defined by attaching three edges to $v_0$, which extend radially from $v_0$ to the vertices of $\Delta$. The graph $\Gamma_{\a,\b,\g}$ is defined as
\[
\Gamma_{\a,\b,\g}:= \ws\times \Gamma_0.
\]
 We give $\Gamma_{\a,\b,\g}$ the cyclic order determined by giving the ends of $X_{\a,\b,\g}$ the ordering $-Y_{\a,\b},$ $-Y_{\b,\g},$ $Y_{\a,\g}$ (read left to right). See Figure~\ref{fig::40}.

\begin{figure}[ht!]
	\centering
	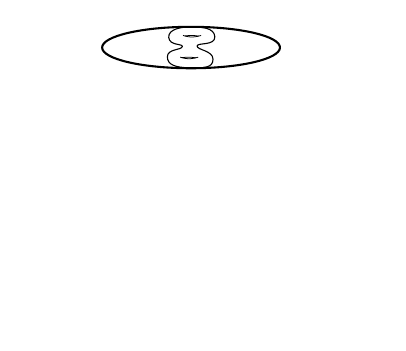
	\caption{\textbf{The 4-manifold with embedded graph $(X_{\a,\b,\g},\Gamma_{\a,\b,\g})$.} \label{fig::40}}
\end{figure}

A natural question is whether the holomorphic triangle map $F_{\a,\b,\g,\frs}$, defined by the formula
\begin{equation}
F_{\a,\b,\g,\frs}(\xs,\ys)=\sum_{\substack{\psi\in \pi_2(\xs,\ys,\zs)\\ \mu(\psi)=0\\
\frs_{\ws}(\psi)=\frs}}\# \cM(\psi) U^{n_{\ws}(\psi)} \cdot \zs,\label{eq:holtrianglemapdef}
\end{equation}
 is chain homotopic to the graph cobordism map for $(X_{\a,\b,\g},\Gamma_{\a,\b,\g})$. We answer this question in the affirmative:

\begin{thm}\label{thm:triplesandgraphcobordismmaps}Suppose that $(\Sigma, \ve{\alpha},\ve{\beta},\ve{\gamma},\ws)$ is a multi-pointed Heegaard triple, and let $(X_{\a,\b,\g}, \Gamma_{\a,\b,\g}):(Y_{\a,\b}\sqcup Y_{\b,\g},\ws\sqcup \ws)\to (Y_{\a,\g},\ws)$ denote the ribbon graph cobordism described above. If $\frs\in \Spin^c(X_{\a,\b,\g})$, the graph cobordism map $F_{X_{\a,\b,\g},\Gamma_{\a,\b,\g},\frs}^B$ is chain homotopic to the holomorphic triangle map 
\[
F_{\a,\b,\g,\frs}\colon \CF^-(\Sigma,\ve{\alpha},\ve{\beta},\ws,\frs|_{Y_{\a,\b}})\otimes_{\bF_2[U]} \CF^-(\Sigma,\ve{\beta},\ve{\gamma},\ws,\frs|_{Y_{\b,\g}})\to \CF^-(\Sigma,\ve{\alpha},\ve{\gamma},\ws,\frs|_{Y_{\a,\g}}),
\]
defined in equation~\eqref{eq:holtrianglemapdef}. 
\end{thm}

\begin{rem} A sketch of a similar result was communicated to the author by Lipshitz, Ozsv\'{a}th and Thurston \cite{LOTPersonalCommunication}. 
	\end{rem}

\begin{proof} We can obtain a handle decomposition for the cobordism $X_{\a,\b,\g}$ by examining the handle decomposition of the trace cobordism described in Section~\ref{sec:handledecomptracecobordism}. We start with a Morse function $f_{\b}$ on $U_{\b}$ which has $|\ws|$ index 0 critical points, and $|\b|$ index 1 critical points, whose ascending manifolds intersect $\Sigma$ along the $\bs$ curves, and has $\Sigma$ as its maximal level set. By adapting the handle decomposition for the trace cobordism from Section~\ref{sec:handledecomptracecobordism}, we can give $X_{\a,\b,\g}$ the following handle decomposition:
	\begin{itemize}
		\item a 1-handle attached for each index 0 critical point of $f_{\b}$, with one foot attached at an index 0 critical point of $f_{\b}$ in $U_{\b}\subset Y_{\a,\b}$, and the other foot attached at the corresponding critical point of $f_{\b}$ in $-U_{\b}\subset Y_{\b,\g}$;
		\item a collection of $|\bs|$ 2-handles, whose framed attaching link $\bL$ is formed by taking the descending manifolds of the index 1 critical points of $f_{\b}$ in $U_{\b}\subset Y_{\a,\b}$, and concatenating them across the 1-handles with their mirrors in $-U_{\b}\subset Y_{\b,\g}$.
	\end{itemize}

We can isotope the handles in this decomposition so that each of the 1-handles are attached with one foot at a basepoint $w\in \ws\subset Y_{\a,\b}$ and the other foot at the corresponding basepoint in $Y_{\b,\g}$. Furthermore, we can perform an isotopy of the graph $\Gamma_{\a,\b,\g}$, so that the two edges connected to $Y_{\a,\b}$ and $Y_{\b,\g}$, as well as the trivalent vertex, are contained in the interior of the corresponding 1-handle.

	The cobordism map $F_{X_{\a,\b,\g}, \Gamma_{\a,\b,\g},\frs}^B$ is thus equal to  the connected sum graph cobordism map $E_1^B$ from Section~\ref{sec:connectedsumsandgraphTQFT} (with $Y_{\a,\b}$ playing the role of $Y_1$, and $Y_{\b,\g}$ playing the role of $Y_2$), followed by the 2-handle map for surgery on $\bL$.
	
	We will  write $Y_{\a,\b}\, \#_{\ws} Y_{\b,\g}$ for the manifold obtain by adding $|\ws|$ connected sum tubes between $Y_{\a,\b}$ and $Y_{\b,\g}$, and we will abuse notation slightly and write $\ws$ for the basepoints in the connected sum regions of $Y_{\a,\b}\, \#_{\ws} Y_{\b,\g}$. The graph $\Gamma_{\a,\b,\g}$ intersects $Y_{\a,\b}\, \#_{\ws} Y_{\b,\g}$ at $\ws$. 

It is convenient to start our computation of $F_{X_{\a,\b,\g},\Gamma_{\a,\b,\g},\frs}$  at the diagram $(\Sigma\sqcup \bar{\Sigma}, \as\cup \bar{\gs}, \bs\cup \bar{\bs}, \ws\sqcup \ws)$. Hence we begin by composing with the transition map
\[
\id_{\CF^-(\Sigma,\a,\b)}\otimes \Psi_{(\Sigma,\b,\g)\to (\bar{\Sigma},\bar{\g},\bar{\b})}.
\]
 We will omit writing this transition map for most of the argument, to condense the notation, however it will reappear at the end.

 By Proposition~\ref{prop:OSmapsaregraphcobmaps},  we know that the graph cobordism map $E_1^B$ is chain homotopic to the Ozsv\'{a}th-Szab\'{o} intertwining map 
 \[
 \cE_1:\CF^-(\Sigma,\ve{\alpha},\ve{\beta},\ws,\frs|_{Y_{\a,\b}})\otimes_{\bF_2[U]} \CF^-(\bar{\Sigma},\bar{\ve{\gamma}},\ve{\bar{\beta}},\ws,\frs|_{Y_{\b,\g}})\]
 \[\to \CF^-(\Sigma\, \#_{\ws} \bar{\Sigma}, \as\cup \bar{\gs}, \bs\cup \bar{\bs}, \ws, \frs|_{Y_{\a,\b}}\# \frs|_{Y_{\b,\g}}),
 \] 
 defined by the formula
\[
\cE_1(-,-):=F_{\a\cup \bar{\g},\b\cup \bar{\g} ,\b\cup \bar{\b}}(F_1^{\bar{\g},\bar{\g}}(-)\otimes F_1^{\b,\b}(-)).
\]

We now pick curves $\Ds$ on $\Sigma\, \#_{\ws} \bar{\Sigma}$ as in Section~\ref{sec:doubleddiagram}, for a doubled diagram. Adapting the proof of Lemma~\ref{lem:randomtrianglemapis2-handlemap}, we see that the triple $(\Sigma\, \#_{\ws} \bar{\Sigma}, \as\cup \bar{\gs}, \bs\cup \bar{\bs}, \Ds,\ws)$ is (after performing a sequence of handleslides and isotopies) subordinate to a bouquet for the framed link $\bL\subset U_{\b\cup \bar{\b}}$.

 Thus the graph cobordism map $F_{X_{\a,\b,\g},\Gamma_{\a,\b,\g},\frs}(-,-)$ is chain homotopic to the composition
\begin{equation}
F_{\a\cup \bar{\g}, \b\cup \bar{\b}, \Dt}(F_{\a\cup \bar{\g},\b\cup \bar{\g} ,\b\cup \bar{\b}}(F_1^{\bar{\g},\bar{\g}}(-)\otimes  F_1^{\b,\b}(-))\otimes \Theta_{\b\cup \bar{\b}, \Dt}^+).\label{eq:trianglemap=graphTQFT1}\end{equation} The associativity relations for the quadruple $(\as\cup \bar{\gs}, \bs\cup \bar{\gs},\bs\cup \bar{\bs}, \Ds)$ imply that equation~\eqref{eq:trianglemap=graphTQFT1} is chain homotopic to
\begin{equation}F_{\a\cup \bar{\g}, \b\cup \bar{\g}, \Dt}(F_1^{\bar{\g},\bar{\g}}(-)\otimes F_{\b\cup \bar{\g},\b\cup \bar{\b},\Dt}( F_1^{\b,\b}(-)\otimes \Theta^+_{\b\cup \bar{\b},\Dt})).\label{eq:trianglemap=graphTQFT}\end{equation} The final Heegaard diagram in this composition is the double of the diagram $(\Sigma,\as,\gs)$, so we must post-compose with a transition map to undo the doubling operation.  Proposition~\ref{prop:changeofdiagramsmapcompundouble} shows that the transition map associated to undoing the doubling operation satisfies
\begin{equation}\Psi_{(\Sigma\, \#_{\ws} \bar{\Sigma},\a\cup \bar{\g},\Dt)\to (\Sigma,\a,\g)}(-)\simeq F_{3}^{\bar{\g},\bar{\g}}F_{\a\cup \bar{\g}, \Dt, \g\cup \bar{\g}}(- \otimes \Theta_{\Dt, \g\cup\bar{\g}}^+).\label{eq:trianglemap=graphTQFT3}\end{equation} Composing equation~\eqref{eq:trianglemap=graphTQFT} with equation~\eqref{eq:trianglemap=graphTQFT3}, we see that the  graph cobordism map for $(X_{\a,\b,\g}, \Gamma_{\a,\b,\g})$ is chain homotopic to
\begin{equation}F_{3}^{\bar{\g},\bar{\g}}F_{\a\cup \bar{\g}, \Dt, \g\cup \bar{\g}}(F_{\a\cup \bar{\g}, \b\cup \bar{\g}, \Dt}(F_1^{\bar{\g},\bar{\g}}(-)\otimes F_{\b\cup \bar{\g},\b\cup \bar{\b},\Dt}( F_1^{\b,\b}(-)\otimes \Theta_{\b\cup \bar{\b},\Dt}^+))\otimes \Theta_{\Dt, \g\cup \bar{\g}}^+ ).\label{eq:trianglemap=graphTQFT4}
\end{equation} 
Applying the associativity relations for the quadruple $(\as\cup \bar{\gs}, \bs\cup \bar{\gs}, \Ds, \gs\cup \bar{\gs})$ to the left two triangle maps in equation~\eqref{eq:trianglemap=graphTQFT4}, we conclude that equation~\eqref{eq:trianglemap=graphTQFT4} is chain homotopic to
\begin{equation}F_3^{\bar{\g},\bar{\g}}F_{\a\cup \bar{\g},\b\cup \bar{\g}, \g\cup \bar{\g}}(F_1^{\bar{\g},\bar{\g}}(-)\otimes F_{\b\cup \bar{\g},\Dt, \g\cup \bar{\g}}(F_{\b\cup \bar{\g},\b\cup \bar{\b},\Dt}( F_1^{\b,\b}(-)\otimes \Theta_{\b\cup \bar{\b},\Dt}^+)\otimes \Theta_{\Dt, \g\cup \bar{\g}}^+)).\label{eq:trianglemap=graphTQFT5}\end{equation}  Using the holomorphic triangle counts from Proposition~\ref{prop:generalized1-handlesandtriangles}, we conclude that equation~\eqref{eq:trianglemap=graphTQFT5} is equal to
\begin{equation}F_{\a,\b,\g}(-\otimes F_3^{\bar{\g},\bar{\g}}( F_{\b\cup \bar{\g},\Dt, \g\cup \bar{\g}}(F_{\b\cup \bar{\g},\b\cup \bar{\b},\Dt}(F_1^{\b,\b}(-)\otimes \Theta_{\b\cup \bar{\b},\Dt}^+)\otimes \Theta_{\Dt, \g\cup \bar{\g}}^+))).\label{eq:trianglemap=graphTQFT6}\end{equation} 
The composition
\begin{equation}F_3^{\bar{\g},\bar{\g}}( F_{\b\cup \bar{\g},\Dt, \g\cup \bar{\g}}(F_{\b\cup \bar{\g},\b\cup \bar{\b},\Dt}(F_1^{\b,\b}(-)\otimes \Theta_{\b\cup \bar{\b},\Dt}^+)\otimes \Theta_{\Dt, \g\cup \bar{\g}}^+)\label{eq:trianglemap=graphTQFT7}\end{equation} is chain homotopic to the transition map
\begin{equation}\Psi_{(\bar{\Sigma},\bar{\g},\bar{\b})\to (\Sigma,\b,\g)},\label{eq:trianglemap=graphTQFT8}\end{equation} since equation~\eqref{eq:trianglemap=graphTQFT7} represents the composition of a change of maps for doubling, followed by the transition map for undoing the doubling operation, by Propositions~\ref{prop:changeofdiagramsmapcomp} and \ref{prop:changeofdiagramsmapcompundouble}. Using the fact that equations~\eqref{eq:trianglemap=graphTQFT7} and~\eqref{eq:trianglemap=graphTQFT8} are chain homotopic, our expression for the graph cobordism map for $(X_{\a,\b,\g}, \Gamma_{\a,\b,\g})$ from equation~\eqref{eq:trianglemap=graphTQFT6} reduces to
\begin{equation}
F_{\a,\b,\g}(-\otimes \Psi_{(\bar{\Sigma},\bar{\g},\bar{\b})\to (\Sigma,\b,\g)}(-)).\label{eq:trianglecobordism1}\end{equation} On the other hand, we started the proof by composing with the transition map 
\[
\id_{\CF^-(\Sigma,\a,\b)}\otimes \Psi_{(\Sigma,\b,\g)\to (\bar{\Sigma},\bar{\g},\bar{\b})}.\] Composing equation~\eqref{eq:trianglecobordism1} with this transition map, which we have been omitting until now, leaves just $F_{\a,\b,\g}(-,-)$, completing the proof. 
\end{proof}

\section{Duality and the graph TQFT}
\label{sec:traceandcotrace}
In this section, we prove Theorem~\ref{thm:dualityv1} by computing the maps induced by the trace and cotrace cobordisms.

\subsection{Turning around graph cobordisms}

If $(W,\Gamma):(Y_1,\ve{w}_1)\to (Y_2,\ve{w}_2)$ is a graph cobordism, then we can turn around $(W,\Gamma)$ to get a graph cobordism 
\[
(W^\vee,\Gamma^\vee):(-Y_2,\ve{w}_2)\to (-Y_1,\ve{w}_1).
\]
 Here we give $\Gamma^\vee$ the same cyclic order as $\Gamma$. In this section, we extend the duality result of \cite{OSTriangles}*{Theorem~3.5} to graph cobordisms, by proving that the graph cobordism map for $(W^\vee,\Gamma^\vee)$ is the dual of the cobordism map for $(W,\Gamma)$:

\begin{prop}\label{prop:turningaroundgraphcobordisms}If $(W,\Gamma):(Y_1,\ve{w}_1)\to (Y_2,\ve{w}_2)$ is a ribbon graph cobordism, then
	\[
	F_{W^\vee,\Gamma^\vee,\frs}^A\simeq   (F_{W,\Gamma,\frs}^A)^{\vee},\] with respect to the natural pairing between $\CF^-(Y,\ws,\frs)$ and $\CF^-(-Y, \ws, \frs)$. The same holds for the type-$B$ graph cobordism maps.
\end{prop}
\begin{proof}It is sufficient to show the claim for a cobordism obtained by attaching a single 0-, 1-, 3- or 4-handle, a collection of 2-handles, or a graph cobordism with underlying 4-manifold $Y\times [0,1]$.

Note that when we turn around a $k$-handle, we get a $4-k$ handle. Duality between the  1- and 3-handle maps follows exactly as in \cite{OSTriangles}*{Section~5} and is immediate from the formulas. The argument for the 2-handle maps is the same as in \cite{OSTriangles}*{Section~5}: if $W$ is a 2-handle cobordism whose map can be computed by counting triangles in the triple $(\Sigma,\as,\bs,\bs')$, then the turned around cobordism $W^\vee$ is a 2-handle cobordism, whose map can be computed by counting triangles in the triple $(\Sigma,\bs,\bs',\as)$. Since the holomorphic triangles on the two triples are in bijection, it is easy to see that the  maps $F_{\a,\b,\b',\frs}(-,\Theta_{\b,\b'}^+)$ and $F_{\b,\b',\a,\frs}(\Theta_{\b,\b'}^+,-)$ are dual to each other.

We now consider the new maps appearing in the graph TQFT. These were summarized in Section~\ref{sec:outlineofconstruction}. Firstly, it is clear that the 0-handle and 4-handle maps are dual to each other. The remaining maps are the graph action maps (which give the graph cobordism maps for graph cobordisms with underlying 4-manifold $Y\times [0,1]$). If $\Gamma:V_0\to V_1$ is an embedded flow graph in $Y$ between two disjoint collections of vertices $V_0,V_1\subset Y$, then the map $A_{\Gamma}:\CF^-(Y,V_0)\to \CF^-(Y,V_1)$ is defined as a composition of free-stabilization maps $S_{w}^{\pm}$ and relative homology maps $A_{\lambda}$, for various edges $\lambda$ and vertices $w$ of a subdivision of  $\Gamma$.  The graph action map $B_{\Gamma}$ is defined by replacing each instance of $A_{\lambda}$ in the formula for $A_{\Gamma}$ with $B_{\lambda}$.  We claim that
\begin{equation}(A_{\Gamma})^\vee\simeq B_{\bar{\Gamma}^\vee}.\label{eq:dualgraphaction}\end{equation} Here $\bar{\Gamma}^\vee:V_1\to V_0$ is the flow graph in $Y$ obtained by turning $\Gamma$ around (i.e. viewing it as a flow graph from $V_1$ to $V_0$) and reversing all of the cyclic orders. To establish equation~\eqref{eq:dualgraphaction}, observe that from the definition of the graph action map \cite{ZemGraphTQFT}*{Section~7}, the formula for $B_{\bar{\Gamma}^\vee}$ can be obtained by taking the formula for $A_{\Gamma}$, reversing the order of all maps (hence reversing the cyclic orders at each vertex), replacing every free-stabilization $S_w^+$ with the corresponding free de-stabilization $S_w^-$ (and vice versa), and replacing each $A_\lambda$ map with $B_\lambda$. Just like with the 1-handle and 3-handle maps, it is straightforward to see that the maps $S_w^+$ and $S_w^-$ are dual to each other. Similarly, the two maps
\[
B_\lambda:\CF^-(\Sigma,\bs,\as)\to \CF^-(\Sigma,\bs,\as)
\]
 and
\[
A_{\lambda}:\CF^-(\Sigma,\as,\bs)\to \CF^-(\Sigma,\as,\bs)
\] 
are dual to each other, since they count the same holomorphic curves, with the same factor, since the roles of the $\ve{\alpha}$ and $\ve{\beta}$ curves have been changed on the two diagrams. Equation~\eqref{eq:dualgraphaction} follows. It follows that
	\[(F_{W,\Gamma,\frs}^A)^\vee=F_{W^\vee,\bar{\Gamma}^\vee,\frs}^B.\] Applying Lemma~\ref{lem:reversecyclicordering} shows that
	\[F_{W^\vee,\bar{\Gamma}^\vee,\frs}^B\simeq F_{W^\vee,\Gamma^\vee,\frs}^A,\] completing the proof for the type-$A$ maps. The same argument works for the type-$B$ maps.
\end{proof}

\subsection{Trace and cotrace cobordism maps}

In this section, we prove that the trace and cotrace graph cobordisms induce the canonical trace and cotrace maps.

\begin{customthm}{\ref{thm:dualityv1}}
If $(Y,\ws)$ is a multi-pointed 3-manifold, the trace graph cobordism $(Y\times [0,1],\ws\times [0,1]): (Y\sqcup -Y,\ws\sqcup \ws)\to \emptyset$ induces the canonical trace map
\[
\tr:\CF^-(Y,\ws,\frs)\otimes_{\bF_2[U]} \CF^-(-Y,\ws,\frs)\to \bF_2[U].
\] Similarly, the cotrace graph cobordism $(Y\times [0,1],\ws\times [0,1]):\emptyset\to (Y\sqcup -Y,\ws\sqcup \ws)$ induces the canonical cotrace map
\[
\cotr: \bF_2[U]\to \CF^-(Y,\ws,\frs)\otimes_{\bF_2[U]} \CF^-(-Y,\ws,\frs).
\] The formulas hold for both the type-$A$ and type-$B$ graph cobordism maps.
\end{customthm}
\begin{proof}[Proof of Theorem~\ref{thm:dualityv1}] We first consider the trace cobordism. Pick a diagram $(\Sigma,\ve{\alpha},\ve{\beta},\ws)$ for $Y$. We note that according to \cite{OSTriangles}*{Proposition~4.3}, the 4-manifold  $X_{\a,\b,\a}$ is diffeomorphic to $Y\times [0,1]\setminus N(U_{\a}\times \{\tfrac{1}{2}\})$, where $U_{\a}\subset Y$ denotes the $\alpha$-handlebody. Each $\as$ curve determines a compressing disk in $U_{\a}$. The union of this disk in $U_{\a}$, together with its image in $-U_{\a}$, determines a 2-sphere in $Y_{\a,\a}\subset \d X_{\a,\b,\a}$.  Let $(X',\Gamma')$ denote the graph cobordism obtained by attaching $|\as|$ 3-handles to $(X_{\a,\b,\a},\Gamma_{\a,\b,\a})$, along these 2-spheres in $Y_{\a,\a}$. The cobordism $(W',\Gamma')$ is diffeomorphic to $(Y\times [0,1],\ws\times [0,1])$ with $|\ws|$ 4-balls removed, and graph $\Gamma'$ obtained by adding a strand from each 3-sphere in the boundary of $W'$ to one of the components of $\ws\times [0,1]$.

Let $\ve{\alpha}'$ be small Hamiltonian translates of the $\ve{\alpha}$ curves. Note that $X_{\a,\b,\a}$ is diffeomorphic to $X_{\a,\b,\a'}$.  Using Theorem~\ref{thm:triplesandgraphcobordismmaps}, the cobordism map for $X_{\a,\b,\a'}$ is chain homotopic to the holomorphic triangle map $F_{\a,\b,\a'}$. The type-$B$ graph cobordism map for the trace cobordism 
\[
F_{Y\times [0,1], \ws\times [0,1]}^B:\CF^-(\Sigma,\as,\bs)\otimes \CF^-(\Sigma, \bs, \as)\to \bF_2[U]
\]
 can thus be computed as the composition of the change of diagrams map $\id \otimes \Psi_{\b}^{\a\to \a'}$, followed by the triangle map $F_{\a,\b,\a'}$, followed by  $|\ve{\alpha}|$ 3-handle maps and $|\ws|$ 4-handle maps. If we identify $\CF^-(S^3,w)$ with $\bF_2[U]$ via the 4-handle map, the graph cobordism map for the trace cobordism thus takes the form
	\begin{align*}&F_{Y\times [0,1],\ws\times [0,1]}^B(\ve{x}\otimes \ve{y})\\=&F_3^{\a,\a'}(F_{\a,\b,\a'}(\ve{x}\otimes \Psi_{\b}^{\a\to \a'}(\ve{y})))\\=&\langle F_{\a,\b,\a'}(\ve{x}\otimes \Psi_{\b}^{\a\to \a'}(\ve{y})), \Theta^-_{\a,\a'}\rangle'\\
	=&\tr( F_{\a',\a,\b}(\Theta^+_{\a',\a}\otimes \ve{x})\otimes\Psi_{\b}^{\a\to \a'}(\ve{y}))\\
	=&\tr( \Psi^{\beta}_{\alpha\to \alpha'}(\ve{x})\otimes \Psi_{\b}^{\a\to \a'}(\ve{y}) )\\
	=&\tr(\xs\otimes\ys). \end{align*} The first equality follows from the topological reasoning of the previous paragraph. The second equality is the definition of the 3-handle map. The third equality follows from observing that $F_{\a,\b,\a'}$ and $F_{\a',\a,\b}$ count the same holomorphic triangles, and also noting that $\Theta_{\a,\a'}^-=\Theta_{\a',\a}^+$. The fourth equality is obtained by observing that the triangle map in the fourth line computes the transition map. The final equality follows by noting that 
	\[
	\Psi_{\b}^{\a\to \a'}: \CF^-(\Sigma,\bs,\as)\to \CF^-(\Sigma,\bs,\as'),
	\]
	 is the dual of 
	\[
	\Psi_{\a'\to \a}^{\b}: \CF^-(\Sigma,\as',\bs)\to \CF^-(\Sigma,\as,\bs),
	\]
	 and that $\Psi_{\a'\to \a}^{\b}\circ \Psi_{\a\to \a'}^{\b}\simeq \id_{\CF^-(\Sigma,\a,\b)}$.

As the graph cobordism $(Y\times [0,1],\ws\times [0,1])$ has no vertices of valence 3 or greater, it follows that \[F^A_{Y\times [0,1],\ws\times [0,1]}\simeq F^B_{Y\times [0,1],\ws\times [0,1]}\] by Lemma~\ref{lem:reversecyclicordering}, so the same formula holds for the type-$A$ graph cobordism maps.

	The statement about the cotrace cobordism follows from the formula for the trace cobordism map, together with Proposition~\ref{prop:turningaroundgraphcobordisms}, since the cotrace cobordism is obtained by turning around the trace cobordism.
\end{proof}

\section{Mixed invariants and non-separating cuts}
\label{sec:mixedinvariants}

In this section, we prove the following:

\begin{customthm}{\ref{thm:mixedinvariantmappingtorus}}
Suppose $X^4$ is a closed, oriented 4-manifold with $b_2^+(X)> 1$ and $Y^3\subset X$ is a closed, oriented, connected and non-separating 3-dimensional submanifold. Write $W$ for the cobordism obtained by cutting $X$ along $Y$. Suppose $\frs\in \Spin^c(W)$ is a $\Spin^c$ structure whose restrictions to both copies of $Y$ in $\d W$ agree. Suppose further that at least one of the following holds:
\begin{enumerate}
\item $\frs|_Y$ is non-torsion, or
\item $b_2^+(W)>0$.
\end{enumerate}
 If $\xi\in \Lambda^*(H_1(W)/\Tors)\otimes \bF_2[U]$, then the $\bF_2$ mixed invariants of $X$ satisfy
\[
\Lef\big(F_{W,\frs}(\xi\otimes -):\HF^+_{\red}(Y,\frs|_{Y})\to \HF^+_{\red}(Y,\frs|_{Y})\big)=\sum_{\substack{\frt\in \Spin^c(X)\\ \frt|_W=\frs}}\Phi_{X,\frt}(\xi).
\]
\end{customthm}

By specializing to mapping tori, we obtain the following:

\begin{customcor}{\ref{cor:mixedinvariantofactualmappingtori}}Suppose $Y^3$ is a closed, oriented 3-manifold and $\phi:Y\to Y$ is an orientation preserving diffeomorphism such that  $b_2^+(X_\phi)>1$. If $\frs\in \Spin^c(Y)$ is non-torsion and  $\phi_*(\frs)=\frs$, then the mixed invariants of $X_\phi$ satisfy
\[
\Lef\big(\phi_*:\HF^+_{\red}(Y,\frs)\to \HF^+_{\red}(Y,\frs)\big)=\sum_{\substack{\frt\in \Spin^c(X_\phi)\\ \frt|_Y=\frs}}\Phi_{X_\phi,\frt}(1).
\]
\end{customcor}

A key step in our proof of Theorem~\ref{thm:mixedinvariantmappingtorus} is a general result about computing mixed invariants using graph cobordisms, which is stated in Theorem~\ref{thm:mixed-invariants-from-graphs}. Using this result, we prove Theorem~\ref{thm:mixedinvariantmappingtorus} in Section~\ref{sec:proof:1.1}.

\subsection{Mixed invariants and graph cobordisms}

In this section, we prove a general result about computing the mixed invariants using graph cobordisms, Theorem~\ref{thm:mixed-invariants-from-graphs}. It will be used when we prove Theorem~\ref{thm:mixedinvariantmappingtorus}.

\begin{define}
\label{def:admissible-1}
 Suppose that $X^4$ is a closed, connected, and oriented 4-manifold. We say that a tuple $\cC=(N,\Gamma,\frs_1,\frs_2)$ is a \emph{graphed-decorated cut} of $X$ if it satisfies the following:
 \begin{enumerate}[label=(C-\arabic*), ref=C-\arabic*,leftmargin=*, widest=III]
\item\label{def:cut-1} $N^3\subset X$ is an oriented 3-manifold dividing $X$ into the composition of two cobordisms, $W_1$ and $W_2$.
\item\label{def:cut-2} $\frs_1\in \Spin^c(W_1)$ and $\frs_2\in \Spin^c(W_2)$ are $\Spin^c$ structures which have a common restriction, $\frs\in \Spin^c(N)$.
\item\label{def:cut-3} $\Gamma$ is a ribbon graph in $X$ which intersects each component of $N$ transversely and non-trivially. Furthermore, $\Gamma$ consists of a tree $T$ with loops spliced on away from the vertices of $T$, with the configuration shown on the left hand side of Figure~\ref{fig::26}.
 \end{enumerate}
If $\cC$ is a graph-decorated cut, we say that $\cC$ is \emph{semi-admissible} if it satisfies at least one of the following:
 \begin{enumerate}[label=(SA-\arabic*), ref=SA-\arabic*,leftmargin=*, widest=IIII]
 \item\label{SA-1} There is a surface $\Sigma$ of positive square in $X\setminus N$, and both of the following maps vanish:
 \[
 \begin{split}
F_{W_1,\Gamma_1,\frs_1}^{B}&\colon \boldHF^\infty(S^3)\to \boldHF^\infty(N,\ws,\frs), \quad \text{and}\\
F_{W_2,\Gamma_2,\frs_2}^{B}&\colon  \boldHF^\infty(N,\ws,\frs)\to \boldHF^\infty(S^3).
\end{split}
\]
 \item\label{SA-2}  $c_1(\frs)\in H^2(N)$ is non-torsion.
 \end{enumerate}
\end{define}

Note that Definition~\ref{def:admissible-1}, $N$, $W_1$ and $W_2$ are not required to be connected.

The following is a small but helpful extension of well known facts about the path cobordism maps:
\begin{lem}
Suppose $X$ is a closed, connected and oriented 4-manifold and $\cC=(N,\Gamma,\frs_1,\frs_2)$ is a graph-decorated cut.  Suppose additionally that one of the following two conditions is satisfied:
\begin{enumerate}
\item $b_2^+(W_i)>0$ for $i=1,2$.
\item $c_1(\frs)$ is non-torsion in $H^2(N)$, where $\frs=\frs_i|_N$.
\end{enumerate}
Then $F_{W_i,\Gamma_i,\frs_i}^{B}=0$ on $\boldHF^\infty$, for $i=1,2$.
\end{lem}

\begin{proof} First, suppose $c_1(\frs)$ is non-torsion on $N$. We claim 
\begin{equation}
\boldHF^\infty(N,\ws,\frs)=0,\label{eq:hfinfty=0}
\end{equation}
which obviously implies the claim.  Indeed if $c_1(s)$ is non-torsion, then $c_1(\frs)$ is non-torsion on some component of $N$. Call this component $N_0$ and let $\frs_0=\frs|_N$. It follows from \cite{OSTrianglesandSymplectic}*{Lemma~2.3} that $\boldHF^\infty(N_0,\frs_0)=0$ (where the latter group is the singly based version of the Heegaard Floer complex). Since adding a basepoint has the effect on homology of tensoring over $\bF_2$ with a 2 dimensional vector space $\bF_2\oplus \bF_2$ (see \cite{OSLinks}*{Proposition~6.5}), and the definition of the complex for a disjoint union is the tensor product over $\bF_2[[U]]$, equation~\eqref{eq:hfinfty=0} quickly follows from the K\"{u}nneth theorem (e.g. using Lemma~\ref{lem:classificationfgPID}).
 
Next, we consider the case when $b_2^+(W_i)>0$ for both $i=1,2$. It is sufficient to show that both maps $F_{W_i,\Gamma_i,\frs_i}^B$ vanish on $\boldHF^\infty$. Consider $(W_1,\Gamma_1)$.  We can decompose $(W_1,\Gamma_1)$ as a composition $(Z_2,\Gamma_2')\circ (Z_1,\Gamma_1')$, where $Z_1$ is a disjoint union of 4-dimensional 1-handlebodies (i.e. is obtained by attaching 0-handles and 1-handles to $\emptyset$), while $(Z_2,\Gamma_2')$ has the property that each component of $Z_2$ has exactly one incoming boundary component, and exactly one outgoing boundary component, and $\Gamma_2'$ is a disjoint union of arcs, each of which connects the incoming boundary of $Z_2$ to the outgoing boundary.
 By the composition law,
\begin{equation}
F_{W_1,\Gamma_1,\frs_1}^{B}=F_{Z_2,\Gamma_2',\frs_1|_{Z_2}}^{B}\circ F_{Z_1,\Gamma_1',\frs_1|_{Z_1}}^{B},\label{eq:composition-law-W_1}
\end{equation}
since $Z_1$ is a 1-handlebody. We note that $b_2^+(Z_2)>0$, so at least one component of $Z_2$ has positive $b_2^+$. The map $F_{Z_2,\Gamma_2',\frs_1|_{Z_2}}^{B}$ is, by definition, the tensor product over $\bF_2[[U]]$ of the maps for each component. It follows from \cite{OSTriangles}*{Lemma~8.2}
that the map for one of the components of $Z_2$ vanishes on $\boldHF^\infty$. An easy extension of Lemma~\ref{lem:0-map-tensor} implies that the map $F_{Z_2,\Gamma_2',\frs_1|_{Z_2}}^B$ itself vanishes on $\boldHF^\infty$. Equation~\eqref{eq:composition-law-W_1} implies the map for $W_1$ vanishes on $\boldHF^\infty$.  The same argument works for $W_2$.
\end{proof}

Given a semi-admissible graph-decorated cut $\cC$, we define a mixed invariant $\Phi_{X,\cC}$, via the following formula:
\begin{equation}
\Phi_{X,\cC}:=\left\langle \Theta^-, \left(F_{W_2,\Gamma_2,\frs_2}^{B}\circ \delta^{-1} \circ F_{W_1,\Gamma_1,\frs_1}^{B}\right)(1)\right\rangle \in \bF.\label{eq:mixed-invariant-graph}
\end{equation}

\begin{thm}\label{thm:mixed-invariants-from-graphs}
 Suppose that $X^4$ is a closed, oriented 4-manifold with $b_2^+(X)>1$, and $\cC=(N,\Gamma,\frs_1,\frs_2)$ is a semi-admissible, graph-decorated cut of $X$. Then
\begin{equation}
\Phi_{X,\cC}= \sum_{\substack{\frt\in \Spin^c(X)\\ \frt|_{W_1}=\frs_1\\ 
\frt|_{W_2}=\frs_2}} \Phi_{X,\frt}(\xi_1\wedge\cdots \wedge \xi_n),
\label{eq:composition-law-mixed-invarariant-statement}
\end{equation}
where $\xi_1,\dots, \xi_n\in H_1(X)/\Tors$ are the classes of the loops of the graph $\Gamma$.
\end{thm}

\begin{proof} We focus first on~\eqref{SA-2}, i.e. when $c_1(\frs)$ is non-torsion. 
Our proof is an adaptation of \cite{OSTrianglesandSymplectic}*{Proposition~2.5} to the setting of graph cobordisms. See also \cite{JabukaMarkProduct}*{Theorem~8.17}.

We introduce some general notation. Suppose that a cobordism $Z$ decomposes as a composition $Z_2\circ Z_1$. If $\frU\subset \Spin^c(Z_2)\times \Spin^c(Z_1)$, then we write $\frS(Z, \frU)$ for the set of $\Spin^c$ structures on $Z$ which restrict to an element of $\frU$.

 We pick an embedded surface $\Sigma\subset N$ such that $\langle c_1(\frs),\Sigma \rangle \neq 0$, which we push slightly into $W_1$. 
Define the following cobordisms
\[
W_1'=N(\Sigma),\quad W_2'=W_1\setminus N(\Sigma)\quad \text{ and } \quad W_3'=W_2.
\]
Write $N'$ for $\d N(\Sigma)\iso S^1\times \Sigma$. We may perturb $\Sigma$ and choose an appropriate regular neighborhood so that $N(\Sigma)$ is disjoint from $\Gamma$. We add a leaf to $\Gamma$ so that it intersects $N'$ transversely in a single point. Abusing notation we write $\Gamma$ for the resulting graph. Write $\Gamma_i'$ for $\Gamma \cap W_i'$, and $\frs_1':=\frs_1|_{W_1'}$, $\frs_2'=\frs_1|_{W_2'}$ and $\frs_3'=\frs_2$. Write $W_{2,3}'$ for $W_{2}'\cup W_{3}'$, and similarly write $\Gamma_{2,3}'$ for $\Gamma\cap W_{2,3}'$.

Consider the following commutative diagram:
 \begin{equation}
\begin{tikzcd}\, &
  \boldHF^-(\emptyset)
  \arrow{d}{F_{W_1',\Gamma_1',\frs_1'}^B}\\
\HF^+(N',\frs')
	\arrow{r}{\iso}[swap]{\delta}
 	\arrow{d}[swap]{F_{W_2',\Gamma_2',\frs_2'}^B}& 
\boldHF^-(N',\frs')
 	\arrow{d}{F_{W_2', \Gamma_2',\frs_2'}^B}\\
\HF^+(N,\frs)\arrow{r}{\iso}[swap]{\delta} 
 	\arrow{d}[swap]{F_{W_3',\Gamma_3', \frs_3'}^B}&
\boldHF^-(N,\frs)\\
\HF^+(\emptyset)&
 \end{tikzcd}
 \label{eq:commutative-diagram-factorizations}
\end{equation}
Note that $\delta(H^1(N'))$ vanishes in $H^2(W_{1})\iso H_2(W_{1},\d W_{1})$, since the coboundary map is Poincar\'{e} dual to the inclusion of $H_2(N')$ into $H_2(W_1,\d W_1)$, and $\Sigma$ may be isotoped into $\d W_1$. Hence
\[
F_{W_1,\Gamma_1,\frs_1}^B=F_{W_2',\Gamma_2',\frs_2'}^B\circ F_{W_1',\Gamma_1',\frs_1'}^B
\]
by the composition law.
Furthermore, by considering equation~\eqref{eq:commutative-diagram-factorizations} and using the $\Spin^c$ composition law, we obtain that
\begin{equation}
\Phi_{X,\cC}=
 \bigg\langle \Theta^-,\bigg(\sum_{\frt\in \frS(W_{2,3}',\frs_2',\frs_3')} F_{W_{2,3}',\Gamma_{2,3}',\frt}^{B}\circ \delta^{-1} \circ F_{W_{1}',\Gamma_{1}',\frs_{1}'}^B\bigg)(1)\bigg\rangle,
 \label{eq:Phi-composition-1}
\end{equation}
where $\frS(W_{2,3}',\frs_2',\frs_3')$ is the set of $\Spin^c$ structures on $W_{2,3}'$ which restrict to $\frs_2'$ and $\frs_3'$.

Since $b_2^+(X)>1$, we may pick a surface $\Sigma'$ of positive square in $X$ which is disjoint from $\Sigma$. We may pick $\Sigma'$ generically, so that it is disjoint from $\Gamma$, then add a leaf to $\Gamma$, so that it intersects $\d N(\Sigma')$ transversely in a single point (abusing notation, we call the new graph $\Gamma$).

Set 
\[
W_1''=W_1'=N(\Sigma),\quad W_2''=X\setminus (N(\Sigma)\cup N(\Sigma')), \quad \text{and} \quad W_3''=N(\Sigma').
\]

Since $\Sigma'$ is a surface of positive square
\[
\delta H^1(\d N(\Sigma'))=\{0\}\subset H^2(W_{2,3}').
\]
In particular, if we let $\frU_{2,3}''\subset \Spin^c(W_3'')\times \Spin^c(W_2'')$ denote the set of pairs $(\fru|_{W_3''}, \fru|_{W_2''})$ where $\fru\in \frS(W_{2,3}',\frs_2',\frs_3')$, then
\[
\frS(W_{2,3}',\frU_{2,3}'')=\frS(W_{2,3}',\frs_2',\frs_3').
\]

The composition law implies that
\[
\sum_{\frt\in \frS(W_{2,3}',\frs_2',\frs_3')} F_{W_{2,3}', \Gamma_{2,3}',\frt}=\sum_{(\fru_3,\fru_2)\in \frU_{2,3}''} F_{W_3'',\Gamma_{3}'',\fru_3}^B\circ F_{W_2'',\Gamma_{2}'',\fru_2}^{B}.
\]
By considering a commutative diagram like the one shown in equation~\eqref{eq:commutative-diagram-factorizations}, we manipulate equation~\eqref{eq:Phi-composition-1} to become
\begin{equation}
\Phi_{X,\cC}=
 \bigg\langle \Theta^-,\bigg(\sum_{(\fru_3,\fru_{2})\in\frU_{2,3}''} F_{W_3'', \Gamma_3'', \fru_3}^B\circ \delta^{-1} \circ F_{W_{2}'', \Gamma_{2}'', \fru_{2}}^B\circ F_{W_1'', \Gamma_1'', \frs_1''}^B\bigg)(1)\bigg\rangle.
 \label{eq:Phi-composition-2}
\end{equation}
Using the $\Spin^c$ composition law, this is easily rearranged to
\begin{equation}
\Phi_{X,\cC}= \bigg\langle \Theta^-,\bigg(\sum_{(\fru_3,\fru_{1,2})\in \frU''} F_{W_3'', \Gamma_3'',\fru_3}^B\circ \delta^{-1} \circ F_{W_{1,2}'', \Gamma_{1,2}'', \fru_{1,2}}^B\bigg)(1)\bigg\rangle,
\label{eq:almost-sum-of-mixed-invariants}
\end{equation}
where $\frU''\subset \Spin^c(W_{3}'')\times \Spin^c(W_{1,2}'')$ consists of all elements which have restriction equal to an element in $\frU''_{2,3}\times \{\frs_1'\}\subset \Spin^c(W_3'')\times \Spin^c(W_2'')\times \Spin^c(W_1'')$.

Using Proposition~\ref{prop:spliceinloopsforUandH_1}, we see that
\[
 F_{W_{1,2}'',\Gamma_{1,2}'',\fru_{1,2}}^B=F_{W_{1,2}'', \fru_{1,2}}(\xi_1\wedge\cdots \wedge \xi_n).
 \]
  Similarly $F_{W_3'',\Gamma_3'',\fru_3}^B$ is the ordinary cobordism map for $W_3''$, since $\Gamma_3''$ is an arc.
  In particular, Equation~\eqref{eq:almost-sum-of-mixed-invariants} rearranges to give
    \[
\Phi_{X,\cC}=\sum_{\frt\in \frS(X,\frU'')} \Phi_{X,\frt}  (\xi_1\wedge\cdots \wedge \xi_n).
  \]
  Finally, it is an easy exercise to see that $\frS(X,\frU'')=\{\frt\in \Spin^c(X): \frt|_{W_1}=\frs_1, \frt|_{W_2}=\frs_2\}$, completing the proof in case~\eqref{SA-2}.

When instead~\eqref{SA-1} is satisfied, the proof follows from the same reasoning as \eqref{SA-2}, by factoring $\Phi_{X,\cC}$ through a neighborhood of a surface of positive square.
\end{proof}

\subsection{Proof of Theorem~\ref{thm:mixedinvariantmappingtorus}}

\label{sec:proof:1.1}

\begin{proof}[Proof of Theorem~\ref{thm:mixedinvariantmappingtorus}] Suppose, for simplicity, that $\xi=\xi_1\wedge\cdots \wedge \xi_n$, where $\xi_i\in H_1(W)/\Tors$. We view $X$ as being the union of $W$ and $Y\times [0,1]$. Let $\Gamma\subset W$ be a graph constructed by taking a path connecting the two boundary components of $\d W$, and attaching loops representing $\xi_1,\dots, \xi_n$, with cyclic ordering as on the left side of Figure~\ref{fig::48}. Let $\hat{\Gamma}$ denote an extension of $\Gamma$ into $Y\times [0,1]$, such that $\hat{\Gamma}\cap (Y\times [0,1])$ consists of two disjoint arcs, one connected to each boundary component.

We decompose $(X,\widehat{\Gamma})$ as in Figure~\ref{fig:59}. We use the central copy of $Y\sqcup -Y$ as the cut, $N$. Lemma~\ref{lem:phi=brokenpathcobordism} identifies the map for the copy of $Y\sqcup -Y\times [0,1]$ with a broken path as $\id\otimes \Phi_w^\vee$.

 If $c_1(\frs)$ is non-torsion, then it follows from \cite{OSTrianglesandSymplectic}*{Lemma~2.3} that $\boldHF^\infty(Y\sqcup -Y,\frs|_Y\sqcup \frs|_Y)=0$.  In the case that $b_2^+(W)>0$, equation~\eqref{eq:Phi-null-homotopic} implies that $\Phi_w=0$ on $\boldHF^\infty(Y,\frs|_Y)$. Hence, Lemma~\ref{lem:0-map-tensor} implies that $\id\otimes \Phi_w^\vee$ vanishes on $\boldHF^\infty(Y\sqcup -Y,\frs|_Y\sqcup \frs|_Y)$. Furthermore $F_{W,\frs}(\xi\otimes -)\otimes \id$ similarly vanishes on $\boldHF^\infty(Y\sqcup -Y,\frs|_Y\sqcup \frs|_Y)$ by \cite{OSTriangles}*{Lemma~8.2} and Lemma~\ref{lem:0-map-tensor} since $b_2^+(W)>0$. 
 
 In both cases, $\cC=(N,\widehat{\Gamma}, \frs, \frs|_Y)$ determines a semi-admissible graph-decorated cut of $X$, so Theorem~\ref{thm:mixed-invariants-from-graphs} implies that
\[
\Phi_{X,\cC}=\sum_{\substack{\frt\in \Spin^c(X)\\ \frt|_W=\frs}} \Phi_{X,\frt}(\xi_1\wedge \cdots \wedge \xi_n).
\]
By definition, $\Phi_{X,\cC}$ is the coefficient of $U^{-1}$ in the expression
\[
\left(\tr\circ (F_{W,\frs}(\xi\otimes -)\otimes\id)\circ \delta^{-1}\circ (\id\otimes \Phi^{\vee}_w)\circ \cotr\right)(1),
\]
which is 
\[
\Lef\left(F_{W,\frs}(\xi\otimes -)\colon \HF^+_{\red}(Y,\frs)\to \HF^+_{\red}(Y,\frs)\right),
\]
by Proposition~\ref{prop:algebraicmappingtorus}. The proof is complete.
 \end{proof}

 \begin{figure}[ht!]
 	\centering
 	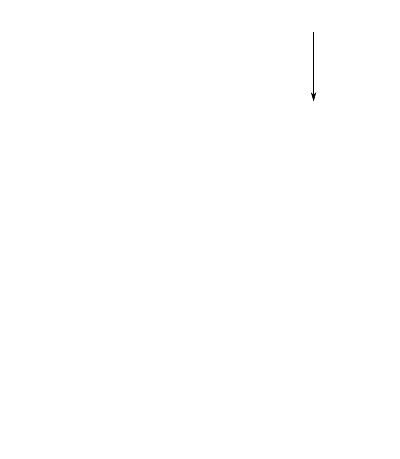
 	\caption{\textbf{Decomposing $(X,\widehat{\Gamma})$ into graph cobordisms.}}
 	\label{fig:59}
 \end{figure}

\subsection{Mixed invariants of \texorpdfstring{$Y\times S^1$}{Y x S1} for non-torsion \texorpdfstring{$\Spin^c$}{Spinc} structures}
\label{sec:YxS1}

We now consider the 4-manifold $ Y\times S^1$, which has more structure than a general mapping torus.  The projection map $\pi:Y\times S^1\to Y$ induces a map
\[
\pi^*:\Spin^c(Y)\to \Spin^c(Y\times S^1).
\] 
Similarly, the inclusion map $i\colon Y\to Y\times \{pt\}$ induces a map
\[
i^*:\Spin^c(Y\times S^1)\to \Spin^c(Y),
\]
Clearly, $i^*\circ \pi^*=\id_{\Spin^c(Y)}.$

 \begin{define}We say that a $\Spin^c$ structure on $Y\times S^1$ is $S^1$-invariant if it is in the image of  $\pi^*:\Spin^c(Y)\to \Spin^c(Y\times S^1)$.  
 \end{define}
 
 We prove the following:
 \begin{customcor}{\ref{cor:invariantsofYxS1}}If $Y^3$ has $b_1(Y)>1$ and $\frs\in \Spin^c(Y)$ is non-torsion, then 
 \[\Phi_{Y\times S^1,\pi^*(\frs)}(1)=\chi(\HF^+(Y,\frs)).\] Furthermore, if $\frt\in \Spin^c(Y\times S^2)$ is not $S^1$-invariant, then
 \[\Phi_{Y\times S^1,\frt}(1)=0.\]
 \end{customcor}
 \begin{proof}The proof is a straightforward combination of Corollary~\ref{cor:mixedinvariantofactualmappingtori} with the adjunction inequality \cite{OSTriangles}*{Theorem~1.5} (compare~\cite{BaldridgeSWCircleActions}*{Lemma~5}).
 
  As in Lemma~\ref{lem:CWhomologyofmappingtori}, write
 \begin{equation}
 H_2(Y\times S^1)\iso H_1(Y)\oplus H_2(Y).\label{eq:homologyYxS1}
 \end{equation}
  The first summand is generated by tori of the form 
 \[F_\gamma:=\gamma\times S^1,\] where $\gamma\in H_1(Y)$. The second summand is generated by $H_2(Y)$ under the inclusion $Y\times \{pt\}\subset Y\times S^1$.  If $[F_\gamma]$ is in the first summand of equation~\eqref{eq:homologyYxS1}, and $[F]$ is in the second summand, then
 \[\pi^*(\frs+\PD[\gamma])=\pi^*(\frs)+\PD[F_\gamma] \qquad \text{and} \qquad i^*(\frt+\PD[F])=i^*(\frt).\]  Hence it is sufficient to show that
 \[
 \Phi_{\pi^*(\frs)+\PD[F]}(1)=0
 \] for any surface $F$ embedded in $Y\times \{pt\}$, which represents a non-zero class in  $H_2(Y)$. If $[F]\neq 0\in H_2(Y)$, we can pick an element $\gamma\in H_1(Y)$ such that $\#(F\cap \gamma)>0$ since $H_2(Y)$ is torsion-free. However, $F_\gamma$ is a genus one surface of self intersection number zero, on which $c_1(\pi^*(\frs)+\PD[F])=c_1(\pi^*(\frs))+2 \PD[F]$ evaluates non-trivially. As this violates the adjunction inequality, we conclude that $\Phi_{\pi^*(\frs)+\PD[F]}(1)=0$.
 \end{proof}

\section{Perturbed coefficients}
\label{sec:perturbed}
In this section, we prove a refinement of Theorem~\ref{thm:mixedinvariantmappingtorus} using twisted coefficients.

\subsection{Perturbed Heegaard Floer complexes}

\label{sec:perturbed-complexes}

In this section, we recall the definition of perturbed Heegaard Floer homology, which is a special case of a more general construction of Heegaard Floer homology with twisted coefficients \cite{OSDisks}*{Section~8}. The construction is due to Ozsv\'{a}th and Szab\'{o} \cite{OSGenusBounds}*{Section~3.1}. See also the work of Jabuka and Mark \cite{JabukaMarkProduct}. We mostly follow the treatment \cite{JuhaszZemkeConcordanceSurgery}, because several details regarding naturality and functoriality were carefully written down there.

We recall that the Novikov ring $\Lambda$ is the ring of formal sums $\sum_{\alpha\in \R} c_\alpha e^{\alpha}$ where $c_{\alpha}\in \bF_2$, and for each $N\in \R$, the set $\{\alpha\in \R: c_{\alpha}\neq 0, \alpha\le N\}$ is finite. The ring $\Lambda$ is a field, since $(1+x)^{-1}=\sum_{n\in \N} x^n$, if $x$ has leading term $e^z$ with $z>0$. If $\omega$ is a closed 2-form on  $Y^3$, then we may endow $\Lambda$ with the structure of an $\bF_2[H_2(Y)]$-module, via the formula
\[
e^{\alpha}\cdot e^{h}=e^{\alpha+\int_h \omega}.
\] 
We write $\Lambda_{\omega}$ for $\Lambda$ with this module structure.

If $M$ is a $\Lambda$-module, and $a,b\in M$, we write
\[
a\doteq b
\]
if there is a $z\in \R$ such that $e^z\cdot a=b$.

We define a chain complex $\CF^-(Y, \ws, \frs; \Lambda_{\omega})$, as follows. If $\cH$ is a weakly admissible Heegaard diagram for $(Y,\ws)$, we let $\CF^-(\cH, \frs; \Lambda_{\omega})$ be the free $\Lambda[U]$-module generated by intersection points $\xs\in \bT_{\a}\cap \bT_{\b}$ satisfying $\frs_{\ws}(\xs)=\frs$. We pick a compressing disk in $U_{\a}$ for each curve in $\as$, and similarly we pick a compressing disk in $U_{\b}$ for each curve in $\bs$. Given a homology class of disks $\phi\in \pi_2(\xs,\ys)$ on $\cH$, we write $\cD(\phi)$ for the domain of $\phi$, which we view as an integral 2-chain on $\Sigma$. The boundary $\d \cD(\phi)$ can be viewed as a 1-chain $A_{\a}+ A_{\b}$, where $A_{\a}$ is a 1-chain in $\as$, and $A_{\b}$ is a 1-chain in $\bs$, such that $\d A_{\a}=-\d A_{\b}=\ys-\xs$. 
We may radially cone $\cD(\phi)$ along the compressing disks to obtain a 2-chain $\tilde{\cD}(\phi)$ in $Y$. The boundary of $\tilde{\cD}(\phi)$ only depends on $\xs$ and $\ys$. We write
\[
\tilde{A}_\omega(\phi):=\int_{\tilde{\cD}(\phi)} \omega.
\]
 We define the differential on $\CF^-(\cH,\frs;\Lambda_{\omega})$ via the formula
\[
\d \xs=\sum_{\substack{\phi\in \pi_2(\xs,\ys) \\ \mu(\phi)=1}} \# (\cM(\phi)/\R)e^{\tilde{A}_{\omega}(\phi)} U^{n_{\ws}(\phi)}\cdot \ys,
\]
extended equivariantly over $\Lambda[U]$.

Clearly, if $\omega$ is the zero 2-form, then there is a chain isomorphism
\[
\CF^-(\cH,\frs;\Lambda_{\omega})\iso\CF^-(\cH,\frs)\otimes_{\bF_2} \Lambda.
\]

It is helpful to also consider a perturbed complex which has been completed over the $U$ variable. We write $\boldCF^-(\cH,\frs;\Lambda_\omega)$ for this complex, which is a module over $\Lambda[[U]]$.

\begin{rem}\label{rem:naturality}
Naturality of the perturbed Heegaard Floer complexes is explored in \cite{JuhaszZemkeConcordanceSurgery}*{Section~6} and is slightly subtle. In general, the perturbed transition maps for changing Heegaard diagrams are only well-defined up to an overall factor of $e^z$ (see \cite{JuhaszZemkeConcordanceSurgery}*{Theorem~3.1} for the statement of naturality, and \cite{JuhaszZemkeConcordanceSurgery}*{Section~6.5} for examples of non-trivial monodromy). Furthermore the transition maps are usually only well-defined in this sense if we restrict to one $\Spin^c$ structure at a time. Of course, if $\omega=0$ then we can consider multiple $\Spin^c$ structures at the same time.
\end{rem}

We now briefly discuss one constituent of the transition maps: the triangle maps. Suppose that $(\Sigma,\as,\bs,\ws)$ is a Heegaard diagram for $Y$, and $\bs'$ is a set of attaching curves obtained from $\bs$ by a sequence of handleslides and isotopies. Furthermore, we assume that $(\Sigma,\as,\bs,\bs',\ws)$ is weakly admissible. We can define a chain complex $\bCF^-(\Sigma,\bs,\bs',\ws;\Lambda_{\omega})$ as follows. It is generated over $\Lambda[[U]]$ by intersection points $\bT_{\b}\cap \bT_{\b'}$. The differential counts holomorphic disks, weighted by the factor $e^{\tilde{A}_{\omega}(\phi)} U^{n_{\ws}(\phi)}$. Here, $\tilde{A}_{\omega}(\phi)$ is the $\omega$-area of the 2-chain obtained by coning both the $\bs$ and $\bs'$ components of the boundary of $\cD(\phi)$ into the $U_{\b}$ handlebody.

A top degree generator $\Theta_{\b,\b'}^{\omega}\in \CF^-(\Sigma,\bs,\bs',\ws;\Lambda_{\omega})$ is specified in \cite{JuhaszZemkeConcordanceSurgery}*{Equation~6.4}, which we call the \emph{$\omega$-canonical generator}. We recall the construction. Let $\Theta_{\b,\b'}\in \CF^-(\Sigma,\bs,\bs',\ws)$ be a top degree generator of the unperturbed complex. Since $H^2(U_{\b})=0$, we pick any smooth 1-form $\eta$ on $U_{\b}$ such that $d \eta=\omega|_{U_{\b}}$. We define a map
\[
\Psi_{\eta}\colon \CF^-(\Sigma,\bs,\bs',\ws)\otimes \Lambda\to \CF^-(\Sigma,\bs,\bs',\ws;\Lambda_{\omega})
\]
via the formula
\[
\theta \mapsto e^{\int_{\g_\theta} \eta} \cdot \theta,
\]
extended equivariantly over $\Lambda[U]$. Here, $\g_{\theta}$ denotes the 1-chain obtained by coning the intersection point $\theta$ into the chosen compressing disks for $\bs$ and $\bs'$. We set 
\begin{equation}
\Theta_{\b,\b'}^\omega:=\Psi_{\eta}(\Theta_{\b,\b'}\otimes 1_\Lambda). \label{eq:def-omega-canonical}
\end{equation}
The class $[\Theta^{\omega}_{\b,\b'}]$ is well-defined up overall multiplication by some $e^z$.

 The transition map $\Psi_{\a;\omega}^{\b\to \b'}$ is defined by counting holomorphic triangles, weighted by their $\omega$-area, with one input equal to $\Theta_{\b,\b'}^{\omega}$. An important aspect is that if $\psi$ is a class of triangles, then the representatives of $\psi$ are weighted by the $\omega$-area of a 2-chain $\tilde{\cD}(\psi)$ in $Y$, obtained by coning $\cD(\psi)\subset \Sigma$. In particular, we cone both the $\bs$ and $\bs'$ boundaries of $\cD(\psi)$ into $U_{\b}$, whereas we cone the $\as$ boundary into $U_{\a}$. (This is opposed to constructing a coned-off 2-chain in the 4-manifold $X_{\a,\b,\b'}$).

A related situation is when $\cH=(\Sigma,\as,\bs,\ws)$ is a diagram for $Y=(S^1\times S^2)^{\# n}$ and $\omega=d \eta$ on all of $Y$. In this case, equation~\eqref{eq:def-omega-canonical} adapts to give a class $\Theta_{\a,\b}^{\omega}\in \CF^-(\Sigma,\as,\bs,\ws;\Lambda_{\omega})$, which we also call the $\omega$-canonical class.

\begin{lem}\label{lem:omega-canonical-class-well-def}
 Suppose $\cH=(\Sigma,\as,\bs,\ws)$ represents $Y=(S^1\times S^2)^{\# n}$ and $\omega$ is a 2-form on $Y$ such that $\omega=d\eta$. Then the $\omega$-canonical class of $\cH$ is preserved by the perturbed transition maps for changing Heegaard diagrams.
\end{lem}
\begin{proof} We focus on the case of a beta-handleslide. Invariance under alpha-handleslides is an easy adaptation, and we leave invariance under isotopies of $\Sigma$ and stabilizations to the reader. Suppose that $\bs'$ is obtained from $\bs$ by a sequence of handleslides and isotopies, and that $(\Sigma,\as,\bs,\bs',\ws)$ is weakly admissible. We make the following claim:
\begin{equation}
\Psi_{\a;\omega}^{\b\to \b'} \circ \Psi_{\eta}\doteq \Psi_{\eta} \circ \Psi_{\a}^{\b\to \b'}.
\label{eq:perturbed-transition-map-canonical-gen}
\end{equation}
Applying equation~\eqref{eq:perturbed-transition-map-canonical-gen} to $\Theta_{\a,\b}$ (the unperturbed canonical class) yields the main claim, since the unperturbed transition maps preserve the top degree generator, by naturality. Unpacking the definition, equation~\eqref{eq:perturbed-transition-map-canonical-gen} follows from the fact that if $\psi\in \pi_2(\theta_1,\theta_2,\theta_3)$ is a class of triangles on $(\Sigma,\as,\bs,\bs',\ws)$, then
\[
\tilde{A}_{\omega}(\psi)=\int_{\tilde{\cD}(\psi)} d\eta =\int_{\gamma_{\theta_3}}\eta-\int_{\gamma_{\theta_2}} \eta -\int_{\gamma_{\theta_1}} \eta, 
\]
by Stokes' theorem, since $\d \tilde{\cD}(\psi)=\gamma_{\theta_3}-\gamma_{\theta_2}-\gamma_{\theta_1}$. 
\end{proof}

\subsection{Perturbed Heegaard Floer complexes and duality}

We now discuss duality and the Heegaard-Floer complexes with perturbed coefficients. A similar account may be found in \cite{JabukaMarkProduct}*{Section~4}. If $\cH=(\Sigma,\as,\bs,\ws)$ is a diagram for $Y$, then $\cH^{\vee}:=(\Sigma,\bs,\as,\ws)$ is a diagram for $-Y$. Furthermore, if $\phi\in \pi_2(\xs,\ys)$ is a class of disks on $\cH$, then there is a corresponding class $\phi^\vee\in \pi_2(\ys,\xs)$ on $\cH^\vee$ which has the same domain as $\phi$. We may use the same compressing disks for both diagrams. Since the orientation of $\Sigma$ is unchanged, 
\[
\tilde{A}_{\omega}(\phi)=\tilde{A}_{\omega}(\phi^{\vee}).
\]
Consequently, the complex $\CF^-(\cH^\vee,\frs;\Lambda_{\omega})$ is chain isomorphic to $\Hom_{\Lambda[U]}(\CF^-(\cH,\frs;\Lambda_{\omega}),\Lambda[U])$.
In particular, there is a trace pairing
\[
\tr_{\omega}\colon \CF^-(\cH,\frs;\Lambda_{\omega})\otimes_{\Lambda[U]} \CF^-(\cH^\vee,\frs;\Lambda_{\omega})\to \Lambda[U],
\]
analogous to the unperturbed case. Dually, there is a perturbed cotrace map $\cotr_{\omega}$ in the opposite direction.

Note that our trace pairing is slightly different than the one in \cite{JabukaMarkProduct}*{Section~4} because we consider a $\Lambda[U]$-equivariant pairing, instead of a $\Lambda$-sesquilinear pairing. The incorporation of the $U$ variable into the pairing is similar to the unperturbed setting. Compare Lemma~\ref{lem:F-pairing-equivalence}. The reason our pairing is not sesquilinear with respect to $\Lambda$ is that Jabuka and Mark use different orientations on $\Sigma$ for $Y$ and $-Y$, whereas we use the same orientation of $\Sigma$ for $Y$ and $-Y$, but reverse the roles of $\as$ and $\bs$.

\subsection{Perturbed cobordism maps}
\label{sec:perturbed-cobordism-maps-overview}
In this section, we describe the perturbed graph cobordism maps, and several refinements involving $\Spin^c$ structures.

If $(W,\Gamma)$ is a ribbon graph cobordism from $(Y_1,\ws_1)$ to $(Y_2,\ws_2)$, then there is a perturbed graph cobordism map
\[
F_{W,\Gamma,\frs;\omega}^B\colon \CF^-(Y_1,\ws_1,\frs_1; \Lambda_{\omega_1})\to \CF^-(Y_2,\ws_2,\frs_2;\Lambda_{\omega_2}),
\]
which is well-defined up to an overall factor of $e^z$, for some $z\in \R$. Ozsv\'{a}th and Szab\'{o} described the construction for the path cobordism maps in \cite{OSGenusBounds}*{Section~3.1}, and a closely related construction is described in \cite{OSTriangles}*{Section~3.1}. A survey proving invariance of the construction may be found in \cite{JuhaszZemkeConcordanceSurgery}*{Section~3}.

We describe the construction of $F_{W,\Gamma,\frs;\omega}^B$ presently, focusing first on the case when $\Gamma$ is a path connecting $Y_1$ to $Y_2$ (which we drop from the notation).
 We decompose $W$ as $W=W_3\circ W_2\circ W_1$, where $W_i$ is obtained by attaching $i$-handles. Then, 
\begin{equation}
F_{W,\frs;\omega}=F_{W_3;\omega|_{W_3}}\circ F_{W_2,\frs|_{W_2};\omega|_{W_2}}\circ F_{W_1;\omega|_{W_1}}.\label{eq:F-W-S-w-def}
\end{equation}
The maps $F_{W_1;\omega|_{W_1}}$ and $F_{W_3;\omega|_{W_3}}$ are the perturbed 1-handle and 3-handle maps of \cite{OSGenusBounds}  (see \cite{JuhaszZemkeConcordanceSurgery}*{Section~7.3} for an exposition with additional details). The map $F_{W_2, \frs|_{W_2};\omega|_{W_2}}$ is a holomorphic triangle map, as follows. We pick a Heegaard triple $(\Sigma,\as,\bs,\bs',\ws)$ which is subordinate to a bouquet for $\bL$, in the sense of \cite{OSTriangles}*{Definition~4.2}. Given a Morse function on $W_2$ with only index 2 critical points, and a gradient-like vector field which induces the framed link $\bL$ in $Y_1$, there is an embedding of the 4-manifold $X_{\a,\b,\b'}$ into $W(Y_1,\bL)=W_2$. The map $F_{W_2,\frs|_{W_2};\omega|_{W_2}}$ counts holomorphic triangles on $(\Sigma,\as,\bs,\bs',\ws)$ which satisfy $\frs_{\ws}(\psi)=\frs|_{W_2}$. Furthermore, a holomorphic triangle is weighted by its $\omega$-area, as follows. A homology class $\psi\in \pi_2(\xs,\ys,\zs)$ determines an immersed surface in $\Sigma\times \Delta$, whose relative homology class is well-defined. The boundary is mapped to $\as\times e_{\a}\cup \bs\times e_{\b}\cup \bs'\times e_{\b'}$. We may cone this surface into $U_{\a}\times e_{\a}\cup U_{\b}\times e_{\b}\cup U_{\b'}\times e_{\b'}$ to obtain a 2-chain $\tilde{\cD}(\psi)$ in $X_{\a,\b,\b'}\subset W(Y_1,\bL)$. We note that the boundary of this 2-chain depends only on $\xs$, $\ys$ and $\zs$. The 2-handle map counts holomorphic triangles weighted by $e^{\int_{\tilde{\cD}(\psi)} \omega} U^{n_{w}(\psi)}$.

We now sketch the additional steps that are necessary to construct the perturbed graph cobordism maps. The main extra step is to prove that the graph action map is well-defined on the perturbed complexes. The free-stabilization maps $S_w^+$ and $S_w^-$ are well-defined by the same argument as the perturbed 1-handle and 3-handle maps; see, e.g. \cite{JuhaszZemkeConcordanceSurgery}*{Section~7.3}. The relative homology actions may be defined analogously to the unperturbed case using equation~\eqref{def:rel-homology-action}, except with an additional weighting of $e^{\tilde{A}_\omega(\phi)}$. The relations ~\ref{rel:R1}--\ref{rel:R9} hold by the same arguments as in the unperturbed case. By combining the arguments from \cite{JuhaszZemkeConcordanceSurgery}*{Section~7} and \cite{ZemGraphTQFT}*{Sections~7 and 10}, one proves that $F_{W,\Gamma,\frs;\omega}^{B}$ is well-defined up to an overall factor of $e^{z}$.

More generally, if $\frS\subset \Spin^c(W)$ is a set, all of whose elements have the same restrictions to $Y_1$ and $Y_2$, then we can define a cobordism map
\begin{equation}
F_{W,\Gamma,\frS ;\omega}^B\colon \boldCF^-(Y_1,\ws_1,\frs_1;\Lambda_{\omega_1})\to \boldCF^-(Y_2,\ws_2,\frs_2;\Lambda_{\omega_2}), \label{eq:refinement-spinc-sets-graph-cob}
\end{equation}
where $\omega_i=\omega|_{Y_i}$. The difference from the case of a single $\Spin^c$ structure is that in the 2-handle portion of equation~\eqref{eq:F-W-S-w-def}, we define $F_{W_2,\frS|_{W_2};\omega|_{W_2}}$ to count all triangles with $\frs_{\ws}(\psi)\in \frS|_{W_2}$.

Note that the condition that each $\frs\in \frS$ restrict to $\frs_1$ and $\frs_2$ is necessary because the perturbed Heegaard Floer groups are only natural if one restricts to one $\Spin^c$ structure at a time, unless $\omega=0$. Since we are working over the ring $\Lambda[[U]]$, we do not require $\frS$ to be a finite set, since \cite{OSTriangles}*{Theorem~3.3} can be used to show that the map is well-defined over the power series ring.

\begin{rem} 
The procedure of assigning $\omega$-area to a triangle class used in the definition of the 2-handle map is slightly different than the procedure to assign $\omega$-area to a triangle class used in the definition of the transition maps. For the 2-handle map above, we construct a 2-chain in the 4-manifold $X_{\a,\b,\b'}$. For the transition map in Section~\ref{sec:perturbed-complexes}, we cone the triangle to obtain a 2-chain in $Y$.
\end{rem}

\begin{lem}\label{lem:well-defined-spinc-sum}
Suppose that $\frS$ is a set of $\Spin^c$ structures on $W$, all of which have the same restriction to $Y_1$, and the same restriction to $Y_2$. Then the map $F_{W,\Gamma,\frS,\omega}^B$ is well-defined up to an overall factor of $e^z$, for $z\in \R$. We can relax the requirement that all elements of $\frS$ have the same restriction to $Y_1$ if $\omega|_{Y_1}=0$, and similarly for $Y_2$.
\end{lem}
\begin{proof}[Proof sketch] The content of the lemma which differs from the case of a single $\Spin^c$ structure is that when we change the auxiliary data used to construct the cobordism map, each $F_{W,\Gamma,\frs;\omega}^B$ changes by the same factor of $e^z$, where $z$ is independent of $\frs$.

Many details are described in \cite{JuhaszZemkeConcordanceSurgery}*{Section~7}. We refer the reader there for additional details, however we will cover one step of the proof, to illustrate associativity of holomorphic quadrilaterals with perturbed coefficients, which we will use later. We focus on answering why we can consider multiple 4-dimensional $\Spin^c$ structures at once, whereas we can often not consider multiple 3-dimensional $\Spin^c$ structures at once (see Remark~\ref{rem:naturality}).

We consider the case that $W$ has a Morse function $f$ with only index 2 critical points. Furthermore, we assume that the graph $\Gamma$ is a collection of arcs which connect $Y_1$ to $Y_2$. There is an induced, framed link $\bL$ in $Y_1$. We pick a Heegaard triple $(\Sigma,\as,\bs,\bs',\ws)$ which is subordinate to a bouquet for $\bL$. We suppose that $\bs''$ is obtained from $\bs'$ by a sequence of handleslides and isotopies, and furthermore, that the quadruple $(\Sigma,\as,\bs,\bs',\bs'',\ws)$ is weakly admissible. The 2-form $\omega|_{Y_{\b,\b'}}$ is a boundary, since 3 and 4-handles are attached to $Y_{\b,\b'}\subset X_{\a,\b,\b'}$ to obtain $W(Y_1,\bL)$. We use equation~\eqref{eq:def-omega-canonical} to give $\omega$-canonical classes $\Theta_{\b,\b'}^{\omega}$ and $\Theta_{\b,\b''}^\omega$. We will show that
\begin{equation}
\Psi_{\a;\omega}^{\b'\to \b''}\circ  F_{\a,\b,\b', \frS;\omega}(-, \Theta_{\b,\b'}^{\omega})\dotsimeq F_{\a,\b,\b'',\frS;\omega}(-,\Theta_{\b,\b''}^{\omega}).
\label{eq:perturbed-associativity-example}
\end{equation}
This computation is a key step toward proving that the perturbed 2-handle map is independent of the choice of Heegaard triple subordinate to $\bL$, and is also unchanged by handleslides amongst the components of $\bL$. Equation~\eqref{eq:perturbed-associativity-example} will follow from a perturbed associativity argument for holomorphic quadrilaterals, with careful attention paid to $\omega$-areas.

We now consider the 4-ended 4-manifold $X_{\a,\b,\b',\b''}$, which we view as being obtained by stacking $X_{\a,\b,\b'}$ and $X_{\a,\b',\b''}$ along $Y_{\a,\b'}$. Since $X_{\a,\b',\b''}$ embeds in $Y_2\times [0,1]$,  we may view $X_{\a,\b,\b',\b''}$ as being embedded in $W(Y_1,\bL)$, via an embedding $I$. We define smooth maps
\[
\Pi_1, \Pi_2\colon X_{\a,\b,\b',\b''}\to X_{\a,\b,\b'},
\]
as follows. Write $I_0$ for the standard inclusion of $X_{\a,\b,\b'}$ into $W(Y_1,\bL)$. We write $X_{\a,\b,\b',\b''}$ as $X_{\a,\b,\b'}\cup X_{\a,\b',\b''}$. On $X_{\a,\b,\b'}$, we set $\Pi_1$ to be $I_0$, except on a small collar neighborhood $Y_{\a,\b'}\times [0,1]$ of $Y_{\a,\b'}$, where the $[0,1]$ component of $\Pi_1$ is modified slightly using a bump function (this will ensure $\Pi_1$ to be smooth).  On $X_{\a,\b',\b''}$, we have $\Pi_1$ map  $(z,w)\in \Sigma\times \Delta$ to $z\in Y_2$ (viewing the Heegaard surface as being embedded in $Y_2$). Similarly, if $(x,t)\in U_{\a}\times e_{\a}$, we set $\Pi_1(x,t)=x\in Y_2$, viewing $U_{\a}$ as being embedded in $Y_2$. Analogously, if $(x,t)\in U_{\b'}\times e_{\b'}$ or $U_{\b''}\times e_{\b''}$, then $\Pi_1(x,t)=x$, where we are viewing $U_{\b'}=U_{\b''}\subset Y_2$.

The map $\Pi_2$ is similar. We view $X_{\a,\b,\b',\b''}$ as the union of $X_{\a,\b,\b''}$ and $X_{\b,\b',\b''}$. On $X_{\a,\b,\b''}$, we define $\Pi_2$ to be $I_0$ (noting that $X_{\a,\b,\b''}=X_{\a,\b,\b'}$) except on a small collar neighborhood of $Y_{\b,\b'}$. Similarly, $\Pi_2$ maps $X_{\b,\b',\b''}$ into $Y_{\b,\b'}$, in a similar manner to how $\Pi_1$ mapped $X_{\a,\b',\b''}$ into $Y_2$. We illustrate the maps $\Pi_1$ and $\Pi_2$ in Figure~\ref{fig:60}. Some care is required in defining the smooth structure at the corner points of $X_{\a,\b,\b',\b''}$ to ensure that $\Pi_i$ are smooth, though we leave these details to the reader.

 \begin{figure}[ht!]
 	\centering
 	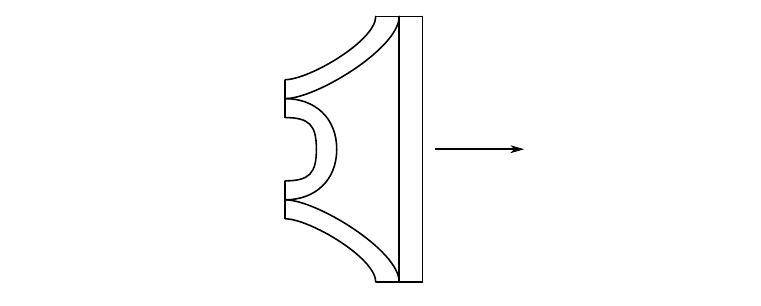
 	\caption{\textbf{A schematic of the maps $\Pi_1$ and $\Pi_2$.}}
 	\label{fig:60}
 \end{figure}

We note that $\Pi_1,$ $I$ and $\Pi_2$ are homotopic as continuous maps, relative to $Y_{\a,\b}$ and $Y_{\a,\b''}$. Let 
\[
h_1\colon X_{\a,\b,\b',\b''}\times [0,1]\to W(Y_1,\bL)
\]
 be a smooth homotopy between $\Pi_1$ and $I$ such that $h_1(z,0)=I(z)$ and $h_1(z,1)=\Pi_1(z)$. Also, let $h_2$ be a smooth homotopy between $\Pi_2$ and $I$.

The chain homotopy appearing in equation~\eqref{eq:perturbed-associativity-example} is defined by counting holomorphic rectangles with perturbed coefficients via the formula
\[
H(-)=e^z\cdot  H_{\a,\b,\b',\b'', \frS;\omega}(-,\Theta_{\b,\b'}^{\omega}, \Theta_{\b',\b''}^{\omega}),
\]
for some $z\in \R$ (to be determined in the proof). Note that the $\omega$-canonical classes in the definition of $H$ are not the same $\omega$-canonical classes used in the definition of the maps in equation~\eqref{eq:perturbed-associativity-example}, since the 2-forms are different. Instead, canonical classes appearing as inputs of $H$ are computed with respect to the embedding $I$ of $X_{\a,\b,\b',\b''}$ into $W(Y_1,\bL)$.

Suppose now that $\psi_{\a,\b,\b'}\in \pi_2(\xs,\theta_{\b,\b'},\ys)$ and $\psi_{\a,\b',\b''}\in \pi_2(\ys,\theta_{\b',\b''}, \zs)$ are homotopy classes of triangles, and consider the rectangle class $\psi_{\a,\b,\b'}*\psi_{\a,\b',\b''}$. Write $S$ for $\tilde{\cD}(\psi_{\a,\b,\b'}* \psi_{\a,\b',\b''})$, where we have performed the coning of $S$ in  $X_{\a,\b,\b',\b''}$. Stokes' theorem implies that
\[
\int_{\d(S\times [0,1])} h_1^*(\omega)=\int_{S\times [0,1]} h_1^* (d \omega)=0,
\]
since $\omega$ is closed. Of course,
\[
\d (S\times [0,1])=\d S\times [0,1]+S\times \{1\}-S\times \{0\}.
\]
We note that 
\begin{equation}
\d S=-\gamma_{\xs}-\gamma_{\theta_{\b,\b'}}-\gamma_{\theta_{\b',\b''}}+\gamma_{\zs}+C_{\a}+C_{\b}+C_{\b'}+C_{\b''}.\label{eq:boundary-coned-class-S}
\end{equation}
In equation~\eqref{eq:boundary-coned-class-S}, $\gamma_{\xs}\subset Y_{\a,\b}$ denotes the 1-chain obtained by extending $\xs$ into $U_{\a}$ and $U_{\b}$ via the radial foliations on the chosen compressing disks. The other $\g$ terms are defined similarly.  Also, $C_{\a}$ denotes the 1-chain $c_{\a}\times e_{\a}$, where $c_\a$ is the sum of the center points of each compressing disk in $U_{\a}$, and $e_{\a}$ is the $\as$-segment of $\d \Box$. The other $C$ terms are defined analogously. Note that the integrals over $-\g_{\xs}\times [0,1]$ and $\g_{\zs}\times [0,1]$ vanish, since $h_1$ is constant on $Y_{\a,\b}$ and $Y_{\a,\b''}$.

We note that $\int_{S\times \{0\}} h_1^*(\omega)$ is the $\omega$-area of the quadrilateral class $\psi_{\a,\b,\b'}*\psi_{\a,\b',\b''}$, as counted by the quadrilateral map, while $\int_{S\times \{1\}} h_1^*(\omega)$ is the sum of the $\omega$-areas of $\psi_{\a,\b,\b'}$ and $\psi_{\a,\b',\b''}$, as would be contributed in the composition 
\[
\Psi_{\a;\omega}^{\b'\to \b''}\circ F_{\a,\b,\b',\frS;\omega}(-,\Theta_{\b,\b'}^{\omega}).
\]

Next, we note that $h^*_1(\omega)|_{Y_{\b,\b'}\times [0,1]}$ is a boundary, since $\omega$ extends over a 1-handlebody which bounds $Y_{\b,\b'}$. For the same reason, $h^*_1(\omega)|_{Y_{\b',\b''}\times [0,1]}$ is also a boundary. Let $\eta_{\b,\b'}$ and $\eta_{\b',\b''}$ be 1-forms on $Y_{\b,\b'}\times [0,1]$ and $Y_{\b',\b''}\times [0,1]$ such that 
\[
d\eta_{\b,\b'}=h^*_1(\omega)|_{Y_{\b,\b'}\times [0,1]}, \quad \text{and} \quad d\eta_{\b',\b''}=h^*_1(\omega)|_{Y_{\b',\b''}\times [0,1]}.
\]
 Hence
 \[
\int_{\g_{\theta_{\b',\b''}}\times [0,1]}h_1^*(\omega)=\int_{\g_{\theta_{\b',\b''}}\times [0,1]} d \eta_{\b',\b''}=\int_{\g_{\b',\b''}\times \{1\}-\g_{\b',\b''}\times \{0\}+\d \g_{\b',\b''}\times [0,1]} \eta_{\b',\b''}. 
 \]
 The integral $\int_{\g_{\b',\b''}\times \{1\}} \eta_{\b',\b''}$ is the quantity used to define the $\omega$-canonical class $\Theta_{\b',\b''}^{\omega}$ in the definition of the transition map $\Psi_{\a;\omega}^{\b'\to \b''}$. The integral $\int_{\g_{\b',\b''}\times \{0\}} \eta_{\b',\b''}$ is the quantity used to define the $\omega$-canonical class $\Theta_{\b',\b''}^{\omega}$ when it appears as the input for the quadrilateral map. A similar pattern holds for $\int_{\gamma_{\theta_{\b,\b'}}\times [0,1]} h_1^*(\omega)$. 
 
 The remaining area terms are
 \begin{equation}
 \int_{(C_{\a}+C_{\b}+C_{\b'}+C_{\b''})\times [0,1]} \omega+ \int_{(\d \g_{\b,\b'})\times [0,1]} \eta_{\b,\b'} +\int_{(\d \g_{\b',\b''})\times [0,1]} \eta_{\b',\b''}. \label{eq:remaining-terms-Pi-1}
 \end{equation}
Importantly, equation~\eqref{eq:remaining-terms-Pi-1} is independent of the choices of triangle classes and intersection points.

One may perform a similar analysis using $\Pi_2$ and $h_2$, and relate the $\omega$-area of a rectangle class which decomposes as a splice of two triangle classes $\psi_{\a,\b,\b''}*\psi_{\b,\b',\b''}$, with the $\omega$-areas of $\psi_{\a,\b,\b''}$ and $\psi_{\b,\b',\b''}$.

 By counting the ends of moduli spaces of quadrilaterals, weighted by these areas, one obtains that there are constants $z_1,z_2\in \R$ such that
\[
\begin{split}
&e^{z_1}\cdot \Psi_{\a;\omega}^{\b\to \b'} \circ F_{\a,\b,\b',\frS;\omega}(-,\Theta_{\b,\b''}^\omega)+e^{z_2}\cdot F_{\a,\b,\b'',\frS;\omega}(-,F_{\b,\b',\b'';\omega}(\Theta_{\b,\b'}^{\omega}, \Theta_{\b',\b''}^{\omega}))\\
=&[\d, H_{\a,\b,\b',\b'',\frS;\omega}(-,\Theta_{\b,\b'}^{\omega},\Theta_{\b',\b''}^{\omega})].
\end{split}
\]
By Lemma~\ref{lem:omega-canonical-class-well-def}, $F_{\b,\b',\b'';\omega}(\Theta_{\b,\b'}^{\omega}, \Theta_{\b',\b''}^{\omega})\doteq\Theta_{\b,\b''}^{\omega}$, so the proof of equation~\eqref{eq:perturbed-associativity-example} is complete.
\end{proof}

For a full proof of Lemma~\ref{lem:well-defined-spinc-sum}, there are additional steps which need to be verified. One must consider more general changes between Heegaard triples which are subordinate to a bouquet of the framed link. Additionally, one must consider changing the Morse function $f$ via cancellations of index 1/2 or index 2/3 pairs of critical points, as well as handleslides amongst the index 2 critical points. The remaining details follow closely Ozsv\'{a}th and Szab\'{o}'s original proof \cite{OSTriangles}, while keeping track of areas. This analysis in the perturbed setting is carried out in \cite{JuhaszZemkeConcordanceSurgery}*{Section~7}, and is a similar flavor to the step described above.

We now state several additional properties. If $\omega$ vanishes on $\d W$, then by a simple adaptation of the argument in \cite{JuhaszZemkeConcordanceSurgery}*{Section~3}, we have
\begin{equation}
F_{W,\Gamma,\frS;\omega}^B\doteq \sum_{\frs\in \frS} e^{\langle (\frs-\frs_0)\cup \omega,[W,\d W]\rangle} \cdot (F_{W,\Gamma, \frs}^B\otimes \id_{\Lambda}),
\label{eq:sum-over-spin^c-structures}
\end{equation}
where $\frs_0$ is any choice of base $\Spin^c$ structure.

Another relation concerns the effect of changing the 2-form by a boundary. If $\omega$ is a closed 2-form on $W$, and $\eta$ is a 1-form which vanishes on $\d W$, then
\begin{equation}
F_{W,\Gamma,\frS;\omega}^B\doteq F_{W,\Gamma,\frS;\omega+d\eta}^B. \label{eq:change-2-form-boundary}
\end{equation}
See \cite{JuhaszZemkeConcordanceSurgery}*{Lemma~3.4} for a proof.

The refinement over sets of $\Spin^c$ structures in equation~\eqref{eq:refinement-spinc-sets-graph-cob} is not entirely sufficient for our purposes, so we define several further refinements. Suppose we have an $n$-fold composition of graph cobordisms $(W,\Gamma)=(W_n,\Gamma_n)\circ \cdots \circ (W_1,\Gamma_1)$, where each $W_i$ is a cobordism from $Y_{i-1}$ to $Y_i$. For notational simplicity, we assume that the graphs are all arcs, which we will omit from the notation. Let $\omega$ be a closed 2-form on $W$.

 Suppose we have a chosen subset
\[
\frU\subset \Spin^c(W_n)\times \cdots \times \Spin^c(W_1).
\]
We say that $\frU$ is \emph{$\omega$-compatible} if whenever $\omega$ is non-zero on a component $Y$ of $\d W_i$, all elements of $\frU$ have the same restriction to $Y$.

If $\frU$ is $\omega$-compatible, then we  define a decomposed cobordism map
\begin{equation}
F_{W_n|\cdots |W_1, \frU;\omega}\colon \boldCF^-(Y_0,\frs_0;\Lambda_{\omega|_{Y_0}})\to \boldCF^-(Y_n,\frs_n;\Lambda_{\omega|_{Y_n}}), \label{eq:decomposed-map-def-perturbed}
\end{equation}
as follows. The definition is essentially the sum over $\frU$ of the composition of the cobordism maps. In more detail, the map is constructed by picking one sequence of Heegaard diagrams and Heegaard triples to compute the cobordism map for each $W_i$. Since we are using coefficients in $\Lambda[[U]]$, a single sequence of Heegaard diagrams and triples can be used to compute the maps for all $\Spin^c$ structures, as long as we use weakly admissible Heegaard diagrams and triples. With respect to this fixed sequence of diagrams, we define
\begin{equation}
F_{W_n|\cdots|W_1,\frU;\omega}:=\sum_{(\fru_n,\dots, \fru_1)\in \frU} F_{W_n,\fru_n;\omega_n}\circ \cdots \circ F_{W_1,\fru_1;\omega_1},
\label{eq:decomposed-perturbed-map}
\end{equation}
where each $F_{W_i,\fru_i;\omega_i}$ is a composition as in equation~\eqref{eq:F-W-S-w-def}, together with change of diagrams maps (which we have omitted). Also $\omega_i=\omega|_{W_i}$.

It is also helpful to define a mixed version of the decomposed map. Suppose that $W_n,\dots, W_1$ is a sequence of cobordisms, as above, and $\frU$ is an $\omega$-compatible subset of $\Spin^c(W_n)\times \cdots \times \Spin^c(W_1)$. Suppose additionally that at least one of the following holds:
\begin{enumerate}[label=(M-\arabic*), ref=M-\arabic*,leftmargin=*, widest=IIII]
\item\label{pert:mixed-map-1} There is a $k\in \{2,\dots, n-1\}$ such that if $(\fru_n,\dots, \fru_1)\in \frU$, then $F_{W_n,\fru_n;\omega_n}\circ \cdots \circ F_{W_k,\fru_k;\omega_k}$ and $F_{W_{k-1}, \fru_{k-1};\omega_{k-1}}\circ \cdots \circ F_{W_{1}, \fru_1;\omega_1}$ both vanish on $\boldHF^\infty$.
\item\label{pert:mixed-map-2} $\boldHF^\infty(Y_0,\frs_0;\Lambda_{\omega|_{Y_0}})=\{0\}$ or $\boldHF^\infty(Y_n, \frs_n;\Lambda_{\omega|_{Y_n}})=\{0\}$.
\end{enumerate}
If either~\eqref{pert:mixed-map-1} or~\eqref{pert:mixed-map-2} holds, then there is a mixed map
\begin{equation}
F^{\mix}_{W_n|\cdots|W_1,\frU;\omega}\colon \boldCF^-(Y_0,\frs_0;\Lambda_{\omega|_{Y_0}})\to \boldCF^+(Y_{n}, \frs_n;\Lambda_{\omega|_{Y_n}}). \label{eq:mixed-perturbed-map}
\end{equation}
The mixed map is defined similarly to equation~\eqref{eq:decomposed-perturbed-map}, but with $\delta^{-1}$ inserted between the maps for $W_k$ and $W_{k-1}$ if \eqref{pert:mixed-map-1} is satisfied, and a $\delta^{-1}$ at the start or end if instead~\eqref{pert:mixed-map-2} is satisfied. Note that the resulting map is independent of the choice of $k$ where $\delta^{-1}$ is inserted, since the following diagram commutes on the nose (not up to an overall factor of $e^z$):
\[
\begin{tikzcd}[row sep=1cm,column sep=.35cm]
\cdots \ar[r]
&\boldHF^+(Y_i,\frs_i;\Lambda_{\omega|_{Y_i}})\ar[r, "\delta"]\ar[d, "F_{W_{i+1}, \fru_{i+1};\omega_{i+1}}"]
&\boldHF^-(Y_i,\frs_i;\Lambda_{\omega|_{Y_i}})\ar[r]\ar[d,"F_{W_{i+1}, \fru_{i+1};\omega_{i+1}}"]
&\boldHF^\infty(Y_i,\frs_i;\Lambda_{\omega|_{Y_i}}) \ar[r] \ar[d, "F_{W_{i+1}, \fru_{i+1};\omega_{i+1}}"]
&\cdots
\\
\cdots \ar[r]&
\boldHF^+(Y_{i+1},\frs_{i+1};\Lambda_{\omega|_{Y_{i+1}}})\ar[r, "\delta"]
&\boldHF^-(Y_{i+1},\frs_{i+1};\Lambda_{\omega|_{Y_{i+1}}})\ar[r]
&\boldHF^\infty(Y_{i+1},\frs_{i+1};\Lambda_{\omega|_{Y_{i+1}}}) \ar[r]
&\cdots
\end{tikzcd}
\]

The same argument as in Lemma~\ref{lem:well-defined-spinc-sum} implies that the decomposed maps $F_{W_n|\cdots|W_1,\frU;\omega}$ and $F_{W_n|\cdots|W_1, \frU;\omega}^{\mix}$ are also well-defined up to an overall factor of $e^z$.  The natural extension works to construct decomposed graph cobordism maps $F_{W_n|\cdots|W_1,\Gamma, \frU;\omega}^B$.

\subsection{Composition laws}
We now state several versions of the composition law. The most basic version of the composition law is the following:

\begin{lem}\label{lem:perturbed-composition-law-1}
 Suppose that 
\[
(W_1,\Gamma_1)\colon (Y,\ws)\to (Y',\ws'),\quad \text{and}\quad (W_2,\Gamma_2)\colon (Y',\ws')\to (Y'',\ws'')
\]
are graph cobordisms.  Write $(W,\Gamma)=(W_2,\Gamma_2)\circ (W_1,\Gamma_1)$. If $\frS_1\subset \Spin^c(W_1)$ and $\frS_2\subset \Spin^c(W_2)$ are sets of $\Spin^c$ structures which all have the same restrictions to $Y$, $Y'$ and $Y''$, write $\frS(W,\frS_1,\frS_2)$ for the set of $\Spin^c$ structures on $W$ which restrict to an element of $\frS_1$ and an element of $\frS_2$. Then
\begin{equation}
F_{W,\Gamma,\frS;\omega}^B\dotsimeq F_{W_2,\Gamma_2,\frS_2;\omega_2}^B\circ F_{W_1,\Gamma_1,\frS_1;\omega_1}^B,
\label{eq:perturbed-composition-law}
\end{equation}
where $\omega_i=\omega|_{W_i}$. If $\omega$ vanishes on one of the 3-manifolds $Y$, $Y'$ or $Y''$, we may relax the requirement that all elements of $\frS_1$ and $\frS_2$ have the same restriction to that 3-manifold. 
\end{lem}

Lemma~\ref{lem:perturbed-composition-law-1} follows by adapting the standard proofs of the composition law \cite{OSTriangles}*{Theorem~3.4} \cite{ZemGraphTQFT}*{Theorem~C} to the perturbed setting by using a perturbed version of associativity, as in our proof-sketch of Lemma~\ref{lem:well-defined-spinc-sum}, above. See also  \cite{JuhaszZemkeConcordanceSurgery}*{Proposition~3.2} for a related discussion.

The decomposed maps in equation~\eqref{eq:decomposed-map-def-perturbed} satisfy the following composition laws:
\begin{lem}\label{lem:perturbed-composition-law-2} Suppose that $(W,\Gamma)$ is decomposed into a composition $(W_n,\Gamma_n)\circ \cdots \circ (W_1,\Gamma_1)$ and $\omega$ is a closed 2-form on $W$.
\begin{enumerate}
\item Suppose $n<k<1$, and
\[
\frU\subset\Spin^c(W_n)\times \cdots \times \Spin^c(W_k)\quad \text{and} \quad \frT\subset \Spin^c(W_{k-1})\times \cdots \times \Spin^c(W_1)
\]
are $\omega$-compatible subsets. Write $\Gamma_{n,\dots, k}$ for $(W_n\circ \cdots\circ  W_k) \cap \Gamma$, and define $\Gamma_{k-1,\dots, 1}$ similarly. Define $\omega_{n,\dots, k}$ and $\omega_{k-1,\dots, 1}$ analogously, by restriction. Then
\[
F_{W_n|\cdots|W_k, \Gamma_{n,\dots, k},\frU;\omega_{n,\dots, k}}^B\circ F_{W_{k-1}|\cdots|W_1, \Gamma_{k-1,\dots, 1},\frT;\omega_{k-1,\dots, 1}}^B\dotsimeq F_{W_n|\cdots |W_1,\Gamma, \frU\times \frT;\omega}^B.
\]
\item If $\frU\subset \Spin^c(W_n)\times \cdots \times \Spin^c(W_1)$ is $\omega$-compatible, then
\[
F_{W_n|\cdots| W_1,\Gamma;\frU;\omega}^B\dotsimeq F_{W,\Gamma, \frS(W, \frU);\omega}^B
\]
where $\frS(W,\frU)$ denotes the set of $\Spin^c$ structures on $W$ whose restriction lies in $\frU$.
\end{enumerate}
\end{lem}

Similar composition laws hold for the mixed maps for decomposed cobordisms.

\subsection{The perturbed trace formula}

In this section, we describe a perturbed version of the trace cobordism formula, Theorem~\ref{thm:dualityv1}. The key step is to describe a perturbed version of the triangle cobordism formula of Theorem~\ref{thm:triplesandgraphcobordismmaps}. Suppose $(\Sigma,\as,\bs,\gs,\ws)$ is a Heegaard triple and $\omega$ is a closed 2-form on $X_{\a,\b,\g}$. As described in Section~\ref{sec:perturbed-cobordism-maps-overview}, if $\frs\in \Spin^c(X_{\a,\b,\g})$, then there is a perturbed holomorphic triangle map
\[
F_{\a,\b,\g,\frs;\omega}\colon \CF^-(\Sigma,\ve{\alpha},\ve{\beta},\frs_{\a,\b};\Lambda_{\omega_{\a,\b}})\otimes_{\Lambda[U]} \CF^-(\Sigma,\ve{\beta},\ve{\gamma},\frs_{\b,\g};\Lambda_{\omega_{\b,\g}})\to \CF^-(\Sigma,\ve{\alpha},\ve{\gamma},\frs_{\a,\g};\Lambda_{\omega_{\a,\g}}),
\]
obtained by counting holomorphic triangles representing $\frs$, weighted by $e^{\tilde{A}_\omega(\psi)} U^{n_{\ws}(\psi)}$. Here, $\frs_{\a,\b}=\frs|_{Y_{\a,\b}}$ and $\omega_{\a,\b}=\omega|_{Y_{\a,\b}}$, and similarly for the $\Spin^c$ structures and 2-forms on the other two ends of $X_{\a,\b,\g}$.

Analogously, we have a perturbed graph cobordism map for $(X_{\a,\b,\g},\Gamma_{\a,\b,\g})$. The following is an analog of Theorem~\ref{thm:triplesandgraphcobordismmaps}:\begin{thm}\label{thm:perturbed-triangle-map}
 Suppose $(\Sigma,\as,\bs,\gs,\ws)$ is a strongly $\frs$-admissible Heegaard triple, and $(X_{\a,\b,\g},\Gamma_{\a,\b,\g})$ is the associated graph cobordism. Let $\omega$ denote a closed 2-form on $X_{\a,\b,\g}$. Then
 \[
F_{X_{\a,\b,\g},\Gamma_{\a,\b,\g},\frs;\omega}^B\dotsimeq F_{\a,\b,\g,\frs;\omega},
 \]
 where $\dotsimeq$ denotes chain homotopic up to an overall factor of $e^z$.
\end{thm}
\begin{proof}[Sketch of proof]
We now sketch the necessary changes needed to adapt the proof of Theorem~\ref{thm:triplesandgraphcobordismmaps} to the present setting of perturbed complexes. The argument follows the same outline as in the unperturbed setting. We first define a perturbed version of the intertwining map for connected sums, and show that it coincides with the perturbed graph cobordism map for connected sums. Then we compose the perturbed intertwining map with the 2-handle portion of the cobordism map for $(X_{\a,\b,\g},\Gamma_{\a,\b,\g})$ and use associativity to identify the composition with the perturbed holomorphic triangle map on $(\Sigma,\as,\bs,\gs,w)$. We now discuss a few more details of the proof, but will leave the remaining bookkeeping to the interested reader.

We first discuss the perturbed version of the intertwining map for connected sums.
  Suppose that $(\Sigma_1,\as_1,\bs_1,w_1)$ and $(\Sigma_2,\as_2,\bs_2,w_2)$ are Heegaard diagrams for $Y_1$ and $Y_2$, respectively, and let $X_1$ denote the 1-handle cobordism from $Y_1\sqcup Y_2$ to $Y_1\# Y_2$. We remark that the 4-manifold $X_1':=X_{\a_1\cup \a_2, \b_1\cup \a_2, \b_1\cup \b_2}$ used to define the intertwining map naturally embeds into the 1-handle cobordism $X_1$. We can see this as follows. Attach $g(\Sigma_2)$ 1-handles to  $Y_1$ and $g(\Sigma_1)$ 1-handles to $Y_2$ to obtain $Y'=Y_1\# (S^1\times S^2)^{\# g(\Sigma_2)} \sqcup Y_2\#(S^1\times S^2)^{\# g(\Sigma_1)}$, the incoming boundary of $X_1'$. A handle decomposition of $X_1'$ is described in the proof of Theorem~\ref{thm:triplesandgraphcobordismmaps}. Namely, it consists of one 1-handle, connecting the 2-components, as well as $g(\Sigma_1)+g(\Sigma_2)$ 2-handles. Using the explicit description therein of the 2-handles given in the proof of Theorem~\ref{thm:triplesandgraphcobordismmaps}, it is easy to see that the 2-handles may be chosen so that each is dual to the co-core of one of the 1-handles attached to $Y_1$ or $Y_2$. In particular, all the 1-handles and 2-handles cancel, except for the single 1-handle which connects the two components. 
  
  Hence, given a closed 2-form on $X_1$, we can naturally restrict $\omega$ to $X_1'$, and define a perturbed intertwining map. The same associativity argument as in Proposition~\ref{prop:OSmapsaregraphcobmaps} implies that the perturbed intertwining map coincides with the perturbed graph cobordism map $G_1^B$.

The next step of the proof of Theorem~\ref{thm:triplesandgraphcobordismmaps} is to consider the composition of the perturbed 1-handle map $G^B_{1;\omega|_{X_1}}$, connecting $(\Sigma,\as,\bs,w)$ and $(\bar{\Sigma}, \bar{\gs}, \bar{\bs},w)$, followed by the perturbed 2-handle map associated to the cobordism $X_{\a,\b,\g}$. The argument in the perturbed setting is not substantially different than the argument in the unperturbed setting: one uses associativity of the perturbed holomorphic polygon maps, as in the proof-sketch of Lemma~\ref{lem:well-defined-spinc-sum} above, in place of ordinary holomorphic polygon maps, following the argument of Theorem~\ref{thm:triplesandgraphcobordismmaps}. We leave the details to the interested reader.
\end{proof}

\begin{cor}
Suppose $(Y,\ws)$ is a multi-pointed 3-manifold, and let $\omega$ be a closed 2-form on $Y$. Let $\pi^*\omega$ denote the pullback of $\omega$ under the projection $Y\times [0,1]\to Y$. The perturbed trace graph cobordism $(Y\times [0,1],\ws\times [0,1];\pi^*\omega): (Y\sqcup -Y,\ws\sqcup \ws)\to \emptyset$ induces the canonical trace map
\[
\tr_{\omega}:\CF^-(Y,\ws,\frs;\Lambda_{\omega})\otimes_{\Lambda[U]} \CF^-(-Y,\ws,\frs;\Lambda_{\omega})\to \Lambda[U].
\]
 Similarly, the perturbed cotrace graph cobordism $(Y\times [0,1],\ws\times [0,1];\pi^*\omega):\emptyset\to (Y\sqcup -Y,\ws\sqcup \ws)$ induces the canonical cotrace map
\[
\cotr_{\omega}: \Lambda[U]\to \CF^-(Y,\ws,\frs;\Lambda_{\omega})\otimes_{\Lambda[U]} \CF^-(-Y,\ws,\frs;\Lambda_{\omega}).
\]
\end{cor}
\begin{proof}
The proof follows by the same argument as the proof of Theorem~\ref{thm:dualityv1}, using Theorem~\ref{thm:perturbed-triangle-map} in place of Theorem~\ref{thm:triplesandgraphcobordismmaps}.
\end{proof}

\subsection{Computing mixed invariants using perturbed coefficients}

In this Section, we discuss using perturbed coefficients to compute the mixed invariants.

\begin{define}
 Suppose that $X^4$ is a closed, oriented 4-manifold. Suppose $\cC=(N,\Gamma,\frs_1,\frs_2)$ is a graph-decorated cut of $X$ (i.e. is a tuple satisfying ~\eqref{def:cut-1}--\eqref{def:cut-3}), and $\omega$ is a closed 2-form on $X$.   We say that $\cC$ is an \emph{$\omega$-admissible cut} if $F_{W_i,\Gamma_i, \frs_i;\omega_i}^{B}$ vanishes on $\boldHF^\infty$ with coefficients in $\Lambda[[U]]$, for $i=1,2$, and at least one of the following holds:
 \begin{enumerate}[label=($\omega$-A-\arabic*), ref=$\omega$-A-\arabic*,leftmargin=*, widest=IIII]
 \item\label{omega-A-1} There is a closed, oriented surface $\Sigma\subset X\setminus N$ of positive square.
 \item\label{omega-A-2} $\omega$ restricts non-trivially to $H^2(N;\R)$.
 \item \label{omega-A-3} $c_1(\frs)$ is non-torsion on $N$. 
 \end{enumerate}
\end{define}

If $\cC=(N,\Gamma,\frs_1,\frs_2)$ is an $\omega$-admissible cut of $X$, then we may define a perturbed mixed invariant
\[
\Phi_{X,\cC;\omega}\in \Lambda,
\]
by adapting equation~\eqref{eq:mixed-invariant-graph} in the obvious way.

\begin{rem} If $c_1(\frs)$ is non-torsion on $N$, or $\omega$ is non-trivial in $H^2(N;\R)$, then $\boldHF^\infty(N,\frs;\Lambda_{\omega})=0$.  See \cite{JabukaMarkProduct}*{Corollary~8.8}. 
\end{rem}

We have the following analog of Theorem~\ref{thm:mixed-invariants-from-graphs}:
\begin{thm}
\label{thm:perturbed-mixed-invariants-graphs} Suppose that $X^4$ is a closed, oriented 4-manifold with $b_2^+(X)>1$, $\omega$ is a closed 2-form, and $\cC=(N,\Gamma,\frs_1,\frs_2)$ is an $\omega$-admissible, graph-decorated cut of $X$. Then
\begin{equation}
\Phi_{X,\cC;\omega}\doteq \sum_{\substack{\frt\in \Spin^c(X)\\ \frt|_{W_1}=\frs_1\\ 
\frt|_{W_2}=\frs_2}} e^{\langle (\frt-\frt_0) \cup \omega, [X] \rangle}\cdot \Phi_{X,\frt}(\xi_1\wedge\cdots \wedge \xi_n),
\label{eq:perturbed-mixed-graph-statement}
\end{equation}
where $\frt_0$ is any chosen $\Spin^c$ structure on $X$, and $\xi_1,\dots, \xi_n\in H_1(X)/\Tors$ are the classes of the loops of the graph $\Gamma$.
\end{thm}
\begin{proof}We focus on~\eqref{omega-A-2}. Cases~\eqref{omega-A-1} and~\eqref{omega-A-3} follow from the same line of reasoning. The proof follows from almost exactly the same line of reasoning as our proof of Theorem~\ref{thm:mixed-invariants-from-graphs}, though care needs to be taken with regards to sums of perturbed cobordism maps, since the cobordism maps are only well-defined up to an overall factor of $e^z$. To handle this, we work with the cobordism maps which are indexed over sets of $\Spin^c$ structures, as described in Section~\ref{sec:perturbed-cobordism-maps-overview}. Additionally, one also needs to use the maps for decomposed 4-manifolds, and their mixed versions, as in equations~\eqref{eq:decomposed-perturbed-map} and~\eqref{eq:mixed-perturbed-map}. The necessary composition laws are in Lemmas~\ref{lem:perturbed-composition-law-1} and ~\ref{lem:perturbed-composition-law-2}.

There is one additional caveat. Care must be taken when cutting along 3-manifolds where multiple 3-dimensional $\Spin^c$ structures are encountered simultaneously, because of naturality issues (c.f. Remark~\ref{rem:naturality}). The only 3-manifold where this occurs is $N''$, which is the boundary of a neighborhood of a surface of positive square. However, the inclusion map $H_2(N'')\to H_2(X)$ is trivial, and hence $H^2(X;\R)\to H^2(N'';\R)$ is trivial. In particular, $\omega$ is null-homologous on $N''$. Given this fact, it is easy to find a 1-form $\eta$ on $X$ such that $\omega+\d \eta$ vanishes on a neighborhood of $N''$ in $X$. By equation~\eqref{eq:change-2-form-boundary}, changing $\omega$ by a boundary only changes $\Phi_{X,\cC;\omega}$ by an overall factor of $e^{z}$. Compare \cite{JuhaszZemkeConcordanceSurgery}*{Lemma~4.1}. In particular,  we may assume the group corresponding to $N''$ is perturbed by the zero 2-form, and hence we can work with multiple $\Spin^c$ structures on $N''$ at the same time.

Paying attention to the above considerations, the proof of the present theorem proceeds nearly identically to Theorem~\ref{thm:perturbed-mixed-invariants-graphs} until equation~\eqref{eq:almost-sum-of-mixed-invariants}. By the same logic, one obtains that
\begin{equation}
\Phi_{X,\cC;\omega}\doteq \left\langle \Theta^-, F^{\mix}_{W_{3}''|W_{1,2}'', \frU'';\omega}(\xi_1\wedge\cdots \wedge \xi_n)\right\rangle.
\label{eq:perturbed-mixed-invariant-almost-done}
\end{equation}
Since $\omega$ vanishes on the boundaries of $W_{3}''$ and $W_{1,2}''$, we may use equation~\eqref{eq:sum-over-spin^c-structures} to decompose each cobordism map in equation~\eqref{eq:perturbed-mixed-invariant-almost-done} as a weighted sum of the ordinary cobordism maps, which is easy to rearrange into equation~\eqref{eq:perturbed-mixed-graph-statement}.
\end{proof}

\subsection{Mapping tori and manifolds with non-separating cuts}

We have the following perturbed version of Theorem~\ref{thm:mixedinvariantmappingtorus}:

\begin{customthm}{\ref{thm:non-separating-cut-mixed-invariant}}
Suppose $X^4$ is a closed, oriented 4-manifold with $b_2^+(X)> 1$ and $Y^3\subset X$ is a closed, oriented, connected and non-separating 3-dimensional submanifold. Write $W$ for the cobordism obtained by cutting $X$ along $Y$. Suppose $\frs\in \Spin^c(W)$ is a $\Spin^c$ structure whose restrictions to both copies of $Y$ in $\d W$ agree, $\omega$ is a closed 2-form on $W$, and $\xi\in \Lambda^*(H_1(W)/\Tors)\otimes \bF_2[U]$. Furthermore, suppose that at least one of the following holds:
\begin{enumerate}
\item $c_1(\frs)$ is non-torsion on $Y$,
\item $[\omega|_Y]\neq 0\in H^2(Y;\R)$, or
\item $b_2^+(W)>0$.
\end{enumerate}
 Then the $\bF_2$ mixed invariants of $X$ satisfy
\[
\begin{split}
&\Lef\big(F_{W,\frs;\omega|_{W}}(\xi\otimes -):\HF^+_{\red}(Y,\frs|_{Y};\Lambda_{\omega|_Y})\to \HF^+_{\red}(Y,\frs|_{Y};\Lambda_{\omega|_Y})\big)\\
\doteq&\sum_{\substack{\frt\in \Spin^c(X)\\ \frt|_W=\frs}}e^{\langle (\frt-\frt_0)\cup \omega, [X]\rangle }\cdot \Phi_{X,\frt}(\xi),
\end{split}
\]
where $\frt_0$ denotes any choice of base $\Spin^c$ structure.
\end{customthm}
\begin{proof} The proof follows from the same line of reasoning as Theorem~\ref{thm:mixedinvariantmappingtorus}, but using Theorem~\ref{thm:perturbed-mixed-invariants-graphs} instead of Theorem~\ref{thm:mixed-invariants-from-graphs}.
\end{proof}

An immediate corollary is the following:

\begin{customcor}{\ref{cor:perturbed-S1xY}} Suppose that $Y^3$ has $b_1(Y)>1$, $\frs\in \Spin^c(Y)$, and $\omega$ is a closed 2-form which induces a non-zero element of $H^2(Y;\R)$. Then
\[
\Phi_{Y\times S^1, \pi^*(\frs)}(1)=\chi(\HF^+_{\red}(Y,\frs;\Lambda_\omega)),
\]
and $\Phi_{Y\times S^1,\frt}(1)=0$ if $\frt$ is not $S^1$-invariant.
\end{customcor}

\bibliographystyle{custom} 
\bibliography{biblio}

\end{document}